\documentclass[11pt,reqno]{amsart}
\usepackage[a4paper,left=2.0cm,right=2.0cm,top=2.5cm,bottom=2.5cm]{geometry}
\usepackage{amsmath}
\usepackage{amssymb}
\usepackage{amsthm,cite}
\usepackage{array}
\usepackage{enumerate}
\usepackage{xcolor}
\usepackage[amsstyle]{cases}
\usepackage{tikz}
\usetikzlibrary{arrows.meta,calc}
\usepackage{microtype}
\usepackage[linkcolor=blue,citecolor=cyan,urlcolor=blue,colorlinks,hypertexnames=false]{hyperref}

\newtheorem{theorem}{Theorem}[section]
\newtheorem{lemma}[theorem]{Lemma}
\newtheorem{corollary}[theorem]{Corollary}
\newtheorem{proposition}[theorem]{Proposition}
\theoremstyle{remark}
\newtheorem{remark}[theorem]{Remark}
\theoremstyle{definition}
\newtheorem{definition}[theorem]{Definition}
\numberwithin{equation}{section}
\allowdisplaybreaks

\begin{document}

\title{A Conditional No-Splash Criterion for Two-Fluid Euler--MHD Interface}

\author{Chengchun Hao}
\address{State Key Laboratory of Mathematical Sciences, Academy of Mathematics and Systems Science, Chinese Academy of Sciences, Beijing 100190, China \newline
	\and
	School of Mathematical Sciences, University of Chinese Academy of Sciences, Beijing 100049, China
}
\email{hcc@amss.ac.cn}

\subjclass[2020]{Primary 35Q35, 35R35; Secondary 76B03, 76W05, 76B47}
\keywords{two-fluid interface, free boundary problem, splash singularity,
magnetohydrodynamics, interface vorticity, chord-arc functional}

\begin{abstract}
We study a two-dimensional incompressible two-fluid system: an inviscid, irrotational Euler
phase and a viscous, conducting MHD phase, separated by a free interface. We prove a conditional
quantitative no-splash criterion for smooth solutions in a specified admissible class. In addition
to a bound on the trace of the Euler velocity, an initial bound on the magnetic trace and bounds on
the bulk vorticity, this class imposes bounds on the interface Biot--Savart potential and its
tangential derivative, fixed exponential vertical localization, and the normal derivative of the
bulk vorticity at the interface together with one tangential derivative of it. It also
requires a Biot--Savart representation of the Euler velocity up to a bounded-gradient harmonic
part and a positive initial chord-arc bound. Under these assumptions the chord-arc functional has
a uniform positive lower bound on $[0,T]$, depending only on the interval length, the decay
width and the a priori bound, after an enlargement that absorbs the surface tension and the
densities. The proof
combines a maximum-principle analysis of the interface-vorticity equation, in which the magnetic
field enters through the Lorentz term and the pressure jump, with a geometric reduction in the style
of Fefferman--Ionescu--Lie.
\end{abstract}

\maketitle

\section{Introduction}

We study the motion of two incompressible fluids that occupy complementary regions of the plane and are separated by a free interface $\Gamma_{t}$. The question we address is whether such an interface can develop a \emph{splash}-type self-intersection in finite time: two distinct material points of the curve meeting at a single point while the curve remains regular everywhere else. Together with the \emph{splat} (a whole arc collapsing onto a segment), the splash is one of the two possible mechanisms of self-intersection (Figure~\ref{fig:splash-splat}); both were introduced by Castro, C\'ordoba, Fefferman, Gancedo and G\'omez-Serrano \cite{CaCoFeGaGoPNAS,CaCoFeGaGo,CaCoFeGaGo2}.

\begin{figure}[!htbp]
\centering
\begin{tikzpicture}[>=Stealth, font=\small, line join=round, scale=1, transform shape]
  \draw[line width=1.6pt]
    (0.4,0.5) .. controls (0.9,1.7) and (1.3,2.7) .. (2.0,3.2)
             .. controls (2.7,2.7) and (3.1,1.7) .. (3.6,0.5);
  \fill (2.0,3.2) circle (2.4pt);
  \node[anchor=west, font=\scriptsize] at (2.14,3.25) {$P$ (splash point)};
  \fill (1.30,2.30) circle (2pt);
  \fill (2.70,2.30) circle (2pt);
  \node[anchor=east, font=\scriptsize] at (1.18,2.30) {$\alpha_{1}$};
  \node[anchor=west, font=\scriptsize] at (2.82,2.30) {$\alpha_{2}$};
  \node[font=\scriptsize, text width=3.7cm, align=center] at (2.0,-0.10)
    {two distinct material points of one smooth curve coincide at $P$};
  \begin{scope}[shift={(5.4,0)}]
    \draw[line width=1.6pt]
      (0.2,0.8) .. controls (1.2,2.7) and (2.6,2.7) .. (3.6,0.8);
    \draw[line width=1.4pt] (0.2,0.8) -- (3.6,0.8);
    \foreach \yy in {2.2,1.6,1.1}{
      \draw[->, black!60, line width=0.7pt] (1.9,\yy) -- (1.9,0.95);
    }
    \fill (0.2,0.8) circle (2pt);
    \fill (3.6,0.8) circle (2pt);
    \node[font=\scriptsize, text width=3.8cm, align=center] at (1.9,-0.10)
      {a whole arc (an interval of parameters) collapses onto the segment};
  \end{scope}
\end{tikzpicture}
\caption{A splash (left) and a splat (right), the two self-intersection mechanisms of
\cite{CaCoFeGaGoPNAS,CaCoFeGaGo,CaCoFeGaGo2}. In a splash two distinct material points
$\alpha_{1}\neq\alpha_{2}$ of an otherwise regular curve meet at a single point $P$; in a splat
an entire arc of the interface collapses onto a line segment.}
\label{fig:splash-splat}
\end{figure}
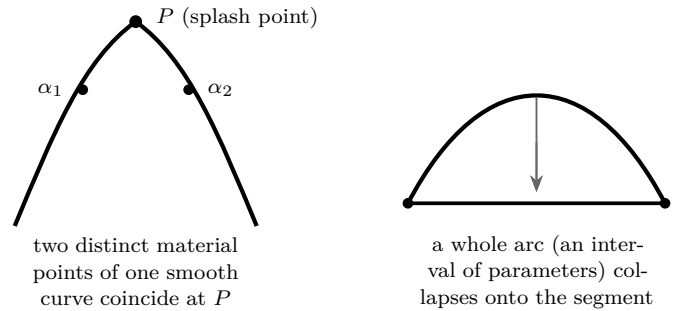

For a single fluid bounded by vacuum the answer is affirmative: \cite{CaCoFeGaGo,CaCoFeGaGo2} construct genuine finite-time splash and splat singularities for the free-surface Euler equations, with viscous analogs in \cite{CoSh2,Cosh3}. The picture changes completely when the vacuum is replaced by a second fluid. Fefferman, Ionescu and Lie \cite{CfAiVl} proved that a two-fluid Euler interface cannot form a splash, and Bulut \cite{AB} extended the result to the two-fluid Euler--Navier--Stokes system, where the viscous fluid carries non-trivial interior vorticity. This paper asks the same question for the magnetohydrodynamic counterpart in which the viscous fluid is electrically conducting: does the Lorentz force, which is quadratic in the magnetic field and has no hydrodynamic analog, create new singular interface behavior?

We consider a two-dimensional two-fluid system. The upper phase $\Omega_{2}(t)$ is a viscous, electrically conducting fluid governed by the incompressible magnetohydrodynamics (MHD) equations, with velocity field $u_{2}(t,x)$ and magnetic field $B(t,x)$, both of constant density $\rho_{2}>0$. The lower phase $\Omega_{1}(t)$ is an inviscid, irrotational Euler fluid of constant density $\rho_{1}>0$ with velocity $u_{1}(t,x)$. The two phases are separated by the free interface $\Gamma_{t}=\{z(t,\alpha):\alpha\in\mathbb R\}$ and move with it; surface tension of coefficient $\sigma\geqslant 0$ is admitted. We write $v_{j}(t,\alpha)=u_{j}(t,z(t,\alpha))$, $j=1,2$, for the traces of the velocities on the interface. The system is introduced formally in Section \ref{sec:model}.

A distinctive feature of this system is that the magnetic field is tangent to the interface. The standard jump condition for a perfectly conducting fluid against a non-conducting one gives $H\cdot n=0$ on $\Gamma_{t}$, where $H(t,\alpha)=B(t,z(t,\alpha))$ is the magnetic field on the interface and $n$ is the interface normal. In two dimensions, the identity $(B\cdot\nabla)B=\tfrac12\nabla|B|^{2}+(\mathrm{curl}\,B)B^{\perp}$, together with the tangency $H|_{\Gamma_{t}}\parallel\partial_\alpha z$, yields
\[
(B\cdot\nabla)B\cdot\partial_\alpha z=\tfrac12\partial_\alpha|H|^{2}\qquad\text{on }\Gamma_t.
\]
Thus, up to the interface parameterization, the Lorentz term is a tangential derivative of the same
form as the pressure and kinetic terms $\tfrac12\partial_\alpha|v|^{2}$, and it is controlled by the
interface trace of $H$ alone (Lemma \ref{MHD:lorentz}). The magnetic field enters the proof here and
again in the magnetic pressure of the stress balance (Proposition \ref{pressure.jump}).

The splash and splat singularities of \cite{CaCoFeGaGo,CaCoFeGaGo2} are structurally stable \cite{CGSstab}, and neither viscosity, surface tension nor a magnetic field prevents finite-time self-intersection in the one-fluid setting \cite{CaCoFeGaGo2,Cosh3,HaoMHDsplash}; the three-dimensional viscous and non-resistive conducting case is treated in \cite{HLZ}, and a splash-squeeze scenario for ideal MHD against vacuum in \cite{CEH}. On the two-fluid side, besides the complex-analytic proof of \cite{CfAiVl} there is an elementary elliptic argument of Coutand and Shkoller \cite{CoShNS}. C\'ordoba, Enciso and Grubic \cite{CoEnGr} show that the one- and two-fluid theories differ: they build stationary two-fluid Euler solutions whose interface is a smooth curve self-intersecting at a single point, so the no-splash theorems are a feature of time evolution rather than of static geometry.

Closely related two-phase models tell a parallel story: splash singularities are absent for the surface quasi-geostrophic sharp-front and Muskat equations \cite{GaSt} (whose global regularity was further studied in \cite{MuskatGlobal}), while the Rayleigh--Taylor instability, interface turnover and breakdown were analyzed in \cite{RTMuskat,Turning} and the local well-posedness of the Hele-Shaw and Muskat interface evolution in \cite{InterfaceEvo}. On the magnetic side, free-boundary incompressible MHD requires a Rayleigh--Taylor sign condition and is known to lose well-posedness without it \cite{HaoMHDill}; yet, as recently shown by Hao and Yang \cite{HaoMHDsplash}, a single conducting fluid against vacuum can nevertheless form a splash. These singularity questions sit against the backdrop of the foundational well-posedness theory for free-boundary incompressible flows \cite{CoShWP,ShatahZeng,Lindblad,AmbroseMasmoudi}.

The main result, Theorem \ref{thm}, is a conditional quantitative no-splash criterion for the
Euler--MHD interface, stated for the solution class in Definition \ref{adm}. Let a solution on
$[0,T]$ satisfy the a priori bounds collected there, on the trace of the Euler velocity and on the
initial magnetic trace, on the bulk vorticity, its Biot--Savart potential and tangential derivative,
its exponential vertical decay, and on the trace of $(\Delta u_{2})\cdot\tau$ together with one
tangential derivative of it, and let its chord-arc
functional be nondegenerate at time $0$,
which is condition \eqref{e10}. Then the functional stays bounded below on $[0,T]$ by a constant
depending only on $T$, on the decay width $R_{0}$ and on the enlarged a priori bound $A$; the
dependence on $\sigma,\rho_{1},\rho_{2}$ is absorbed into $A$ (Remark \ref{rm:normalize}).

Theorem \ref{thm} is an a priori separation criterion of Beale--Kato--Majda type for the
chord-arc functional: on any fixed interval on which the quantities of Definition \ref{adm} stay
within the stated bounds, the functional has a uniform positive lower bound. Corollary
\ref{cor:criterion} records the contrapositive: a drop below that bound forces at least one of the
admissibility requirements to fail.
Remark \ref{rm:example} gives a steady admissible solution with $\omega=0$.

Conversely, the hypotheses of Definition \ref{adm} are quantitative bounds on a solution that
already exists: a smooth solution on $[0,T]$ satisfying the seven conditions of that definition is
$(A,\rho_{1},\rho_{2})$-admissible as soon as $A$ exceeds the sup-norms appearing there. Two of
these conditions are specific to the present setting, the exponential vertical decay of the bulk
vorticity and the boundedness of the harmonic drift $c+w$, with bounded gradient, on the whole
plane. Producing such solutions on a prescribed time interval, with bounds uniform over a class
of data, is a question of existence theory for the free-boundary Euler--MHD system. The constant
$c$ of the theorem is a function of $(T,A,R_{0})$ only, the same for every admissible solution.
Remark \ref{rm:hypset} compares the hypotheses with those of \cite{CfAiVl} and \cite{AB}.

The mechanism isolated by the proof can be stated in one sentence. The magnetic field reaches the
interface equations only through $\partial_{\alpha}|H|^{2}$, once from the Lorentz term and once
from the magnetic pressure, and both occurrences are total tangential derivatives; the tangency
$H\cdot n=0$ is what makes the first one total. The no-splash mechanism is therefore the one of the
non-conducting case, and what the theorem establishes is that the magnetic field creates no new
obstruction. Without the tangency the Lorentz term would contribute
$(\mathrm{curl}\,B)B^{\perp}\cdot\partial_{\alpha}z$ instead, which is not a tangential derivative,
and the source term of the interface-vorticity equation would then contain a quantity that the
trace of $H$ alone does not control.

The present paper is organized as follows (Figure~\ref{fig:flow} records the logical dependence of
the proof).
In Section \ref{sec:model} we introduce the Euler--MHD two-fluid system, its interface conditions and its Lagrangian reformulation, state the notion of an admissible solution (Definition \ref{adm}) and the main theorem (Theorem \ref{thm}). Section \ref{sec:vorticity} is devoted to the interface vorticity: we derive the evolution equation satisfied by $\omega_{\Gamma}$, the bounds \eqref{e19} and \eqref{om.C1} actually used in the proof, the interior gradient bound \eqref{e9.int} for the Euler velocity, the two places where the magnetic field enters the vorticity equation, and the rigidity of the magnetic trace that replaces its a priori bound by a condition at time $0$. Section \ref{sec:geometry} contains the geometric reduction of a candidate splash point (chord-arc function, conjugate pairs and the local description of the interface). Section \ref{sec:main} proves the main theorem, and Appendix~\ref{sec:singular} collects the singular-integral lemmas used in Step 2 of that proof.

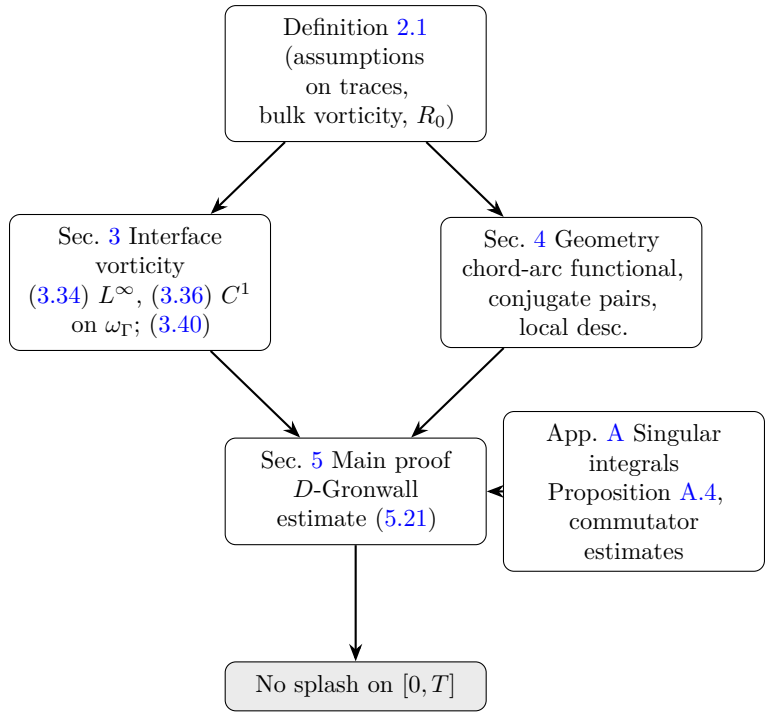
\begin{figure}[htbp]
\centering
\begin{tikzpicture}[>=Stealth, font=\small, line join=round, scale=0.92, transform shape,
    box/.style={draw, rounded corners, align=center, inner sep=5pt, text width=3.4cm}]
  \node[box] (adm) at (4.6,8.4)
    {Definition \ref{adm}\\ (assumptions on traces,\\ bulk vorticity, $R_{0}$)};
  \node[box] (vort) at (1.5,5.4)
    {Sec.~\ref{sec:vorticity} Interface\\ vorticity\\ \eqref{e19} $L^{\infty}$, \eqref{om.C1} $C^{1}$\\ on $\omega_{\Gamma}$; \eqref{e9.int}};
  \node[box] (geom) at (7.7,5.4)
    {Sec.~\ref{sec:geometry} Geometry\\ chord-arc functional,\\ conjugate pairs, local desc.};
  \node[box] (main) at (4.6,2.4)
    {Sec.~\ref{sec:main} Main proof\\ $D$-Gronwall\\ estimate \eqref{D.gronwall}};
  \node[box] (sing) at (8.6,2.4)
    {App.~\ref{sec:singular} Singular\\ integrals\\ Proposition \ref{prop:key},\\ commutator estimates};
  \node[box, fill=black!8] (conc) at (4.6,-0.4)
    {No splash on $[0,T]$};
  \draw[->, thick] (adm) -- (vort);
  \draw[->, thick] (adm) -- (geom);
  \draw[->, thick] (vort) -- (main);
  \draw[->, thick] (geom) -- (main);
  \draw[->, thick] (sing) -- (main);
  \draw[->, thick] (main) -- (conc);
\end{tikzpicture}
\caption{Logical dependence of the proof. The admissibility assumptions
(Definition~\ref{adm}) feed both the interface-vorticity estimates of
Section~\ref{sec:vorticity} and the geometric reduction of Section~\ref{sec:geometry}; their outputs
combine in the $D$-Gronwall estimate \eqref{D.gronwall} of Section~\ref{sec:main}, whose Step~2
uses the singular-integral estimates of Appendix~\ref{sec:singular}.}
\label{fig:flow}
\end{figure}

\section{The two-fluid Euler--MHD problem}\label{sec:model}

\subsection{The equations and the interface conditions}
Throughout, $I$ denotes a time interval. For every $t\in I$, let
$z=z(t,\alpha)=(z_{1}(t,\alpha),z_{2}(t,\alpha))$
parameterize an embedded curve, the interface, which separates $\mathbb{R}^{2}$ into two
open regions $\Omega_{1}(t),\Omega_{2}(t)$.  Writing
$\Gamma_{t}:=\{z(t,\alpha):\alpha\in\mathbb{R}\}$, we have
\[
\mathbb{R}^{2}=\Omega_{1}(t)\cup\Omega_{2}(t)\cup\Gamma_{t}.
\]
We assume that $\Omega_{1}(t)$ and $\Omega_{2}(t)$ are open, connected and simply connected sets, and that $\Gamma_{t}=\partial\Omega_{1}(t)=\partial\Omega_{2}(t)$; the configuration is depicted in Figure \ref{fig:configuration}. Simple connectivity is
automatic in this setting: an embedded curve which is a bounded perturbation of a line and has
the periodicity $z(t,\alpha+2\pi)=z(t,\alpha)+(2\pi,0)$ of Definition \ref{adm} divides the
plane into two components, each homeomorphic to a half-plane.

\begin{figure}[htbp]
	\centering
	\begin{tikzpicture}[>=Stealth, font=\small, line join=round, scale=1, transform shape]
		
		\def\Xmax{9.2}
		\def\Ymax{2.25}
		\def\Amp{0.72}
		\def\prof{\Amp*sin(\x*360/\Xmax)}
		
		\draw[black!35, dashed] plot[domain=-0.95:0, samples=50] (\x,{\prof});
		\draw[black!35, dashed] plot[domain=\Xmax:{\Xmax+0.95}, samples=50] (\x,{\prof});
		
		\fill[black!12] (0,\Ymax) -- plot[domain=0:\Xmax, samples=240] (\x,{\prof}) -- (\Xmax,\Ymax) -- cycle;
		\draw[black!30, dashed] (0,-\Ymax) rectangle (\Xmax,\Ymax);
		
		\draw[line width=1.6pt] plot[domain=0:\Xmax, samples=240] (\x,{\prof});
		
		\draw[<->, black!60] (0,\Ymax+0.26) -- node[above, font=\scriptsize] {$2\pi$ (one period)} (\Xmax,\Ymax+0.26);
		
		\node[anchor=west, align=left] at (0.45,\Ymax-0.55)
		{$\Omega_2(t)$: viscous, electrically conducting (MHD)};
		\node[anchor=west, align=left] at (0.45,-\Ymax+0.50)
		{$\Omega_1(t)$: inviscid, irrotational (Euler)};
		\node[font=\small, inner sep=2pt] at (8.35,-0.62) {$\Gamma_t$};
		
		\foreach \xx in {1.35,3.25,5.30}{
			\pgfmathsetmacro{\yy}{\Amp*sin(\xx*360/\Xmax)}
			\pgfmathsetmacro{\ss}{\Amp*(360/\Xmax)*pi/180*cos(\xx*360/\Xmax)}
			\pgfmathsetmacro{\nn}{sqrt(1+\ss*\ss)}
			\pgfmathsetmacro{\cx}{\xx-0.15*\ss/\nn}
			\pgfmathsetmacro{\cy}{\yy+0.15/\nn}
			\draw[->, black!80, line width=0.6pt] ({\cx-0.38/\nn},{\cy-\ss*0.38/\nn}) -- ({\cx+0.38/\nn},{\cy+\ss*0.38/\nn});
		}
		\node[anchor=south, font=\scriptsize, black!80] at (3.29,0.86) {$H$};
		
		\pgfmathsetmacro{\xp}{6.90}
		\pgfmathsetmacro{\yp}{\Amp*sin(\xp*360/\Xmax)}
		\pgfmathsetmacro{\sl}{\Amp*(360/\Xmax)*pi/180*cos(\xp*360/\Xmax)}
		\pgfmathsetmacro{\tn}{sqrt(1+\sl*\sl)}
		\pgfmathsetmacro{\tx}{1/\tn}
		\pgfmathsetmacro{\ty}{\sl/\tn}
		\pgfmathsetmacro{\nx}{\ty}
		\pgfmathsetmacro{\ny}{-\tx}
		\fill (\xp,\yp) circle (1.8pt);
		\pgfmathsetmacro{\ox}{\xp-0.20*\nx}
		\pgfmathsetmacro{\oy}{\yp-0.20*\ny}
		\draw[black!45, dotted] (\xp,\yp) -- (\ox,\oy);
		\draw[->, thick] (\ox,\oy) -- ++(\tx*0.85,\ty*0.85)
		node[above right, font=\scriptsize, inner sep=2pt] {$\tau$};
		\draw[->, thick] (\xp,\yp) -- ++(\nx*0.85,\ny*0.85)
		node[right, font=\scriptsize, inner sep=2pt] {$\nu$};
		\node[anchor=north west, font=\scriptsize, inner sep=2pt] at (\xp+0.10,\yp-0.16) {$z(t,\alpha)$};
		
	\end{tikzpicture}
	\caption{The two-fluid Euler--MHD configuration, drawn in one
		period $x_{1}\in[0,2\pi]$ of the horizontally periodic window; the dashed lines delimit that
		window and are not physical boundaries, and the dashed arcs are the periodic copies of the
		interface. The arrows marked $H$ are the magnetic field at the
		interface, which is tangent to it. }
	\label{fig:configuration}
\end{figure}
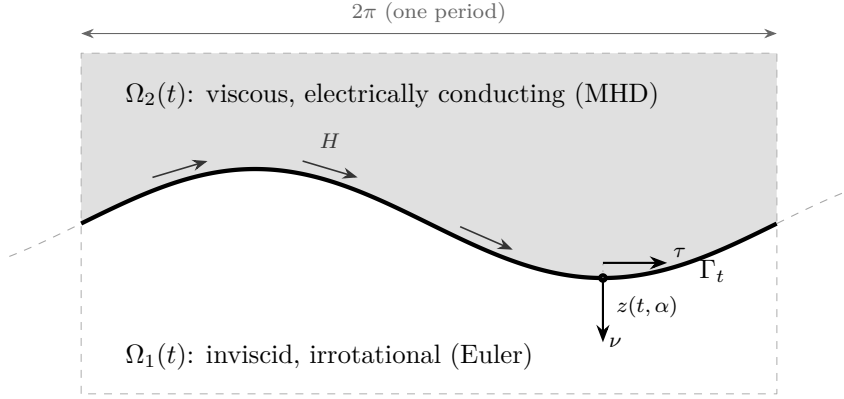

Throughout, the setting is $2\pi$-periodic in the horizontal direction, in the sense made
precise in Definition \ref{adm}; consequently $\Omega_{1}(t)$ and $\Omega_{2}(t)$ are unbounded,
and every integral over $\Omega_{2}(t)$ is understood in the principal-value sense: one integrates over $y_{1}\in\mathbb{R}$ in the principal-value sense,
equivalently one sums the periodized kernel over the periods. The resulting fields are
$2\pi$-periodic in $x_{1}$, as they must be.

The region $\Omega_{1}(t)$ is filled with an Euler fluid of constant density $\rho_{1}>0$
and velocity field $u_{1}(t)\in C^{3}(\overline{\Omega_{1}(t)})$, which is incompressible
in $\Omega_{1}(t)$:
\begin{equation}\label{e1}
\text{div}\ u_{1}=0 \quad \text{in }\Omega_{1}(t),
\end{equation}
and we assume it evolves according to the Euler equation
\begin{equation}\label{e2}
\rho_{1}(\partial_{t}u_{1}+(u_{1}\cdot\nabla)u_{1})=-\nabla p_{1} \quad \text{in }\Omega_{1}(t).
\end{equation}
The flow is moreover irrotational,
\begin{equation}\label{e3}
\text{curl}\ u_{1}=0 \quad \text{in }\Omega_{1}(t).
\end{equation}

The region $\Omega_{2}(t)$ is filled with a magnetohydrodynamic fluid of constant density
$\rho_{2}>0$, velocity field $u_{2}(t)\in C^{3}(\overline{\Omega_{2}(t)})$ and magnetic
field $B\in C^{3}(\overline{\Omega_{2}(t)})$, also incompressible in $\Omega_{2}(t)$:
\begin{equation}\label{e4}
\text{div}\ u_{2}=0 \quad \text{in }\Omega_{2}(t),
\end{equation}
and it evolves according to the MHD equations
\begin{subnumcases}{\label{mhd}}
\rho_{2}(\partial_{t}u_{2}+(u_{2}\cdot\nabla)u_{2})-\Delta u_{2}=-\nabla p_{2}+(B\cdot\nabla)B, \label{mhd.2}\\
\partial_{t}B+(u_{2}\cdot\nabla)B=(B\cdot\nabla)u_{2}, \\
\text{div}\ B=0.
\end{subnumcases}

Two normalizations are in force in \eqref{mhd}: the kinematic viscosity of the conducting fluid and
the magnetic permeability are both equal to one, so that the viscous term reads $-\Delta u_{2}$ and
the Lorentz force reads $(B\cdot\nabla)B$. This is a choice of units and not a restriction. Every
step of the argument uses only that the viscous term is a positive multiple of $\Delta u_{2}$ and
that the Lorentz force is a positive multiple of $(B\cdot\nabla)B$; restoring general coefficients
multiplies the constants of the estimates by powers of those coefficients and replaces the magnetic
pressure $\tfrac12|H|^{2}$ by $\tfrac{\mu}{2}|H|^{2}$ throughout, but changes no estimate in
structure. The constants in Theorem \ref{thm} are therefore to be read for the normalized system
\eqref{mhd}, with the densities $\rho_{1},\rho_{2}$ and the surface tension $\sigma$ entering only
through the enlargement of $A$ described in Remark \ref{rm:normalize}.

We also introduce the Lagrangian coordinates $x=x(t,y)$, obtained by solving
\begin{align}\label{eq.Lag}
	\frac{dx}{dt}=\begin{cases}
		u_{1}(t,x(t,y)), &y \in \Omega_1(0),\\
		u_{2}(t,x(t,y)), &y \in \Omega_2(0),
	\end{cases}\qquad x(0,y)=y,
\end{align}
so that for each $t$ the map $y\mapsto x(t,y)$ is a diffeomorphism of $\Omega_{j}(0)$ onto
$\Omega_{j}(t)$, $j=1,2$.

We write, for the traces on the interface, $v_{1}(t,\alpha)=u_{1}(t,z(t,\alpha))$, $v_{2}(t,\alpha)=u_{2}(t,z(t,\alpha))$, $H(t,\alpha)=B(t,z(t,\alpha))$, $q_{1}(t,\alpha)=p_{1}(t,z(t,\alpha))$, $q_{2}(t,\alpha)=p_{2}(t,z(t,\alpha))$ and 
\[
n(t,\alpha)=-[(\partial_{\alpha}z)(t,\alpha)]^{\bot}=(\partial_{\alpha}z_{2},-\partial_{\alpha}z_{1}). 
\]

Since $\Gamma_{t}$ is a material surface, the following boundary conditions hold:
\begin{subnumcases}{\label{e6}}
(\partial_{t}z-v_{1})\cdot n=0, \\
\partial_{t}z-v_{2}=0, \\
\left(\big(q_{2}-q_{1}+\tfrac{1}{2}|H|^{2}\big)I_{2\times 2}-2D(u_2)-H\otimes H\right)n=\sigma K n, \\
H\cdot n=0, \label{e6.m}
\end{subnumcases}
where all quantities are evaluated at $(t,\alpha)$, except $u_{2}$ and its derivatives, which are
evaluated at $(t,z(t,\alpha))$. Here $\sigma\geqslant 0$ is the coefficient of surface tension and 
\[
K=K_{z(t,\alpha)}(\alpha)=\frac{(\partial_{\alpha}z_{1})(\partial_{\alpha}^{2}z_{2})-(\partial_{\alpha}z_{2})(\partial_{\alpha}^{2}z_{1})}{[(\partial_{\alpha}z_{1})^{2}+(\partial_{\alpha}z_{2})^{2}]^{\frac{3}{2}}}|_{(t,\alpha)}
\] 
denotes the curvature of the interface described by $\alpha\to z(t,\alpha)$ at the point $z(t,\alpha)$.

The third condition of \eqref{e6} is the balance of stresses across $\Gamma_{t}$. The
parameterization is fixed so that $\Omega_{2}(t)$ lies to the left of
$\alpha\mapsto z(t,\alpha)$; then $n=-[\partial_{\alpha}z]^{\bot}$ is the (unnormalized)
normal pointing out of $\Omega_{2}(t)$ into $\Omega_{1}(t)$, and $K$ is the curvature with
respect to this orientation, so that $K>0$ where $\Omega_{2}(t)$ is convex. The Cauchy
stress of the conducting fluid is
\[
T_{2}=-\big(q_{2}+\tfrac12|H|^{2}\big)I+2D(u_{2})+H\otimes H,
\] 
where $D(u):=\tfrac12(\nabla u+\nabla u^{\mathrm T})$
is the deformation-rate tensor and the last two terms
are the viscous and the Maxwell stresses, the latter written for the Lorentz force in the
form $(B\cdot\nabla)B$; the inviscid phase contributes the pure pressure stress
$T_{1}=-q_{1}I$. The jump of the tractions across $\Gamma_{t}$ balances the surface
tension, $(T_{1}-T_{2})n=\sigma Kn$, which is the third condition. By \eqref{e6.m} the
field $H$ is tangent to the interface, so $H\otimes H$ annihilates $n$ and the magnetic
field enters the normal balance only through the magnetic pressure $\tfrac12|H|^{2}$.
Also, $q_{2}$ denotes the {\em fluid} pressure alone, and accordingly in the trace equation \eqref{e15} of Lemma \ref{MHD} the Lorentz term
appears as the total derivative $\tfrac12\partial_{\alpha}|H|^{2}$ rather than being absorbed into $\partial_{\alpha}q_{2}$. The tangential
component of the balance reads $\tau\cdot D(u_{2})n=0$, with
$\tau:=\partial_{\alpha}z/|\partial_{\alpha}z|$ the unit tangent: the standard vanishing of
tangential stress for constant surface tension, a condition on the conducting fluid alone,
since the inviscid phase exerts no tangential traction. In the static case
$u_{2}\equiv0$, $H\equiv0$ this reduces to the Young--Laplace law $q_{2}-q_{1}=\sigma K$.
Only the normal component of the third condition is used in the proof of Theorem \ref{thm}:
Proposition \ref{pressure.jump} uses it to express the pressure jump in closed form, and both
Lemma \ref{nu.omega} and the interface-vorticity equation of Lemma \ref{curl} use that
expression, so that the magnetic pressure and the viscous normal stress enter the estimates on
$\omega_{\Gamma}$ of the interface-vorticity section. The tangential component $\tau\cdot
D(u_{2})n=0$ is not used in the proof. It does not prescribe $\omega$ on $\Gamma_{t}$: the
interface vorticity is an unknown, and the parabolic argument bounds $u=\omega_{\Gamma}/a$ by the
maximum principle applied to the linear equation \eqref{u.eq.lin} on the whole line, with
$\partial_{\nu}\omega$ entering only through the source term. The tangential balance is related to
that normal derivative: by Lemma \ref{nu.omega} the identity
$\partial_{\nu}\omega=(\Delta u_{2})\cdot\tau$ holds on $\Gamma_{t}$, and \eqref{nu.omega.eq}
rewrites this boundary quantity in terms of the traces, of $\partial_{t}(v_{2}\cdot\tau)$, of the
curvature and of the pressure data. That identity is a consistency relation, not an estimate: it
does not by itself bound $\partial_{\nu}\omega$ in terms of the other quantities listed in
Definition \ref{adm}, and the bound that the proof requires is therefore imposed as a hypothesis,
in \eqref{e11.nu}.

The chord-arc function of $z:I\times\mathbb{R}\to\mathbb{R}^{2}$ is
\begin{equation}\label{e7}
F(z)(t,\alpha,\beta):=
\begin{cases}
\frac{|z(t,\alpha)-z(t,\beta)|}{|\alpha-\beta|}\qquad &\text{ if }\alpha\neq\beta,\\
|z_\alpha(t,\alpha)|\qquad &\text{ if }\alpha=\beta.
\end{cases}
\end{equation}

The following symbols keep the same meaning throughout the paper:
$z(t,\alpha)$ is the parameterization of the interface $\Gamma_{t}$;
$a:=|\partial_{\alpha}z|$ is the arclength element, $\tau:=\partial_{\alpha}z/a$ the unit
tangent and $\nu:=n/|n|$ the unit normal; $K$ is the curvature, with the sign convention
$\partial_{\tau}\tau=-K\nu$ and $\partial_{\tau}\nu=K\tau$; $v_{1}$ and $v_{2}$ are the traces
of the two velocities on $\Gamma_{t}$ and $H$ is the trace of the magnetic field;
$\omega:=\mathrm{curl}\,u_{2}$ is the bulk vorticity of the conducting fluid, $\Phi_{2}$ its
Biot--Savart potential, and $\omega_{\Gamma}:=(v_{1}-v_{2})\cdot\partial_{\alpha}z$ is the
interface vorticity; $u:=\omega_{\Gamma}/a$; $c(t)$ and $w(t,\cdot)$ are the constant and the
gradient-harmonic remainder occurring in item (vi) below; and
$\partial_{s}:=a^{-1}\partial_{\alpha}$ denotes differentiation with respect to arclength.
Primes always denote $\partial_{\alpha}$. A few letters carry a second meaning inside individual
proofs, always introduced where it happens and never carried over; they are listed here so that the
global meanings above can be restored at any point. $G$ is the prescribed function of Proposition
\ref{prop:key}, realized in Section \ref{sec:main} by the quantity \eqref{FG.def}. $e$ is the
forcing of \eqref{u.e} in Lemma \ref{curl}, and the unit vector along the chord in Proposition
\ref{chord-arc prop} and in Section \ref{sec:main}. $d$ is a coefficient of \eqref{u.coeff} in Lemma
\ref{curl}, the chord length in Lemma \ref{chord} and Proposition \ref{chord-arc prop}, and the
parameter of the kernel $h_{d}$ in Step 3 of the proof of Theorem \ref{thm}. $c$ is a coefficient of
\eqref{u.coeff} in Lemma \ref{curl}, the drift $c(t)$ of item (vi), and the constant of \eqref{e13}.
$a$ and $b$ are the arclength element and the tangential velocity $v_{1}\cdot\tau$ of
\eqref{curl.not} in Section \ref{sec:vorticity}, and the amplitude $f_{n}(\mu)-f_{m}(\rho)$ and the
parameter difference $\mu-\rho$ in the proof of \eqref{I.C1}. Finally $\varsigma$, $\varrho$, $\Xi$,
$\widetilde\Xi$ are a signed distance, a radius and two covering maps in the proof of Lemma
\ref{chord}.

\subsection{Admissible solutions}
Two logically different kinds of assumption enter the definition below. The first is
\emph{qualitative}: that a smooth solution of the two-fluid
system exists on the whole interval $[0,T]$, with an interface that remains a regular curve and does
not self-intersect. The second is
\emph{quantitative}: that a finite list of norms of that solution is bounded by a single constant
$A$, and that the bulk vorticity decays at a fixed rate $R_{0}$. The theorem is stated for solutions
satisfying both, and its conclusion is a lower bound that depends on $T$, $A$ and $R_{0}$ but not
otherwise on the solution.

\begin{definition}\label{adm}
Fix $T>0$, $A\geqslant 1$ and $\rho_{1},\rho_{2}>0$. A smooth solution
$(z,u_{1},u_{2},B,p_{1},p_{2})$ of the two-fluid Euler--MHD system of Section \ref{sec:model} on
$[0,T]$ is \emph{$(A,\rho_{1},\rho_{2})$-admissible} if the flow is $2\pi$-periodic in the $e_{1}$
direction, that is, all fluid quantities are $2\pi$-periodic in $x_{1}$ and the interface is
parameterized so that $z(t,\alpha+2\pi)=z(t,\alpha)+(2\pi,0)$, and if, with
$v_{1}(t,\alpha)=u_{1}(t,z(t,\alpha))$,
$v_{2}(t,\alpha)=u_{2}(t,z(t,\alpha))$ and $H(t,\alpha)=B(t,z(t,\alpha))$, the following
conditions hold.
\begin{enumerate}
\item[(i)] (regular parameterization of the interface) The map $t\mapsto z(t,\cdot)$ is
continuous from $[0,T]$ into $C^{4}(\mathbb{R};\mathbb{R}^{2})$, the $C^{4}$-norm being taken in
the variable $\alpha$ at fixed $t$, $\partial_{t}z$ exists and is continuous on
$[0,T]\times\mathbb{R}$, and
\begin{equation}\label{e8}
\sup_{t\in[0,T]}\big\|z(t,\cdot)-(\,\cdot\,,0)\big\|_{C^{4}(\mathbb{R})}\leqslant  A
\qquad\text{and}\qquad
\inf_{\alpha\in\mathbb{R},\,t\in[0,T]}|z_{\alpha}(t,\alpha)|\geqslant \frac{1}{A},
\end{equation}

\item[(ii)] (chord-arc condition at time $0$)
\begin{equation}\label{e10}
\inf_{\alpha,\beta\in\mathbb{R}}|F(z)(0,\alpha,\beta)|\geqslant \frac{1}{A}.
\end{equation}

\item[(iii)] (regularity of the interface traces) $v_{1}$ and $H$ are continuous in $t$ with
values in $C^{3}$, respectively $C^{2}$, in $\alpha$, and
\begin{equation}\label{e9}
\sup_{t\in[0,T]}\Big(\|v_{1}(t,\cdot)\|_{C^{3}(\mathbb{R})}
+\|\partial_{t}v_{1}(t,\cdot)\|_{C^{1}(\mathbb{R})}\Big)\leqslant  A,
\qquad
\|H(0,\cdot)\|_{C^{2}(\mathbb{R})}\leqslant  A,
\end{equation}
all norms being taken in $\alpha$ at fixed $t$. 

\item[(iv)] (a priori bounds for the conducting fluid) The bulk vorticity
$\omega:=\mathrm{curl}\,u_{2}$ is bounded,
\begin{equation}\label{e11}
\|\omega\|_{L^{\infty}(\{(t,x):\,t\in[0,T],\,x\in\Omega_{2}(t)\})}\leqslant  A;
\end{equation}
the Biot--Savart potential of the bulk vorticity,
\[
\Phi_{2}(t,x):=\frac{1}{2\pi}\int_{\Omega_{2}(t)}\frac{(x-y)^{\bot}}{|x-y|^{2}}\,\omega(t,y)\,dy,
\]
and its tangential derivative along the interface are bounded,
\begin{equation}\label{e11.bis}
\sup_{t\in[0,T],\,\alpha\in\mathbb{R}}\Big(\big|\Phi_{2}(t,z(t,\alpha))\big|
+\big|\partial_{\alpha}\big[\Phi_{2}(t,z(t,\alpha))\big]\big|\Big)\leqslant  A;
\end{equation}
and $(\Delta u_{2})\cdot\tau$, together with one tangential derivative, is bounded,
\begin{equation}\label{e11.nu}
\sup_{t\in[0,T],\,\alpha\in\mathbb{R}}\Big(
\big|(\Delta u_{2})\cdot\tau\big|
+\big|\partial_{\alpha}\big[(\Delta u_{2})\cdot\tau\big]\big|
\Big)\leqslant A,
\end{equation}
where $\tau=\partial_{\alpha}z/|\partial_{\alpha}z|$ and $u_{2}$ is evaluated at
$(t,z(t,\alpha))$.

\item[(v)] (exponential localization of the bulk vorticity) There is $R_{0}>0$ such that
\begin{equation}\label{e11.loc}
|\omega(t,x)|\leqslant A\,e^{-|x_{2}|/R_{0}}
\qquad\text{for every }t\in[0,T]\text{ and every }x\in\Omega_{2}(t).
\end{equation}

\item[(vi)] (structure of the Euler velocity) There are $c(t)\in\mathbb{R}^{2}$ and a
gradient-harmonic field $w(t,\cdot)=\nabla\psi(t,\cdot)$, with $\Delta\psi=0$ in $\Omega_{1}(t)$,
such that
\begin{equation}\label{e9.bs}
\begin{aligned}
u_{1}(t,x)&=c(t)+w(t,x)+\frac{1}{2\pi}\,\mathrm{p.v.}\!\int_{\mathbb{R}}\frac{(x-z(t,\beta))^{\bot}}{|x-z(t,\beta)|^{2}}\,\omega_{\Gamma}(t,\beta)\,d\beta\\
&\quad+\frac{1}{2\pi}\int_{\Omega_{2}(t)}\frac{(x-y)^{\bot}}{|x-y|^{2}}\,\omega(t,y)\,dy,
\end{aligned}
\end{equation}
for every $t\in[0,T]$ and $x\in\Omega_{1}(t)$, where $\omega_{\Gamma}$ is the interface vorticity
of \eqref{e16}; moreover $w(t,\cdot)$ has a $C^{1}$ extension to $\mathbb{R}^{2}$, still written
$w(t,\cdot)$, such that
\begin{equation}\label{e9.norm}
\sup_{t\in[0,T]}|c(t)|\leqslant A,
\qquad
\sup_{t\in[0,T]}\Big(\|w(t)\|_{L^{\infty}(\mathbb{R}^{2})}
+\|\nabla w(t)\|_{L^{\infty}(\mathbb{R}^{2})}\Big)\leqslant A.
\end{equation}

\item[(vii)] (initial regularity of the interface vorticity) Writing
$a=|\partial_{\alpha}z|$ and $u=\omega_{\Gamma}/a$, the initial values satisfy
\begin{equation}\label{e8.init}
\|u(0,\cdot)\|_{L^{\infty}(\mathbb{R})}
+\|\partial_{\alpha}u(0,\cdot)\|_{L^{\infty}(\mathbb{R})}\leqslant A.
\end{equation}
\end{enumerate}
\end{definition}

\begin{remark}\label{rm:nu.flux}
The bound \eqref{e11.nu} is independent of the other items of Definition \ref{adm} and of the
tangential stress condition, which supplies only the consistency identity of Lemma \ref{nu.omega};
its relation to the joint space--time bound (ST) and to an interior $C^{2}$ bound on $v_{2}$ is
discussed in Remark \ref{rm:hypset}. By Lemma \ref{nu.omega}, the quantity in \eqref{e11.nu} equals
$\rho_{2}\partial_{t}(v_{2}\cdot\tau)$ plus terms involving only the traces $v_{2}$, $q_{1}$, $K$,
$|H|^{2}$ and their tangential derivatives, namely $\partial_{\tau}q_{1}$,
$\sigma\partial_{\tau}K$, $\partial_{\tau}|H|^{2}$, $\partial_{\tau}^{2}(v_{2}\cdot\tau)$,
$\partial_{\tau}[K(v_{2}\cdot\nu)]$ and $v_{2}\cdot\partial_{t}\tau$; for the last,
$\partial_{t}\tau=a^{-1}\big(\partial_{\alpha}v_{2}-\tau\,(\tau\cdot\partial_{\alpha}v_{2})\big)$,
so that $v_{2}\cdot\partial_{t}\tau$ is a tangential derivative of traces of $v_{2}$ as well. Thus
\eqref{e11.nu} controls the time derivative of the tangential velocity of the conducting fluid, for
which no other item provides a bound: item (iii) bounds $\partial_{t}v_{1}$, and the kinematic
conditions \eqref{e6} relate the two traces only through instantaneous quantities.
\end{remark}

\begin{remark}\label{rm:loc}
The absolute value in \eqref{e11.loc} is written because no convention fixes which side of the
interface the conducting fluid occupies (Figure \ref{fig:configuration} depicts one of the two
cases). By \eqref{e8} the curve lies in the strip $|x_{2}|\leqslant A$, and since it is periodic and
embedded it separates the plane into two components, one containing $\{x_{2}>A\}$ and the other
$\{x_{2}<-A\}$; hence $\Omega_{2}(t)$, being one of them, meets at most one of the two half-strips.
Where $|x_{2}|\leqslant A$ the weights $e^{-|x_{2}|/R_{0}}$ and $e^{\mp x_{2}/R_{0}}$ differ by at
most $e^{A/R_{0}}$, and on the half-strip that $\Omega_{2}(t)$ does meet, $e^{-|x_{2}|/R_{0}}$
coincides with one of $e^{\pm x_{2}/R_{0}}$; replacing $|x_{2}|$ by $x_{2}$ or by $-x_{2}$
therefore changes \eqref{e11.loc} only by a constant factor, absorbed into $A$ by Remark
\ref{rm:normalize}.
\end{remark}

\begin{remark}\label{rm:bs}
Item (vi) records the standard representation of a divergence-free velocity by the vorticity it
induces, the interface vortex sheet and the bulk vorticity of the conducting fluid, plus a
harmonic-gradient remainder; it is not an additional restriction on $u_{1}$, only the choice of the
decomposition. The second bound of \eqref{e9.norm} is imposed on an extension, because
$w(t,\cdot)$ is defined only on $\Omega_{1}(t)$; the quantifier is existential, so that admissibility
depends on the existence of an extension and not on the choice of one.
Such an extension exists as soon as $w(t,\cdot)$ is $C^{1}$ up to $\Gamma_{t}$ and bounded with
bounded gradient on $\Omega_{1}(t)$: Whitney's extension theorem applied to the closed set
$\overline{\Omega_{1}(t)}$ produces a $C^{1}$ extension whose norm is bounded by a constant
multiple of the $C^{1}$ norm on $\overline{\Omega_{1}(t)}$, so that $A$ is replaced by $CA$; the
enlargement is absorbed by Remark \ref{rm:normalize}. The estimates below use \eqref{e9.norm} along
segments that may cross $\Gamma_{t}$ (see the discussion following \eqref{V.decomp}); the bounds are
therefore imposed on the whole plane.
\end{remark}

\begin{remark}[The hypothesis set in relation to \cite{CfAiVl} and \cite{AB}]\label{rm:hypset}
	Write {\em (ST)} for the joint space-time bound
	$\|z-(\alpha,0)\|_{C^{4}([0,T]\times\mathbb{R})}\leqslant A$ for the parameterization. In
	\cite{CfAiVl} hypothesis (1.7) is (ST) together with
	$\inf|z_{\alpha}|\geqslant A^{-1}$, and in \cite{AB} hypotheses (7) and (9) are the same two
	conditions, while (8) is the bound
	$\|v\|_{C^{2}(\{(t,x)\in[0,T]\times\mathbb{R}^{2}:\,x\in\Omega_{NS}(t)\})}\leqslant A$ on the
	velocity of the viscous phase, taken \emph{jointly} in $(t,x)$ and not merely on each time slice.
	Since both papers have no magnetic field, the comparison is made
	in the non-conducting limit $B\equiv0$, in which the present setting reduces to the two-fluid
	Euler--Navier--Stokes setting of \cite{AB}.
	
	\begin{center}
		\small
		\begin{tabular}{>{\raggedright\arraybackslash}p{3.3cm}>{\raggedright\arraybackslash}p{2.9cm}>{\raggedright\arraybackslash}p{3.2cm}>{\raggedright\arraybackslash}p{4.3cm}}
			\hline
			requirement & \cite{CfAiVl} & \cite{AB} & Theorem \ref{thm} (here) \\
			\hline
			regularity of $z$ & (ST), joint in $(t,\alpha)$ & (ST), joint in $(t,\alpha)$ & \eqref{e8}: $C^{4}$ in $\alpha$ at fixed $t$; \emph{no} derivative in $t$ \\
			$\inf|z_{\alpha}|\geqslant A^{-1}$ & (1.7) & (9) & \eqref{e8}: same \\
			chord-arc at $t=0$ & (1.10) & Def.~1.1 & \eqref{e10}: same \\
			trace of the Euler velocity & (1.8), joint $C^{3}$ in $(t,\alpha)$ & derived from (ST) & \eqref{e9}: $C^{3}$ in $\alpha$ and $\partial_{t}v_{1}\in C^{1}$ in $\alpha$; no joint bound \\
			magnetic trace at $t=0$ & --- & --- & \eqref{e9}: $\|H(0)\|_{C^{2}}\leqslant A$; on $[0,T]$ by Lemma \ref{B.trace} \\
			bound at time $0$ & (1.11), $L^{\infty}$ of $u_{2}$ & $L^{\infty}$ at $t=0$ & \eqref{e8.init}: $L^{\infty}$ of $u(0)$ \emph{and} of $\partial_{\alpha}u(0)$ \\
			interior bound on the second fluid & none needed ($\mathrm{curl}\,u_{2}=0$) & \textbf{(8)} $\|v\|_{C^{2}}\leqslant A$, \emph{joint} in $(t,x)$ & \textbf{none} \\
			bulk vorticity $\omega$ & absent & implied by (8) & \eqref{e11}: $\|\omega\|_{L^{\infty}}\leqslant A$ \\
			Biot--Savart potential $\Phi_{2}$ & absent & implied by (8) & \eqref{e11.bis} \\
			exponential vertical decay of $\omega$ & --- & \emph{not assumed} & \eqref{e11.loc}: \emph{additional} \\
			$\partial_{\nu}\omega$ at $\Gamma_{t}$ & --- & C$^{0}$ part follows from (ST) and (8) via the identity & \eqref{e11.nu}: \emph{assumed directly} \\
			bounded harmonic part of $u_{1}$ & not needed & not needed & \eqref{e9.norm}: \emph{additional} \\
			\hline
		\end{tabular}
	\end{center}
	
	Dropping (ST) makes
	\eqref{e11.nu} and \eqref{e8.init} necessary: \eqref{e8.trace} yields \eqref{e8.init} by the
	product rule, and Lemma \ref{nu.omega} expresses $\partial_{\nu}\omega=(\Delta u_{2})\cdot\tau$
	through the traces and $\partial_{t}(v_{2}\cdot\tau)$. This expression shows that the first
	summand of \eqref{e11.nu} is implied by (ST) together with (8): the term $\partial_{t}(v_{2}\cdot\tau)$
	is controlled by the momentum equation, whose right-hand side involves $\Delta v$ and $\nabla p$,
	both of which are bounded under (8) once the pressure is recovered from the equations. The second summand, $|\partial_{\alpha}[(\Delta u_{2})\cdot\tau]|$, is of a
	different nature: differentiating the identity in $\alpha$ produces the mixed derivative
	$\partial_{\alpha}\partial_{t}(v_{2}\cdot\tau)$, hence $\nabla\partial_{t}v$. In \cite{AB} the
	bound (8) is joint in $(t,x)$, so this mixed derivative is controlled there and the second summand
	of \eqref{e11.nu} follows from (ST) and (8) as well; were (8) a spatial $C^{2}$ bound only, the
	second summand would require in addition a $C^{2,1}$ bound on $v$ in the fluid region. The present
	theorem assumes \eqref{e11.nu} directly and dispenses with (ST) altogether.
	The interior bound (8) likewise implies \eqref{e11} and \eqref{e11.bis}. No hypothesis of
	Definition \ref{adm} controls an interior derivative of $u_{2}$, or any derivative of $z$ in which
	$t$ occurs. The two rows marked \emph{additional} have no counterpart in \cite{AB}: the exponential
	vertical decay \eqref{e11.loc} and the bounded-drift condition \eqref{e9.norm} arise because the
	velocities are recovered from their vorticity through the Biot--Savart representation \eqref{e9.bs},
	rather than from an assumed interior regularity. Remark \ref{rm:example} gives a solution of the
	class with $\omega\equiv0$.
	
	Concerning \eqref{e11.loc}, the decay is used at three places, all of them in the estimate of the
	bulk term of the Biot--Savart representation: in the proof of the boundary gradient bound, where it
	produces the constant $C(A,R_{0})$ of \eqref{P.int}; in the proof of the at-most-linear growth
	\eqref{growth} of $u_{1}$, which makes the Lagrangian flow \eqref{eq.Lag} well defined; and in the
	estimate of the term $I_{1}$ of the geometric reduction, where $\omega$ is integrated against the
	kernel of \eqref{V.decomp.U}. In the first two places it suffices that the weight be
	integrable against the logarithmically singular kernel of \eqref{P.est}, that is, that
	$\int_{\mathbb{R}}(1+|\log|h||)\,e^{-|h|/R_{0}}\,dh<\infty$; this holds for the exponential weight
	and for any weight decaying like $|y_{2}|^{-2-\delta}$ with $\delta>0$, in which case the constants
	$C(A,R_{0})$ are replaced by constants depending on $\delta$ instead of $R_{0}$. In the third place
	the far field is estimated by $\int e^{-|y_{2}|/R_{0}}|y|^{-2}\,dy$ (Step 3 of the proof of Theorem
	\ref{thm}), so there it suffices that $\omega$ be bounded and that the weight decay in $y_{2}$;
	the exponential weight has this property, as does any bounded weight decaying like
	$|y_{2}|^{-2-\delta}$. These three estimates
	use \eqref{e11.loc} and the boundedness \eqref{e11} of $\omega$, and no other hypothesis of
	Definition \ref{adm}.
	
	The hypotheses of the three results are of different kinds. Bulut's Euler--Navier--Stokes theorem
	assumes a uniform interior $C^{2}$ bound on the viscous velocity, joint in space and time, whereas
	the theorem above assumes
	the bulk-vorticity bounds \eqref{e11}--\eqref{e11.loc}, the boundary flux bound \eqref{e11.nu}, and
	the Euler representation and drift conditions \eqref{e9.bs}--\eqref{e9.norm}. The
	Fefferman--Ionescu--Lie theorem concerns two irrotational Euler phases; the class of Definition
	\ref{adm} allows bulk vorticity in the conducting phase. The three results therefore apply to
	different conditional classes.
	
	The comparison turns on one point: whether \eqref{e11.nu} stays bounded for as long as
	the solution remains smooth. In the Euler--Navier--Stokes setting it is a consequence of (ST)
	together with (8), so there the criterion is a continuation principle; in the class of Definition
	\ref{adm} it is assumed. The
	theorem is conditional on it in the first place, and Remark \ref{rm:nu.flux} identifies
	it as the only item that controls a time derivative of the trace of the conducting phase.
\end{remark}

The admissibility assumptions leave the tangential derivatives of $\omega_{\Gamma}$ free; these are
supplied by the interface-vorticity estimates of Section \ref{sec:vorticity}. By no-slip,
$\partial_{t}z=v_{2}$ and
$\omega_{\Gamma}=(v_{1}-\partial_{t}z)\cdot\partial_{\alpha}z$, so that by the product rule
\begin{equation}\label{e8.trace}
\partial_{\alpha}\omega_{\Gamma}
=(\partial_{\alpha}v_{1}-\partial_{t}\partial_{\alpha}z)\cdot\partial_{\alpha}z
+(v_{1}-\partial_{t}z)\cdot\partial_{\alpha}^{2}z,
\end{equation}
and analogously for $\partial_{t}\omega_{\Gamma}$, which involves
$\partial_{t}^{2}z$. Item (i) imposes no bound on $\partial_{t}\partial_{\alpha}z$ or on $\partial_{t}^{2}z$, so
\eqref{e8.trace} does not bound $\partial_{\alpha}\omega_{\Gamma}$ by the admissibility
assumptions; neither $\|v_{2}\|_{C^{1}(\mathbb{R})}$ nor $\|\omega_{\Gamma}\|_{C^{1}(\mathbb{R})}$,
the norms being taken in the variable $\alpha$, belongs to the a priori information. From the
assumptions alone one obtains
\[
\|\partial_{t}z\|_{L^{\infty}([0,T]\times\mathbb{R})}=\|v_{2}\|_{L^{\infty}}
=\|v_{1}-u\tau\|_{L^{\infty}}\leqslant A+\|u\|_{L^{\infty}},
\]
in which the right-hand side involves the unknown $u$, together with \eqref{e8.init}, which
bounds $u$ at time $0$ only.

Item (i) requires four derivatives of $z$ in $\alpha$ because the interface-vorticity equation is
differentiated once in $\alpha$, and the resulting equation
contains two tangential derivatives of the curvature $K$, hence four derivatives of $z$ in the
parameter. No regularity in $t$ beyond that of item (i) is used.

The seven items enter the proof as follows. Item (i) fixes the range
$1/A\leqslant a\leqslant A+1$, the curvature bound $|K|\leqslant C(A)$ and the $C^{4}$ data for
the local graphs of Proposition \ref{chord-arc prop}; item (ii) is used in Lemma \ref{chord-arc}
and to make $\bar t>0$; item (iii) gives the Lorentz bound of Lemma \ref{MHD:lorentz} and part of
the bounds on $\Lambda$ and $f$ of Lemma \ref{curl.lin}, the trace of $H$ itself being recovered
from its value at time $0$ in Lemma \ref{B.trace}; the three parts of item (iv) enter in
Step 3, in Step 2(d) and in Lemma \ref{curl.lin} respectively; item (v) makes the bulk term
integrable in the far field, in Lemma \ref{grad.euler}, \eqref{growth} and Step 3; item (vi)
supplies the decomposition of $v_{1}$ and the bound on the harmonic drift; item (vii) supplies the
initial data of the two maximum principles. Each item is used at least once, and none only to
shorten a statement.

The following consequence of the definition supplies the interior gradient bound for the
Euler velocity that the geometric reduction needs. The
argument has two parts, which rest on different inputs. The first
part is local: it computes the Hessian of the potential of $u_{1}$ on the interface, and it
uses only the hypotheses of Definition \ref{adm}. The second part is the maximum principle on
$\Omega_{1}(t)$, and it uses the bounds on the interface vorticity derived in the
interface-vorticity section; it is stated at the end of that section, as Lemma \ref{grad.euler}.

\begin{lemma}\label{grad.euler.bdry}
Let $(z,u_{1},u_{2},B,p_{1},p_{2})$ be an $(A,\rho_{1},\rho_{2})$-admissible solution. Then
\begin{equation}\label{e9.bdry}
\sup_{t\in[0,T]}\sup_{x\in\Gamma_{t}}|\nabla u_{1}(t,x)|\leqslant C(A).
\end{equation}
\end{lemma}
\begin{proof}
By \eqref{e3} and the simple connectedness of $\Omega_{1}(t)$ there is a function
$\psi=\psi(t,\cdot)$ with $\Delta\psi=0$ in $\Omega_{1}(t)$ and $u_{1}=\nabla\psi$. Hence
$\nabla u_{1}=D^{2}\psi$, and every entry of $D^{2}\psi$ is harmonic in $\Omega_{1}(t)$.

\emph{The boundary values.} Let $\tau=z_{\alpha}/|z_{\alpha}|$ and
$\nu=n/|n|$ be the unit tangent and the unit normal of $\Gamma_{t}$; then $(\tau,\nu)$ is an
orthonormal frame, and with $K$ the curvature of \eqref{e6} one has
$\partial_{\tau}\tau=-K\nu$ and $\partial_{\tau}\nu=K\tau$, with $|K|\leqslant C(A)$ by
\eqref{e8}. On $\Gamma_{t}$ we have $\psi_{\tau}=v_{1}\cdot\tau$ and
$\psi_{\nu}=v_{1}\cdot\nu$, both bounded by $A$. Differentiating along $\Gamma_{t}$ and using
the two frame relations,
\[
\psi_{\tau\tau}=\partial_{\tau}(v_{1}\cdot\tau)+K\,(v_{1}\cdot\nu),\qquad
\psi_{\tau\nu}=\partial_{\tau}(v_{1}\cdot\nu)-K\,(v_{1}\cdot\tau),\qquad
\psi_{\nu\nu}=-\psi_{\tau\tau},
\]
where $\psi_{\tau\tau},\psi_{\tau\nu},\psi_{\nu\nu}$ denote the entries of the Hessian
$D^{2}\psi$ in the frame $(\tau,\nu)$: the first two follow from
$\partial_{\tau}\nabla\psi=D^{2}\psi\,\tau$ together with the frame relations, and the third
from $\Delta\psi=\psi_{\tau\tau}+\psi_{\nu\nu}=0$, the Laplacian being the trace of the
Hessian. Since $\|v_{1}\|_{C^{1}}\leqslant A$ by \eqref{e9} and
$|K|\leqslant C(A)$, all second derivatives of $\psi$ are bounded on $\Gamma_{t}$ by
$C(A)$; as $(\tau,\nu)$ is an orthonormal frame, this gives
$\sup_{\Gamma_{t}}|\nabla u_{1}|\leqslant C(A)$.
\end{proof}

The interior version of \eqref{e9.bdry}, used in the geometric reduction, is
Lemma \ref{grad.euler}, at the end of the interface-vorticity section.

\subsection{The main theorem}
\begin{theorem}\label{thm}
Let $(z,u_{1},u_{2},B,p_{1},p_{2})$ be an $(A,\rho_{1},\rho_{2})$-admissible solution
of the two-fluid Euler--MHD system on the time interval $[0,T]$. Then
\begin{equation}\label{e13}
\inf_{\alpha,\beta\in\mathbb{R},\,t\in[0,T]}|F(z)(t,\alpha,\beta)|\geqslant c
\end{equation}
for some constant $c>0$ that depends only on $T$, $A$, the decay
width $R_{0}$ of \eqref{e11.loc}, the surface tension coefficient $\sigma$ and the
densities $\rho_{1},\rho_{2}$. 
\end{theorem}

\begin{remark}\label{rm:dependence}
	The dependence on $\sigma$ and on the densities is inessential:
	enlarging $A$ preserves admissibility, so after the enlargement of Remark \ref{rm:normalize} the
	constant may equally be read as depending only on $T$, on the decay width $R_{0}$ and on the
	enlarged a priori bound.
	The initial chord-arc bound \eqref{e10} is part of admissibility, and the proof below gives an
	explicit lower bound, so \eqref{e13} upgrades \eqref{e10} to a bound that is uniform on $[0,T]$
	and over the whole class.
\end{remark}

\begin{corollary}[Quantitative separation criterion]\label{cor:criterion}
Fix $T>0$, $A\geqslant1$, $R_{0}>0$ and $\rho_{1},\rho_{2}>0$, and let $c>0$ be the constant of
Theorem \ref{thm}. Let $(z,u_{1},u_{2},B,p_{1},p_{2})$ be a smooth solution on $[0,T]$ whose
interface is an embedded curve for every $t\in[0,T]$, and suppose that conditions (i)--(vii) of
Definition \ref{adm} hold with the constant $A$ and the decay rate $R_{0}$, the initial conditions
being imposed at $t=0$. Then
\[
\inf_{(t,\alpha,\beta)\in[0,T]\times\mathbb{R}^{2}}F(z)(t,\alpha,\beta)\geqslant c,
\]
so the interface remains uniformly separated on $[0,T]$ and no splash occurs there. Conversely, if
this infimum is less than $c$, then at least one of the quantitative requirements (i)--(vii) fails.
This includes failure of the initial chord-arc bound \eqref{e10}, one of the stated upper or lower
norm bounds, the velocity representation in item (vi), or the exponential localization
\eqref{e11.loc} with rate $R_{0}$.
\end{corollary}
\begin{proof}
Under the stated hypotheses the solution is $(A,\rho_{1},\rho_{2})$-admissible, so \eqref{e13}
applies; the converse is its contrapositive.
\end{proof}

\begin{remark}\label{rm:local.in.time}
The constant $c$ depends on $T$, and the dependence is quantitative: the proof gives
$c\approx\varepsilon_{1}C_{3}^{-1/2}\varepsilon_{4}^{e^{C_{0}T}}$ with $C_{0}$ depending only on
$A$ and $R_{0}$, so $c$ decreases double exponentially in the length of the interval. Theorem
\ref{thm} and Corollary \ref{cor:criterion} are accordingly local in time: they exclude a splash on
any interval on which the a priori bounds hold, with a threshold chosen for the length of that
interval. The argument yields no bound that is uniform over all horizons, since the $D$-Gronwall
estimate of Section \ref{sec:main} controls $D(\bar t)$ by a power of $D(t_{0})$ whose exponent is
exponential in $\bar t-t_{0}$.
\end{remark}

The first part of the proof is analytic and is contained in Section \ref{sec:vorticity}: the
analysis of the interface-vorticity equation. The uniform $L^{\infty}$ and $C^{1}$ bounds, in the
variable $\alpha$, on the interface vorticity $\omega_{\Gamma}=(v_{1}-v_{2})\cdot\partial_{\alpha}z$
and on the conducting-fluid trace $v_{2}$ are derived from the equation (Lemmas \ref{om.Linf} and
\ref{om.C1.est}), and the normal derivative $\partial_{\nu}\omega$ of the bulk vorticity at the
interface is expressed in closed form (Lemma \ref{nu.omega}), its bound being the flux hypothesis
\eqref{e11.nu}. In the Euler--MHD derivation the magnetic field enters through the Lorentz term and
the magnetic pressure. By the tangency condition $H\cdot n=0$ of \eqref{e6.m}, $B$ is parallel to the
interface, so that $(B\cdot\nabla)B\cdot\partial_{\alpha}z=\frac{1}{2}\partial_{\alpha}|H|^{2}$ on
$\Gamma_{t}$ (Lemma \ref{MHD:lorentz}) and the Lorentz force enters the tangential momentum balance
as a total tangential derivative; the magnetic pressure in the pressure jump is a total tangential
derivative as well. The magnetic field therefore reaches the interface-vorticity equation only
through $\partial_{\alpha}|H|^{2}$, twice, once from the Lorentz term and once from the magnetic
pressure, each with coefficient $\frac{1}{2\rho_{2}a}$ (see \eqref{u.e}); without the tangency
condition, \eqref{lorentz.vector} contributes instead
$(\mathrm{curl}\,B)\,B^{\bot}\cdot\partial_{\alpha}z$, which is not a tangential derivative. The
second part, Section \ref{sec:geometry} and Appendix~\ref{sec:singular}, is geometric and kinematic;
it adapts Fefferman--Ionescu--Lie and, beyond the bounds that Section \ref{sec:vorticity} imports
from the equations, uses neither the momentum equations nor the stress balance, so it is
independent of the magnetic field. In the non-conducting limit $B\equiv0$ the theorem reduces
to a no-splash statement for the two-fluid Euler--Navier--Stokes system of the same type as that of
Bulut \cite{AB}.

\begin{remark}\label{rm:example}
The example below is a steady solution with a flat interface: $z(\alpha)=(\alpha,0)$, so
$F(z)(t,\alpha,\beta)\equiv1$ and \eqref{e13} holds with $c=1$. It is an admissible solution in
the case $\omega=0$, with a nonzero magnetic field and a nonzero interface vorticity. Since $\omega\equiv0$ and $w\equiv0$ there, the two restrictions
\eqref{e11.loc} and \eqref{e9.norm} hold degenerately, and the remaining conditions are checked
below.

Let $\Omega_{2}=\{x_{2}>0\}$,
$\Omega_{1}=\{x_{2}<0\}$ and $z(\alpha)=(\alpha,0)$, and fix $c_{0}\in\mathbb{R}\setminus\{0\}$ and
$b_{0}\neq0$. Set
\[
u_{1}\equiv(c_{0},0),\qquad u_{2}\equiv0,\qquad B\equiv(b_{0},0),\qquad
p_{1}\equiv p_{0},\qquad p_{2}\equiv p_{0}-\tfrac12b_{0}^{2}.
\]
Then $\mathrm{div}\,u_{j}=0$, $\mathrm{curl}\,u_{1}=0$ and $\mathrm{div}\,B=0$; the induction
equation holds because $B$ and $u_{2}$ are constant, and the momentum equations reduce to
$\nabla p_{j}=0$ and $(B\cdot\nabla)B=0$, both of which are satisfied. The interface conditions
of \eqref{e6} hold: $\partial_{t}z=0=v_{2}$, $v_{1}=(c_{0},0)$ and $H=(b_{0},0)$ are tangent to
$\Gamma$, the viscous stress vanishes because $u_{2}$ is constant, and the normal balance reads
$q_{2}-q_{1}+\tfrac12|H|^{2}=0$, which is exactly $p_{2}-p_{1}=-\tfrac12b_{0}^{2}$. Also
$\omega=0$, so \eqref{e11} and \eqref{e11.loc} hold for every $R_{0}>0$; $\Phi_{2}=0$, so
\eqref{e11.bis} holds; and $\omega_{\Gamma}=(v_{1}-v_{2})\cdot\partial_{\alpha}z=c_{0}$, so the
interface vorticity is a nonzero constant. Conditions \eqref{e11.nu} and \eqref{e8.init} also
hold: $\Delta u_{2}=0$, hence
$(\Delta u_{2})\cdot\tau=0$ and \eqref{e11.nu} holds; and $u=\omega_{\Gamma}/a=c_{0}$ is constant,
so $\|u(0)\|_{L^{\infty}}+\|\partial_{\alpha}u(0)\|_{L^{\infty}}=|c_{0}|$ and \eqref{e8.init}
holds. Finally \eqref{e9.bs} holds with
$c(t)=(\tfrac{c_{0}}{2},0)$ and $w=0$: the Birkhoff--Rott integral of the constant
$\omega_{\Gamma}=c_{0}$ over the flat line equals $(-\tfrac{c_{0}}{2}\,\mathrm{sgn}\,x_{2},0)$, which is
$(\tfrac{c_{0}}{2},0)$ on $\Omega_{1}$. Finally \eqref{e9} holds, since
$\|v_{1}(t,\cdot)\|_{C^{3}}=|c_{0}|$ and $\|H(0,\cdot)\|_{C^{2}}=|b_{0}|$ for every $t$; the
uniform bound on $H$ required by the proof is supplied by Lemma \ref{B.trace}, which here reads
$H(t,\alpha)=(b_{0},0)=\lambda_{0}\partial_{\alpha}z$ with $\lambda_{0}\equiv b_{0}$. Thus all the
conditions of Definition \ref{adm} are
satisfied for $A$ large enough, with a nonzero magnetic field and a nonzero interface
vorticity, and the interface is flat.

No member of the class with $\omega\not\equiv0$ is constructed here. The items (iv)--(vi) are not consequences of one another: the
Biot--Savart potential and its tangential derivative are not controlled by
$\|\omega\|_{L^{\infty}}$ alone, the exponential localization \eqref{e11.loc} is a separate
restriction, and the flux bound \eqref{e11.nu} controls the time derivative of the tangential
velocity of the conducting phase, for which no other item provides a bound (Remark
\ref{rm:nu.flux}). A construction of a non-trivial member of the class would have to
produce all three.
\end{remark}

\section{The interface vorticity}\label{sec:vorticity}
Six groups of results of this section are used in the proof of Theorem \ref{thm}: the
Biot--Savart representation of Lemma \ref{Euler formula}; the rigidity of the magnetic trace of
Lemma \ref{B.trace}, which recovers $\|H\|_{C^{2}}$ from the value of $H$ at time $0$; the bounds on
the interface vorticity of Lemmas \ref{om.Linf} and \ref{om.C1.est} together with the bound on
$v_{2}$ that follows from them; the interface-vorticity equation \eqref{omega.eq} of Lemma
\ref{curl}, from which those bounds are derived; the closed formula for the normal derivative of
the bulk vorticity at the interface of Lemma \ref{nu.omega}, together with the flux hypothesis
\eqref{e11.nu} that bounds it; and the interior gradient bound for the Euler velocity of Lemma
\ref{grad.euler}, which the geometric reduction needs and which is proved at the end of this
section. The trace equations of Lemmas \ref{Euler} and \ref{MHD} and the
pressure-jump identity of Proposition \ref{pressure.jump} are the two ingredients of
\eqref{omega.eq}; the magnetic field enters through the first as the Lorentz term and through the
second as the magnetic pressure. We begin by
writing the equations of the two fluids on the closures of
$\Omega_{1}(t)$ and $\Omega_{2}(t)$, evaluated on the interface, and read off an evolution equation
for the \emph{interface vorticity}, the jump of the tangential velocity across $\Gamma_t$.

\begin{lemma}\label{Euler}
Let $(z,u_{1},p_{1})$ be the Euler phase of an $(A,\rho_{1},\rho_{2})$-admissible solution in the
sense of Definition \ref{adm}, and let $v_{1}$ and $q_{1}$ be the traces of $u_{1}$ and $p_{1}$
on $\Gamma_{t}$ as in \eqref{e6}. Then
\begin{equation}\label{e14}
\partial_{t}(v_{1}\cdot\partial_{\alpha}z)-\partial_{\alpha}(v_{1}\cdot\partial_{t}z)+\frac{1}{2}\partial_{\alpha}[|v_{1}|^{2}]+\frac{1}{\rho_{1}}\partial_{\alpha}q_{1}=0
\end{equation}
for all $t\in[0,T]$ and $\alpha\in\mathbb{R}$.
\end{lemma}
\begin{proof}
Recall that $v_{1}(t,\alpha)=u_{1}(t,z(t,\alpha))$. By the chain rule,
\begin{equation}\label{euler.chain}
\partial_{\alpha}v_{1}^{j}=\sum_{k=1}^{2}(\partial_{k}u_{1}^{j})\,\partial_{\alpha}z_{k},
\qquad
\partial_{t}v_{1}^{j}=(\partial_{t}u_{1}^{j})+\sum_{k=1}^{2}(\partial_{k}u_{1}^{j})\,\partial_{t}z_{k},
\end{equation}
where, here and below, $u_{1}$ and all of its derivatives are evaluated at
$(t,z(t,\alpha))$. Using \eqref{euler.chain},
\[
\begin{aligned}
\partial_{t}(v_{1}\cdot\partial_{\alpha}z)
&=\sum_{j=1}^{2}\Big[(\partial_{t}v_{1}^{j})(\partial_{\alpha}z_{j})+(v_{1}^{j})(\partial_{t\alpha}z_{j})\Big] \\
&=\sum_{j=1}^{2}\Big[(\partial_{t}u_{1}^{j})(\partial_{\alpha}z_{j})+\sum_{k=1}^{2}(\partial_{k}u_{1}^{j})(\partial_{t}z_{k})(\partial_{\alpha}z_{j})+(u_{1}^{j})(\partial_{t\alpha}z_{j})\Big],
\end{aligned}
\]
and, in the same way,
\[
\partial_{\alpha}(v_{1}\cdot\partial_{t}z)
=\sum_{j=1}^{2}\Big[\sum_{k=1}^{2}(\partial_{k}u_{1}^{j})(\partial_{\alpha}z_{k})(\partial_{t}z_{j})+(u_{1}^{j})(\partial_{t\alpha}z_{j})\Big],
\]
whereas
\[
\frac{1}{2}\partial_{\alpha}[|v_{1}|^{2}]
=\sum_{j=1}^{2}(v_{1}^{j})(\partial_{\alpha}v_{1}^{j})
=\sum_{j=1}^{2}\sum_{k=1}^{2}(u_{1}^{j})(\partial_{k}u_{1}^{j})(\partial_{\alpha}z_{k}).
\]
Since $\partial_{t\alpha}z_{j}=\partial_{\alpha t}z_{j}$, the terms involving the second
derivatives of $z$ cancel, and we obtain
\[
\begin{aligned}
&\partial_{t}(v_{1}\cdot\partial_{\alpha}z)-\partial_{\alpha}(v_{1}\cdot\partial_{t}z)+\frac{1}{2}\partial_{\alpha}[|v_{1}|^{2}] \\
=&\sum_{j=1}^{2}(\partial_{t}u_{1}^{j})(\partial_{\alpha}z_{j})
+\underbrace{\sum_{j,k=1}^{2}(\partial_{k}u_{1}^{j})\Big[(\partial_{t}z_{k})(\partial_{\alpha}z_{j})-(\partial_{\alpha}z_{k})(\partial_{t}z_{j})\Big]}_{=:S_{2}}
+\underbrace{\sum_{j,k=1}^{2}(u_{1}^{j})(\partial_{k}u_{1}^{j})(\partial_{\alpha}z_{k})}_{=:S_{3}}.
\end{aligned}
\]
In $S_{2}$ only the off-diagonal terms survive, and regrouping them gives
\begin{equation}\label{euler.S2}
S_{2}=(\mathrm{curl}\,u_{1})\big[(\partial_{t}z_{1})(\partial_{\alpha}z_{2})-(\partial_{\alpha}z_{1})(\partial_{t}z_{2})\big],
\qquad \mathrm{curl}\,u_{1}:=\partial_{1}u_{1}^{2}-\partial_{2}u_{1}^{1}.
\end{equation}
Next we substitute the Euler equation, written componentwise as
\[
\partial_{t}u_{1}^{j}=-\sum_{k=1}^{2}(u_{1}^{k})(\partial_{k}u_{1}^{j})-\rho_{1}^{-1}(\partial_{j}p_{1}),
\]
into the first sum. Together with $S_{3}$ this produces
\[
-\sum_{j,k=1}^{2}(u_{1}^{k})(\partial_{k}u_{1}^{j})(\partial_{\alpha}z_{j})
+S_{3}-\rho_{1}^{-1}\sum_{j=1}^{2}(\partial_{j}p_{1})(\partial_{\alpha}z_{j}).
\]
Renaming the summation indices in $S_{3}$ and in the first term, the two convective
contributions combine into
\begin{equation}\label{euler.S3}
\begin{aligned}
-\sum_{j,k=1}^{2}(u_{1}^{k})(\partial_{k}u_{1}^{j})(\partial_{\alpha}z_{j})+S_{3}
&=\sum_{j,k=1}^{2}(u_{1}^{k})\Big[(\partial_{j}u_{1}^{k})-(\partial_{k}u_{1}^{j})\Big](\partial_{\alpha}z_{j})\\
&=(\mathrm{curl}\,u_{1})\big[(u_{1}^{2})(\partial_{\alpha}z_{1})-(u_{1}^{1})(\partial_{\alpha}z_{2})\big].
\end{aligned}
\end{equation}
Finally, by the chain rule,
$\partial_{\alpha}q_{1}=(\partial_{1}p_{1})(\partial_{\alpha}z_{1})+(\partial_{2}p_{1})(\partial_{\alpha}z_{2})$.
Collecting the identities above, every term of
$\partial_{t}(v_{1}\cdot\partial_{\alpha}z)-\partial_{\alpha}(v_{1}\cdot\partial_{t}z)+\frac{1}{2}\partial_{\alpha}[|v_{1}|^{2}]$
other than the pressure term is a multiple of $\mathrm{curl}\,u_{1}$; see
\eqref{euler.S2} and \eqref{euler.S3}. Since the Euler fluid is irrotational, \eqref{e3}
gives $\mathrm{curl}\,u_{1}=0$, and therefore
\[
\partial_{t}(v_{1}\cdot\partial_{\alpha}z)-\partial_{\alpha}(v_{1}\cdot\partial_{t}z)+\frac{1}{2}\partial_{\alpha}[|v_{1}|^{2}]
=-\rho_{1}^{-1}\partial_{\alpha}q_{1},
\]
which is precisely \eqref{e14}.
\end{proof}
\begin{lemma}\label{MHD}
Let $(z,u_{2},B,p_{2})$ be the conducting phase of an $(A,\rho_{1},\rho_{2})$-admissible solution
in the sense of Definition \ref{adm}, and let $v_{2}$ and $q_{2}$ be the traces of $u_{2}$ and
$p_{2}$ on $\Gamma_{t}$ as in \eqref{e6}. Then
\begin{equation}\label{e15}
\partial_{t}(v_{2}\cdot\partial_{\alpha}z)-\partial_{\alpha}(v_{2}\cdot\partial_{t}z)+\frac{1}{2}\partial_{\alpha}[|v_{2}|^{2}]+\frac{1}{\rho_{2}}\partial_{\alpha}q_{2}-\frac{1}{\rho_{2}}(\Delta u_{2})\cdot(\partial_{\alpha}z)|_{(t,\alpha)}-\frac{1}{\rho_{2}}(B\cdot\nabla)B\cdot\partial_{\alpha}z=0
\end{equation}
for all $t\in[0,T]$ and $\alpha\in\mathbb{R}$.
\end{lemma}
\begin{proof}
We argue as in the proof of Lemma \ref{Euler}, but now we must keep track of the term produced by $\mathrm{curl}\,u_{2}$, which does \emph{not} vanish in $\Omega_{2}(t)$. Throughout, $u_{2}$ and its derivatives are evaluated at $(t,z(t,\alpha))$. First,
\[
\partial_{t}(v_{2}\cdot\partial_{\alpha}z)=\sum_{j=1}^{2}\Big[(\partial_{t}u_{2}^{j})(\partial_{\alpha}z_{j})+(\partial_{1}u_{2}^{j})(\partial_{t}z_{1})(\partial_{\alpha}z_{j})+(\partial_{2}u_{2}^{j})(\partial_{t}z_{2})(\partial_{\alpha}z_{j})+(u_{2}^{j})(\partial_{t\alpha}z_{j})\Big],
\]
and
\[
\partial_{\alpha}(v_{2}\cdot\partial_{t}z)=\sum_{j=1}^{2}\Big[(\partial_{1}u_{2}^{j})(\partial_{\alpha}z_{1})(\partial_{t}z_{j})+(\partial_{2}u_{2}^{j})(\partial_{\alpha}z_{2})(\partial_{t}z_{j})+(u_{2}^{j})(\partial_{t\alpha}z_{j})\Big],
\]
and
\[
\frac{1}{2}\partial_{\alpha}[|v_{2}|^{2}]=\sum_{j=1}^{2}\Big[(u_{2}^{j})(\partial_{1}u_{2}^{j})(\partial_{\alpha}z_{1})+(u_{2}^{j})(\partial_{2}u_{2}^{j})(\partial_{\alpha}z_{2})\Big].
\]
Summing the three identities the terms containing $\partial_{t\alpha}z_{j}$ cancel, and, recalling that $\mathrm{curl}\,u_{2}=\partial_{1}u_{2}^{2}-\partial_{2}u_{2}^{1}$, we get
\[
\begin{aligned}
&\partial_{t}(v_{2}\cdot\partial_{\alpha}z)-\partial_{\alpha}(v_{2}\cdot\partial_{t}z)+\frac{1}{2}\partial_{\alpha}[|v_{2}|^{2}] \\
=&\sum_{j=1}^{2}\Big[(\partial_{t}u_{2}^{j})(\partial_{\alpha}z_{j})+(u_{2}^{j})(\partial_{1}u_{2}^{j})(\partial_{\alpha}z_{1})+(u_{2}^{j})(\partial_{2}u_{2}^{j})(\partial_{\alpha}z_{2})\Big] \\
&+(\mathrm{curl}\,u_{2})\Big[(\partial_{t}z_{1})(\partial_{\alpha}z_{2})-(\partial_{t}z_{2})(\partial_{\alpha}z_{1})\Big].
\end{aligned}
\]
Now we insert the equations of motion and the interface condition $\partial_{t}z=v_{2}$ from \eqref{e6}. By \eqref{mhd.2},
\[
\partial_{t}u_{2}^{j}=-\sum_{k=1}^{2}(u_{2}^{k})(\partial_{k}u_{2}^{j})+\frac{1}{\rho_{2}}\Delta u_{2}^{j}-\frac{1}{\rho_{2}}(\partial_{j}p_{2})+\frac{1}{\rho_{2}}\sum_{k=1}^{2}(B_{k})(\partial_{k}B_{j}).
\]
Substituting this and using $\partial_{t}z=v_{2}$, the right-hand side becomes
\[
\begin{aligned}
&\sum_{j=1}^{2}\Big[\Big(-\sum_{k=1}^{2}(v_{2}^{k})(\partial_{k}u_{2}^{j})+\frac{1}{\rho_{2}}\Delta u_{2}^{j}-\frac{1}{\rho_{2}}(\partial_{j}p_{2})+\frac{1}{\rho_{2}}\sum_{k=1}^{2}(B_{k})(\partial_{k}B_{j})\Big)(\partial_{\alpha}z_{j}) \\
&\qquad\qquad+(v_{2}^{j})(\partial_{1}u_{2}^{j})(\partial_{\alpha}z_{1})+(v_{2}^{j})(\partial_{2}u_{2}^{j})(\partial_{\alpha}z_{2})\Big] \\
&\qquad+(\mathrm{curl}\,u_{2})\Big[(v_{2}^{1})(\partial_{\alpha}z_{2})-(v_{2}^{2})(\partial_{\alpha}z_{1})\Big].
\end{aligned}
\]
Here the terms of the form $(v_{2}^{i})(\partial_{j}u_{2}^{k})(\partial_{\alpha}z_{\ell})$ (which collect the convective contribution, the contribution of the term $\tfrac12\partial_{\alpha}[|v_{2}|^{2}]$ and the $\mathrm{curl}\,u_{2}$ term) cancel identically. Writing $\omega:=\mathrm{curl}\,u_{2}$ (the bulk vorticity of \eqref{e17}) for this cancellation,
\[
-\sum_{j,k}(v_{2}^{k})(\partial_{k}u_{2}^{j})(\partial_{\alpha}z_{j})+\sum_{j}(v_{2}^{j})\big[(\partial_{1}u_{2}^{j})(\partial_{\alpha}z_{1})+(\partial_{2}u_{2}^{j})(\partial_{\alpha}z_{2})\big]=\omega\,(v_{2}^{2}\partial_{\alpha}z_{1}-v_{2}^{1}\partial_{\alpha}z_{2}),
\]
while the last term equals $\omega\,(v_{2}^{1}\partial_{\alpha}z_{2}-v_{2}^{2}\partial_{\alpha}z_{1})=-\omega\,(v_{2}^{2}\partial_{\alpha}z_{1}-v_{2}^{1}\partial_{\alpha}z_{2})$, so that the two cancel. Therefore only the following terms survive:
\[
\begin{aligned}
&\sum_{j=1}^{2}\frac{1}{\rho_{2}}(\Delta u_{2}^{j})(\partial_{\alpha}z_{j})-\frac{1}{\rho_{2}}\partial_{\alpha}q_{2}
+\frac{1}{\rho_{2}}\sum_{j,k}(B_{k})(\partial_{k}B_{j})(\partial_{\alpha}z_{j})\\
&=\frac{1}{\rho_{2}}(\Delta u_{2})\cdot\partial_{\alpha}z-\frac{1}{\rho_{2}}\partial_{\alpha}q_{2}
+\frac{1}{\rho_{2}}(B\cdot\nabla)B\cdot\partial_{\alpha}z,
\end{aligned}
\]
where we used the chain rule $\partial_{\alpha}q_{2}=(\partial_{1}p_{2})(\partial_{\alpha}z_{1})+(\partial_{2}p_{2})(\partial_{\alpha}z_{2})$. Moving these to the left-hand side yields exactly \eqref{e15}. 
\end{proof}

We define the interface vorticity,
\begin{equation}\label{e16}
\omega_{\Gamma}(t,\alpha)=(v_{1}(t,\alpha)-v_{2}(t,\alpha))\cdot(\partial_{\alpha}z)(t,\alpha).
\end{equation}

For notational convenience, we let $\omega$ denote the vorticity internal to the MHD fluid, defined by 
\begin{equation}\label{e17}
\omega(t,x)=\text{curl}(u_{2})(t,x)\,\chi_{\Omega_{2}(t)}(x).
\end{equation}

First we recall the Biot--Savart representation of the velocity of the Euler fluid in terms of the interface vorticity $\omega_{\Gamma}$ and the bulk vorticity $\omega$ (see \cite[Proposition 2.1]{AB}). By the structural hypothesis \eqref{e9.bs} this representation is an identity for $u_{1}$ up to the bounded harmonic part $c+w$ of condition (vi).
\begin{lemma}\label{Euler formula}
With the notation above, let $c(t)$ and $w(t,\cdot)$ be as in \eqref{e9.norm}. Then, for every
$t\in[0,T]$ and every $\alpha\in\mathbb{R}$,
\begin{equation}\label{e18}
\begin{aligned}
v_{1}(t,\alpha)=&\,c(t)+w(t,z(t,\alpha))+\frac{\omega_{\Gamma}(t,\alpha)\partial_{\alpha}z(t,\alpha)}{2|\partial_{\alpha}z(t,\alpha)|^{2}}\\
&+\frac{1}{2\pi}\,\mathrm{p.v.}\!\int_{\mathbb{R}}\frac{(z(t,\alpha)-z(t,\beta))^{\bot}}{|z(t,\alpha)-z(t,\beta)|^{2}}\omega_{\Gamma}(t,\beta)d\beta \\
&+\frac{1}{2\pi}\int_{\Omega_{2}(t)}\frac{(z(t,\alpha)-y)^{\bot}}{|z(t,\alpha)-y|^{2}}\omega(t,y)dy.
\end{aligned}
\end{equation}
\end{lemma}

\begin{proof}
Fix $t\in[0,T]$ and drop it from the notation. By the structural hypothesis \eqref{e9.bs},
\begin{equation}\label{e18.bs}
\begin{aligned}
u_{1}(x)&=c+w(x)+\frac{1}{2\pi}\,\mathrm{p.v.}\!\int_{\mathbb{R}}\frac{(x-z(\beta))^{\bot}}{|x-z(\beta)|^{2}}\omega_{\Gamma}(\beta)\,d\beta\\
&\quad+\frac{1}{2\pi}\int_{\Omega_{2}}\frac{(x-y)^{\bot}}{|x-y|^{2}}\omega(y)\,dy
\qquad\text{for }x\in\Omega_{1}.
\end{aligned}
\end{equation}
Define
\[
U_{1}(x)=\frac{1}{2\pi}\,\mathrm{p.v.}\!\int_{\mathbb{R}}\frac{(x-z(\beta))^{\bot}}{|x-z(\beta)|^{2}}\omega_{\Gamma}(\beta)d\beta.
\]
For $x$ at a positive distance from $\Gamma$ the integrand is smooth; the principal value is
required only in order that the far-field tail converge, and it produces a $C^{4}$ dependence on
$x$, the differentiation being justified by the smoothness of the solution assumed in Definition
\ref{adm}. Hence $U_{1}\in C^{4}(\Omega_{1};\mathbb{R}^{2})$, and by direct differentiation
\[
\text{div} \ U_{1}=0 \quad \text{and} \quad \text{curl} \ U_{1}=0 \quad \text{in} \ \Omega_{1}.
\]
Since $z$ is a regular $C^{4}$ curve and $|\partial_{\alpha}z|\geqslant1/A$, there is $\varepsilon_{0}>0$ such that
\[
z(\alpha)-\varepsilon[z_{\alpha}(\alpha)]^{\bot}\in\Omega_{1} \ \text{for any} \ \alpha\in\mathbb{R} \ \text{and} \ \varepsilon\in(0,\varepsilon_{0}).
\]
By the classical jump relation for the Birkhoff--Rott integral (see \cite[Proposition 2.1]{AB}),
\[
\lim_{\varepsilon\to 0}U_{1}(z(\alpha)-\varepsilon[z_{\alpha}(\alpha)]^{\bot})=\frac{\omega_{\Gamma}(\alpha)}{2}\frac{z_{\alpha}(\alpha)}{|z_{\alpha}(\alpha)|^{2}}+\frac{1}{2\pi}\,\mathrm{p.v.}\!\int_{\mathbb{R}}\frac{(z(\alpha)-z(\beta))^{\bot}}{|z(\alpha)-z(\beta)|^{2}}\omega_{\Gamma}(\beta)d\beta
\]
uniformly for any $\alpha\in\mathbb{R}$. The harmonic part $w$ tends to $w(z(\alpha))$: by
\eqref{e9.norm} it is bounded with gradient bounded by $A$ on the whole plane, hence uniformly
continuous there and, in particular, it extends continuously to $\overline{\Omega_{1}}$. Letting $x\to z(\alpha)$ from
inside $\Omega_{1}$ in \eqref{e18.bs}, and using that the bulk term is continuous across $\Gamma$,
\[
\begin{aligned}
v_{1}(\alpha)=&\lim_{x\to z(\alpha)}u_{1}(x) \\
=&\,c+w(z(\alpha))+\frac{\omega_{\Gamma}(\alpha)}{2}\frac{z_{\alpha}(\alpha)}{|z_{\alpha}(\alpha)|^{2}}+\frac{1}{2\pi}\,\mathrm{p.v.}\!\int_{\mathbb{R}}\frac{(z(\alpha)-z(\beta))^{\bot}}{|z(\alpha)-z(\beta)|^{2}}\omega_{\Gamma}(\beta)d\beta \\
&+\frac{1}{2\pi}\int_{\Omega_{2}}\frac{(z(\alpha)-y)^{\bot}}{|z(\alpha)-y|^{2}}\omega(y)dy,
\end{aligned}
\]
which is \eqref{e18}.
\end{proof}

\begin{remark}
The principal value in the second term of \eqref{e18} cannot be dispensed with: by the periodicity
$z(t,\beta+2\pi)=z(t,\beta)+(2\pi,0)$ the second component of the kernel decays only like
$|\alpha-\beta|^{-1}$ as $|\beta|\to\infty$, so the corresponding integral is merely conditionally
convergent, whereas the first component is integrable because $|z_{2}(t,\beta)|$ is bounded by
\eqref{e8}.
\end{remark}

\begin{remark}\label{rm:constant}
The terms $c(t)$ and $w(t)$ of \eqref{e9.norm} enter the proof of Theorem \ref{thm} only
through \eqref{e18}, and neither plays any role in the estimates, for the following reasons. The
Euler velocity is used below through the differences $V_{1}(0)-V_{2}(0)$ of Step 1 and through
the interface quantities $U_{1},U_{2},W_{1},W_{2}$ of \eqref{V.decomp}, and \eqref{V.decomp}
itself is obtained from the Biot--Savart part of \eqref{e18}: writing
\[
V_{n}(\mu)=Q^{-1}\big(c(t)+w(t,z(t,\alpha_{n}+\beta_{n}(\mu)))\big)+U_{n}(\mu)+W_{n}(\mu),
\]
the constant $Q^{-1}(c(t))$ cancels from the difference $V_{1}(0)-V_{2}(0)$ of Step 1, while the
harmonic part $w$ contributes
\[
\begin{aligned}
\big|\big(Q^{-1}w(t,z(t,\alpha_{1}))-Q^{-1}w(t,z(t,\alpha_{2}))\big)\cdot e_{2}\big|
&\leqslant\big|w(t,z(t,\alpha_{1}))-w(t,z(t,\alpha_{2}))\big|\\
&\leqslant \|\nabla w(t)\|_{L^{\infty}(\mathbb{R}^{2})}\,|z(t,\alpha_{1})-z(t,\alpha_{2})|
\leqslant A\tilde d,
\end{aligned}
\]
where the first inequality uses $|Q^{-1}v|=|v|$ for the orthogonal $Q$ and the second is the mean
value theorem along the segment joining the two points, together with the gradient bound of
\eqref{e9.norm}. That bound is available here even though the segment lies in $\Omega_{2}(t)$,
because the hypothesis is imposed on the whole plane rather than only inside $\Omega_{1}(t)$. The
resulting $O(A\tilde d)$ term is
absorbed in the $D$-Gronwall estimate \eqref{D.gronwall.a.e}. The identities \eqref{Ue1.pm} and \eqref{Ue2.disp},
and the conclusion $|G(0)|\lesssim\epsilon\log\frac1\epsilon$ of Proposition \ref{prop:key} (here
$\epsilon$ is the width parameter of the localized analysis, fixed in
Section \ref{sec:main}), concern
$U_{n}$ and $W_{n}$ only and are therefore unaffected by $c$ and $w$. In item (d) of Step 2 the
$C^{1}$-bound $\|U_{n}\|_{C^{1}((-2\sqrt{\epsilon},2\sqrt{\epsilon}))}\lesssim_{A}1$ is checked by
adding the contribution of $w$, which is bounded by $A$ together with its $\mu$-derivative, since
$\|\nabla w\|_{L^{\infty}(\mathbb{R}^{2})}\leqslant A$ by \eqref{e9.norm} and
$\|\beta_{n}\|_{C^{4}}\lesssim1$. The harmonic part does not enter $F^{\pm}$ or $G$ at all:
these are defined in \eqref{FG.def} from the localized velocities $U_{n}$, and \eqref{V.decomp}
splits $V_{n}=Q^{-1}(c)+Q^{-1}w_{n}+U_{n}+W_{n}$, so that $Q^{-1}(c)+Q^{-1}w_{n}$ is precisely
the summand that is dropped when passing from $V_{n}$ to $U_{n}$.
\end{remark}

We estimate $\omega_{\Gamma}$ and $\omega$.

The next lemma records a rigidity property of the magnetic trace: it is determined by its initial
value and by the geometry of the interface, the induction equation entering only through this
identity.

\begin{lemma}\label{B.trace}
Let $(z,u_{1},u_{2},B,p_{1},p_{2})$ be an $(A,\rho_{1},\rho_{2})$-admissible solution and let
$H(t,\alpha)=B(t,z(t,\alpha))$. Then
\begin{equation}\label{B.trace.id}
H(t,\alpha)=\lambda_{0}(\alpha)\,\partial_{\alpha}z(t,\alpha),
\qquad
\lambda_{0}(\alpha):=\frac{H(0,\alpha)\cdot\partial_{\alpha}z(0,\alpha)}
{|\partial_{\alpha}z(0,\alpha)|^{2}},
\end{equation}
for every $(t,\alpha)\in[0,T]\times\mathbb{R}$. In particular
\begin{equation}\label{B.trace.bound}
\sup_{t\in[0,T]}\|H(t,\cdot)\|_{C^{2}(\mathbb{R})}
\leqslant C(A)\big(1+\|H(0,\cdot)\|_{C^{2}(\mathbb{R})}\big).
\end{equation}
\end{lemma}
\begin{proof}
By \eqref{e6.m} the field is tangent to $\Gamma_{t}$ at every time, so
$H(t,\alpha)=\lambda(t,\alpha)\partial_{\alpha}z(t,\alpha)$ with
$\lambda=H\cdot\partial_{\alpha}z/|\partial_{\alpha}z|^{2}$. Since $\partial_{t}z=v_{2}$ by
\eqref{e6}, the chain rule gives
$\partial_{t}H=\big(\partial_{t}B+(v_{2}\cdot\nabla)B\big)_{|x=z(t,\alpha)}$, and the induction
equation of \eqref{mhd}, evaluated at $x=z(t,\alpha)$ where $u_{2}=v_{2}$, gives
\[
\partial_{t}H=(H\cdot\nabla)u_{2}=\lambda\,(\partial_{\alpha}z\cdot\nabla)u_{2}
=\lambda\,\partial_{\alpha}v_{2}.
\]
On the other hand $\partial_{t}\partial_{\alpha}z=\partial_{\alpha}v_{2}$, so that
$\partial_{t}(\lambda\partial_{\alpha}z)=(\partial_{t}\lambda)\partial_{\alpha}z
+\lambda\partial_{\alpha}v_{2}$; comparing the two expressions for $\partial_{t}H$ yields
$(\partial_{t}\lambda)\partial_{\alpha}z=0$, and $\partial_{\alpha}z\neq0$ by \eqref{e8}, hence
$\partial_{t}\lambda=0$. Therefore $\lambda(t,\alpha)=\lambda(0,\alpha)=\lambda_{0}(\alpha)$,
which is \eqref{B.trace.id}.

For \eqref{B.trace.bound}, the product rule and
$\|\partial_{\alpha}z(t,\cdot)\|_{C^{2}}\leqslant1+A$ by \eqref{e8} give
$\|H(t,\cdot)\|_{C^{2}}\leqslant C\|\lambda_{0}\|_{C^{2}}(1+A)$, while
$\|\lambda_{0}\|_{C^{2}}\leqslant C(A)\|H(0,\cdot)\|_{C^{2}}$ because
$|\partial_{\alpha}z(0,\alpha)|\geqslant1/A$ and $\|z(0,\cdot)\|_{C^{4}}\leqslant A$.
\end{proof}

The following identity shows that the Lorentz term is of the same
(total-derivative) nature as the pressure and $\tfrac12|v|^{2}$ terms.

\begin{lemma}\label{MHD:lorentz}
Let $B$ be the magnetic field of the MHD fluid and $H(t,\alpha)=B(t,z(t,\alpha))$. Assume
the interface condition $H\cdot n=0$, i.e., \eqref{e6.m}. Then, on $\Gamma_{t}$,
\begin{equation}\label{lorentz.identity}
(B\cdot\nabla)B\cdot\partial_{\alpha}z=\frac{1}{2}\partial_{\alpha}|H|^{2}.
\end{equation}
In particular, under \eqref{e9} and Lemma \ref{B.trace}, the trace bound of \eqref{B.trace.bound} being absorbed into $A$ as in Remark \ref{rm:normalize},
\begin{equation}\label{lorentz.bound}
\Big\|(B\cdot\nabla)B\cdot\partial_{\alpha}z\Big\|_{L^{\infty}([0,T]\times\mathbb{R})}
=\frac{1}{2}\Big\|\partial_{\alpha}|H|^{2}\Big\|_{L^{\infty}([0,T]\times\mathbb{R})}\leqslant A^{2}.
\end{equation}
\end{lemma}
\begin{proof}
Since $n=-[(\partial_{\alpha}z)]^{\bot}=(\partial_{\alpha}z_{2},-\partial_{\alpha}z_{1})$,
the condition $H\cdot n=0$ amounts to $H_{1}(\partial_{\alpha}z_{2})-H_{2}(\partial_{\alpha}z_{1})=0$,
i.e., $B$ is tangent to the interface: $B|_{\Gamma_{t}}=\lambda\,\partial_{\alpha}z$ for some
scalar $\lambda$.

We use the two-dimensional identity
\begin{equation}\label{lorentz.vector}
(B\cdot\nabla)B=\frac{1}{2}\nabla|B|^{2}+(\mathrm{curl}\,B)\,B^{\bot},\qquad B^{\bot}:=(-B_{2},B_{1}),
\end{equation}
where $\mathrm{curl}\,B=\partial_{1}B_{2}-\partial_{2}B_{1}$. Taking the inner product of
\eqref{lorentz.vector} with $\partial_{\alpha}z$ and evaluating on $\Gamma_{t}$, the first
term gives, by the chain rule,
\[
\frac{1}{2}(\nabla|B|^{2})\cdot\partial_{\alpha}z
=\frac{1}{2}\partial_{\alpha}\big(|B(t,z(t,\alpha))|^{2}\big)=\frac{1}{2}\partial_{\alpha}|H|^{2}.
\]
For the second term, since $B|_{\Gamma_{t}}=\lambda\,\partial_{\alpha}z$ we have
$B^{\bot}|_{\Gamma_{t}}=\lambda\,(\partial_{\alpha}z)^{\bot}$, and
$(\partial_{\alpha}z)^{\bot}\cdot\partial_{\alpha}z=0$; hence
\[
(\mathrm{curl}\,B)\,B^{\bot}\cdot\partial_{\alpha}z
=(\mathrm{curl}\,B)\,\lambda\,\big[(\partial_{\alpha}z)^{\bot}\cdot\partial_{\alpha}z\big]=0.
\]
This proves \eqref{lorentz.identity}. Finally \eqref{lorentz.bound} follows from
$\|H\|_{L^{\infty}}+\|\partial_{\alpha}H\|_{L^{\infty}}\leqslant A$, the $C^{1}$-bound furnished by
\eqref{e9} and Lemma \ref{B.trace}.
\end{proof}

\begin{remark}\label{rm:no.induction}
Only the trace of the magnetic field enters the estimates. The Lorentz force reaches the interface
equations through the identity of Lemma \ref{MHD:lorentz}, which requires only the tangency
\eqref{e6.m}, and the magnetic pressure of Proposition \ref{pressure.jump} involves only
$|H|^{2}$; the constraint $\mathrm{div}\,B=0$ is used nowhere, and the induction equation only in
Lemma \ref{B.trace}, where it is what makes the trace bound \eqref{B.trace.bound} follow from the
value of $H$ at time $0$. Once that bound is available, the estimates of Sections \ref{sec:main}
and \ref{sec:singular} are independent of the magnetic evolution: they apply verbatim to any
tangent field whose trace obeys \eqref{B.trace.bound}, whether or not it solves the induction
equation. What the no-splash mechanism depends on is not the magnetic evolution but the
total-derivative structure that the tangency confers on the Lorentz term.
\end{remark}

The next consequence of the tangency condition is the point at which the magnetic field enters
the estimates of this paper. Because $H$ is tangent to $\Gamma_{t}$, the Lorentz force exerts no
tangential traction, and its normal component is the isotropic magnetic pressure
$\tfrac12|H|^{2}$; the stress balance therefore determines the pressure jump exactly, with the
magnetic pressure as one of its terms.

\begin{proposition}\label{pressure.jump}
Let $(z,u_{1},u_{2},B,p_{1},p_{2})$ be an $(A,\rho_{1},\rho_{2})$-admissible solution, let
$\tau=\partial_{\alpha}z/|\partial_{\alpha}z|$ and $\nu=n/|n|$ be the unit tangent and the unit
normal of $\Gamma_{t}$, and let $K$ be the curvature of \eqref{e6}. Then, on $\Gamma_{t}$,
\begin{equation}\label{pressure.jump.eq}
q_{2}-q_{1}=\sigma K-\frac{1}{2}|H|^{2}-2\,\partial_{\tau}(v_{2}\cdot\tau)-2K\,(v_{2}\cdot\nu).
\end{equation}
\end{proposition}
\begin{proof}
Taking the normal component of the stress balance \eqref{e6} and using \eqref{e6.m} gives
\[
q_{2}-q_{1}+\tfrac12|H|^{2}-2\,\nu\cdot D(u_{2})\nu=\sigma K.
\]
In the normal coordinates $(t,s)\mapsto\gamma(t)+s\nu(t)$, where $\gamma$ is the arclength
parameterization of $\Gamma_{t}$, one has $\partial_{\nu}\nu=0$ and hence
$\nu\cdot D(u_{2})\nu=\partial_{\nu}(u_{2}\cdot\nu)$ at $s=0$. Write
$u_{2}=u^{\tau}\tau+u^{\nu}\nu$ on $\Gamma_{t}$, so that $u^{\tau}=v_{2}\cdot\tau$ and
$u^{\nu}=v_{2}\cdot\nu$. The frame relations
are $\partial_{\tau}\tau=-K\nu$ and $\partial_{\tau}\nu=K\tau$. In those coordinates
the metric is $(1+sK)^{2}dt^{2}+ds^{2}$, and the standard formula for the divergence in
orthogonal coordinates gives, at $s=0$,
\[
\mathrm{div}\,X=\partial_{\tau}X^{\tau}+\partial_{\nu}X^{\nu}+KX^{\nu}
\]
for every field $X=X^{\tau}\tau+X^{\nu}\nu$. Applying this to the divergence-free field $u_{2}$
gives
\[
\partial_{\nu}u^{\nu}=-\partial_{\tau}u^{\tau}-Ku^{\nu}
=-\partial_{\tau}(v_{2}\cdot\tau)-K\,(v_{2}\cdot\nu),
\]
and \eqref{pressure.jump.eq} follows. 
\end{proof}

\begin{lemma}\label{nu.omega}
Let $(z,u_{1},u_{2},B,p_{1},p_{2})$ be an $(A,\rho_{1},\rho_{2})$-admissible solution, let
$\tau=\partial_{\alpha}z/|\partial_{\alpha}z|$ and $\nu=n/|n|$ be the unit tangent and the unit
normal of \eqref{e6}, and let $\omega=\mathrm{curl}\,u_{2}$. Then, on $\Gamma_{t}$,
\begin{equation}\label{nu.omega.eq}
\begin{aligned}
\partial_{\nu}\omega
={}&\rho_{2}\,\partial_{t}(v_{2}\cdot\tau)-\rho_{2}\,(v_{2}\cdot\partial_{t}\tau)
+\partial_{\tau}q_{1}+\sigma\,\partial_{\tau}K-\partial_{\tau}|H|^{2}\\
&-2\,\partial_{\tau}^{2}(v_{2}\cdot\tau)-2\,\partial_{\tau}\big[K\,(v_{2}\cdot\nu)\big],
\end{aligned}
\end{equation}
and the identity $(\Delta u_{2})\cdot\tau=\partial_{\nu}\omega$.
\end{lemma}
\begin{proof}
\emph{Step 1: Prandtl's relation.} By the no-slip condition $\partial_{t}z=v_{2}$ of \eqref{e6}
the curve $t\mapsto z(t,\alpha)$ is a trajectory of $u_{2}$, so that the material derivative of
$u_{2}$ along it is $\partial_{t}v_{2}$:
\[
\big(\partial_{t}u_{2}+(u_{2}\cdot\nabla)u_{2}\big)\big|_{x=z(t,\alpha)}
=\frac{d}{dt}u_{2}(t,z(t,\alpha))=\partial_{t}v_{2}(t,\alpha).
\]
Evaluating the momentum equation \eqref{mhd.2} at $x=z(t,\alpha)$ and taking the tangential
component therefore gives
\begin{equation}\label{prandtl.1}
\rho_{2}\,\partial_{t}v_{2}\cdot\tau
=-\partial_{\tau}q_{2}+(\Delta u_{2})\cdot\tau+(B\cdot\nabla)B\cdot\tau.
\end{equation}
Three identities reduce the right-hand side. First, $\mathrm{div}\,u_{2}=0$ gives
$\Delta u_{2}=\nabla\,\mathrm{div}\,u_{2}+\nabla^{\perp}\mathrm{curl}\,u_{2}
=\nabla^{\perp}\omega$, the sign being the one for which
$\nabla^{\perp}f=(-\partial_{2}f,\partial_{1}f)$; since
$\nu=(z_{\alpha,2},-z_{\alpha,1})/|\partial_{\alpha}z|$ one has
$\nabla^{\perp}\omega\cdot\tau=\partial_{\nu}\omega$, and hence
$(\Delta u_{2})\cdot\tau=\partial_{\nu}\omega$. Second, contracting the identity
\eqref{lorentz.vector} with $\tau$ instead of $\partial_{\alpha}z$ gives
$(B\cdot\nabla)B\cdot\tau=\frac12\partial_{\tau}|B|^{2}+(\mathrm{curl}\,B)(B^{\perp}\cdot\tau)$,
and $B^{\perp}\cdot\tau=0$ on $\Gamma_{t}$ because $B|_{\Gamma_{t}}\parallel\tau$ by
\eqref{e6.m}; hence $(B\cdot\nabla)B\cdot\tau=\frac12\partial_{\tau}|H|^{2}$. Third,
$\partial_{t}v_{2}\cdot\tau=\partial_{t}(v_{2}\cdot\tau)-v_{2}\cdot\partial_{t}\tau$.
Substituting the three into \eqref{prandtl.1} yields
\begin{equation}\label{prandtl.2}
\partial_{\nu}\omega=\rho_{2}\,\partial_{t}(v_{2}\cdot\tau)
-\rho_{2}\,(v_{2}\cdot\partial_{t}\tau)+\partial_{\tau}q_{2}-\frac12\partial_{\tau}|H|^{2}.
\end{equation}

\emph{Step 2: the pressure.} Differentiating the pressure jump of Proposition
\ref{pressure.jump} along $\Gamma_{t}$ gives
\[
\partial_{\tau}q_{2}=\partial_{\tau}q_{1}+\sigma\,\partial_{\tau}K-\frac12\partial_{\tau}|H|^{2}
-2\,\partial_{\tau}^{2}(v_{2}\cdot\tau)-2\,\partial_{\tau}\big[K\,(v_{2}\cdot\nu)\big],
\]
and inserting this into \eqref{prandtl.2} gives \eqref{nu.omega.eq}.
\end{proof}

\begin{lemma}\label{curl}
Let $(z,u_{1},u_{2},B,p_{1},p_{2})$ be an $(A,\rho_{1},\rho_{2})$-admissible solution as in
Theorem \ref{thm} and let $\omega_{\Gamma}$ be defined by \eqref{e16}. Put
\begin{equation}\label{curl.not}
a:=|\partial_{\alpha}z|,\qquad
\tau:=\frac{\partial_{\alpha}z}{a},\qquad
u:=\frac{\omega_{\Gamma}}{a},\qquad
\gamma:=\frac{\rho_{1}-\rho_{2}}{\rho_{2}},\qquad
b:=\frac{v_{1}\cdot\partial_{\alpha}z}{a},
\end{equation}
and
\begin{equation}\label{curl.Theta}
\Theta:=(\Delta u_{2})\cdot\partial_{\alpha}z+(B\cdot\nabla)B\cdot\partial_{\alpha}z.
\end{equation}
Then, on $[0,T]\times\mathbb{R}$,
\begin{equation}\label{omega.eq}
\begin{aligned}
\partial_{t}\omega_{\Gamma}
=&-\gamma\,\partial_{t}(v_{1}\cdot\partial_{\alpha}z)
+\frac{\gamma}{2}\,\partial_{\alpha}|v_{1}|^{2}
-\gamma\,\partial_{\alpha}\Big[\frac{\omega_{\Gamma}\,(v_{1}\cdot\partial_{\alpha}z)}{|\partial_{\alpha}z|^{2}}\Big]
-\frac{1}{2}\,\partial_{\alpha}\Big[\frac{\omega_{\Gamma}^{2}}{|\partial_{\alpha}z|^{2}}\Big]\\
&-\frac{1}{\rho_{2}}\partial_{\alpha}(q_{1}-q_{2})
-\frac{1}{\rho_{2}}\Theta.
\end{aligned}
\end{equation}
Equivalently, $u$ satisfies the
quasilinear parabolic equation
\begin{equation}\label{u.eq}
\partial_{t}u=\frac{2}{\rho_{2}a^{2}}\,u''-c\,u'+d\,u+e,
\end{equation}
where $':=\partial_{\alpha}$ and
\begin{equation}\label{u.coeff}
c:=\frac{u}{a}+\frac{\gamma b}{a}+\frac{2a'}{\rho_{2}a^{3}},
\qquad
d:=-\frac{\gamma b'}{a}-\frac{\partial_{t}a}{a},
\end{equation}
\begin{equation}\label{u.e}
\begin{aligned}
e:=&-\frac{\gamma}{a}\partial_{t}(v_{1}\cdot\partial_{\alpha}z)
+\frac{\gamma}{2a}\partial_{\alpha}|v_{1}|^{2}\\
&+\frac{1}{\rho_{2}a}\Big[\sigma\,\partial_{\alpha}K-\frac{1}{2}\partial_{\alpha}|H|^{2}
-\partial_{\alpha}\Big(\frac{2}{a}\partial_{\alpha}(v_{1}\cdot\tau)\Big)
-2\partial_{\alpha}\big(K\,(v_{1}\cdot\nu)\big)\Big]\\
&-\frac{1}{\rho_{2}}(\Delta u_{2})\cdot\tau
-\frac{1}{2\rho_{2}a}\partial_{\alpha}|H|^{2}.
\end{aligned}
\end{equation}
The two occurrences of $\partial_{\alpha}|H|^{2}$ in \eqref{u.e}, both with coefficient
$\frac{1}{2\rho_{2}a}$, have different origins: the Lorentz force produces one through the
identity of Lemma \ref{MHD:lorentz}, the magnetic pressure in Proposition \ref{pressure.jump} the
other. They are combined at the end of the proof into the single term
$-\frac{1}{\rho_{2}a}\partial_{\alpha}|H|^{2}$ of $f$.
Also, writing $\tilde c:=\gamma b/a+2a'/(\rho_{2}a^{3})$, so that $c=u/a+\tilde c$, the
diffusion coefficient is strictly positive and bounded, and the coefficient $\tilde c$ is bounded
together with its first derivative:
\begin{equation}\label{u.coeff.bound}
\frac{1}{a}\in[(A+1)^{-1},A],\qquad
\frac{2}{\rho_{2}a^{2}}\in\Big[\frac{2}{\rho_{2}(A+1)^{2}},\frac{2A^{2}}{\rho_{2}}\Big],
\qquad |\tilde c|_{1}\leqslant C(A,\sigma,\rho_{1},\rho_{2}),
\end{equation}
where $|f|_{1}:=\|f\|_{L^{\infty}([0,T]\times\mathbb{R})}
+\|\partial_{\alpha}f\|_{L^{\infty}([0,T]\times\mathbb{R})}$. The coefficient
$d=-\gamma b'/a-\partial_{t}a/a$ does not enter \eqref{u.coeff.bound}: by
\eqref{lin.a} it equals $-\gamma b'/a-(\tau\cdot\partial_{\alpha}v_{1}-\partial_{\alpha}u)/a$, so
that $\|d\|_{L^{\infty}}$ requires a bound on $\|\partial_{\alpha}u\|_{L^{\infty}}$, which is not
available at this stage. The forcing $e$ likewise contains the term
$-\gamma a^{-1}\partial_{t}(v_{1}\cdot\partial_{\alpha}z)=-\gamma a^{-1}\partial_{t}(ab)$, whose
boundedness uses \eqref{lin.a} as well. The quasilinear form \eqref{u.eq} is therefore recorded only
for orientation; the estimates below use the \emph{linear} form
\eqref{u.eq.lin}, whose coefficients $\Lambda$ and $f$ are bounded by \eqref{e8}, \eqref{e9},
\eqref{e11.nu} and \eqref{B.trace.bound} alone (Lemma \ref{curl.lin}).
\end{lemma}

\begin{proof}
Multiplying \eqref{e14} by $\rho_{1}$ and \eqref{e15} by $\rho_{2}$ and subtracting, the three
derivative terms combine into
\[
\partial_{t}\big[(\rho_{1}v_{1}-\rho_{2}v_{2})\cdot\partial_{\alpha}z\big]
-\partial_{\alpha}\big[(\rho_{1}v_{1}-\rho_{2}v_{2})\cdot\partial_{t}z\big]
+\frac{1}{2}\partial_{\alpha}\big[\rho_{1}|v_{1}|^{2}-\rho_{2}|v_{2}|^{2}\big],
\]
and the remaining terms give $-\partial_{\alpha}(q_{1}-q_{2})-\Theta$, the two pressure
contributions canceling with the coefficients $\rho_{1}\rho_{1}^{-1}=\rho_{2}\rho_{2}^{-1}=1$
and the terms of \eqref{curl.Theta} surviving with the sign shown. We rewrite the
left-hand side in terms of $v_{1}$ and $\omega_{\Gamma}$ alone. By \eqref{e16} and the first
condition of \eqref{e6},
\begin{equation}\label{curl.jump}
v_{1}-v_{2}=\frac{\omega_{\Gamma}}{|\partial_{\alpha}z|^{2}}\,\partial_{\alpha}z,
\end{equation}
so that $v_{2}\cdot\partial_{\alpha}z=v_{1}\cdot\partial_{\alpha}z-\omega_{\Gamma}$ and
$|v_{1}-v_{2}|^{2}=\omega_{\Gamma}^{2}/|\partial_{\alpha}z|^{2}$. Hence
\[
(\rho_{1}v_{1}-\rho_{2}v_{2})\cdot\partial_{\alpha}z
=(\rho_{1}-\rho_{2})(v_{1}\cdot\partial_{\alpha}z)+\rho_{2}\omega_{\Gamma},
\]
which gives the first two terms of the right-hand side of \eqref{omega.eq}. For the remaining
two we use the no-slip condition $\partial_{t}z=v_{2}$ of \eqref{e6} and \eqref{curl.jump},
which gives $v_{2}=v_{1}-\omega_{\Gamma}\partial_{\alpha}z/|\partial_{\alpha}z|^{2}$ and
therefore
$v_{1}\cdot v_{2}=|v_{1}|^{2}-\omega_{\Gamma}(v_{1}\cdot\partial_{\alpha}z)/|\partial_{\alpha}z|^{2}$.
A direct computation then yields
\[
\begin{aligned}
&-\partial_{\alpha}\big[(\rho_{1}v_{1}-\rho_{2}v_{2})\cdot\partial_{t}z\big]
+\frac{1}{2}\partial_{\alpha}\big[\rho_{1}|v_{1}|^{2}-\rho_{2}|v_{2}|^{2}\big]\\
&=-\rho_{1}\partial_{\alpha}(v_{1}\cdot v_{2})
+\rho_{2}\partial_{\alpha}|v_{2}|^{2}
+\frac{\rho_{1}}{2}\partial_{\alpha}|v_{1}|^{2}
-\frac{\rho_{2}}{2}\partial_{\alpha}|v_{2}|^{2}\\
&=-\rho_{2}\partial_{\alpha}(v_{1}\cdot v_{2})
+\frac{\rho_{2}}{2}\partial_{\alpha}\big[|v_{1}|^{2}+|v_{2}|^{2}\big]
+(\rho_{1}-\rho_{2})\partial_{\alpha}\Big[\frac{1}{2}|v_{1}|^{2}-v_{1}\cdot v_{2}\Big]\\
&=\frac{\rho_{2}}{2}\partial_{\alpha}|v_{1}-v_{2}|^{2}
+(\rho_{1}-\rho_{2})\partial_{\alpha}\Big[-\frac{1}{2}|v_{1}|^{2}
+\frac{\omega_{\Gamma}\,(v_{1}\cdot\partial_{\alpha}z)}{|\partial_{\alpha}z|^{2}}\Big],
\end{aligned}
\]
the last identity using $|v_{1}-v_{2}|^{2}=\omega_{\Gamma}^{2}/|\partial_{\alpha}z|^{2}$ and
the formula for $v_{1}\cdot v_{2}$ above. Dividing by $\rho_{2}$ and collecting gives
\eqref{omega.eq}.

To obtain \eqref{u.eq}, we divide \eqref{omega.eq} by $a$ and use
$\omega_{\Gamma}=au$, $\omega_{\Gamma}/|\partial_{\alpha}z|^{2}=u/a$ and
$\frac{1}{a}\partial_{t}\omega_{\Gamma}=\partial_{t}u+\frac{\partial_{t}a}{a}u$. The Burgers
term becomes $-\frac{1}{2a}\partial_{\alpha}[u^{2}]=-\frac{u}{a}u'$. For the pressure term we
use Proposition \ref{pressure.jump}: by \eqref{curl.jump} one has
$v_{2}=v_{1}-u\,\tau$, so that $v_{2}\cdot\tau=v_{1}\cdot\tau-u$ and
$v_{2}\cdot\nu=v_{1}\cdot\nu$, and therefore
\begin{equation}\label{curl.qjump}
q_{2}-q_{1}=\sigma K-\frac{1}{2}|H|^{2}-\frac{2}{a}\partial_{\alpha}(v_{1}\cdot\tau)+\frac{2u'}{a}-2K\,(v_{1}\cdot\nu).
\end{equation}
Since $-\rho_{2}^{-1}\partial_{\alpha}(q_{1}-q_{2})=\rho_{2}^{-1}\partial_{\alpha}(q_{2}-q_{1})$,
the contribution of the pressure is
$\rho_{2}^{-1}\partial_{\alpha}(2u'/a)=\frac{2}{\rho_{2}a^{2}}u''-\frac{2a'}{\rho_{2}a^{3}}u'$
plus the bounded terms of \eqref{u.e}. Finally, by Lemma \ref{MHD:lorentz} and
$(\Delta u_{2})\cdot\partial_{\alpha}z=a\,(\Delta u_{2})\cdot\tau$,
\[
-\frac{1}{\rho_{2}a}\Theta
=-\frac{1}{\rho_{2}}(\Delta u_{2})\cdot\tau-\frac{1}{2\rho_{2}a}\partial_{\alpha}|H|^{2},
\]
which is the last line of \eqref{u.e}. Collecting the terms gives \eqref{u.coeff} and
\eqref{u.e}. The bound for $\tilde c$ in \eqref{u.coeff.bound} follows from \eqref{e8}
and \eqref{e9}, and so does the range of $1/a$ and of $2/(\rho_{2}a^{2})$, which is
$1/A\leqslant a\leqslant A+1$, that is, \eqref{e8} together with
$|z_{\alpha}-(1,0)|\leqslant A$. As noted in the statement, neither $d$ nor $|e|_{1}$ is bounded
here: by \eqref{lin.a} both involve $\partial_{\alpha}u$, which \eqref{e8} and \eqref{e9} do not
control.
\end{proof}

\begin{lemma}\label{curl.lin}
Let $(z,u_{1},u_{2},B,p_{1},p_{2})$ be an $(A,\rho_{1},\rho_{2})$-admissible solution, let
$a,u,\tau,\nu,\gamma,b$ be as in \eqref{curl.not}, let $K$ be the curvature of \eqref{e6}, and
let
\[
\partial_{s}:=\frac{1}{a}\,\partial_{\alpha}
\]
denote differentiation with respect to arclength on $\Gamma_{t}$. Then \eqref{omega.eq} is
equivalent to the {\em linear} parabolic equation
\begin{equation}\label{u.eq.lin}
\partial_{t}u=\frac{2}{\rho_{2}}\,\partial_{s}^{2}u+\Lambda\,u+f,
\end{equation}
where
\begin{equation}\label{u.lin.coeff}
\begin{aligned}
\Lambda&:=-\frac{\tau\cdot\partial_{\alpha}v_{1}}{a}
-\gamma\,\frac{\partial_{\alpha}b}{a}
-\gamma\,K\,(v_{1}\cdot\nu),\\
f&:=-\gamma\,(\partial_{t}v_{1})\cdot\tau
+\frac{1}{\rho_{2}a}\Big[\sigma\,\partial_{\alpha}K
-\partial_{\alpha}\Big(\frac{2}{a}\,\partial_{\alpha}b\Big)
-2\,\partial_{\alpha}\big(K\,(v_{1}\cdot\nu)\big)\Big]\\
&\qquad-\frac{1}{\rho_{2}}(\Delta u_{2})\cdot\tau
-\frac{1}{\rho_{2}a}\,\partial_{\alpha}|H|^{2}.
\end{aligned}
\end{equation}
Also, in view of \eqref{e8}, \eqref{e9}, \eqref{e11.nu} and the trace bound
\eqref{B.trace.bound}, writing
$|g|_{1}:=\|g\|_{L^{\infty}([0,T]\times\mathbb{R})}
+\|\partial_{\alpha}g\|_{L^{\infty}([0,T]\times\mathbb{R})}$,
\begin{equation}\label{u.lin.bound}
|\Lambda|_{1}+|f|_{1}\leqslant C(A,\sigma,\rho_{1},\rho_{2}).
\end{equation}
\end{lemma}

\begin{proof}
We record three consequences of the kinematic conditions \eqref{e6}, all obtained by
writing tangential fields in the orthonormal frame $(\tau,\nu)$, whose derivatives are the frame
relations $\partial_{\alpha}\tau=-aK\nu$ and $\partial_{\alpha}\nu=aK\tau$ of Proposition
\ref{pressure.jump}. Write $':=\partial_{\alpha}$. By the no-slip condition
$\partial_{t}z=v_{2}$ and by \eqref{curl.jump}, which says $v_{2}=v_{1}-u\tau$, we have
$\partial_{t}z=v_{1}-u\tau$ throughout.

\emph{(i) Stretching of arclength.} Since $a^{2}=\partial_{\alpha}z\cdot\partial_{\alpha}z$,
\[
\partial_{t}a=\tau\cdot\partial_{\alpha}\partial_{t}z
=\tau\cdot\big(\partial_{\alpha}v_{1}-u'\tau-u\,\partial_{\alpha}\tau\big)
=\tau\cdot\partial_{\alpha}v_{1}-u',
\]
because $\tau\cdot\tau=1$ and $\tau\cdot\partial_{\alpha}\tau=0$; that is,
\begin{equation}\label{lin.a}
\partial_{t}a=\tau\cdot\partial_{\alpha}v_{1}-\partial_{\alpha}u.
\end{equation}
Thus the rate of stretching of arclength is the tangential velocity gradient of the inviscid
phase \emph{minus} $\partial_{\alpha}u$.

\emph{(ii) Rotation of the tangent.} From $\tau=\partial_{\alpha}z/a$, together with
\eqref{lin.a} and $\partial_{\alpha}\tau=-aK\nu$,
\[
\partial_{t}\tau=\frac{\partial_{\alpha}\partial_{t}z}{a}-\frac{\partial_{t}a}{a}\tau
=\frac{\partial_{\alpha}v_{1}}{a}+K\,u\,\nu
-\frac{\tau\cdot\partial_{\alpha}v_{1}}{a}\,\tau.
\]

\emph{(iii) Time derivative of $b$.} Since $b=v_{1}\cdot\tau$ (indeed
$b=(v_{1}\cdot\partial_{\alpha}z)/a$), the preceding relation gives
\begin{equation}\label{lin.b}
\partial_{t}b=(\partial_{t}v_{1})\cdot\tau+\frac{v_{1}\cdot\partial_{\alpha}v_{1}}{a}
+K\,u\,(v_{1}\cdot\nu)-\frac{\tau\cdot\partial_{\alpha}v_{1}}{a}\,b.
\end{equation}
Finally, decomposing in the orthonormal frame as
$\partial_{\alpha}v_{1}=(\tau\cdot\partial_{\alpha}v_{1})\,\tau+\mu\,\nu$ with
$\mu:=\partial_{\alpha}v_{1}\cdot\nu$, one has
\begin{equation}\label{lin.decomp}
v_{1}\cdot\partial_{\alpha}v_{1}=b\,(\tau\cdot\partial_{\alpha}v_{1})+(v_{1}\cdot\nu)\,\mu,
\qquad
\tfrac12\,\partial_{\alpha}|v_{1}|^{2}=b\,(\tau\cdot\partial_{\alpha}v_{1})+(v_{1}\cdot\nu)\,\mu.
\end{equation}

Now divide \eqref{omega.eq} by $a$, expanding its left-hand side as
$\partial_{t}\omega_{\Gamma}/a=\partial_{t}u+(\partial_{t}a/a)u$, and rewrite its right-hand side
using $\partial_{t}(v_{1}\cdot\partial_{\alpha}z)=\partial_{t}(ab)$,
$\omega_{\Gamma}(v_{1}\cdot\partial_{\alpha}z)/|\partial_{\alpha}z|^{2}=ub$ and
$\omega_{\Gamma}^{2}/|\partial_{\alpha}z|^{2}=u^{2}$. This gives
\[
\partial_{t}u=-\frac{\partial_{t}a}{a}u-\gamma\frac{\partial_{t}a}{a}b-\gamma\partial_{t}b
+\frac{\gamma}{2a}\partial_{\alpha}|v_{1}|^{2}
-\gamma\frac{b}{a}u'-\gamma\frac{b'}{a}u-\frac{u}{a}u'
+\frac{1}{\rho_{2}a}\partial_{\alpha}(q_{2}-q_{1})-\frac{1}{\rho_{2}a}\Theta.
\]
Substituting \eqref{lin.a} into the first two terms produces $+\frac{u}{a}u'$ and
$+\gamma\frac{b}{a}u'$, which cancel the two terms $-\frac{u}{a}u'$ (coming from the Burgers
term $-\frac12\partial_{\alpha}[u^{2}]$) and $-\gamma\frac{b}{a}u'$ (coming from the density
term $-\gamma\partial_{\alpha}[bu]$) exactly:
\[
\partial_{t}u=-\frac{\tau\cdot\partial_{\alpha}v_{1}}{a}u
-\gamma\frac{\tau\cdot\partial_{\alpha}v_{1}}{a}b-\gamma\partial_{t}b
+\frac{\gamma}{2a}\partial_{\alpha}|v_{1}|^{2}-\gamma\frac{b'}{a}u
+\frac{1}{\rho_{2}a}\partial_{\alpha}(q_{2}-q_{1})-\frac{1}{\rho_{2}a}\Theta.
\]
Substituting \eqref{lin.b} and \eqref{lin.decomp} in turn, the terms carrying
$b\,(\tau\cdot\partial_{\alpha}v_{1})$ and those carrying $(v_{1}\cdot\nu)\mu$ cancel in pairs as
well, leaving
\[
\partial_{t}u=-\frac{\tau\cdot\partial_{\alpha}v_{1}}{a}u-\gamma\frac{b'}{a}u
-\gamma K\,(v_{1}\cdot\nu)\,u-\gamma(\partial_{t}v_{1})\cdot\tau
+\frac{1}{\rho_{2}a}\partial_{\alpha}(q_{2}-q_{1})-\frac{1}{\rho_{2}a}\Theta.
\]
It remains to evaluate the last two groups. By Proposition \ref{pressure.jump} with
$v_{2}\cdot\tau=v_{1}\cdot\tau-u=b-u$ and $v_{2}\cdot\nu=v_{1}\cdot\nu$ (both from
\eqref{curl.jump}),
\[
q_{2}-q_{1}=\sigma K-\tfrac12|H|^{2}-\frac{2}{a}(b'-u')-2K\,(v_{1}\cdot\nu),
\]
and since
$\frac{1}{\rho_{2}a}\partial_{\alpha}\big(\frac{2}{a}u'\big)
=\frac{2}{\rho_{2}a^{2}}u''-\frac{2a'}{\rho_{2}a^{3}}u'
=\frac{2}{\rho_{2}}\partial_{s}^{2}u$, we get
\[
\frac{1}{\rho_{2}a}\partial_{\alpha}(q_{2}-q_{1})
=\frac{2}{\rho_{2}}\partial_{s}^{2}u
+\frac{1}{\rho_{2}a}\Big[\sigma\,\partial_{\alpha}K
-\partial_{\alpha}\Big(\frac{2}{a}\partial_{\alpha}b\Big)
-2\,\partial_{\alpha}\big(K\,(v_{1}\cdot\nu)\big)\Big]
-\frac{1}{2\rho_{2}a}\partial_{\alpha}|H|^{2},
\]
while Lemma \ref{MHD:lorentz} and the definition \eqref{curl.Theta} give
\[
-\frac{1}{\rho_{2}a}\Theta
=-\frac{1}{\rho_{2}}(\Delta u_{2})\cdot\tau-\frac{1}{2\rho_{2}a}\partial_{\alpha}|H|^{2}.
\]
Adding the last three relations gives \eqref{u.eq.lin} and \eqref{u.lin.coeff}. The magnetic
field contributes the two identical terms $-\frac{1}{2\rho_{2}a}\partial_{\alpha}|H|^{2}$, one
produced by the Lorentz force through Lemma \ref{MHD:lorentz} and one by the magnetic pressure
through Proposition \ref{pressure.jump}; they combine into the single term
$-\frac{1}{\rho_{2}a}\partial_{\alpha}|H|^{2}$ appearing in $f$.

Finally, \eqref{u.lin.bound} holds because neither $\Lambda$ nor $f$ contains $u$: each is a
polynomial expression in $\partial_{\alpha}^{k}z$ with $k\leqslant4$, in
$\partial_{\alpha}^{k}v_{1}$ with $k\leqslant3$ together with $\partial_{t}v_{1}$ and
$\partial_{\alpha}\partial_{t}v_{1}$, in $\partial_{\alpha}^{k}H$ with
$k\leqslant2$, in the factor $1/a$, and (through the last line of \eqref{u.lin.coeff}) in
$(\Delta u_{2})\cdot\tau$ together with its $\alpha$-derivative. All of these are bounded by
\eqref{e8}, \eqref{e9}, \eqref{e11.nu} and the trace bound \eqref{B.trace.bound}. 
\end{proof}

\begin{lemma}\label{om.Linf}
Let $(z,u_{1},u_{2},B,p_{1},p_{2})$ be an $(A,\rho_{1},\rho_{2})$-admissible solution. There is
a constant $C_{1}=C_{1}(A,T,\sigma,\rho_{1},\rho_{2})$ such that
\begin{equation}\label{e19}
\|\omega_{\Gamma}\|_{L^{\infty}([0,T]\times\mathbb{R})}\leqslant C_{1}.
\end{equation}
\end{lemma}
\begin{proof}
Since $\omega_{\Gamma}=au$ and $a\leqslant A+1$ by \eqref{e8}, it suffices to bound
$\|u\|_{L^{\infty}([0,T]\times\mathbb{R})}$. The function $u$ is smooth enough for the
classical maximum principle. Indeed $u=\omega_{\Gamma}/a$ with
$\omega_{\Gamma}=(v_{1}-v_{2})\cdot\partial_{\alpha}z$ and
$v_{2}=u_{2}(t,z(t,\alpha))$; since $u_{2}\in C^{3}(\overline{\Omega_{2}(t)})$ and
$z(t,\cdot)\in C^{4}$ in $\alpha$, the composition $v_{2}$ is of class $C^{3}$ in $\alpha$,
whence $\omega_{\Gamma}\in C^{3}$ and $u\in C^{3}$ in $\alpha$. The linear equation
\eqref{u.eq.lin} therefore holds in the classical sense, and $\partial_{t}u$ is continuous;
the envelope argument below is accordingly legitimate. The initial datum of \eqref{e8.init} is used for the bound only, not for the regularity.

Put $\lambda:=\|\Lambda\|_{L^{\infty}([0,T]\times\mathbb{R})}$ and
$\varphi:=\|f\|_{L^{\infty}([0,T]\times\mathbb{R})}$; both are finite by \eqref{u.lin.bound}. Being
$2\pi$-periodic and smooth in $\alpha$, the function $u(t,\cdot)$ attains its supremum and its
infimum on $[0,2\pi]$ for each $t$. Write
\[
\begin{aligned}
M(t)&:=\max_{\alpha}u(t,\alpha),\qquad m(t):=\min_{\alpha}u(t,\alpha),\\
N(t)&:=\|u(t)\|_{L^{\infty}(\mathbb{R})}=\max\{M(t),-m(t)\}.
\end{aligned}
\]
By Danskin's theorem on the envelope of an extremum over a compact set that depends on a parameter, $M$ and $m$ are Lipschitz, and for almost every $t$
\[
M'(t)=\partial_{t}u(t,\alpha_{t})\quad\text{for some }\alpha_{t}\text{ with }
u(t,\alpha_{t})=M(t),
\]
the analogous statement holding for $m$ and $-m$ at some minimizing point. At a maximizing point
$\alpha_{t}$ we have $\partial_{\alpha}u(t,\alpha_{t})=0$, and since
$\partial_{s}^{2}u=\frac{1}{a}\partial_{\alpha}\big(\frac{1}{a}\partial_{\alpha}u\big)
=\frac{u''}{a^{2}}-\frac{a'u'}{a^{3}}$, this gives
\[
\partial_{s}^{2}u(t,\alpha_{t})
=\frac{\partial_{\alpha}^{2}u(t,\alpha_{t})}{a(t,\alpha_{t})^{2}}\leqslant0.
\]
Hence, evaluating \eqref{u.eq.lin} at $(t,\alpha_{t})$ and discarding the non-positive leading
term,
\[
M'(t)\leqslant\Lambda(t,\alpha_{t})\,M(t)+f(t,\alpha_{t})
\leqslant\lambda\,N(t)+\varphi,
\]
where the last step uses $M(t)\leqslant N(t)$ together with $\lambda\geqslant0$; if $M(t)<0$
the inequality is immediate. Applying the same argument to $-u$, which solves
\eqref{u.eq.lin} with $f$ replaced by $-f$ because the equation is linear, bounds $-m'(t)$ by
$\lambda N(t)+\varphi$ as well. Since $N=\max\{M,-m\}$, we conclude that
\[
N'(t)\leqslant\lambda\,N(t)+\varphi\qquad\text{for a.e. }t\in[0,T],
\]
and Gronwall's inequality gives
\begin{equation}\label{e19.gronwall}
\|u(t)\|_{L^{\infty}(\mathbb{R})}\leqslant e^{\lambda T}\big(\|u(0)\|_{L^{\infty}(\mathbb{R})}+\varphi T\big)
\qquad\text{for every }t\in[0,T].
\end{equation}
For the initial value we use hypothesis \eqref{e8.init}, which gives
$\|u(0)\|_{L^{\infty}(\mathbb{R})}\leqslant A$; as $\omega_{\Gamma}=au$ and
$a\leqslant A+1$ by \eqref{e8}, this also gives $|\omega_{\Gamma}(0)|\leqslant A(A+1)$. 
Combining this with \eqref{e19.gronwall} and
$\|\omega_{\Gamma}\|_{L^{\infty}}\leqslant(A+1)\|u\|_{L^{\infty}}$ proves \eqref{e19}, with
$C_{1}=(A+1)e^{\lambda T}(A+\varphi T)$.
\end{proof}

\begin{lemma}\label{om.C1.est}
Let $(z,u_{1},u_{2},B,p_{1},p_{2})$ be an $(A,\rho_{1},\rho_{2})$-admissible solution. There is a
constant $C_{2}=C_{2}(A,T,\sigma,\rho_{1},\rho_{2})$ such that
\begin{equation}\label{om.C1}
\sup_{t\in[0,T]}\|\omega_{\Gamma}(t,\cdot)\|_{C^{1}(\mathbb{R})}\leqslant C_{2}
\end{equation}
and
\begin{equation}\label{e11.trace.der}
\sup_{t\in[0,T]}\|v_{2}(t,\cdot)\|_{C^{1}(\mathbb{R})}\leqslant C_{2}.
\end{equation}
\end{lemma}

\begin{proof}
Write $\zeta:=\partial_{\alpha}u$; the letter $\zeta$ (rather than $w$) is used because $w$ denotes the gradient-harmonic remainder of Definition \ref{adm}(vi). The solution is smooth, so \eqref{u.eq.lin} may be differentiated with respect to $\alpha$; the estimate of the resulting equation uses the bound \eqref{u.lin.bound} on the coefficients and on their first derivatives. Differentiating and using
\[
\partial_{\alpha}\partial_{s}^{2}u
=\partial_{s}^{2}\zeta-\frac{2a'}{a^{3}}\,\zeta'
-\Big(\frac{a''}{a^{3}}-\frac{3(a')^{2}}{a^{4}}\Big)\zeta
\]
gives
\begin{equation}\label{zeta.eq.lin}
\partial_{t}\zeta=\frac{2}{\rho_{2}}\,\partial_{s}^{2}\zeta+c_{1}\,\zeta'+c_{0}\,\zeta+g,
\end{equation}
where
\[
c_{1}:=-\frac{4a'}{\rho_{2}a^{3}},\qquad
c_{0}:=\Lambda-\frac{2}{\rho_{2}}\Big(\frac{a''}{a^{3}}-\frac{3(a')^{2}}{a^{4}}\Big),\qquad
g:=(\partial_{\alpha}\Lambda)\,u+\partial_{\alpha}f.
\]
By \eqref{e8} and \eqref{u.lin.bound},
\[
\|c_{1}\|_{L^{\infty}([0,T]\times\mathbb{R})}+\|c_{0}\|_{L^{\infty}([0,T]\times\mathbb{R})}
\leqslant C(A,\sigma,\rho_{1},\rho_{2})=:\kappa_{0},
\]
and by Lemma \ref{om.Linf} together with \eqref{u.lin.bound},
\[
\|g\|_{L^{\infty}([0,T]\times\mathbb{R})}
\leqslant\|\partial_{\alpha}\Lambda\|_{L^{\infty}}\|u\|_{L^{\infty}([0,T]\times\mathbb{R})}
+\|\partial_{\alpha}f\|_{L^{\infty}}
\leqslant C(A,\sigma,\rho_{1},\rho_{2})\,(C_{1}+1)=:G_{0},
\]
with $C_{1}$ the constant of Lemma \ref{om.Linf}. Equation \eqref{zeta.eq.lin} is again linear,
with coefficients independent of $\zeta$, so the argument of Lemma \ref{om.Linf} applies verbatim:
at a point where $\zeta(t,\cdot)$ attains its supremum one has $\zeta'=\partial_{\alpha}\zeta=0$ and hence
$\partial_{s}^{2}\zeta=\partial_{\alpha}^{2}\zeta/a^{2}\leqslant0$, so that Danskin's theorem and
\eqref{zeta.eq.lin} yield, for $N(t):=\|\zeta(t)\|_{L^{\infty}(\mathbb{R})}$,
\[
N'(t)\leqslant\kappa_{0}\,N(t)+G_{0}\qquad\text{for a.e. }t\in[0,T].
\]
The inequality at a {\em negative} minimum follows by applying the same reasoning to
$-\zeta$, which solves \eqref{zeta.eq.lin} with $g$ replaced by $-g$; linearity is used here, and
the bound obtained is two-sided. Gronwall's inequality gives
\begin{equation}\label{zeta.gronwall}
\|\zeta(t)\|_{L^{\infty}(\mathbb{R})}\leqslant e^{\kappa_{0} T}\big(\|\zeta(0)\|_{L^{\infty}(\mathbb{R})}+G_{0}T\big)
\qquad\text{for every }t\in[0,T].
\end{equation}
The initial value is bounded by hypothesis \eqref{e8.init}, which asserts exactly
$\|u(0)\|_{L^{\infty}}+\|\partial_{\alpha}u(0)\|_{L^{\infty}}\leqslant A$, hence
$\|\zeta(0)\|_{L^{\infty}}\leqslant A$. Enlarging the constant if necessary, this gives
\eqref{zeta.gronwall} with $C_{2}=C_{2}(A,T,\sigma,\rho_{1},\rho_{2})$.

Since $\omega_{\Gamma}=au$ and $a$ is of class $C^{1}$ in $\alpha$ by \eqref{e8}, we have
$\partial_{\alpha}\omega_{\Gamma}=a'u+a\zeta$, whence
$\|\partial_{\alpha}\omega_{\Gamma}\|_{L^{\infty}}\leqslant C_{2}$; together with \eqref{e19}
this proves \eqref{om.C1}. Finally \eqref{curl.jump} gives $v_{2}=v_{1}-u\tau$, so that, $\tau$
being of class $C^{3}$ in $\alpha$ by \eqref{e8},
\[
\sup_{t\in[0,T]}\|v_{2}(t,\cdot)\|_{C^{1}(\mathbb{R})}
\leqslant\sup_{t\in[0,T]}\|v_{1}(t,\cdot)\|_{C^{1}(\mathbb{R})}
+C(A)\sup_{t\in[0,T]}\|u(t,\cdot)\|_{C^{1}(\mathbb{R})}
\leqslant C_{2},
\]
which is \eqref{e11.trace.der}, the last bound being \eqref{e19} together with
\eqref{zeta.gronwall} applied to $\zeta=\partial_{\alpha}u$.
\end{proof}

\begin{lemma}\label{grad.euler}
Let $(z,u_{1},u_{2},B,p_{1},p_{2})$ be an $(A,\rho_{1},\rho_{2})$-admissible solution whose
chord-arc functional is positive on $[0,T]$, that is,
$\inf_{t\in[0,T],\,\alpha\ne\beta}F(z)(t,\alpha,\beta)>0$. Then
there is $A_{*}=A_{*}(A,T,\sigma,\rho_{1},\rho_{2})\geqslant A$ such that
\begin{equation}\label{e9.int}
\sup_{t\in[0,T]}\|\nabla u_{1}(t)\|_{L^{\infty}(\Omega_{1}(t))}\leqslant A_{*}.
\end{equation}
\end{lemma}
\begin{proof}
The positivity of the chord-arc functional is used only qualitatively, to make the quantities
$M_{t}$ and $\kappa_{*}$ below finite; the bound \eqref{e9.int} itself does not depend on it. In
the proof of Theorem \ref{thm} the lemma is applied with $T$ replaced by $\bar{t}$, and there the
required positivity is \eqref{g.range}.

We continue the argument of Lemma \ref{grad.euler.bdry}: as there, $u_{1}=\nabla\psi$ with
$\Delta\psi=0$ in $\Omega_{1}(t)$, so that every entry $h$ of $\nabla u_{1}=D^{2}\psi$ is
harmonic in $\Omega_{1}(t)$, and \eqref{e9.bdry} gives
$\sup_{\Gamma_{t}}|\nabla u_{1}|\leqslant C(A)$.

\emph{The maximum principle.} Fix $t$ and an entry $h$ of $\nabla u_{1}$; $h$ is
harmonic in $\Omega_{1}(t)$ and continuous on $\overline{\Omega_{1}(t)}$, because
$u_{1}\in C^{3}(\overline{\Omega_{1}(t)})$. We establish the finiteness needed to justify a
maximum-principle argument on this horizontally unbounded domain. By the structural representation \eqref{e9.bs} one has
$u_{1}=c+w+U_{1}+\Phi_{2}$ on $\Omega_{1}(t)$, where $U_{1}$ and $\Phi_{2}$ denote the
interface and the bulk Biot--Savart terms respectively.
The first two summands are bounded by $A$ in view of \eqref{e9.norm}, and
$\|\omega_{\Gamma}\|_{L^{\infty}}+\|\omega_{\Gamma}\|_{C^{0,1}}\leqslant C(A)$ by
\eqref{e19} and \eqref{om.C1}; it remains to bound the two Biot--Savart terms, which we do by an explicit
periodization.

\emph{The periodized kernel.} Since $z(t,\alpha)-(\alpha,0)$ is $2\pi$-periodic in $\alpha$, the
curve $\Gamma_{t}$ and the region $\Omega_{2}(t)$ are invariant under $x_{1}\mapsto x_{1}+2\pi$,
and $\omega_{\Gamma}$, $\omega$ are $2\pi$-periodic in their first variable. The principal value
prescribed after \eqref{e9.bs} is therefore exactly this periodization: an integral over
$\mathbb{R}$, or over $\Omega_{2}(t)$, against a $2\pi$-periodic density equals the integral over
one period against
\begin{equation}\label{P.kernel}
\begin{aligned}
P(h,t)&:=\mathop{\mathrm{sym}}\sum_{k\in\mathbb Z}\frac{(-h,\ t-2\pi k)}{(t-2\pi k)^{2}+h^{2}}\\
&=\frac12\Big(-\frac{\sinh h}{\cosh h-\cos t},\ \frac{\sin t}{\cosh h-\cos t}\Big),
\qquad h\ne0,\ t\in\mathbb R,
\end{aligned}
\end{equation}
where $\mathrm{sym}\sum$ denotes the limit of the sums over $|k|\leqslant N$. The first component is
absolutely summable; the second has $O(k^{-1})$ terms, but its symmetric pairs are $O(k^{-2})$.
The closed form obtained above is the standard partial-fraction identity for the cotangent and fixes
precisely the principal-value convention prescribed after \eqref{e9.bs}. The components satisfy
\begin{equation}\label{P.est}
\begin{aligned}
\int_{0}^{2\pi}|P_{1}(h,t)|\,dt&=\pi,\\
\int_{0}^{2\pi}|P_{2}(h,t)|\,dt
&=\log\frac{\cosh h+1}{\cosh h-1}=2\log\coth\frac{|h|}{2}\\
&\leqslant
\begin{cases}
C\big(1+\log\tfrac1{|h|}\big),&0<|h|\leqslant1,\\[1mm]
C,&|h|>1.
\end{cases}
\end{aligned}
\end{equation}
the first identity follows from $\int_{0}^{2\pi}\frac{dt}{\cosh h-\cos t}
=\frac{2\pi}{|\sinh h|}$ together with the factor $\tfrac12$ in \eqref{P.kernel}; for the second,
split the integral at $t=\pi$ and substitute $u=\cos t$ on each half to obtain the logarithmic
expression in \eqref{P.est}, whose small- and large-$|h|$ bounds are immediate. Hence
\begin{equation}\label{P.int}
\int_{\mathbb{R}}\int_{0}^{2\pi}|P(x_{2}-y_{2},x_{1}-y_{1})|\,e^{-|y_{2}|/R_{0}}\,dy_{1}\,dy_{2}\leqslant C(R_{0}),
\end{equation}
uniformly in $x$, because the inner integral in $y_{1}$ is bounded by
$C(1+|\log|x_{2}-y_{2}||)$ for $|x_{2}-y_{2}|\leqslant1$ and by $C$ otherwise, by
\eqref{P.est}, while the remaining integral in $y_{2}$ against the exponential weight converges:
the logarithmic singularity is integrable and the exponential weight makes the contribution of
$|y_{2}|\geqslant1$ finite.

\emph{The bulk term.} By \eqref{e11.loc} the density $\omega$ decays exponentially,
$|\omega(t,y)|\leqslant A e^{-|y_{2}|/R_{0}}$, and by \eqref{e11} it is bounded by $A$; moreover
$x\in\Omega_{1}(t)$ and $y\in\Omega_{2}(t)$ lie on opposite sides of
$\Gamma_{t}$, so $h=x_{2}-y_{2}$ vanishes only on a set of measure zero in the integration and
\eqref{P.est} applies there. Hence
\[
|\Phi_{2}(t,x)|\leqslant\frac{A}{2\pi}\int_{\mathbb{R}}\int_{0}^{2\pi}
|P(x_{2}-y_{2},x_{1}-y_{1})|\,e^{-|y_{2}|/R_{0}}\,dy_{1}\,dy_{2}\leqslant C(A,R_{0})
\]
by \eqref{P.int}.

\emph{The interface term.} We show that $U_{1}(t,x)$ is bounded for each fixed admissible
solution, with a bound that may depend on that solution through its chord--arc constant. Such a
bound is needed only to justify the maximum principle; the boundary estimate
$\sup_{\Gamma_t}|\nabla u_1|\leqslant C(A)$ does not use it.
Set
\[
\kappa_{*}:=\inf_{t\in[0,T],\,\alpha\ne\beta}F(z)(t,\alpha,\beta)>0,
\]
the positivity being the hypothesis of the lemma. This constant is positive because: near
$\alpha=\beta$ use
$|z_{\alpha}|\geqslant A^{-1}$ and the uniform $C^{4}$ bound; for $|\alpha-\beta|\geqslant 4A$ use
$F\geqslant 1-2A/|\alpha-\beta|\geqslant \tfrac12$; on the remaining compact set modulo the simultaneous
period shift, positivity follows from embeddedness and compactness in $(t,\alpha,\beta)$.
For fixed $t$, the one-period form of the interface integral is
\[
U_{1}(t,x)=\frac1{2\pi}\int_{0}^{2\pi}
P\big(x_{2}-z_{2}(t,\beta),\,x_{1}-z_{1}(t,\beta)\big)\,
\omega_{\Gamma}(t,\beta)\,d\beta,
\]
with the symmetric principal value at the singular point; this follows from
$z(t,\beta+2\pi k)=z(t,\beta)+(2\pi k,0)$ and the definition of $P$ in \eqref{P.kernel}. The
periodized kernel $P(h,r)$ is smooth, bounded, and $2\pi$-periodic in $r$ on every set
$h^{2}+\mathrm{dist}(r,2\pi\mathbb{Z})^{2}\geqslant r_{0}^{2}>0$: this follows directly from its closed
form in \eqref{P.kernel}, since the denominator $\cosh h-\cos r$ vanishes only at
$(h,r)=(0,2\pi k)$ and $P$ decays as $|h|\to\infty$.
Let $d:=\mathrm{dist}(x,\Gamma_{t})$. If $d\geqslant r_{0}$, the whole periodized integrand is bounded,
so $|U_{1}(t,x)|\leqslant C(A,r_{0})$. If $d<r_{0}$, choose $r_{0}$ smaller than a tubular-neighborhood radius of $\Gamma_t$; such a
radius is positive uniformly for this fixed solution on the compact time interval, by the uniform
$C^{2}$ bound and the positive chord--arc constant $\kappa_{*}$. Then the nearest point is unique
modulo the period, say $z(t,\alpha_{0})$, and
$x=z(t,\alpha_{0})+d\nu(t,\alpha_{0})$. Choose a fixed small arclength radius $\delta>0$,
depending only on the curvature bound, and split the periodized integral into the local arc
$|s|<\delta$ about this point and its complement. The chord--arc lower bound $\kappa_{*}$ ensures
that every point of the complementary arc, including all its periodic translates, stays a positive
distance (depending on $\kappa_{*},\delta,r_{0}$) from $x$; hence that contribution is bounded by
$C(A,\kappa_{*},\delta,r_{0})$ using the boundedness of $P$ away from its singular set.
For the local arc, the $k=0$ term is the singular kernel and the terms $k\ne0$ in
\eqref{P.kernel} are smooth and bounded there by a constant depending only on $A$; the latter
contribute at most $C(A)\|\omega_{\Gamma}\|_{L^{\infty}}$. Use arclength $s$ for the singular
term and write $\widetilde\omega(s)\,ds=\omega_{\Gamma}(t,\beta)\,d\beta$; by \eqref{om.C1}
and the $C^{4}$ bound for $z$, which gives $1/A\leqslant a\leqslant A+1$ and
$|\partial_{\alpha}(1/a)|\leqslant C(A)$, hence
$\|\widetilde\omega\|_{C^{0,1}}\leqslant C(A)$. In the orthonormal frame
$(\tau,\nu)$ at $s=0$,
where $\nu=-\tau^{\perp}$, Taylor's theorem gives
\[
x-z(t,s)=-s\tau+d\nu+O(A s^{2}),\qquad
(x-z(t,s))^{\perp}=d\tau+s\nu+O(A s^{2}),
\]
and consequently
\[
\frac{(x-z(t,s))^{\perp}}{|x-z(t,s)|^{2}}
=\frac{(d,s)}{s^{2}+d^{2}}+O(A),\qquad |s|<\delta,
\]
where the two components are in the $(\tau,\nu)$ frame. The constant part
$\widetilde\omega(0)$ of the density contributes at most
$C(A)\big(\int_{-\delta}^{\delta}\frac{d}{s^{2}+d^{2}}ds+
|\mathrm{p.v.}\int_{-\delta}^{\delta}\frac{s}{s^{2}+d^{2}}ds|\big)\leqslant C(A)$.
For the difference part, $|\widetilde\omega(s)-\widetilde\omega(0)|\leqslant C(A)|s|$ and
$|(d,s)|/(s^{2}+d^{2})=(s^{2}+d^{2})^{-1/2}$, so its integral is bounded by
$C(A)\int_{-\delta}^{\delta}|s|/(s^{2}+d^{2})^{1/2}ds\leqslant C(A)\delta$; the $O(A)$ kernel
remainder is bounded as well. Thus $|U_{1}(t,x)|\leqslant C(A,\kappa_{*})$ uniformly for
$x\in\Omega_{1}(t)$ and $t\in[0,T]$.

Combining the four terms,
\begin{equation}\label{u1.bdd}
|u_{1}(t,x)|\leqslant C(A,R_{0},\kappa_{*}),\qquad x\in\Omega_{1}(t),\ t\in[0,T],
\end{equation}
so $u_{1}$ is bounded for this fixed solution. We also need $h=\nabla u_{1}$ to be bounded in order
to use the barrier at infinity; boundedness of $u_{1}$ alone would not imply this. Choose a tubular
neighborhood of $\Gamma_{t}$ of radius smaller than the reach used above. On its closure modulo the
horizontal period, $h$ is bounded because $u_{1}\in C^{3}(\overline{\Omega_{1}(t)})$ and is periodic.
On the complement, every ball of radius $r_{0}/4$ is contained in $\Omega_{1}(t)$, and the interior
gradient estimate for the harmonic components of $u_{1}$ gives
$|h(x)|\leqslant C r_{0}^{-1}\|u_{1}\|_{L^{\infty}(\Omega_{1}(t))}$. Thus $h$ is bounded on
$\Omega_{1}(t)$ by a finite constant $M_{t}$, which may depend on the particular solution.
Choose $x_{0}\in\Omega_{2}(t)$, set $d_{0}:=\mathrm{dist}(x_{0},\Gamma_{t})>0$, and fix
$\eta>0$. For any fixed $x\in\Omega_{1}(t)$ take $R>|x-x_{0}|$ so large that
$M_{t}-\eta\log R\leqslant\sup_{\Gamma_{t}}h-\eta\log d_{0}$. Applying the maximum principle to
$h-\eta\log|x-x_{0}|$ on $D_{R}:=\Omega_{1}(t)\cap B_{R}(x_{0})$ then yields
\[
h(x)-\eta\log|x-x_{0}|\leqslant\sup_{\Gamma_{t}}h-\eta\log d_{0}.
\]
Letting $\eta\downarrow0$ gives $h(x)\leqslant\sup_{\Gamma_{t}}h$. Repeating the argument for $-h$
shows $\sup_{\Omega_{1}(t)}|h|\leqslant\sup_{\Gamma_{t}}|h|\leqslant C(A)$, which is
\eqref{e9.int}.
\end{proof}

\begin{remark}\label{rm:normalize}
Both constants of this section depend on $T$ as well as on $A$, $\sigma$ and the densities: $C_{1}$ and $C_{2}$ contain the Gronwall factors $e^{\lambda T}$ and $e^{\kappa_{0}T}$ of \eqref{e19.gronwall} and \eqref{zeta.gronwall}, and $C_{2}$ additionally contains the constant $C_{1}$, since the forcing $\partial_{\alpha}\Lambda\,u+\partial_{\alpha}f$ of the differentiated equation \eqref{zeta.eq.lin} is bounded using the estimate for $u$ itself. Since every statement of this paper is monotone in $A$, we may
replace $A$ by any larger constant and keep the letter $A$; taking
$A':=\max\{A,C_{1},C_{2},\sigma,A_{*},C(A)(1+\|H(0,\cdot)\|_{C^{2}(\mathbb{R})})\}$, where $A_{*}$ is
the constant of Lemma \ref{grad.euler} and the last entry is the bound \eqref{B.trace.bound}, the
solution is $(A',\rho_{1},\rho_{2})$-admissible and the five bounds
\[
\begin{aligned}
&\|\omega_{\Gamma}\|_{L^{\infty}}\leqslant A',\qquad
\sup_{t\in[0,T]}\|\omega_{\Gamma}(t,\cdot)\|_{C^{1}(\mathbb{R})}\leqslant A',\qquad
\sup_{t\in[0,T]}\|v_{2}(t,\cdot)\|_{C^{1}(\mathbb{R})}\leqslant A',\\
&\sup_{t\in[0,T]}\|\nabla u_{1}(t)\|_{L^{\infty}(\Omega_{1}(t))}\leqslant A',\qquad
\sup_{t\in[0,T]}\|H(t,\cdot)\|_{C^{2}(\mathbb{R})}\leqslant A'
\end{aligned}
\]
hold, the first three being taken in the variable $\alpha$ alone, as in Lemmas \ref{om.Linf} and
\ref{om.C1.est}. Renaming $A'\mapsto A$, we may use all five with the constant $A$ of Definition
\ref{adm}. From this point on these five bounds are used freely with the constant $A$,
and no further dependence on the densities or on $\sigma$ occurs anywhere in the estimates of
the geometric-reduction and singular-integral sections; the dependence asserted in Theorem \ref{thm}
is exactly the one inherited through $A'$. 
\end{remark}

\section{Geometry of a candidate splash point}\label{sec:geometry}
We collect the elementary properties of the chord-arc function $F(z)$ that we will
need. They show that, as soon as $F(z)$ becomes small, the interface develops a pair of
almost-touching points; this is the configuration analyzed in the proof of the main theorem.

\begin{lemma}\label{chord}\mbox{}
	\begin{enumerate}[\rm (1)]
		\item  There is $\varepsilon_{1}=\varepsilon_{1}(A)>0$ sufficiently small such that
		\begin{equation}\label{e20}
			F(z)(t,\alpha,\beta)\geqslant\varepsilon_{1}, \quad \text{if } |\alpha-\beta|\leqslant\varepsilon_{1}\ \text{and}\ (t,\alpha,\beta)\in[0,T]\times\mathbb{R}\times\mathbb{R}.
		\end{equation} 
		
		\item There is a constant $\varepsilon_{2}=\varepsilon_{2}(A)$ sufficiently small with the following properties: if
		\begin{equation}\label{e21}
			\inf_{\alpha,\beta\in\mathbb{R}}F(z)(t,\alpha,\beta)=\vartheta\leqslant\varepsilon_{2},
		\end{equation}
		for some $t\in[0,T]$, then there are points $\alpha_{1},\alpha_{2}\in\mathbb{R}$ such that
		\begin{equation}\label{e22}
			\begin{aligned}
				|\alpha_{1}-\alpha_{2}|&\in[2\varepsilon_{1},\frac{1}{2\varepsilon_{1}}], \\
				|z(t,\alpha_{1})-z(t,\alpha_{2})|&=\min_{\alpha,\beta\in\mathbb{R},|\alpha-\beta|\in[\varepsilon_{1},\frac{1}{\varepsilon_{1}}]}|z(t,\alpha)-z(t,\beta)|\approx\vartheta.
			\end{aligned}
		\end{equation}
		Also,
\begin{equation}\label{e23}
			\begin{aligned}
				(z(t,\alpha_{1}):z(t,\alpha_{2}))\cap\Gamma_{t}&=\varnothing, \\
				[z(t,\alpha_{1})-z(t,\alpha_{2})]\cdot z_{\alpha}(t,\alpha_{1})&=[z(t,\alpha_{1})-z(t,\alpha_{2})]\cdot z_{\alpha}(t,\alpha_{2})=0,
			\end{aligned}
		\end{equation}
		and
		\begin{equation}\label{e24}
			z_{\alpha}(t,\alpha_{1})\cdot z_{\alpha}(t,\alpha_{2})<0.
		\end{equation}
		where $(p:q)$ denotes the open line segment in the plane between the points $p$ and $q$.
	\end{enumerate}
\end{lemma}
\begin{proof}
For \eqref{e20}, if $\alpha=\beta$ then $F(z)(t,\alpha,\beta)=|z_{\alpha}(t,\alpha)|\geqslant \frac{1}{A}$ by \eqref{e8}.
If $\alpha\ne\beta$, Taylor's formula gives
\[
z(t,\alpha)-z(t,\beta)=z_{\alpha}(t,\beta)(\alpha-\beta)+R,\qquad |R|\leqslant \tfrac12\|z_{\alpha\alpha}\|_{L^{\infty}}|\alpha-\beta|^{2},
\]
so that, dividing by $|\alpha-\beta|$ and using $|z_{\alpha}|\geqslant \frac{1}{A}$ together with $\|z_{\alpha\alpha}\|_{L^{\infty}}\leqslant A$,
\[
F(z)(t,\alpha,\beta)\geqslant |z_{\alpha}(t,\beta)|-\tfrac12\|z_{\alpha\alpha}\|_{L^{\infty}}|\alpha-\beta|\geqslant \frac{1}{A}-A|\alpha-\beta|.
\]
For $|\alpha-\beta|\leqslant \varepsilon_{1}$ with $\varepsilon_{1}\leqslant \frac{1}{2A^{2}}$ this gives
$F(z)(t,\alpha,\beta)\geqslant \frac1A-A|\alpha-\beta|\geqslant \frac1A-\frac{1}{2A}=\frac{1}{2A}\geqslant \varepsilon_{1}$,
which is \eqref{e20}. 

For $(2)$, we assume that $t$ is fixed and drop it from the notation. We begin with the
following observation, which will allow us to take $\alpha_{1},\alpha_{2}$ to be minimizers over
a compact window.

First we show that the infimum in \eqref{e21} is attained at a pair whose parameter
separation lies in $[\varepsilon_{1},\varepsilon_{1}^{-1}]$. On the one hand, by \eqref{e20} the
value of $F(z)$ is not small when $|\alpha-\beta|\leqslant\varepsilon_{1}$. On the other hand,
the periodicity condition of Definition \ref{adm} says that $\alpha\mapsto z(t,\alpha)-(\alpha,0)$
is $2\pi$-periodic, so that $\|z(t,\alpha)-(\alpha,0)\|_{L^{\infty}}\leqslant  A$ by \eqref{e8} and
hence
\[
|z(t,\alpha)-z(t,\beta)|\geqslant|\alpha-\beta|-2A,
\qquad\text{i.e.,}\qquad
F(z)(t,\alpha,\beta)\geqslant1-\frac{2A}{|\alpha-\beta|}\geqslant\frac12
\]
whenever $|\alpha-\beta|\geqslant4A$. Shrinking $\varepsilon_{1}=\varepsilon_{1}(A)$ once more if
necessary, we may assume
\begin{equation}\label{eps1.list}
4A\leqslant\varepsilon_{1}^{-1},\qquad
\varepsilon_{1}\leqslant\frac18,\qquad
\varepsilon_{1}\leqslant\frac{1}{2A^{4}}.
\end{equation}
The last condition is used only in Case 1 of the proof of \eqref{e25} below; the first two are used
throughout.
Since $2\vartheta\leqslant2\varepsilon_{2}\leqslant\frac12$ after shrinking
$\varepsilon_{2}=\varepsilon_{2}(A)$, every pair with $F(z)(t,\alpha,\beta)\leqslant2\vartheta$
must satisfy $\varepsilon_{1}\leqslant|\alpha-\beta|\leqslant\varepsilon_{1}^{-1}$. Applying this
to a nearly minimizing pair of the infimum in \eqref{e21} we obtain $\alpha_{0},\beta_{0}$ with
\begin{equation}\label{eq.ast}
\varepsilon_{1}\leqslant|\alpha_{0}-\beta_{0}|\leqslant\varepsilon_{1}^{-1},
\qquad
|z(\alpha_{0})-z(\beta_{0})|\leqslant2\vartheta\,|\alpha_{0}-\beta_{0}|
\leqslant2\vartheta\,\varepsilon_{1}^{-1}.
\end{equation}
Next we check that the infimum over the window
$|\alpha-\beta|\in[\varepsilon_{1},\varepsilon_{1}^{-1}]$ is actually attained at a pair with
$|\alpha-\beta|\in[2\varepsilon_{1},\frac{1}{2\varepsilon_{1}}]$, which requires excluding the two
boundary annuli. If $\varepsilon_{1}\leqslant|\alpha-\beta|\leqslant2\varepsilon_{1}$ then
$|\alpha-\beta|\leqslant\frac{1}{2A^{2}}$, so the Taylor estimate used in the proof of
\eqref{e20} gives $F(z)(t,\alpha,\beta)\geqslant\frac{1}{2A}$ and therefore
$|z(\alpha)-z(\beta)|\geqslant\frac{|\alpha-\beta|}{2A}\geqslant\frac{\varepsilon_{1}}{2A}$. If
$\frac{1}{2\varepsilon_{1}}\leqslant|\alpha-\beta|\leqslant\varepsilon_{1}^{-1}$, then
$|z(\alpha)-z(\beta)|\geqslant|\alpha-\beta|-2A\geqslant\frac{1}{2\varepsilon_{1}}-2A
\geqslant\frac{1}{4\varepsilon_{1}}$, using $4A\varepsilon_{1}\leqslant\frac12$. Both lower bounds
exceed the upper bound $2\vartheta\,\varepsilon_{1}^{-1}$ obtained in \eqref{eq.ast} as soon as
$2\varepsilon_{2}\varepsilon_{1}^{-1}<\min\{\frac{\varepsilon_{1}}{2A},\frac{1}{4\varepsilon_{1}}\}$,
which holds after a further shrinking of $\varepsilon_{2}=\varepsilon_{2}(A)$. Hence the
infimum of $|z(\alpha)-z(\beta)|$ over $|\alpha-\beta|\in[\varepsilon_{1},\varepsilon_{1}^{-1}]$
coincides with its infimum over the compact window
$|\alpha-\beta|\in[2\varepsilon_{1},\frac{1}{2\varepsilon_{1}}]$, and the latter is attained; let
$\alpha_{1},\alpha_{2}$ realize it. Then \eqref{e22} holds: the constraint on
$|\alpha_{1}-\alpha_{2}|$ is part of the definition of the window, and, writing
$d_{0}:=|z(\alpha_{1})-z(\alpha_{2})|$, we have on the one hand
$d_{0}\leqslant|z(\alpha_{0})-z(\beta_{0})|\leqslant2\vartheta\varepsilon_{1}^{-1}$ and on the
other hand $d_{0}=F(z)(t,\alpha_{1},\alpha_{2})|\alpha_{1}-\alpha_{2}|\geqslant
2\vartheta\varepsilon_{1}$, because $(\alpha_{1},\alpha_{2})$ belongs to the larger window so
that $F(z)(t,\alpha_{1},\alpha_{2})\geqslant\inf_{\alpha,\beta}F(z)(t,\alpha,\beta)=\vartheta$.
Thus $d_{0}\approx_{\varepsilon_{1}}\vartheta$, which is the asserted relation in
\eqref{e22} (the implied constants depend on $\varepsilon_{1}$, hence only on $A$).

For \eqref{e23}, the function $\alpha\mapsto|z(\alpha)-z(\alpha_{2})|^{2}$ is locally minimized at $\alpha=\alpha_{1}$, while $\alpha\mapsto|z(\alpha_{1})-z(\alpha)|^{2}$ is locally minimized at $\alpha=\alpha_{2}$; differentiating these two functions at the corresponding points gives
\begin{align}\label{ast.2}
[z(\alpha_{1})-z(\alpha_{2})]\cdot z_{\alpha}(\alpha_{1})=0,\qquad
[z(\alpha_{1})-z(\alpha_{2})]\cdot z_{\alpha}(\alpha_{2})=0,
\end{align}
which is the second line of \eqref{e23}. We prove the first line of \eqref{e23},
asserting that the open segment $(z(\alpha_{1}):z(\alpha_{2}))$ is disjoint from
$\Gamma_{t}$, an assertion that is used both here (to verify that the curve $\gamma$
below is simple) and in the proof of \eqref{e25}.

Suppose, for a contradiction, that $z(\alpha_{3})\in(z(\alpha_{1}):z(\alpha_{2}))$ for
some $\alpha_{3}\in\mathbb{R}$. Then
$z(\alpha_{3})=(1-\lambda)z(\alpha_{1})+\lambda z(\alpha_{2})$ for a unique
$\lambda\in(0,1)$, so that
\begin{align}\label{ast.3}
|z(\alpha_{1})-z(\alpha_{3})|=\lambda\,d_{0}<d_{0},
\qquad
|z(\alpha_{3})-z(\alpha_{2})|=(1-\lambda)d_{0}<d_{0},
\end{align}
where $d_{0}:=|z(\alpha_{1})-z(\alpha_{2})|$. The vector
$z(\alpha_{3})-z(\alpha_{1})=\lambda\,(z(\alpha_{2})-z(\alpha_{1}))$ is parallel to
$z(\alpha_{2})-z(\alpha_{1})$, hence \eqref{ast.2} also yields
\begin{align}\label{ast.4}
[z(\alpha_{3})-z(\alpha_{1})]\cdot z_{\alpha}(\alpha_{1})=0,
\qquad
[z(\alpha_{3})-z(\alpha_{2})]\cdot z_{\alpha}(\alpha_{2})=0.
\end{align}
Because $\Gamma_{t}$ is an embedded curve, the parameterization $\alpha\mapsto z(\alpha)$ is
injective modulo the periodicity $z(\alpha+2\pi)=z(\alpha)+(2\pi,0)$, and hence injective on
$\mathbb{R}$: the point $z(\alpha_{3})$ therefore has a unique label $\alpha_{3}$, and there is no
freedom in choosing it. The three cases below cover all the possible values of
$|\alpha_{3}-\alpha_{1}|$, so they are exhaustive.

\textit{Case 1: $|\alpha_{3}-\alpha_{1}|\leqslant \varepsilon_{1}$.} Taylor's formula at
$\alpha_{1}$ reads
\[
z(\alpha_{3})-z(\alpha_{1})=z_{\alpha}(\alpha_{1})(\alpha_{3}-\alpha_{1})+R,
\qquad
|R|\leqslant \tfrac12\|z_{\alpha\alpha}\|_{L^{\infty}}|\alpha_{3}-\alpha_{1}|^{2}.
\]
Taking the scalar product with $z_{\alpha}(\alpha_{1})$ and using the first identity in
\eqref{ast.4} together with $|z_{\alpha}(\alpha_{1})|\geqslant \frac1A$ and
$\|z_{\alpha}\|_{L^{\infty}}\leqslant A+1\leqslant2A$ (both by \eqref{e8}) gives
\[
\frac{|\alpha_{3}-\alpha_{1}|}{A^{2}}
\leqslant |z_{\alpha}(\alpha_{1})|^{2}|\alpha_{3}-\alpha_{1}|
=|R\cdot z_{\alpha}(\alpha_{1})|
\leqslant A^{2}|\alpha_{3}-\alpha_{1}|^{2}.
\]
If $\alpha_{3}\ne\alpha_{1}$ we may divide by $|\alpha_{3}-\alpha_{1}|$ and obtain
$|\alpha_{3}-\alpha_{1}|\geqslant \frac{1}{A^{4}}$, which contradicts
$|\alpha_{3}-\alpha_{1}|\leqslant \varepsilon_{1}\leqslant \frac{1}{2A^{4}}$, the last inequality
being part of \eqref{eps1.list}. Hence
$\alpha_{3}=\alpha_{1}$, i.e.,\ $z(\alpha_{3})=z(\alpha_{1})$, which does not belong to
the open segment $(z(\alpha_{1}):z(\alpha_{2}))$.

\textit{Case 2: $\varepsilon_{1}<|\alpha_{3}-\alpha_{1}|\leqslant \varepsilon_{1}^{-1}$.}
By \eqref{ast.3} and the fact, established above, that $d_{0}$ is the infimum of
$|z(\alpha)-z(\beta)|$ over the window $|\alpha-\beta|\in[\varepsilon_{1},\varepsilon_{1}^{-1}]$,
so that every such pair satisfies $|z(\alpha)-z(\beta)|\geqslant d_{0}$, we have
$d_{0}\leqslant |z(\alpha_{1})-z(\alpha_{3})|<d_{0}$, a contradiction.

\textit{Case 3: $|\alpha_{3}-\alpha_{1}|>\varepsilon_{1}^{-1}$.} Using
$|z(\alpha)-z(\beta)|\geqslant |\alpha-\beta|-2A$ (periodicity of
$\alpha\mapsto z(\alpha)-(\alpha,0)$ and \eqref{e8}) together with
$4A\varepsilon_{1}\leqslant \frac12$ and with
$d_{0}\leqslant |z(\alpha_{0})-z(\beta_{0})|\leqslant 2\vartheta\,\varepsilon_{1}^{-1}
\leqslant 2\varepsilon_{2}\varepsilon_{1}^{-1}$ from the estimate \eqref{eq.ast} above and
$\vartheta\leqslant \varepsilon_{2}$, we obtain
\[
|z(\alpha_{1})-z(\alpha_{3})|
\geqslant |\alpha_{3}-\alpha_{1}|-2A
>\varepsilon_{1}^{-1}-2A
\geqslant \tfrac12\varepsilon_{1}^{-1}
>2\varepsilon_{2}\varepsilon_{1}^{-1}\geqslant d_{0},
\]
again contradicting \eqref{ast.3}.

The three cases are exhaustive and each leads to a contradiction, which proves the
first line of \eqref{e23}.

As for \eqref{e24}, we may assume that $\alpha_{1}<\alpha_{2}$ (interchanging the two
points if needed) and we consider the curve $\gamma:[0,2]\to\mathbb{R}^{2}$:
\[
\gamma(\mu):=
\begin{cases}
z((1-\mu)\alpha_{1}+\mu\alpha_{2}), &\text{if} \ \mu\in[0,1], \\
(2-\mu)z(\alpha_{2})+(\mu-1)z(\alpha_{1}), &\text{if} \ \mu\in[1,2].
\end{cases}
\]
We verify that $\gamma$ is a simple closed curve. Since $\Gamma_{t}$ is an embedded curve,
the parameterization $\alpha\mapsto z(\alpha)$ is injective on $\mathbb{R}$, so that the closed arc $z([\alpha_{1},\alpha_{2}])$ meets the two complementary arcs
$z((-\infty,\alpha_{1}))$ and $z((\alpha_{2},\infty))$ only at its endpoints
$z(\alpha_{1}),z(\alpha_{2})$. On the other hand, the open segment
$(z(\alpha_{1}):z(\alpha_{2}))$ is disjoint from $\Gamma_{t}$ by the first line of \eqref{e23},
which we have just proved. Hence the two pieces of $\gamma$ meet only at the points
$z(\alpha_{1})$ and $z(\alpha_{2})$, and $\gamma$ is a simple closed curve in the plane; by the
Jordan curve theorem, $\mathbb{R}^{2}\setminus\gamma=O_{1}\cup O_{2}$, where $O_{1},O_{2}$ are
disjoint connected open sets, $O_{2}$ is bounded, $O_{1}$ is unbounded, and
$\gamma=\partial O_{1}=\partial O_{2}$. Each of the two arcs
$A^{-}:=z((-\infty,\alpha_{1}))$ and $A^{+}:=z((\alpha_{2},\infty))$ is connected and unbounded,
because $|z(\alpha)-(\alpha,0)|\leqslant A$ by \eqref{e8}, and it is disjoint from $\gamma$, by
the injectivity of the parameterization together with the first line of \eqref{e23}; since $O_{2}$
is bounded while $O_{1}$ is the other component of $\mathbb{R}^{2}\setminus\gamma$, such an arc
must lie in $O_{1}$. Therefore
\[
z(\alpha)\in O_{1} \quad \text{for any} \ \alpha\in\mathbb{R}\setminus[\alpha_{1},\alpha_{2}].
\]
Let $v:=z(\alpha_{2})-z(\alpha_{1})$, for $\delta>0$, we consider the sets
\[
\begin{aligned}
&R_{\delta}^{-}:=\{z(\alpha)+\mu v:\alpha\in(\alpha_{1}-\delta,\alpha_{1}),\mu\in[0,\frac{2}{3}]\}, \\
&R_{\delta}^{+}:=\{z(\alpha)-\mu v:\alpha\in(\alpha_{2},\alpha_{2}+\delta),\mu\in[0,\frac{2}{3}]\}.
\end{aligned}
\]
The sets $R_{\delta}^{\pm}$ are connected, being continuous images of the rectangles $(\alpha_{1}-\delta,\alpha_{1})\times[0,\frac23]$ and $(\alpha_{2},\alpha_{2}+\delta)\times[0,\frac23]$ under $(\alpha,\mu)\mapsto z(\alpha)\pm\mu v$. If $\delta$ is sufficiently small, then we claim:
\[
R_{\delta}^{-}\cap\gamma=R_{\delta}^{+}\cap\gamma=\varnothing.
\]
Since $v=z(\alpha_{2})-z(\alpha_{1})$ and $v\cdot z_{\alpha}(\alpha_{j})=0$ for $j=1,2$ by
\eqref{e23}, a point of $R_{\delta}^{-}$ that also lies on $\Gamma_{t}$ gives labels
$\alpha\in(\alpha_{1}-\delta,\alpha_{1})$ and $\beta$ with $z(\beta)-z(\alpha)=\mu v$ for some
$\mu\in(0,\frac23]$, and then
\[
[z(\beta)-z(\alpha)]\cdot z_{\alpha}(\alpha_{1})=\mu\,v\cdot z_{\alpha}(\alpha_{1})=0.
\]
Similarly, a point of $R_{\delta}^{+}$ that also lies on $\Gamma_{t}$ gives
$\alpha\in(\alpha_{2},\alpha_{2}+\delta)$ and $\beta$ with $z(\alpha)-z(\beta)=\mu v$ for some
$\mu\in(0,\frac23]$, and then
$[z(\alpha)-z(\beta)]\cdot z_{\alpha}(\alpha_{2})=\mu\,v\cdot z_{\alpha}(\alpha_{2})=0$. In
both cases $|z(\alpha)-z(\beta)|=\mu|v|\leqslant\frac23d$.

We claim that in both cases $|\alpha-\beta|<\varepsilon_{1}$. If
$|\alpha-\beta|\in[\varepsilon_{1},\varepsilon_{1}^{-1}]$, then
$|z(\alpha)-z(\beta)|\geqslant d$ by the defining property of $d=d_{0}$ established above,
which contradicts $|z(\alpha)-z(\beta)|\leqslant\frac23d$. If $|\alpha-\beta|>\varepsilon_{1}^{-1}$,
then the periodicity estimate $|z(\alpha)-z(\beta)|\geqslant|\alpha-\beta|-2A$ together with
$4A\varepsilon_{1}\leqslant\frac12$ gives
$|z(\alpha)-z(\beta)|>\frac12\varepsilon_{1}^{-1}\geqslant2\varepsilon_{2}\varepsilon_{1}^{-1}\geqslant d$,
again a contradiction. Hence $|\alpha-\beta|<\varepsilon_{1}$, so that \eqref{e20} applies and
yields
\[
|z(\alpha)-z(\beta)|\geqslant\varepsilon_{1}|\alpha-\beta|,
\qquad\text{i.e.,}\qquad
\eta:=|\alpha-\beta|\leqslant\frac{|z(\alpha)-z(\beta)|}{\varepsilon_{1}}
\leqslant\frac{2d}{3\varepsilon_{1}}.
\]
Hence, in the first case both $\alpha$ and $\beta$ lie in the interval
$(\alpha_{1}-\delta-\eta,\alpha_{1}+\eta)$, and in the second case both lie in
$(\alpha_{2}-\eta,\alpha_{2}+\delta+\eta)$.

Now Taylor's formula at $\alpha_{1}$, respectively at $\alpha_{2}$, gives
\[
\frac{z(\beta)-z(\alpha)}{\beta-\alpha}=z_{\alpha}(\alpha_{1})+O(\delta+\eta),
\qquad\text{respectively}\qquad
\frac{z(\alpha)-z(\beta)}{\alpha-\beta}=z_{\alpha}(\alpha_{2})+O(\delta+\eta),
\]
where the implied constant is at most $\|z_{\alpha\alpha}\|_{L^{\infty}}\leqslant A$. Taking
the scalar product with $z_{\alpha}(\alpha_{1})$, respectively with $z_{\alpha}(\alpha_{2})$,
and using the orthogonality relations just derived yields, for $j=1,2$ (with $j=1$ in the
first case above and $j=2$ in the second),
\[
0=|z_{\alpha}(\alpha_{j})|^{2}+O\big(A(\delta+\eta)\,|z_{\alpha}(\alpha_{j})|\big),
\]
while $|z_{\alpha}(\alpha_{j})|^{2}\geqslant A^{-2}$ by \eqref{e8} and
$|z_{\alpha}(\alpha_{j})|\leqslant A+1\leqslant2A$. Since
$\eta\leqslant\frac{2d}{3\varepsilon_{1}}$ and $d\approx\vartheta\leqslant\varepsilon_{2}$,
choosing first $\delta$ and then $\varepsilon_{2}$ small enough (both depending only on $A$)
makes the error terms in the last relation at most $\frac12A^{-2}$, which contradicts
$|z_{\alpha}(\alpha_{j})|^{2}\geqslant A^{-2}$. This proves
the claim.
Since $z(\alpha)\in O_{1}$ for every $\alpha\in(\alpha_{1}-\delta,\alpha_{1})\cup(\alpha_{2},\alpha_{2}+\delta)$, we conclude that
$R_{\delta}^{-}\subset O_{1}$ and $R_{\delta}^{+}\subset O_{1}$.
By the second identity in \eqref{e23} both $z_{\alpha}(\alpha_{1})$ and
$z_{\alpha}(\alpha_{2})$ are orthogonal to the nonzero vector $v$, with $|v|=d$; since the
orthogonal complement of a nonzero vector in $\mathbb{R}^{2}$ is one-dimensional, the two vectors
are parallel or antiparallel. Assume for contradiction that they are parallel with the same
orientation, and let $\hat\nu$ be the unit normal to $v$ chosen so that
$z_{\alpha}(\alpha_{j})\cdot\hat\nu=|z_{\alpha}(\alpha_{j})|$ for $j=1,2$; denote by
$\varsigma(x):=(x-z(\alpha_{1}))\cdot\hat\nu$ the signed distance from the chord line. Since
$v\cdot\hat\nu=0$, the translation $\pm\mu v$ has no component in the direction $\hat\nu$, so for
a point of the form $x=z(\alpha)\pm\mu v$ the signed distance is
\[
\varsigma(x)=\big(z(\alpha)-z(\alpha_{j})\big)\cdot\hat\nu
=(\alpha-\alpha_{j})\,z_{\alpha}(\alpha_{j})\cdot\hat\nu+O\big(A(\alpha-\alpha_{j})^{2}\big),
\]
with $j=1$ and the sign $+$ for $R_{\delta}^{-}$, and $j=2$ and the sign $-$ for
$R_{\delta}^{+}$, the constant in $O(\cdot)$ being at most
$\|z_{\alpha\alpha}\|_{L^{\infty}}\leqslant A$. Taking $\delta\leqslant(2A^{2})^{-1}$, so that
the error term is at most $\frac12A^{-1}|\alpha-\alpha_{j}|$, and using
$z_{\alpha}(\alpha_{j})\cdot\hat\nu=|z_{\alpha}(\alpha_{j})|\geqslant A^{-1}$ from \eqref{e8}, we
obtain
\[
\varsigma(x)\leqslant-\tfrac12A^{-1}(\alpha_{1}-\alpha)<0
\quad\text{for } x=z(\alpha)+\mu v,\ \alpha\in(\alpha_{1}-\delta,\alpha_{1}),
\]
\[
\varsigma(x)\geqslant\tfrac12A^{-1}(\alpha-\alpha_{2})>0
\quad\text{for } x=z(\alpha)-\mu v,\ \alpha\in(\alpha_{2},\alpha_{2}+\delta),
\]
uniformly in $\mu\in[0,\frac23]$. Thus $R_{\delta}^{-}$ and $R_{\delta}^{+}$ lie on opposite
sides of the chord line, separated by it.

To pass from this to a full neighborhood of the midpoint $q=\frac{z(\alpha_{1})+z(\alpha_{2})}{2}$
we use that the two families of translates cover the two sides. The maps
$\Xi(\alpha,\mu):=z(\alpha)+\mu v$ and $\widetilde\Xi(\alpha,\mu):=z(\alpha)-\mu v$ have
Jacobian determinants $\det(z_{\alpha}(\alpha),v)$ and $-\det(z_{\alpha}(\alpha),v)$, which do not
vanish at $(\alpha_{1},\frac12)$, respectively at $(\alpha_{2},\frac12)$, because
$z_{\alpha}(\alpha_{j})$ is orthogonal and not parallel to $v$; both maps send those points to
$q$. By the inverse function theorem there is $\varrho>0$, depending only on $A$ and $d$, such that
$B(q,\varrho)\subset\Xi(U_{1})\cap\widetilde\Xi(U_{2})$, where $U_{1}$ and $U_{2}$ are
neighborhoods of $(\alpha_{1},\frac12)$, respectively of $(\alpha_{2},\frac12)$, on which
$\Xi$, respectively $\widetilde\Xi$, is one-to-one and on which
$|\alpha-\alpha_{j}|\leqslant\delta$ and $|\mu-\frac12|\leqslant\frac16$. Let
$x\in B(q,\varrho)$ and write $(\alpha,\mu)=\Xi^{-1}(x)$. Then $x=z(\alpha)+\mu v$, so by the
computation above $\varsigma(x)$ and $\alpha-\alpha_{1}$ have the same sign; if $\varsigma(x)<0$ then
$\alpha\in(\alpha_{1}-\delta,\alpha_{1})$ and $\mu\in[\frac13,\frac23]$, that is,
$x\in R_{\delta}^{-}$. If instead $\varsigma(x)>0$, the same argument applied to
$(\alpha,\mu)=\widetilde\Xi^{-1}(x)$ gives $x\in R_{\delta}^{+}$. Finally, if $\varsigma(x)=0$, then
$x$ lies on the chord line at distance at most $\varrho$ from $q$, hence on the open segment
$(z(\alpha_{1}):z(\alpha_{2}))$ provided $\varrho<\frac d2$. Diminishing $\varrho$ if necessary so
that this holds, we obtain
\[
B(q,\varrho)\subset(z(\alpha_{1}):z(\alpha_{2}))\cup R_{\delta}^{-}\cup R_{\delta}^{+}
\subset\gamma\cup O_{1}.
\]
This contradicts $q\in\partial(O_{2})$, which follows from $\gamma=\partial O_{2}$: as $O_{2}$ is a
nonempty open set, every neighborhood of a point of its boundary contains points of $O_{2}$,
whereas $\gamma$ has empty interior and $O_{1}\cap O_{2}=\varnothing$. So it follows that
$z_{\alpha}(\alpha_{1})$ and $z_{\alpha}(\alpha_{2})$ are antiparallel, that is
$z_{\alpha}(\alpha_{1})\cdot z_{\alpha}(\alpha_{2})<0$, which is exactly \eqref{e24}.
\end{proof}
\begin{remark}
A pair of points $z(t,\alpha_{1})$, $z(t,\alpha_{2})$ satisfying \eqref{e22} will be called a \emph{conjugate pair}.
\end{remark}
\begin{lemma}\label{chord-arc}
For any $\delta>0$ there is $\varepsilon_{3}=\varepsilon_{3}(A,T,\delta)$ sufficiently small with the following property: if $t\in[0,T]$, $\alpha,\beta\in\mathbb{R}$ satisfy $|\alpha-\beta|\geqslant \delta$, and
\[
[z(t,\alpha):z(t,\beta)]\subset\overline{\Omega_{1}(t)},
\]
then we have
\[
|z(t,\alpha)-z(t,\beta)|\geqslant \varepsilon_{3}.
\]
\end{lemma}
\begin{proof}
We assume that $\delta$ is sufficiently small, and assume for contradiction that the conclusion of the lemma fails. Therefore, there are points $t_{0}\in[0,T]$ and $\alpha_{0},\beta_{0}\in\mathbb{R}$ such that 
\[
|\alpha_{0}-\beta_{0}|\geqslant \delta, \quad |z(t_{0},\alpha_{0})-z(t_{0},\beta_{0})|\leqslant \varepsilon_{3}, \quad [z(t_{0},\alpha_{0}):z(t_{0},\beta_{0})]\subset\overline{\Omega_{1}(t_{0})}.
\]
Throughout the proof we use the interior bound \eqref{e9.int}, which is supplied by Lemma
\ref{grad.euler}.
By the kinematic condition in \eqref{e6} the vector $\partial_{t}z-v_{1}$ is tangential to
$\Gamma_{t}$, so there is a coefficient $A_{1}$ such that
\begin{equation}\label{chord.ar.kin}
\partial_{t}z(t,\alpha)=v_{1}(t,\alpha)+A_{1}(t,\alpha)\partial_{\alpha}z(t,\alpha).
\end{equation}
Since $\partial_{t}z=v_{2}$ and $v_{1}-v_{2}=u\tau$ by \eqref{curl.jump}, comparison with
\eqref{chord.ar.kin} identifies the coefficient explicitly:
\begin{equation}\label{A1.id}
A_{1}=-\frac{u}{a},\qquad a=|\partial_{\alpha}z|,\qquad u=\frac{\omega_{\Gamma}}{a}.
\end{equation}
By the two bounds \eqref{e19} and \eqref{om.C1} of the interface-vorticity section,
\[
\|u\|_{L^{\infty}([0,T]\times\mathbb{R})}+\|\partial_{\alpha}u\|_{L^{\infty}([0,T]\times\mathbb{R})}
\leqslant C(A,T,\sigma,\rho_{1},\rho_{2}),
\]
while $\|a\|_{C^{1}}\leqslant C(A)$ and $|\partial_{\alpha}z|\geqslant1/A$ by \eqref{e8}. Hence
$\|A_{1}\|_{L^{\infty}}+\|\partial_{\alpha}A_{1}\|_{L^{\infty}}\lesssim1$, so that $A_{1}$ is
bounded and Lipschitz in $\alpha$; this uses the equations of motion, since
\eqref{e19} and \eqref{om.C1} do. 

Because the interface moves with the Euler flow only up to the tangential
reparameterization \eqref{chord.ar.kin}, one cannot estimate
$z(t,\alpha_{0})-z(t,\beta_{0})$ by differentiating along the flow directly; instead we
argue through the Lagrangian label of the Euler fluid. Let $\Phi_{t}$ denote the flow of
$u_{1}$, i.e.,
\[
\partial_{t}\Phi_{t}(q)=u_{1}(t,\Phi_{t}(q)),\qquad \Phi_{t_{0}}=\mathrm{id}.
\]
The flow is well defined on $[0,T]$, and $\Phi_{t}$ is a diffeomorphism of
$\Omega_{1}(t_{0})$ onto its image: by \eqref{e9.int} the field $u_{1}(t,\cdot)$ is Lipschitz
with constant $A$ on $\Omega_{1}(t)$, and by \eqref{e9.bs} it has at most linear growth there,
\begin{equation}\label{growth}
|u_{1}(t,x)|\leqslant C_{\mathrm{sol}}(1+|x|),\qquad x\in\Omega_{1}(t),\ t\in[0,T],
\end{equation}
which together with the Lipschitz bound is what makes the trajectories of \eqref{eq.Lag} exist
on $[0,T]$. Indeed, the structural representation \eqref{e9.bs} and \eqref{u1.bdd} give
$|u_{1}(t,x)|\leqslant C_{\mathrm{sol}}<\infty$ for $x\in\Omega_{1}(t)$ by
\eqref{u1.bdd} of Lemma \ref{grad.euler}; the two harmonic
summands are bounded by \eqref{e9.norm}, the interface term by the periodization of \eqref{P.kernel}
together with the Lipschitz bound $\|\omega_{\Gamma}\|_{C^{0,1}}\leqslant C(A)$ of \eqref{om.C1},
and the bulk term by \eqref{P.int} together with the exponential decay \eqref{e11.loc}, whence
\eqref{growth} follows a fortiori. 

We verify that $\Phi_{t}$ maps $\Gamma_{t_{0}}$ onto $\Gamma_{t}$ and $\Omega_{1}(t_{0})$ onto
$\Omega_{1}(t)$. Both statements are consequences of \eqref{eq.Lag} and the kinematic condition,
as follows.

(i) \emph{The interface.} The coefficient $A_{1}$ of \eqref{chord.ar.kin} is bounded and
Lipschitz, so for each $\psi$ the ordinary differential equation
\begin{equation}\label{phi.ode}
\partial_{t}\varphi(t,\psi)=-A_{1}(t,\varphi(t,\psi)),\qquad\varphi(t_{0},\psi)=\psi,
\end{equation}
has a unique solution on $[0,T]$. Set $Z(t,\psi):=z(t,\varphi(t,\psi))$. Then $Z(t_{0},\psi)=z(t_{0},\psi)$,
and by \eqref{chord.ar.kin} and \eqref{phi.ode},
\[
\begin{aligned}
\partial_{t}Z(t,\psi)&=\partial_{t}z(t,\varphi)+z_{\alpha}(t,\varphi)\,\partial_{t}\varphi
=\big[v_{1}+A_{1}z_{\alpha}\big](t,\varphi)-A_{1}(t,\varphi)z_{\alpha}(t,\varphi)\\
&=v_{1}(t,\varphi)=u_{1}(t,Z(t,\psi)).
\end{aligned}
\]
Hence $t\mapsto Z(t,\psi)$ solves the same initial value problem as $t\mapsto\Phi_{t}(z(t_{0},\psi))$,
and by uniqueness of solutions of \eqref{eq.Lag} they coincide:
\begin{equation}\label{phi.identity}
\Phi_{t}\big(z(t_{0},\psi)\big)=z\big(t,\varphi(t,\psi)\big)\in\Gamma_{t}.
\end{equation}
Conversely, $\psi\mapsto\varphi(t,\psi)$ is a homeomorphism of $\mathbb{R}$ (it is the flow of the
bounded Lipschitz field $-A_{1}(t,\cdot)$), so every point of $\Gamma_{t}$ arises in this way.
Therefore $\Phi_{t}(\Gamma_{t_{0}})=\Gamma_{t}$.

(ii) \emph{The two regions.} First, $\Phi_{t}(\Omega_{1}(t_{0}))\subseteq\Omega_{1}(t)$. Let
$q\in\Omega_{1}(t_{0})$. The trajectory $s\mapsto\Phi_{s}(q)$ is continuous on $[0,T]$, and by (i)
it never meets $\Gamma_{s}$: if $\Phi_{s}(q)\in\Gamma_{s}$, then
$q=\Phi_{t_{0}}(\Phi_{s}(q))\in\Gamma_{t_{0}}$, contradicting $q\in\Omega_{1}(t_{0})$. The set
$A:=\{s\in[0,T]:\Phi_{s}(q)\in\Omega_{1}(s)\}$ is therefore both open and closed. It is open: if
$s_{0}\in A$, then $\delta:=\operatorname{dist}(\Phi_{s_{0}}(q),\Gamma_{s_{0}})>0$, and for $s$
close to $s_{0}$ the point $\Phi_{s}(q)$ stays within $\delta/3$ of $\Phi_{s_{0}}(q)$ (continuous
dependence on $s$), while $\Gamma_{s}$ stays within $\delta/3$ of $\Gamma_{s_{0}}$ in Hausdorff
distance, because $z$ is continuous in $t$ with values in $C^{4}$, hence in particular with
values in $C^{0}$, uniformly, by item (i) of Definition \ref{adm}; hence
$\Phi_{s}(q)\notin\Gamma_{s}$ and $\Phi_{s}(q)$ lies in the same component of
$\mathbb{R}^{2}\setminus\Gamma_{s}$ as before, namely $\Omega_{1}(s)$. The same argument applied to
$\Omega_{2}(s)$ shows that the complement of $A$ is open. Since the set contains $s=t_{0}$, where $\Phi_{t_{0}}=\mathrm{id}$, it is all of
$[0,T]$.

Conversely, let $p\in\Omega_{1}(t)$ and let $X(s)$ be the solution of the backward problem
$\partial_{s}X=-u_{1}(s,X)$, $X(t)=p$, continued as far as it stays in $\Omega_{1}(s)$. If it
failed to exist up to $s=t_{0}$, there would be a first time $s_{*}\in[t_{0},t)$ with
$X(s_{*})\in\Gamma_{s_{*}}$; but then $p=\Phi_{t}(X(s_{*}))\in\Phi_{t}(\Gamma_{s_{*}})=\Gamma_{t}$
by (i), contradicting $p\in\Omega_{1}(t)$. Hence $X$ is defined on $[t_{0},t]$, $q:=X(t_{0})$ lies in
$\Omega_{1}(t_{0})$ by the first part applied to the trajectory $s\mapsto X(s)$, and
$\Phi_{t}(q)=p$. Therefore $\Phi_{t}(\Omega_{1}(t_{0}))=\Omega_{1}(t)$.

In view of \eqref{e9.int} the flow does not expand the length of
curves lying in $\overline{\Omega_{1}(t_{0})}$, in the following {\em restricted} sense: if
$q,q'\in\overline{\Omega_{1}(t_{0})}$ are such that the closed segment $[q,q']$ is contained in
$\overline{\Omega_{1}(t_{0})}$, then
\begin{equation}\label{chord.ar.lip}
|\Phi_{t}(q)-\Phi_{t}(q')|\leqslant e^{A|t-t_{0}|}|q-q'|
\qquad\text{for all }t\in[0,T].
\end{equation}
Let $\gamma(s):=q+s(q'-q)$, $s\in[0,1]$, parameterize the segment $[q,q']$, and let
$\Gamma_{t}(s):=\Phi_{t}(\gamma(s))$. Since $\Phi_{t}$ transports $\overline{\Omega_{1}(t_{0})}$
onto $\overline{\Omega_{1}(t)}$, the curve $\Gamma_{t}$ is contained in
$\overline{\Omega_{1}(t)}$ for every $t$; differentiating $\partial_{t}\Phi_{t}=u_{1}(t,\Phi_{t})$
with respect to $s$ gives
$\partial_{t}\big(\partial_{s}\Gamma_{t}(s)\big)=\nabla u_{1}(t,\Gamma_{t}(s))\partial_{s}\Gamma_{t}(s)$,
so that $L(t):=\int_{0}^{1}\big|\partial_{s}\Gamma_{t}(s)\big|\,ds$ satisfies
\[
L'(t)\leqslant\|\nabla u_{1}(t)\|_{L^{\infty}(\Omega_{1}(t))}\,L(t)\leqslant A\,L(t)
\]
whenever $L(t)>0$, whence $L(t)\leqslant e^{A|t-t_{0}|}L(t_{0})=e^{A|t-t_{0}|}|q-q'|$ by Gronwall's
inequality, applied forward for $t>t_{0}$ and backward for $t<t_{0}$. Since the distance between
the endpoints of $\Gamma_{t}$ is at most its length, \eqref{chord.ar.lip} follows. 

On the other hand, the reparameterization $\varphi$ of \eqref{phi.ode} satisfies
\[
\partial_{t}\varphi(t,\psi)=-A_{1}(t,\varphi(t,\psi))
\quad\text{and}\quad
\varphi(t_{0},\psi)=\psi,
\]
by
construction, and \eqref{phi.identity} shows that it realizes the conjugacy between the two
parameterizations of $\Gamma_{t}$. Since
$A_{1}$ is Lipschitz with constant $\lesssim A$, the difference of two solutions satisfies
\[
|\partial_{t}(\varphi(t,\psi)-\varphi(t,\psi'))|\lesssim A|\varphi(t,\psi)-\varphi(t,\psi')|,
\]
and Gronwall's inequality backwards from $t_{0}$ yields
\begin{equation}\label{chord.ar.philip}
|\varphi(0,\alpha_{0})-\varphi(0,\beta_{0})|\gtrsim|\alpha_{0}-\beta_{0}|\geqslant \delta.
\end{equation}
Finally, applying \eqref{chord.ar.lip} and \eqref{phi.identity} with
$q=z(t_{0},\alpha_{0})$, $q'=z(t_{0},\beta_{0})$, which are joined by the closed segment
$[z(t_{0},\alpha_{0}):z(t_{0},\beta_{0})]$ contained in $\overline{\Omega_{1}(t_{0})}$ by
hypothesis, so that \eqref{chord.ar.lip} applies, and using the hypothesis
$|z(t_{0},\alpha_{0})-z(t_{0},\beta_{0})|\leqslant \varepsilon_{3}$, we obtain
\[
|z(0,\varphi(0,\alpha_{0}))-z(0,\varphi(0,\beta_{0}))|
=|\Phi_{0}(z(t_{0},\alpha_{0}))-\Phi_{0}(z(t_{0},\beta_{0}))|
\lesssim\varepsilon_{3},
\]
Together with \eqref{chord.ar.philip} this shows that
$F(z)(0,\varphi(0,\alpha_{0}),\varphi(0,\beta_{0}))\lesssim\varepsilon_{3}/\delta$, which is
strictly smaller than $1/A$ provided $\varepsilon_{3}=\varepsilon_{3}(A,T,\delta)$ is
chosen sufficiently small. This contradicts the chord-arc assumption \eqref{e10} at
time $0$ and completes the proof of the lemma.
\end{proof}
We give a complete description of a neighborhood of a conjugate pair:
\begin{proposition}\label{chord-arc prop}
(1) There are small constants $\varepsilon_{4},\varepsilon_{5}>0$ such that
\begin{equation}\label{eps45.list}
0<\varepsilon_{4}\leqslant \varepsilon_{5}^{100},\qquad
\varepsilon_{5}<\min\{\varepsilon_{1},\varepsilon_{3}\},\qquad
\varepsilon_{5}\leqslant\varepsilon_{2},
\end{equation}
with the following properties: assume that $\inf_{\alpha,\beta}F(z)(t,\alpha,\beta)=\vartheta\leqslant \varepsilon_{4}$ for some $t\in[0,T]$, let $\alpha_{1},\alpha_{2}$ be as in \eqref{e22}, and set
\[
d:=|z(t,\alpha_{2})-z(t,\alpha_{1})|, \quad e:=\frac{z(t,\alpha_{2})-z(t,\alpha_{1})}{|z(t,\alpha_{2})-z(t,\alpha_{1})|}.
\]
Then 
\begin{equation}\label{e25}
(z(t,\alpha_{1}):z(t,\alpha_{2}))\subset\Omega_{2}(t),
\end{equation}
and 
\begin{equation}\label{e26}
z_{\alpha}(t,\alpha_{1})\cdot e=z_{\alpha}(t,\alpha_{2})\cdot e=0, \quad z_{\alpha}(t,\alpha_{1})\cdot z_{\alpha}(t,\alpha_{2})<0.
\end{equation}
In addition we have
\begin{equation}\label{e27}
\begin{aligned}
&\{z(t,\beta)-\rho e:|\beta-\alpha_{1}|\leqslant \varepsilon_{5}, \rho\in(0,\varepsilon_{5})\}\subset\Omega_{1}(t), \\
&\{z(t,\beta)+\rho e:|\beta-\alpha_{2}|\leqslant \varepsilon_{5}, \rho\in(0,\varepsilon_{5})\}\subset\Omega_{1}(t), \\
&\{z(t,\beta)+\rho e:|\beta-\alpha_{1}|\leqslant \varepsilon_{5}, \rho\in(0,d)\}\subset\Omega_{2}(t), \\
&\{z(t,\beta)-\rho e:|\beta-\alpha_{2}|\leqslant \varepsilon_{5}, \rho\in(0,d)\}\subset\Omega_{2}(t).
\end{aligned}
\end{equation}
(2) Let $R:\mathbb{R}^{2}\to\mathbb{R}^{2}$ denote the orientation-preserving rigid motion transformation that satisfies 
\begin{equation}\label{e28}
\begin{aligned}
R(0,0)&=\frac{z(t,\alpha_{1})+z(t,\alpha_{2})}{2}, \\
R(0,\frac{d}{2})&=z(t,\alpha_{2}), \\
R(0,-\frac{d}{2})&=z(t,\alpha_{1}).
\end{aligned}
\end{equation}
Then there exist functions $f_{1},f_{2},\beta_{1},\beta_{2}:(-\varepsilon_{5}^{4},\varepsilon_{5}^{4})\to(-\varepsilon_{5}^{2},\varepsilon_{5}^{2})$ such that
\begin{equation}\label{e29}
R(\rho,f_{1}(\rho))=z(t,\alpha_{1}+\beta_{1}(\rho)), \quad R(\rho,f_{2}(\rho))=z(t,\alpha_{2}+\beta_{2}(\rho)),
\end{equation} 
for any $\rho\in(-\varepsilon_{5}^{4},\varepsilon_{5}^{4})$. In addition,
\begin{equation}\label{e30}
\begin{aligned}
&f_{1}(0)=-\frac{d}{2}, \quad f_{2}(0)=\frac{d}{2}, \\
&\beta_{1}(0)=\beta_{2}(0)=0, \quad f'_{1}(0)=f'_{2}(0)=0, \\
&-\beta'_{1}(\rho)\beta'_{2}(\rho)\approx 1 \quad \text{and} \quad f_{2}(\rho)-f_{1}(\rho)\geqslant d \quad \text{for all}\ \rho\in(-\varepsilon_{5}^{4},\varepsilon_{5}^{4}), \\
&\|f_{1}\|_{C^{4}(-\varepsilon_{5}^{4},\varepsilon_{5}^{4})}+\|f_{2}\|_{C^{4}(-\varepsilon_{5}^{4},\varepsilon_{5}^{4})}
+\|\beta_{1}\|_{C^{4}(-\varepsilon_{5}^{4},\varepsilon_{5}^{4})}+\|\beta_{2}\|_{C^{4}(-\varepsilon_{5}^{4},\varepsilon_{5}^{4})}\lesssim1,
\end{aligned}
\end{equation}
all the implied constants depending only on $A$.
Finally,
\begin{equation}\label{e31}
\begin{aligned}
&R(\{(\rho,y):\rho\in(-\varepsilon_{5}^{4},\varepsilon_{5}^{4}),y\in(f_{2}(\rho),\varepsilon_{5}^{2})\})\subset\Omega_{1}(t), \\
&R(\{(\rho,y):\rho\in(-\varepsilon_{5}^{4},\varepsilon_{5}^{4}),y\in(-\varepsilon_{5}^{2},f_{1}(\rho))\})\subset\Omega_{1}(t), \\
&R(\{(\rho,y):\rho\in(-\varepsilon_{5}^{4},\varepsilon_{5}^{4}),y\in(f_{1}(\rho),f_{2}(\rho))\})\subset\Omega_{2}(t).
\end{aligned}
\end{equation}
\end{proposition}

\begin{remark}
	Such an $R$ in (2) exists and is unique: writing $m:=\frac{1}{2}(z(t,\alpha_{1})+z(t,\alpha_{2}))$ and
	$e=(e_{1},e_{2})$ for the unit vector of \eqref{e26}, the affine map
	\[
	R(\rho,y):=m+\rho\,(e_{2},-e_{1})+y\,e
	\]
	satisfies \eqref{e28}, because $R(0,\pm\frac d2)=m\pm\frac d2 e=z(t,\alpha_{2})$ resp.\ $z(t,\alpha_{1})$.
	Its linear part $Q:=DR$ has the two orthonormal columns $(e_{2},-e_{1})$ and $e$, so
	$Q\in O(2)$ and $\det Q=e_{2}^{2}+e_{1}^{2}=1$; consequently
	\begin{equation}\label{Q.perp}
		(Qv)^{\perp}=Q(v^{\perp})\quad\text{for every }v\in\mathbb{R}^{2},\qquad v^{\perp}:=(-v_{2},v_{1}),
	\end{equation}
	and the change of variables $y\mapsto R(y)$ has Jacobian $\det Q=1$. 
\end{remark}

\begin{proof}
As before, we assume that $t$ is fixed and we drop it from the notation. 

For the first part, Lemma \ref{chord} ensures that there exist $\alpha_{1},\alpha_{2}$ satisfying \eqref{e22} such that $|z(\alpha_{2})-z(\alpha_{1})|\approx\vartheta$, provided $\varepsilon_{4}\leqslant \varepsilon_{2}$. For \eqref{e25}, note that $|\alpha_{1}-\alpha_{2}|\geqslant2\varepsilon_{1}$ by \eqref{e22} and that
$d\approx\vartheta\leqslant\varepsilon_{4}$. Apply the contrapositive of Lemma \ref{chord-arc} with
$\delta:=\varepsilon_{1}$: if $d<\varepsilon_{3}=\varepsilon_{3}(A,T,\varepsilon_{1})$ then the closed
segment $[z(\alpha_{1}):z(\alpha_{2})]$ cannot be contained in $\overline{\Omega_{1}(t)}$, hence it meets
$\Omega_{2}(t)$. On the other hand, the open segment $(z(\alpha_{1}):z(\alpha_{2}))$ does not meet
$\Gamma_{t}$, by the first line of \eqref{e23}; since
$\mathbb{R}^{2}=\Omega_{1}(t)\cup\Omega_{2}(t)\cup\Gamma_{t}$ with $\Omega_{1}(t),\Omega_{2}(t)$ disjoint and
open, the open segment, being connected, is contained in $\Omega_{1}(t)\cup\Omega_{2}(t)$ and must lie
entirely in one of the two; as it meets $\Omega_{2}(t)$, it lies in $\Omega_{2}(t)$, which is \eqref{e25}.
The condition $d<\varepsilon_{3}$ holds as soon as $\varepsilon_{4}$ is chosen so small that
$d\approx\vartheta\leqslant\varepsilon_{4}$ forces $d<\varepsilon_{3}$.

By \eqref{e24}, we get \eqref{e26}. 

Then for \eqref{e27}, we first state the geometric fact on which all the inclusions below rest. By \eqref{e26} and $\|z-(\alpha,0)\|_{C^{4}}\leqslant A$ in \eqref{e8},
\[
|z_{\alpha}(\beta)\cdot e|=|(z_{\alpha}(\beta)-z_{\alpha}(\alpha_{1}))\cdot e|\leqslant A|\beta-\alpha_{1}|;
\]
hence, since $A^{2}\varepsilon_{1}\leqslant \frac14$ (a condition of the form $\varepsilon_{1}\leqslant c(A)$, which we include among those determining $\varepsilon_{1}$ in Lemma \ref{chord}), the arc $z((\alpha_{1}-2\varepsilon_{1},\alpha_{1}+2\varepsilon_{1}))$ is a graph over the line $e^{\perp}$, that is, every line parallel to $e$ meets it in at most one point; the same holds on the corresponding arc around $\alpha_{2}$. If $\beta<\beta'$ belonged to that interval with $z(\beta')-z(\beta)$ parallel to $e$, then $0=\int_{\beta}^{\beta'}z_{\alpha}(s)\cdot e^{\perp}ds$ with $z_{\alpha}(s)\cdot e^{\perp}\geqslant |z_{\alpha}(\alpha_{1})|-A|s-\alpha_{1}|\geqslant \frac{1}{2A}$ on the whole interval (the sign of $e^{\perp}$ being chosen so that $z_{\alpha}(\alpha_{1})\cdot e^{\perp}=|z_{\alpha}(\alpha_{1})|\geqslant \frac1A$), which is impossible.
First we show that the open segments $(z(\beta):z(\beta)+de)$ and $(z(\beta):z(\beta)-\varepsilon_{5}e)$, $|\beta-\alpha_{1}|\leqslant \varepsilon_{5}$, do not meet $\Gamma_{t}$. Suppose first that $z(\beta)+te=z(\gamma)$ with $0<t<d$, or that $z(\beta)-te=z(\gamma)$ with $0<t\leqslant \varepsilon_{5}$, for the respective segment, and consider $|\gamma-\beta|$. If $|\gamma-\beta|<\varepsilon_{1}$, then also $|\gamma-\alpha_{1}|<2\varepsilon_{1}$ (recall $|\beta-\alpha_{1}|\leqslant \varepsilon_{5}<\varepsilon_{1}$), so both parameters lie in $(\alpha_{1}-2\varepsilon_{1},\alpha_{1}+2\varepsilon_{1})$ and the vector $z(\gamma)-z(\beta)=\pm te$ is parallel to $e$, which contradicts the graph property. Hence $|\gamma-\beta|\geqslant\varepsilon_{1}$. If $|\gamma-\beta|\in[\varepsilon_{1},\varepsilon_{1}^{-1}]$, then $(\gamma,\beta)$ is a pair in the window of \eqref{e22}, so that $|z(\gamma)-z(\beta)|\geqslant d$. If $|\gamma-\beta|>\varepsilon_{1}^{-1}$, then $|z(\gamma)-z(\beta)|\geqslant |\gamma-\beta|-2A>\frac12\varepsilon_{1}^{-1}\geqslant d$. Thus in either case $t=|z(\gamma)-z(\beta)|\geqslant d$, contradicting $t<d$ for the first segment and showing that any intersection on the second ray has $t\geqslant d$. Consequently, the positive open segment is disjoint from $\Gamma_{t}$, and the negative ray has no intersection with $\Gamma_t$ for $0<t<d$. To rule out an intersection on the negative segment of length $\varepsilon_{5}$, suppose that $z(\beta)-te\in\Gamma_{t}$ for some $t\in[d,\varepsilon_{5}]$. The set of such intersection parameters is closed in the compact interval $[d,\varepsilon_{5}]$; let $t_{1}$ be its minimum. Then $t_{1}\leqslant\varepsilon_{5}$ and the half-open segment $[z(\beta):z(\beta)-t_{1}e)$ is disjoint from $\Gamma_{t}$ and contains points of $\Omega_{1}$, so the closed segment $[z(\beta):z(\beta)-t_{1}e]$ is contained in $\overline{\Omega_{1}}$. Lemma \ref{chord-arc}, applied with $\delta=\varepsilon_{1}$ (recall $|\beta-\gamma|\geqslant\varepsilon_{1}$), therefore gives $t_{1}\geqslant\varepsilon_{3}$, contradicting $t_{1}\leqslant\varepsilon_{5}<\varepsilon_{3}$. We also impose $\varepsilon_{5}<\varepsilon_{1}$; these conditions are compatible with $\varepsilon_{4}\leqslant\varepsilon_{5}^{100}$, since $\varepsilon_{3}=\varepsilon_{3}(A,T,\varepsilon_{1})$ is fixed once $\varepsilon_{1}$ is.

With this at hand we obtain the inclusions near $z(\alpha_{1})$. The points $z(\alpha_{1})+\rho'e$ with $0<\rho'<d$ lie on the open segment $(z(\alpha_{1}):z(\alpha_{2}))$, hence in $\Omega_{2}$ by \eqref{e25}. Since $e$ is orthogonal to $z_{\alpha}(\alpha_{1})$ by \eqref{e26}, the ray $\rho'\mapsto z(\alpha_{1})+\rho'e$ crosses $\Gamma_{t}$ transversally at $\rho'=0$ and enters $\Omega_{2}$; by \eqref{e26} and $\|z_{\alpha}(\beta)-z_{\alpha}(\alpha_{1})\|\leqslant A|\beta-\alpha_{1}|$ the same transversality holds for every $|\beta-\alpha_{1}|\leqslant \varepsilon_{5}$, and by continuity in $\beta$ the ray $\rho'\mapsto z(\beta)+\rho'e$ enters the same side, namely $\Omega_{2}$. Since by the first part it does not meet $\Gamma_{t}$ again for $\rho'\in(0,d)$, we conclude that
\[
\{z(\beta)+\rho'e:|\beta-\alpha_{1}|\leqslant \varepsilon_{5},\ \rho'\in(0,d)\}\subset\Omega_{2}.
\]
Since $\Gamma_{t}\subset\partial\Omega_{1}$, the same argument applied on the other side of $\Gamma_{t}$ gives
$\{z(\beta)-\rho'e:|\beta-\alpha_{1}|\leqslant \varepsilon_{5},\rho'\in(0,\varepsilon_{5})\}\subset\Omega_{1}$.
These are the first two inclusions of \eqref{e27}; the remaining two follow by the same argument
in a neighborhood of $\alpha_{2}$, with $\alpha_{1}$ and $\alpha_{2}$ interchanged and $e$
replaced by $-e$.

We turn to $(2)$. We define the rigid motion transformation $R$ such that \eqref{e28} holds. We then define the functions $g_{1},h_{1}:(-\varepsilon_{5}^{2},\varepsilon_{5}^{2})\to(-\varepsilon_{5},\varepsilon_{5})$ such that
\[
z(\alpha_{1}+\beta)=R(g_{1}(\beta),h_{1}(\beta)). 
\]
We check that if we choose $\beta=0$,
\[
z(\alpha_{1})=R(g_{1}(0),h_{1}(0)).
\]
Since R is bijective, $g_{1}(0)=0$ and $h_{1}(0)=-\frac{d}{2}$. 

Moreover, by the regularity of $z$ assumed above, $\|g_{1}\|_{C^{4}}+\|h_{1}\|_{C^{4}}\lesssim 1$. 

Next we differentiate $z(\alpha_{1}+\beta)=R(g_{1}(\beta),h_{1}(\beta))$ with respect to $\beta$ and evaluate at $\beta=0$. The rigid motion $R$ has the form $R(x)=Qx+b$ with $Q\in O(2)$ and $b\in\mathbb{R}^{2}$, so its derivative is the constant orthogonal matrix $DR=(\partial_{1}R\ \partial_{2}R)=Q$, where $\partial_{1}R=Qe_{1}$ and $\partial_{2}R=Qe_{2}$ are its two orthonormal columns. By \eqref{e28} we have $R(0,\frac{d}{2})-R(0,0)=\frac{d}{2}e$ and $R(0,-\frac{d}{2})-R(0,0)=-\frac{d}{2}e$; therefore $\partial_{2}R=R(0,\frac{d}{2})-R(0,0)$ up to the factor $\frac{d}{2}$, i.e.,\ $\partial_{2}R=e$ as a unit vector, and consequently $\partial_{1}R\cdot e=0$ because the two columns of an orthogonal matrix are orthogonal. Hence
\[
z_{\alpha}(\alpha_{1})=\partial_{1}R\,g'_{1}(0)+\partial_{2}R\,h'_{1}(0)=\partial_{1}R\,g'_{1}(0)+e\,h'_{1}(0).
\]
Taking the scalar product with $e$ and using $z_{\alpha}(\alpha_{1})\cdot e=0$ from \eqref{e26} gives
\[
0=z_{\alpha}(\alpha_{1})\cdot e=(\partial_{1}R\cdot e)\,g'_{1}(0)+h'_{1}(0)=h'_{1}(0),
\]
so $h'_{1}(0)=0$. Substituting this back yields $z_{\alpha}(\alpha_{1})=\partial_{1}R\,g'_{1}(0)$; since $Q$ is orthogonal, $|\partial_{1}R|=1$, and therefore $|z_{\alpha}(\alpha_{1})|=|g'_{1}(0)|$. By the admissible-parameterization assumption \eqref{e8} we have $|z_{\alpha}(\alpha_{1})|\geqslant A^{-1}$, while the $C^{4}$ bound $\|g_{1}\|_{C^{4}}\lesssim 1$ (coming from the regularity of $z$) gives $|g'_{1}(0)|\lesssim 1$. Thus $|g'_{1}(0)|\approx 1$ (with implied constants depending only on $A$). Therefore, $g_{1}$ is locally invertible in a sufficiently small neighborhood of $0$. So we can define $\beta_{1}:=g_{1}^{-1}$ and $f_{1}:=h_{1}\circ g_{1}^{-1}$, then we can check that
\[
\begin{aligned}
&f_{1}(0)=(h_{1}\circ g_{1}^{-1})(0)=h_{1}(g_{1}^{-1}(0))=h_{1}(0)=-\frac{d}{2}, \\
&\beta_{1}(0)=g_{1}^{-1}(0)=0, \\
&f'_{1}(0)=(h_{1}\circ g_{1}^{-1})'(0)=h'_{1}(g_{1}^{-1}(0))(g_{1}^{-1})'(0)=0, \\
&|\beta'_{1}(0)|=|(g_{1}^{-1})'(0)|\approx 1, \\
&\|f_{1}\|_{C^{4}}+\|\beta_{1}\|_{C^{4}}\lesssim 1.
\end{aligned}
\]
Similarly we can define $g_{2},h_{2}$ such that $z(\alpha_{2}+\beta)=R(g_{2}(\beta),h_{2}(\beta))$, and we can also get $g_{2}$ is locally invertible in a sufficiently small neighborhood of $0$, then we define $\beta_{2}:=g_{2}^{-1}$ and $f_{2}:=h_{2}\circ g_{2}^{-1}$. As before we have
\[
\begin{aligned}
&f_{2}(0)=\frac{d}{2}, \quad
\beta_{2}(0)=0, \quad
f'_{2}(0)=0, \\
&|\beta'_{2}(0)|\approx 1, \quad
\|f_{2}\|_{C^{4}}+\|\beta_{2}\|_{C^{4}}\lesssim 1.
\end{aligned}
\]
We finally prove the last assertion of \eqref{e30}, $f_{2}(\rho)-f_{1}(\rho)\geqslant d$ for
$\rho\in(-\varepsilon_{5}^{4},\varepsilon_{5}^{4})$. By \eqref{e29} the two points
$R(\rho,f_{1}(\rho))$ and $R(\rho,f_{2}(\rho))$ are the image under $R$ of the interface points
with parameters $\alpha_{1}+\beta_{1}(\rho)$ and $\alpha_{2}+\beta_{2}(\rho)$; since $R$ is a rigid
motion it preserves distances, so
\[
|z(\alpha_{2}+\beta_{2}(\rho))-z(\alpha_{1}+\beta_{1}(\rho))|
=\big|R(\rho,f_{2}(\rho))-R(\rho,f_{1}(\rho))\big|=|f_{2}(\rho)-f_{1}(\rho)|.
\]
By \eqref{e22} we have $|\alpha_{2}-\alpha_{1}|\in[2\varepsilon_{1},\frac{1}{2\varepsilon_{1}}]$,
while $\beta_{j}(0)=0$ and $\|\beta'_{j}\|_{L^{\infty}}\lesssim1$ give
$|\beta_{j}(\rho)|\lesssim|\rho|\leqslant\varepsilon_{5}^{4}$; after shrinking $\varepsilon_{5}$
we may therefore assume
\[
\big|(\alpha_{2}+\beta_{2}(\rho))-(\alpha_{1}+\beta_{1}(\rho))\big|
\in[\varepsilon_{1},\varepsilon_{1}^{-1}]\qquad\text{for }|\rho|\leqslant\varepsilon_{5}^{4}.
\]
Since $d$ is by definition the infimum of $|z(\alpha)-z(\beta)|$ over exactly this range of
parameter differences, we conclude $|f_{2}(\rho)-f_{1}(\rho)|\geqslant d$; and since
$f_{2}(0)-f_{1}(0)=d>0$ and $\rho\mapsto f_{2}(\rho)-f_{1}(\rho)$ is continuous on the connected
interval $(-\varepsilon_{5}^{4},\varepsilon_{5}^{4})$ while never vanishing (its modulus is at
least $d$), the difference keeps the positive sign, i.e.,\ $f_{2}(\rho)-f_{1}(\rho)\geqslant d$.

It remains to prove the remaining assertion of the last line of \eqref{e30}, namely
$-\beta'_{1}(\rho)\beta'_{2}(\rho)\approx1$. Since $h'_{1}(0)=h'_{2}(0)=0$, the computation above gives
$z_{\alpha}(\alpha_{j})=\partial_{1}R\,g'_{j}(0)$ for $j=1,2$; the antiparallel
condition $z_{\alpha}(\alpha_{1})\cdot z_{\alpha}(\alpha_{2})<0$ of \eqref{e26}
therefore yields $g'_{1}(0)g'_{2}(0)<0$. Because
$\beta'_{j}(0)=1/g'_{j}(0)$, the two numbers $\beta'_{1}(0)$ and $\beta'_{2}(0)$ have
opposite signs and satisfy $|\beta'_{j}(0)|\approx 1$, so that
$-\beta'_{1}(0)\beta'_{2}(0)\approx 1$. By the bounds
$\|f_{j}\|_{C^{4}}+\|\beta_{j}\|_{C^{4}}\lesssim 1$ and continuity, the same conclusion
holds for every $\rho\in(-\varepsilon_{5}^{4},\varepsilon_{5}^{4})$ after shrinking
$\varepsilon_{5}$ if necessary. This completes the proof of \eqref{e30}. 

For \eqref{e31}, the first two inclusions come from the first two inclusions in \eqref{e27} after transferring to the $(\rho,y)$ coordinates: indeed, writing $z(\beta)-\rho'e=R(g_{1}(\beta),h_{1}(\beta)-\rho')$, the first inclusion of \eqref{e27} gives
$\{R(\rho,f_{1}(\rho)-\rho'):|\rho|<\varepsilon_{5}^{4},\ \rho'\in(0,\varepsilon_{5})\}\subset\Omega_{1}$
(here $\beta=\beta_{1}(\rho)$, so that $|\beta-\alpha_{1}|=|\beta_{1}(\rho)|\lesssim|\rho|<\varepsilon_{5}^{4}\leqslant\varepsilon_{5}$), and this set contains $R(\{(\rho,y):|\rho|<\varepsilon_{5}^{4},\,y\in(-\varepsilon_{5}^{2},f_{1}(\rho))\})$ because $f_{1}(\rho)-\varepsilon_{5}<-\varepsilon_{5}^{2}$ (as $|f_{1}(\rho)|\lesssim d+\rho^{2}\lesssim\varepsilon_{5}^{8}$); the second inclusion follows in the same way from the second inclusion of \eqref{e27}, using $f_{2}(\rho)+\varepsilon_{5}>\varepsilon_{5}^{2}$. So it suffices to show the last inclusion. Set
\[
L:=\bigl\{(\rho,y):|\rho|<\varepsilon_{5}^{4},\ f_{1}(\rho)<y<f_{2}(\rho)\bigr\}.
\]
The set $L$ is connected (it is the image of the connected set $(-\varepsilon_{5}^{4},\varepsilon_{5}^{4})\times(0,1)$ under the continuous map $(\rho,s)\mapsto(\rho,(1-s)f_{1}(\rho)+s f_{2}(\rho))$), and hence so is $R(L)$; moreover $R(L)$ meets $\Omega_{2}$, because
\[
R\bigl(\{(0,y):y\in(f_{1}(0),f_{2}(0))\}\bigr)
=R\bigl(\{(0,y):y\in(-\tfrac d2,\tfrac d2)\}\bigr)
=(z(\alpha_{1}),z(\alpha_{2}))\subset\Omega_{2}
\]
by \eqref{e25} and \eqref{e30}. Assume that $R(L)\not\subset\Omega_{2}$. If $R(L)\cap\Gamma_{t}=\varnothing$, then $R(L)$, being connected and contained in $\Omega_{1}\cup\Omega_{2}$, the union of two disjoint open sets, would have to lie in a single one of them, hence in $\Omega_{2}$, contrary to the assumption; so there is a point $(\rho,y)\in L$ with $R(\rho,y)\in\Gamma_{t}$, and we choose such a point $(\rho_{0},y_{0})$ with $|\rho_{0}|$ minimal. The minimum is attained: letting $\sigma$ denote the infimum of $|\rho|$ over all $(\rho,y)\in L$ with $R(\rho,y)\in\Gamma_{t}$, take a sequence $(\rho_{n},y_{n})$ with $|\rho_{n}|\to\sigma$; passing to a subsequence we may assume $\rho_{n}\to\rho_{\ast}$ and $y_{n}\to y_{\ast}$, where $|\rho_{\ast}|=\sigma<\varepsilon_{5}^{4}$ (all $(\rho_{n},y_{n})$ lie in $L$), $R(\rho_{\ast},y_{\ast})\in\Gamma_{t}$ and $y_{\ast}\in[f_{1}(\rho_{\ast}),f_{2}(\rho_{\ast})]$. If $y_{\ast}=f_{1}(\rho_{\ast})$, then $0<y_{n}-f_{1}(\rho_{n})\to0$, so that for all large $n$ we have $\rho'_{n}:=y_{n}-f_{1}(\rho_{n})\in(0,d)$; but the third inclusion of \eqref{e27}, transferred as above, which reads
$\{R(\rho,f_{1}(\rho)+\rho'):|\rho|<\varepsilon_{5}^{4},\ \rho'\in(0,d)\}\subset\Omega_{2}$,
then places $R(\rho_{n},y_{n})=R(\rho_{n},f_{1}(\rho_{n})+\rho'_{n})$ in $\Omega_{2}$, contradicting $R(\rho_{n},y_{n})\in\Gamma_{t}$. The case $y_{\ast}=f_{2}(\rho_{\ast})$ is identical using the fourth inclusion of \eqref{e27}, i.e.,
$\{R(\rho,f_{2}(\rho)-\rho'):|\rho|<\varepsilon_{5}^{4},\ \rho'\in(0,d)\}\subset\Omega_{2}$.
Thus $(\rho_{\ast},y_{\ast})\in L$ attains the minimum, and we take $(\rho_{0},y_{0}):=(\rho_{\ast},y_{\ast})$. Finally $\rho_{0}\ne0$: otherwise $R(\rho_{0},y_{0})\in R(\{(0,y):y\in(f_{1}(0),f_{2}(0))\})\subset\Omega_{2}$, contradicting $R(\rho_{0},y_{0})\in\Gamma_{t}$. We claim that $z_{\alpha}(\mu_{0})$ is parallel or antiparallel to the vector $e$, where $\mu_{0}\in\mathbb{R}$ is any parameter with $z(\mu_{0})=R(\rho_{0},y_{0})$; such a $\mu_{0}$ exists because $R(\rho_{0},y_{0})\in\Gamma_{t}$. Write $R^{-1}z_{\alpha}(\mu_{0})=(a,b)$; we show that $a=0$. Suppose $a\ne0$. If $b\ne0$, then $\Gamma_{t}$ is locally the graph $\rho=\varphi(y)$ with $\varphi(y_{0})=\rho_{0}$ and $\varphi'(y_{0})=a/b\ne0$; choosing $y$ close to $y_{0}$ so that $|\varphi(y)|<|\rho_{0}|$ (possible because $\rho_{0}\ne0$) and $y\in(f_{1}(\varphi(y)),f_{2}(\varphi(y)))$ (possible because $y_{0}\in(f_{1}(\rho_{0}),f_{2}(\rho_{0}))$ is an open condition and $f_{1},f_{2}$ are continuous), we obtain a point $R(\varphi(y),y)\in\Gamma_{t}\cap R(L)$ with $|\rho|<|\rho_{0}|$. If $b=0$, then $\Gamma_{t}$ is locally the graph $y=\psi(\rho)$, and the same openness reasoning gives $R(\rho,\psi(\rho))\in L$ for all $\rho$ close enough to $\rho_{0}$, in particular for some $\rho$ with $|\rho|<|\rho_{0}|$. In either case the minimality of $|\rho_{0}|$ is contradicted. Hence $a=0$, i.e.,\ $R^{-1}z_{\alpha}(\mu_{0})=(0,k)$ and $z_{\alpha}(\mu_{0})=k\,e$; by \eqref{e8} we have $|k|=|z_{\alpha}(\mu_{0})|\geqslant A^{-1}$, while the $C^{4}$ bound in \eqref{e8} gives the upper estimate $|k|\leqslant A+1$.

Then we use the Taylor expansion around the point $\mu_{0}$,
\[
|z(\mu+\beta)-z(\mu_{0})-\beta z_{\alpha}(\mu_{0})|\lesssim\beta^{2}.
\]
Since $e=\partial_{2}R$ is the second column of $Q$, one has $R^{-1}e=e_{2}$ and therefore
$R^{-1}z_{\alpha}(\mu_{0})=R^{-1}(ke)=(0,k)$; writing $R^{-1}(z(\mu_{0}+\beta))=(\rho,y)$, we get
\[
|(\rho,y)-(\rho_{0},y_{0}+k\beta)|\lesssim\beta^{2}.
\]
Since $y_{0}\in(f_{1}(\rho_{0}),f_{2}(\rho_{0}))$, and since $f_{j}(\rho)=f_{j}(0)+O(\rho^{2})$ (with a constant depending only on $A$) by $f'_{j}(0)=0$ and $\|f_{j}\|_{C^{4}}\lesssim 1$ with $f_{1}(0)=-d/2$, $f_{2}(0)=d/2$, we have $|y_{0}|\lesssim d+\rho_{0}^{2}\lesssim\varepsilon_{5}^{8}$ (recall $d\approx\vartheta\leqslant\varepsilon_{4}\leqslant\varepsilon_{5}^{100}$). Choose $\beta$ with $|\beta|=\frac{\varepsilon_{5}^{4}}{2}$ and $\operatorname{sgn}(\beta)=-\operatorname{sgn}(k)$. Then
\[
|\rho|\leqslant|\rho_{0}|+C\beta^{2}\lesssim\varepsilon_{5}^{4}\leqslant\varepsilon_{5}^{2},
\]
and, because $|k|\geqslant A^{-1}$ and $|k|\leqslant A+1$, we obtain, after shrinking $\varepsilon_{5}$ depending on $A$ if necessary so that $\varepsilon_{5}^{4}\lesssim A^{-1}$ and $A\varepsilon_{5}\lesssim1$,
\[
y=y_{0}+k\beta+O(\beta^{2})\leqslant-\frac{\varepsilon_{5}^{4}}{2A}+O(\varepsilon_{5}^{8})
\leqslant-\frac{\varepsilon_{5}^{4}}{4A},\qquad
y\geqslant-O(\varepsilon_{5}^{8})-(A+1)\frac{\varepsilon_{5}^{4}}{2}\geqslant-\varepsilon_{5}^{3}.
\]
Also, $f_{1}(\rho)=-\frac d2+O(\rho^{2})$ satisfies $|f_{1}(\rho)|\lesssim d+\rho^{2}\lesssim\varepsilon_{5}^{8}$, so that $y<f_{1}(\rho)$ (because $y\leqslant-\frac{\varepsilon_{5}^{4}}{4A}<0$ while $f_{1}(\rho)\geqslant-C\varepsilon_{5}^{8}$, using $\varepsilon_{5}^{4}\lesssim A^{-1}$), and also $y>f_{1}(\rho)-\frac{\varepsilon_{5}^{2}}{2}$ (because $y\geqslant-\varepsilon_{5}^{3}>-\frac{\varepsilon_{5}^{2}}{4}$ while $f_{1}(\rho)-\frac{\varepsilon_{5}^{2}}{2}\leqslant C\varepsilon_{5}^{8}-\frac{\varepsilon_{5}^{2}}{2}\leqslant-\frac{\varepsilon_{5}^{2}}{4}$).
On the other hand the first inclusion of \eqref{e27}, transferred to the coordinates $(\rho,y)$, gives $R(\rho,y)\in\Omega_{1}$ whenever $|\rho|\leqslant2\varepsilon_{5}^{4}$ and $f_{1}(\rho)-\frac{\varepsilon_{5}^{2}}{2}<y<f_{1}(\rho)$: indeed $z(\alpha_{1}+\beta_{1}(\rho))-\rho'e=R(g_{1}(\beta_{1}(\rho)),h_{1}(\beta_{1}(\rho))-\rho')=R(\rho,f_{1}(\rho)-\rho')$ with $\rho'=f_{1}(\rho)-y\in(0,\frac{\varepsilon_{5}^{2}}{2})\subset(0,\varepsilon_{5})$, while $|\beta_{1}(\rho)|\lesssim|\rho|\leqslant2\varepsilon_{5}^{4}\leqslant\varepsilon_{5}$ and, after shrinking $\varepsilon_{5}$ depending on $A$, the functions $\beta_{1},f_{1}$ are defined on $(-2\varepsilon_{5}^{4},2\varepsilon_{5}^{4})$ with $|\beta'_{1}|\lesssim1$ and $\|f_{1}\|_{C^{4}}\lesssim1$ there, because $|g'_{1}(0)|\approx1$ and $\|g_{1}\|_{C^{4}}\lesssim1$ make $g_{1}$ a diffeomorphism with $|g'_{1}|\approx1$ on a neighborhood of $0$ whose size depends only on $A$. The point $(\rho,y)$ constructed above satisfies all these hypotheses, since $|\rho|\leqslant|\rho_{0}|+C\beta^{2}<\varepsilon_{5}^{4}+C\varepsilon_{5}^{8}\leqslant2\varepsilon_{5}^{4}$, $y<f_{1}(\rho)$ and $y>f_{1}(\rho)-\frac{\varepsilon_{5}^{2}}{2}$. The point $z(\mu_{0}+\beta)=R(\rho,y)$ therefore lies at the same time in $\Omega_{1}$ and on $\Gamma_{t}$, which is impossible. This completes the proof of the proposition.
\end{proof}

\section{Proof of the main theorem}\label{sec:main}
Throughout this section and the next, $\epsilon$ denotes the small parameter of the localized
analysis: at a given time it is $\epsilon=\tilde d/2$, where $\tilde d$ is the length of the
shortest admissible chord at that time, so that $\epsilon=d/2$ in the notation of Proposition
\ref{chord-arc prop} (see Step 2 of the proof of Theorem \ref{thm}). Likewise
$(\alpha_{1},\alpha_{2})$ always denotes the conjugate pair furnished by that proposition at the
time under consideration, and $\varepsilon_{1},\dots,\varepsilon_{5}$ are the constants of Lemmas
\ref{chord} and \ref{chord-arc} and of Proposition \ref{chord-arc prop}. A constant written $\lesssim_{A}$ below may
also depend on the fixed time horizon $T$, through $\varepsilon_{3}=\varepsilon_{3}(A,T,\varepsilon_{1})$
and hence through $\varepsilon_{5}$, which accounts for the dependence on $T$ in Theorem
\ref{thm}. The notation is collected for convenience.

\begin{center}
\begin{tabular}{lp{4.35in}}
\hline
symbol & meaning \\
\hline
$\vartheta$ & $\inf_{\alpha,\beta}F(z)(t,\alpha,\beta)$ at the time under consideration (\eqref{e21}) \\
$\varepsilon_{1},\varepsilon_{2}$ & constants of Lemma \ref{chord}, depending only on $A$ \\
$\varepsilon_{3}$ & constant of Lemma \ref{chord-arc}, depending on $A,T,\delta$ (with $\delta=\varepsilon_{1}$ in the applications) \\
$\varepsilon_{4},\varepsilon_{5}$ & constants of Proposition \ref{chord-arc prop}, subject to \eqref{eps45.list} \\
$d$, $\tilde d$ & length of the shortest admissible chord; $d=\tilde d$ in Step 2 \\
$\epsilon$ & $d/2=\tilde d/2$, the width parameter of Appendix~\ref{sec:singular} \\
$D(t)$, $m(t)$ & $D=m^{-2}$, $m$ the shortest admissible chord at time $t$ (\eqref{D.def}) \\
$\tilde D(t)$ & $|z(t,\alpha_{1})-z(t,\alpha_{2})|^{-2}$ for the pair fixed at $t_{2}$ \\
$f,g,\beta_{1},\beta_{2}$ & local graph data of \eqref{e29}--\eqref{e30}; $g=-f_{1}-\epsilon$, $f=f_{2}-\epsilon$ \\
$\omega^{\pm}$ & $\omega^{-}=-\omega_{1}$, $\omega^{+}=\omega_{2}$ (\eqref{omV.def}) \\
$T_{1},T_{2},T_{3},T_{4},H,M$ & singular-integral operators of Appendix~\ref{sec:singular} \\
$\mathcal{C}$, $\|\cdot\|_{Z}$ & test-function class and norm of Appendix~\ref{sec:singular} \\
\hline
\end{tabular}
\end{center}

Three typographically close but logically unrelated symbols occur below:
$\vartheta$, the \emph{value} of the chord-arc functional at the time under
consideration; $\epsilon=d/2$, the \emph{width} of the narrow channel of Appendix~\ref{sec:singular};
and $\varepsilon_{1},\dots,\varepsilon_{5}$, the \emph{thresholds} of Lemma \ref{chord} and
Proposition \ref{chord-arc prop}. They are of very different size: $\vartheta$ is of the order of
$\varepsilon_{2}$, while $\epsilon=d/2$ is at most $\varepsilon_{4}$ and
$\varepsilon_{4}\leqslant\varepsilon_{5}^{100}$.
Throughout, the relation $\approx$ carries implied constants depending only on $A$, unless a
different dependence is stated explicitly.

Let $\psi:\mathbb{R}\to[0,1]$ be a smooth, even function supported in $[-1,1]$ and equal to $1$ in the interval $[-\frac{1}{2},\frac{1}{2}]$, let
\begin{equation}\label{cutoff}
\psi_{r}(\xi):=\psi(\xi/r),\qquad r>0,
\end{equation}
and define the function $S_{t}^{\infty}:\mathbb{R}^{2}\to\mathbb{R}^{2}$
\begin{equation}\label{S.inf}
S_{t}^{\infty}(z):=\frac{1}{2\pi}\,\mathrm{p.v.}\!\int_{\mathbb{R}}\frac{(z-z(t,\beta))^{\bot}}{|z-z(t,\beta)|^{2}}\omega_{\Gamma}(t,\beta)[1-\psi_{\varepsilon_{5}^{10}}(\beta-\alpha_{1})-\psi_{\varepsilon_{5}^{10}}(\beta-\alpha_{2})]d\beta.
\end{equation}
The principal value is needed for the same reason as in \eqref{e18}: by the periodicity
$z(t,\beta+2\pi)=z(t,\beta)+(2\pi,0)$ the second component of the kernel decays only like
$|\beta|^{-1}$ as $|\beta|\to\infty$, so that the tail is merely conditionally convergent.
We verify that the compositions $S_{t}^{\infty}(z(t,\alpha_{m}+\beta_{m}(\cdot)))$ are $C^{3}$
in the parameter, with a bound depending only on $A$. Throughout we use the periodicity condition of Definition \ref{adm}
together with \eqref{e8}, that is $|z(t,\beta)-(\beta,0)|\leqslant A$ for all $\beta$, and the bound
$\|\omega_{\Gamma}\|_{L^{\infty}}\leqslant A$ of \eqref{e19} (Remark \ref{rm:normalize}). By the periodicity
$z(t,\alpha+2\pi k)=z(t,\alpha)+(2\pi k,0)$ we may shift $\alpha_{1},\alpha_{2}$ by a common
multiple of $2\pi$, which translates the whole configuration and leaves the norm in \eqref{S.C3}
unchanged; we therefore choose $k\in\mathbb{Z}$ such that
$\bigl|\frac12(\alpha_{1}+\alpha_{2})+2\pi k\bigr|\leqslant\pi$. Then
$|\alpha_{m}|\leqslant\pi+\frac12|\alpha_{1}-\alpha_{2}|\leqslant\pi+(4\varepsilon_{1})^{-1}$
by \eqref{e22}, and hence, by \eqref{e8},
\begin{equation}\label{S.loc}
|z|\leqslant \pi+(4\varepsilon_{1})^{-1}+A+\varepsilon_{5}^{12}\leqslant C(A)
\qquad\text{for every }z\in B(z(t,\alpha_{m}),\varepsilon_{5}^{12}).
\end{equation}

(a) \emph{The cutoffs remove the two singular parameters.} Since $\psi=1$ on
$[-\frac12,\frac12]$ and $\operatorname{supp}\psi\subseteq[-1,1]$, the factor
$1-\psi_{\varepsilon_{5}^{10}}(\beta-\alpha_{1})-\psi_{\varepsilon_{5}^{10}}(\beta-\alpha_{2})$
vanishes identically on the union of the two intervals
$|\beta-\alpha_{m}|\leqslant \frac12\varepsilon_{5}^{10}$: indeed, if
$|\beta-\alpha_{m}|\leqslant \frac12\varepsilon_{5}^{10}$ then
$|\beta-\alpha_{m'}|\geqslant |\alpha_{1}-\alpha_{2}|-\frac12\varepsilon_{5}^{10}
\geqslant 2\varepsilon_{1}-\frac12\varepsilon_{5}^{10}>\varepsilon_{5}^{10}$ for the other index $m'$
(by \eqref{e22} and $\varepsilon_{5}^{10}\ll\varepsilon_{1}$), so the factor equals $1-1-0=0$.
Hence, on the support of the integrand,
\begin{equation}\label{S.supp}
|\beta-\alpha_{m}|\geqslant \tfrac12\varepsilon_{5}^{10}\quad (m=1,2).
\end{equation}

(b) \emph{Inside the strip the interface consists of the two sheets only.} By \eqref{e29} the two
branches of $\Gamma_{t}$ near $z(t,\alpha_{m})$ are $R(\rho,f_{m}(\rho))$ with
$|\rho|\leqslant \varepsilon_{5}^{4}$, and by \eqref{e31} the sets
$\{(\rho,y):|\rho|<\varepsilon_{5}^{4},\ y>f_{2}(\rho)\}$ and
$\{(\rho,y):|\rho|<\varepsilon_{5}^{4},\ y<f_{1}(\rho)\}$ are mapped by $R$ into $\Omega_{1}(t)$,
while $\{(\rho,y):|\rho|<\varepsilon_{5}^{4},\ f_{1}(\rho)<y<f_{2}(\rho)\}$ is mapped into
$\Omega_{2}(t)$. Since $|f_{j}(\rho)|\leqslant \frac d2+C\rho^{2}\leqslant \varepsilon_{5}^{7}$ by \eqref{e30}, these
three sets cover the strip
\[
\Sigma:=R\big(\{(\rho,y):|\rho|<\varepsilon_{5}^{4},\ |y|<\varepsilon_{5}^{2}\}\big)
\]
minus the two graphs, and therefore
\begin{equation}\label{S.strip}
\Gamma_{t}\cap\Sigma
=R\Big(\{(\rho,f_{1}(\rho)):|\rho|\leqslant \varepsilon_{5}^{4}\}\cup
\{(\rho,f_{2}(\rho)):|\rho|\leqslant \varepsilon_{5}^{4}\}\Big):
\end{equation}
every point of $\Gamma_{t}$ in the strip is $z(t,\alpha_{1}+\beta_{1}(\rho))$ or
$z(t,\alpha_{2}+\beta_{2}(\rho))$ for some $|\rho|\leqslant \varepsilon_{5}^{4}$, and conversely.

(c) \emph{A lower bound for the denominator, uniform in $d$.} Let $z\in B(z(t,\alpha_{m}),r)$ with
$r:=\varepsilon_{5}^{12}$, and let $\beta$ lie in the support of the integrand, so that
$|\beta-\alpha_{m'}|\geqslant \frac12\varepsilon_{5}^{10}$ for $m'=1,2$ by \eqref{S.supp}. Write
$z=R(\rho_{z},y_{z})$ and $z(t,\beta)=R(\rho_{\beta},y_{\beta})$; since $R$ is an isometry,
\begin{equation}\label{S.den}
|z-z(t,\beta)|\geqslant |\rho_{z}-\rho_{\beta}|.
\end{equation}
Here $|\rho_{z}|\leqslant r$, because $z\in B(z(t,\alpha_{m}),r)=B(R(0,f_{m}(0)),r)$ and
$R(0,f_{m}(0))=z(t,\alpha_{m})$ by \eqref{e29}, while $|y_{z}|\leqslant|f_{m}(0)|+r\leqslant \varepsilon_{5}^{7}$.
We distinguish three cases.

If $|\rho_{\beta}|\geqslant \varepsilon_{5}^{4}$, then $|\rho_{z}-\rho_{\beta}|\geqslant \frac12\varepsilon_{5}^{4}$.

If $|\rho_{\beta}|<\varepsilon_{5}^{4}$ and $|y_{\beta}|<\varepsilon_{5}^{2}$, then
$z(t,\beta)\in\Sigma$ and \eqref{S.strip} shows that $z(t,\beta)$ lies on one of the two sheets,
say $z(t,\beta)=z(t,\alpha_{m'}+\beta_{m'}(\rho_{\beta}))$. By \eqref{e30} the map $\beta_{m'}$ is a
diffeomorphism of $(-\varepsilon_{5}^{4},\varepsilon_{5}^{4})$ onto its image; differentiating
\eqref{e29} and taking absolute values gives
\[
|\beta'_{m'}(\rho)|=\frac{\sqrt{1+|f'_{m'}(\rho)|^{2}}}{|z_{\alpha}(t,\alpha_{m'}+\beta_{m'}(\rho))|},
\]
and since $|f'_{m'}|\lesssim1$ while $A^{-1}\leqslant|z_{\alpha}|\leqslant A+1$ by \eqref{e8}, there is a
constant $C_{A}>1$, depending only on $A$ and kept fixed in the rest of this proof, such that
\[
\frac{1}{C_{A}}|\rho|\leqslant|\beta_{m'}(\rho)|\leqslant C_{A}|\rho|,
\qquad |\rho|\leqslant\varepsilon_{5}^{4}.
\]
The hypothesis $|\beta-\alpha_{m'}|\geqslant \frac12\varepsilon_{5}^{10}$, i.e.\
$|\beta_{m'}(\rho_{\beta})|\geqslant \frac12\varepsilon_{5}^{10}$, and the second of these two
inequalities therefore give $|\rho_{\beta}|\geqslant \frac{1}{2C_{A}}\varepsilon_{5}^{10}$; since
$|\rho_{z}|\leqslant r=\varepsilon_{5}^{12}\leqslant\frac{1}{4C_{A}}\varepsilon_{5}^{10}$ after
shrinking $\varepsilon_{5}$ once more, we obtain
$|\rho_{z}-\rho_{\beta}|\geqslant \frac{1}{4C_{A}}\varepsilon_{5}^{10}$.

If $|\rho_{\beta}|<\varepsilon_{5}^{4}$ and $|y_{\beta}|\geqslant \varepsilon_{5}^{2}$, then
$z(t,\beta)\notin\Sigma$ and we use the second coordinate instead:
$|z-z(t,\beta)|\geqslant |y_{z}-y_{\beta}|\geqslant \varepsilon_{5}^{2}-\varepsilon_{5}^{7}\geqslant \frac12\varepsilon_{5}^{2}$.

In all three cases therefore
\[
|z-z(t,\beta)|\geqslant \min\Big\{\tfrac12\varepsilon_{5}^{4},\ \tfrac{1}{4C_{A}}\varepsilon_{5}^{10},\
\tfrac12\varepsilon_{5}^{2}\Big\}\geqslant \tfrac{1}{4C_{A}}\varepsilon_{5}^{10},
\]
Hence the kernel
$K(w):=w^{\bot}/|w|^{2}$ is smooth on $B(z(t,\alpha_{m}),r)\times\operatorname{supp}$ and, for
$k=0,1,2,3$,
\begin{equation}\label{S.deriv}
\big|\partial_{z}^{k}K(z-z(t,\beta))\big|\leqslant C_{k}\,(\varepsilon_{5}^{10})^{-1-k},
\end{equation}
since $\partial_{w}^{k}K$ is homogeneous of degree $-1-k$; the constants $C_{k}$ absorb powers of
$C_{A}$ and therefore depend only on $A$.

(d) \emph{Differentiation under the integral sign.} By \eqref{S.deriv} the integrand of
$\partial_{z}^{3}S_{t}^{\infty}$ is bounded by $C(A)(\varepsilon_{5}^{10})^{-4}|\omega_{\Gamma}(t,\beta)|$
on the part of the support carried by the two sheets, which is a set of measure at most
$4\varepsilon_{5}^{4}$ in $\beta$ by (b). On the remaining part it is bounded by
$C\min\{\varepsilon_{5}^{-16},\varepsilon_{5}^{-8},|\beta|^{-4}\}\,|\omega_{\Gamma}(t,\beta)|$.
For $\beta$ outside the two sheets one has $|\rho_{\beta}|\geqslant\varepsilon_{5}^{4}$ or
$|y_{\beta}|\geqslant\varepsilon_{5}^{2}$: the alternatives
$|\rho_{\beta}|<\varepsilon_{5}^{4}$ and $|y_{\beta}|<\varepsilon_{5}^{2}$ would place
$z(t,\beta)$ in the strip $\Sigma$, hence on one of the two sheets by \eqref{S.strip}. In the first
case $|z-z(t,\beta)|\geqslant|\rho_{z}-\rho_{\beta}|\geqslant\frac12\varepsilon_{5}^{4}$ by (c);
in the second $|z-z(t,\beta)|\geqslant|y_{z}-y_{\beta}|\geqslant\frac12\varepsilon_{5}^{2}$, again
by (c); and in the far field
$|z-z(t,\beta)|\geqslant|\beta|-A-|z|\geqslant|\beta|-C(A)\geqslant\frac12|\beta|$
whenever $|\beta|\geqslant B:=2C(A)$, by \eqref{S.loc}. Hence,
splitting $\beta$ into the two sheets, the set $\{|\beta|\leqslant B\}$, and its complement,
\[
\begin{aligned}
&\int_{\mathbb{R}}\sup_{z\in B(z(t,\alpha_{m}),r)}|\partial_{z}^{3}K(z-z(t,\beta))|\,|\omega_{\Gamma}(t,\beta)|\,d\beta\\
&\quad\leqslant C(A)\Big((\varepsilon_{5}^{10})^{-4}\varepsilon_{5}^{4}
+(\varepsilon_{5}^{-16}+\varepsilon_{5}^{-8})B+B^{-3}\Big)
\leqslant C(A)\big(\varepsilon_{5}^{-36}+1\big),
\end{aligned}
\]
the last inequality because $B$ depends only on $A$ and $\varepsilon_{5}^{-8}B\leqslant
\varepsilon_{5}^{-36}$ after shrinking $\varepsilon_{5}$ once more. Since for fixed $\beta$ in the support the map
$z\mapsto K(z-z(t,\beta))$ is smooth and the bounds are uniform in $z\in B(z(t,\alpha_{m}),r)$, the
dominated convergence theorem justifies differentiating three times under the integral sign. Hence
$S_{t}^{\infty}\in C^{3}(B(z(t,\alpha_{1}),r)\cup B(z(t,\alpha_{2}),r))$ and
\begin{equation}\label{S.C3}
\|S_{t}^{\infty}\|_{C^{3}(B(z(t,\alpha_{1}),\varepsilon_{5}^{12})\cup
B(z(t,\alpha_{2}),\varepsilon_{5}^{12}))}\lesssim_{A}1.
\end{equation}
The constant $\varepsilon_{5}^{-36}$ depends only on $A$ and $T$, and only the bound
\eqref{S.C3} itself is used below.
In particular, the compositions $\mu\mapsto S_{t}^{\infty}(z(t,\alpha_{m}+\beta_{m}(\mu)))$ are
$C^{3}$ on $|\mu|\leqslant2\sqrt{\epsilon}$, since $|z(t,\alpha_{m}+\beta_{m}(\mu))-z(t,\alpha_{m})|
\lesssim|\mu|\lesssim\sqrt{\epsilon}\lesssim\varepsilon_{5}^{50}\ll\varepsilon_{5}^{12}$ by
\eqref{e29}, so the points in question lie in the ball. 

Therefore, letting $Q:\mathbb{R}^{2}\to\mathbb{R}^{2}, Q(x):=R(x)-R(0)$, where $R$ is defined in \eqref{e28}, we have 
\begin{equation}\label{zabeta.id}
\begin{aligned}
\frac{z_{\alpha}(\alpha_{1}+\beta_{1}(\mu))}{|z_{\alpha}(\alpha_{1}+\beta_{1}(\mu))|^{2}}&=\frac{\beta'_{1}(\mu)Q(1,f'_{1}(\mu))}{1+|f'_{1}(\mu)|^{2}}, \\
\frac{z_{\alpha}(\alpha_{2}+\beta_{2}(\mu))}{|z_{\alpha}(\alpha_{2}+\beta_{2}(\mu))|^{2}}&=\frac{\beta'_{2}(\mu)Q(1,f'_{2}(\mu))}{1+|f'_{2}(\mu)|^{2}}.
\end{aligned}
\end{equation} 
All the quantities below are evaluated at one and the same time $t$ (in the proof of
Theorem \ref{thm} we will take $t=t_{2}$). Let
$\omega_{1},\omega_{2}:(-\varepsilon_{5}^{6},\varepsilon_{5}^{6})\to\mathbb{R}$ and
$V_{1},V_{2}:(-\varepsilon_{5}^{6},\varepsilon_{5}^{6})\to\mathbb{R}^{2}$ be given by
\begin{equation}\label{omV.def}
\begin{aligned}
\omega_{1}(\rho)&:=\omega_{\Gamma}(t,\alpha_{1}+\beta_{1}(\rho))\beta'_{1}(\rho), \qquad \omega_{2}(\rho):=\omega_{\Gamma}(t,\alpha_{2}+\beta_{2}(\rho))\beta'_{2}(\rho), \\
V_{1}(\rho)&:=Q^{-1}(v_{1}(t,\alpha_{1}+\beta_{1}(\rho))), \qquad V_{2}(\rho):=Q^{-1}(v_{1}(t,\alpha_{2}+\beta_{2}(\rho))).
\end{aligned}
\end{equation}
\emph{Convention on the integrals in $\rho$.} The functions $\omega_{m}$ of \eqref{omV.def} are defined
on $(-\varepsilon_{5}^{6},\varepsilon_{5}^{6})$, whereas $\beta_{m}$ and $f_{m}$ are defined on
$(-\varepsilon_{5}^{4},\varepsilon_{5}^{4})$; see \eqref{e29} and \eqref{e30}. Since
$\beta_{m}(0)=0$, $|\beta'_{m}(0)|\approx1$ and $\|\beta_{m}\|_{C^{4}}\lesssim1$, the map $\beta_{m}$
is bilipschitz near $0$: with the constant $C_{A}$ of item (c),
$|\beta_{m}(\rho)|\geqslant C_{A}^{-1}|\rho|$ for $|\rho|\leqslant\varepsilon_{5}^{6}$. Hence the
factor $\psi_{\varepsilon_{5}^{10}}(\beta_{m}(\rho))$ vanishes identically for
$C_{A}\varepsilon_{5}^{10}\leqslant|\rho|<\varepsilon_{5}^{6}$. Every integral in $\rho$ occurring in
this section is therefore understood as an integral over the compact interval
$[-\varepsilon_{5}^{9},\varepsilon_{5}^{9}]$, which is contained in
$(-\varepsilon_{5}^{6},\varepsilon_{5}^{6})$ and on which all the functions appearing in the
integrand are defined, the integrand being extended by $0$ outside it; this extension changes
nothing, since the integrand already vanishes for $C_{A}\varepsilon_{5}^{10}\leqslant|\rho|\leqslant
\varepsilon_{5}^{9}$, and $C_{A}\varepsilon_{5}^{10}\leqslant\varepsilon_{5}^{9}$ after shrinking
$\varepsilon_{5}$ if necessary. The factors
$\psi_{\varepsilon_{5}^{10}}(\beta_{m}(\rho))$ equal $1$ only where
$|\beta_{m}(\rho)|\leqslant\frac12\varepsilon_{5}^{10}$ and vanish where
$|\beta_{m}(\rho)|\geqslant\varepsilon_{5}^{10}$; in \eqref{S.supp} the same cutoff is applied to the
parameter $\beta-\alpha_{m}$ rather than to $\beta_{m}(\rho)$. In Appendix~\ref{sec:singular}, $\rho$
is again a summation variable, over compact intervals fixed by the cutoffs
$\psi_{\sqrt{\epsilon}/2^{6}}$.

Using the Biot--Savart formula of Lemma \ref{Euler formula}, together with the identity
\eqref{zabeta.id} for the local term and the definition \eqref{S.inf} of $S_{t}^{\infty}$ for the
far-field term, we get
\begin{equation}\label{V.decomp}
\begin{aligned}
V_{n}(\mu)&=Q^{-1}(c(t))+Q^{-1}\big(w(t,z(t,\alpha_{n}+\beta_{n}(\mu)))\big)
+\frac{\omega_{n}(\mu)(1,f'_{n}(\mu))}{2(1+|f'_{n}(\mu)|^{2})} \\
&\quad+Q^{-1}(S_{t}^{\infty}(z(\alpha_{n}+\beta_{n}(\mu)))) \\
&\quad+\sum_{m=1}^{2}\frac{1}{2\pi}p.v.\int_{\mathbb{R}}\frac{(-f_{n}(\mu)+f_{m}(\rho),\mu-\rho)}{|\mu-\rho|^{2}+|f_{n}(\mu)-f_{m}(\rho)|^{2}}\omega_{m}(\rho)\,\psi_{\varepsilon_{5}^{10}}(\beta_{m}(\rho))\,d\rho \\
&\quad+Q^{-1}\Big(\frac{1}{2\pi}\int_{\Omega_{2}(t)}\frac{(z(t,\alpha_{n}+\beta_{n}(\mu))-y)^{\bot}}{|z(t,\alpha_{n}+\beta_{n}(\mu))-y|^{2}}\omega(t,y)dy\Big) \\
&=:Q^{-1}(c(t))+Q^{-1}w_{n}(\mu)+U_{n}(\mu)+W_{n}(\mu),
\end{aligned}
\end{equation}
In the principal-value integrals we have used \eqref{Q.perp}: since
$z(\alpha_{n}+\beta_{n}(\mu))-z(\alpha_{m}+\beta_{m}(\rho))
=Q\big((\mu,f_{n}(\mu))-(\rho,f_{m}(\rho))\big)$ by \eqref{e29}, applying $Q^{-1}$ to the
orthogonal of this vector produces exactly $\big(f_{m}(\rho)-f_{n}(\mu),\,\mu-\rho\big)$.

Here $V_{n}=Q^{-1}(c(t))+Q^{-1}w_{n}+U_{n}+W_{n}$ with
$w_{n}(\mu):=w(t,z(t,\alpha_{n}+\beta_{n}(\mu)))$ is the following splitting: $U_{n}$ collects the
three contributions of the \emph{interface} vorticity listed above, namely the local
Birkhoff--Rott jump term, the far-field term $S_{t}^{\infty}$ and the two principal-value
integrals, while $W_{n}$ is the contribution of the \emph{bulk} vorticity of the conducting
fluid,
\begin{equation}\label{V.decomp.W}
W_{n}(\mu):=Q^{-1}\Big(\frac{1}{2\pi}\int_{\Omega_{2}(t)}\frac{(z(t,\alpha_{n}+\beta_{n}(\mu))-y)^{\bot}}{|z(t,\alpha_{n}+\beta_{n}(\mu))-y|^{2}}\,\omega(t,y)\,dy\Big),
\end{equation}
and $Q^{-1}(c(t))+Q^{-1}w_{n}$ is the harmonic part of \eqref{e9.norm} (the constant $c(t)$
together with the field $w$, which is bounded with bounded gradient on the whole plane). The constant $c(t)$ cancels from the differences
$V_{1}-V_{2}$ that occur in Step 1, while $w_{1}-w_{2}$ contributes at most $A\tilde d$ to their
$e_{2}$-component, by the gradient bound of \eqref{e9.norm} and the mean value theorem; see Remark
\ref{rm:constant}, where the point that this bound must hold on the whole plane is discussed. Steps 2--3 concern $U_{n}$ and
$W_{n}$ separately, and neither of them contains the harmonic part; the latter is therefore dropped
from the notation in the sequel, in accordance with Remark \ref{rm:constant}.
The factor $\frac{1}{2\pi}$ is the one occurring in the Biot--Savart formula \eqref{e18}. The
first three lines of \eqref{V.decomp} give the three-term formula for $U_{n}$:
\begin{equation}\label{V.decomp.U}
\begin{aligned}
U_{n}(\mu)=&\frac{\omega_{n}(\mu)(1,f'_{n}(\mu))}{2(1+|f'_{n}(\mu)|^{2})}+Q^{-1}(S_{t}^{\infty}(z(\alpha_{n}+\beta_{n}(\mu)))) \\
&+\sum_{m=1}^{2}\frac{1}{2\pi}p.v.\int_{\mathbb{R}}\frac{(-f_{n}(\mu)+f_{m}(\rho),\mu-\rho)}{|\mu-\rho|^{2}+|f_{n}(\mu)-f_{m}(\rho)|^{2}}\omega_{m}(\rho)\,\psi_{\varepsilon_{5}^{10}}(\beta_{m}(\rho))\,d\rho.
\end{aligned}
\end{equation}

\begin{proof}[Proof of Theorem \ref{thm}]
Let $(z,u_{1},u_{2},B,p_{1},p_{2})$ be an $(A,\rho_{1},\rho_{2})$-admissible solution.
We argue by contradiction. Let $K>1$ be a large constant depending only on
$A,T,R_{0}$, to be fixed at the end of the proof, and suppose that
\begin{equation}\label{hyp}
\inf_{(t,\alpha,\beta)\in[0,T]\times\mathbb{R}^{2}}F(z)(t,\alpha,\beta)<\frac{\varepsilon_{4}}{K},
\end{equation}
with $\varepsilon_{4}$ as in Proposition \ref{chord-arc prop}. Write
\[
g(t):=\inf_{\alpha,\beta\in\mathbb{R}}F(z)(t,\alpha,\beta),\qquad t\in[0,T].
\]

\emph{Continuity of $g$.} The function $g$ is continuous at every $t$ with
$g(t)<\min\{\varepsilon_{1},\frac12\}$: indeed, at such a $t$ the infimum is unchanged if one
restricts to
$\varepsilon_{1}\leqslant|\alpha-\beta|\leqslant\varepsilon_{1}^{-1}$, because the
pairs with $|\alpha-\beta|\leqslant\varepsilon_{1}$ are excluded by \eqref{e20} and those with
$|\alpha-\beta|\geqslant\varepsilon_{1}^{-1}$ by the periodicity of
$\alpha\mapsto z(t,\alpha)-(\alpha,0)$ together with \eqref{e8}, which gives
$F(z)(t,\alpha,\beta)\geqslant1-\frac{2A}{|\alpha-\beta|}\geqslant\frac12$ whenever
$|\alpha-\beta|\geqslant4A$, hence a fortiori whenever
$|\alpha-\beta|\geqslant\varepsilon_{1}^{-1}$; shifting $\alpha$ and $\beta$ by one and the same
multiple of $2\pi$ changes neither $|z(t,\alpha)-z(t,\beta)|$ nor $|\alpha-\beta|$, so $\alpha$
may be restricted to one period, after which the set of pairs is compact and depends continuously
on $t$, and the infimum over it is a minimum of a continuous function of $t$. The same
compactness argument is needed again in Step 4 below.

\emph{The two reference times.} By \eqref{hyp} the set
$\{t\in[0,T]:g(t)\leqslant\varepsilon_{4}/K\}$ is non-empty, and it is closed: if $t_{n}\to t^{*}$
with $g(t_{n})\leqslant\varepsilon_{4}/K$, then the compactness argument just given, applied to
pairs that nearly minimize at $t_{n}$, produces a pair at $t^{*}$ with
$F(z)(t^{*},\alpha,\beta)\leqslant\varepsilon_{4}/K$, so that
$g(t^{*})\leqslant\varepsilon_{4}/K$. Let
\begin{equation}\label{tbar.def}
\bar{t}:=\min\big\{t\in[0,T]:g(t)\leqslant\varepsilon_{4}/K\big\}.
\end{equation}
We have $\bar{t}>0$: by \eqref{e10}, $g(0)\geqslant1/A$, while $1/A>\varepsilon_{4}/K$ once
$\varepsilon_{4}\leqslant2/A$ and $K>2A$. Since $g(\bar{t})\leqslant\varepsilon_{4}/K$ and
$g(t)>\varepsilon_{4}/K$ for $t<\bar{t}$, continuity gives
\begin{equation}\label{tbar.range}
g(\bar{t})=\frac{\varepsilon_{4}}{K},
\qquad\text{and}\qquad
g(t)\geqslant\frac{\varepsilon_{4}}{K}\quad\text{for all }t\in[0,\bar{t}].
\end{equation}
Likewise $g(0)\geqslant1/A\geqslant\varepsilon_{4}/2>g(\bar{t})$ once
$\varepsilon_{4}\leqslant\min\{\varepsilon_{2},2/A,e^{-1}\}$, so the set
$\{t\in[0,\bar{t}]:g(t)\geqslant\varepsilon_{4}/2\}$ is non-empty and closed and we may set
\begin{equation}\label{t0.def}
t_{0}:=\max\big\{t\in[0,\bar{t}]:g(t)\geqslant\varepsilon_{4}/2\big\}.
\end{equation}
Then $t_{0}\in[0,\bar{t})$ and, by continuity and the maximality of $t_{0}$,
\begin{equation}\label{t0.value}
\inf_{\alpha,\beta\in\mathbb{R}}F(z)(t_{0},\alpha,\beta)=\frac{\varepsilon_{4}}{2},
\qquad
\inf_{\alpha,\beta\in\mathbb{R}}F(z)(t,\alpha,\beta)\leqslant \frac{\varepsilon_{4}}{2}
\quad\text{for all }t\in[t_{0},\bar{t}].
\end{equation}
The last such time is used because the argument below needs
$g\leqslant\varepsilon_{4}/2$ on the {\em whole} of $[t_{0},\bar{t}]$, which the first crossing of
the level $\varepsilon_{4}/2$ does not give. Together with
\eqref{tbar.range} this gives
\begin{equation}\label{g.range}
\frac{\varepsilon_{4}}{K}\leqslant g(t)\leqslant\frac{\varepsilon_{4}}{2},
\qquad t\in[t_{0},\bar{t}],
\end{equation}
so the shortest admissible chord is positive on $[t_{0},\bar{t}]$, and
Proposition \ref{chord-arc prop} and Lemma \ref{chord}(2) apply at every
$t_{2}\in[t_{0},\bar{t}]$, with $\epsilon=\tilde d/2\in[\varepsilon_{4}/(2K),\varepsilon_{4}/4]$.
More generally, every statement of Sections \ref{sec:geometry} and \ref{sec:main} that is
formulated with the horizon $[0,T]$ is applied below with $T$ replaced by $\bar t$: the a priori
bounds of Definition \ref{adm} hold on $[0,T]\supseteq[0,\bar t]$, and Lemma \ref{grad.euler},
whose hypothesis is the positivity of the chord-arc functional on the horizon, may be used
because of \eqref{g.range}.

For $t\in[t_{0},\bar{t}]$ define
\begin{equation}\label{D.def}
D(t):=\sup\Big\{\frac{1}{|z(t,\alpha)-z(t,\beta)|^{2}}:\alpha,\beta\in\mathbb{R},\ \varepsilon_{1}<|\alpha-\beta|<\frac{1}{\varepsilon_{1}}\Big\}.
\end{equation}
The proof rests on the following a priori estimate: there is a constant $C>0$,
depending only on $A,T,R_{0}$, such that
\begin{equation}\label{D.gronwall}
D(s)\leqslant D(r)+C\int_{r}^{s}D(t')\big(\log D(t')+A\big)\,dt'
\qquad\text{for all }t_{0}\leqslant r\leqslant s\leqslant \bar{t}.
\end{equation}
We establish \eqref{D.gronwall} in several steps; the local form
\begin{equation}\label{D.local}
D(t_{2})\leqslant D(s)+C\int_{s}^{t_{2}}D(t')\big(\log D(t')+A\big)\,dt'
\qquad\text{for all }s\in I(t_{2}):=(t_{2}-\delta(t_{2}),t_{2})\cap[t_{0},\bar{t}]
\end{equation}
is then an immediate consequence, for any $\delta(t_{2})>0$.

\textit{Step 1: reduction to the localized velocity.}

Fix $t_{2}\in[t_{0},\bar{t}]$; by \eqref{t0.value} we have
$\inf_{\alpha,\beta\in\mathbb{R}}F(z)(t_{2},\alpha,\beta)\leqslant\frac{\varepsilon_{4}}{2}\leqslant\varepsilon_{2}$,
so that Lemma \ref{chord}(2) applies at the time $t_{2}$. Let $\alpha_{1},\alpha_{2}$ be the pair of
points furnished by that lemma with $t=t_{2}$, so that $z(t_{2},\alpha_{1})$ and
$z(t_{2},\alpha_{2})$ form a conjugate pair. Set
\[
\tilde{d}:=|z(t_{2},\alpha_{1})-z(t_{2},\alpha_{2})|,
\qquad
\tilde{D}(t):=\frac{1}{|z(t,\alpha_{1})-z(t,\alpha_{2})|^{2}}
\quad\text{for }t\in[0,t_{2}].
\]
Differentiating and then replacing $\partial_{t}z$ by $v_{1}$ yields
\begin{equation}\label{Dtilde.deriv}
\tilde{D}'(t_{2})=-2\tilde{D}(t_{2})^{2}\big[(z(t_{2},\alpha_{1})-z(t_{2},\alpha_{2}))\cdot(v_{1}(t_{2},\alpha_{1})-v_{1}(t_{2},\alpha_{2}))\big].
\end{equation}
The replacement is legitimate: by the kinematic condition for the Euler fluid in
\eqref{e6}, the vector $\partial_{t}z-v_{1}$ is orthogonal to the normal $n$ of
$\Gamma_{t}$ and hence parallel to $\partial_{\alpha}z$, while \eqref{e23} gives
$(z(t_{2},\alpha_{1})-z(t_{2},\alpha_{2}))\cdot\partial_{\alpha}z(t_{2},\alpha_{i})=0$
for $i=1,2$; hence
$(z(t_{2},\alpha_{1})-z(t_{2},\alpha_{2}))\cdot(\partial_{t}z-v_{1})(t_{2},\alpha_{i})=0$
for $i=1,2$.

Let $R$ and $Q:=R-R(0)$ be as in \eqref{e28}. Since
$z(t_{2},\alpha_{1})=R(0,-\frac{\tilde{d}}{2})$ and
$z(t_{2},\alpha_{2})=R(0,\frac{\tilde{d}}{2})$, and since $Q$ is linear and orthogonal,
\[
\begin{aligned}
&(z(t_{2},\alpha_{1})-z(t_{2},\alpha_{2}))\cdot(v_{1}(t_{2},\alpha_{1})-v_{1}(t_{2},\alpha_{2})) \\
&=\Big(Q(0,-\tfrac{\tilde{d}}{2})-Q(0,\tfrac{\tilde{d}}{2})\Big)\cdot(v_{1}(t_{2},\alpha_{1})-v_{1}(t_{2},\alpha_{2})) \\
&=-\tilde{d}\,\big[(Q^{-1}(v_{1}(t_{2},\alpha_{1}))-Q^{-1}(v_{1}(t_{2},\alpha_{2})))\cdot e_{2}\big]
=-\tilde{d}\,\big[(V_{1}(0)-V_{2}(0))\cdot e_{2}\big],
\end{aligned}
\]
where $V_{m}$ is defined in \eqref{omV.def}. The constant $Q^{-1}(c(t_{2}))$ of
\eqref{V.decomp} has canceled in the last difference, so that
$V_{1}(0)-V_{2}(0)=(U_{1}(0)+W_{1}(0))-(U_{2}(0)+W_{2}(0))
+Q^{-1}\big(w(t_{2},z(t_{2},\alpha_{1}))-w(t_{2},z(t_{2},\alpha_{2}))\big)$; the harmonic term
contributes at most $A\tilde d$ to the $e_{2}$-component, by $|Q^{-1}v|=|v|$, the mean value theorem
along the segment joining $z(t_{2},\alpha_{1})$ and $z(t_{2},\alpha_{2})$, whose length is $\tilde d$,
and the gradient bound $\|\nabla w(t_{2})\|_{L^{\infty}(\mathbb{R}^{2})}\leqslant A$ of
\eqref{e9.norm}. The bound must be global here: by \eqref{e25} the chord joining the two points lies in
$\Omega_{2}(t_{2})$, so that no mean value theorem inside $\Omega_{1}(t_{2})$ is available, and the
arc of $\Gamma_{t_{2}}$ joining them has length at least $2\varepsilon_{1}/A$ by \eqref{e8} and
\eqref{e22}, uniformly in $\tilde d$, and is therefore much longer than $\tilde d$.
Combining this with
\eqref{V.decomp} and \eqref{Dtilde.deriv}, and estimating the three contributions separately, we get
\[
|\tilde{D}'(t_{2})|\lesssim\tilde{d}\,\tilde{D}(t_{2})^{2}\Big(\big|(U_{1}(0)-U_{2}(0))\cdot e_{2}\big|+\big|(W_{1}(0)-W_{2}(0))\cdot e_{2}\big|+A\tilde d\Big).
\]
The last summand, of size $O(A\tilde{d})$, is dominated for small $\tilde{d}$ by the bound
$C\tilde{d}|\log\tilde{d}|$ of Steps 2--3, and in any case it is absorbed in the final
Gronwall step: passing to $D=m^{-2}$ it contributes $O(A)$ to the right-hand side of
\eqref{D.gronwall.a.e}, where it is absorbed by $\log D(t)\geqslant\log(\varepsilon_{1}^{2}
\varepsilon_{4}^{-2})\geqslant A$ after shrinking $\varepsilon_{4}$, exactly as the lower bound
$D(t)\geqslant1$ established in Step 4 is used there.

\textit{Step 2: the localized estimate.} We claim that there exists a constant $C$ depending only
on $A$ and $T$ such that
\begin{equation}\label{U.step2.claim}
|(U_{1}(0)-U_{2}(0))\cdot e_{2}|\leqslant  C\tilde{d}|\log(\tilde{d})|.
\end{equation}

For the claim we apply Proposition \ref{chord-arc prop} at the time $t=t_{2}$ and with the pair
$(\alpha_{1},\alpha_{2})$ fixed in Step 1; its parameter $d$ then coincides with $\tilde{d}$, and we
set, in the notation of Figure~\ref{fig:branches},
\[
\begin{aligned}
&\epsilon:=\frac{\tilde{d}}{2}=\frac{d}{2},\qquad
g(\rho):=-f_{1}(\rho)-\epsilon,\qquad f(\rho):=f_{2}(\rho)-\epsilon,\\
&\omega^{-}(\rho):=-\omega_{1}(\rho),\qquad \omega^{+}(\rho):=\omega_{2}(\rho).
\end{aligned}
\]
\begin{figure}[htbp]
\centering
\begin{tikzpicture}[>=Stealth, font=\small, line join=round, scale=1, transform shape]
  \draw[->, black!55] (-2.4,0) -- (2.4,0) node[right, font=\scriptsize] {$x_{1}$};
  \draw[->, black!55] (0,-1.5) -- (0,1.6) node[above, font=\scriptsize] {$x_{2}$};
  \draw[line width=1.6pt]
     plot[domain=-1.75:1.75, samples=120] ({\x}, {0.65 + 0.26*\x*\x});
  \draw[line width=1.6pt]
     plot[domain=-1.75:1.75, samples=120] ({\x}, {-0.65 + 0.26*\x*\x});
  \fill (0,0) circle (2.4pt);
  \node[anchor=south east, font=\scriptsize] at (-0.10,0.02) {splash point};
  \draw[<->, black!70] (1.35,{-0.65+0.26*1.8225}) -- (1.35,{0.65+0.26*1.8225})
     node[midway, right, font=\scriptsize] {$d$};
  \node[font=\scriptsize, anchor=south west] at (1.05,1.35) {$z^{+}(\rho)=(\rho,\,f(\rho)+d/2)$};
  \node[font=\scriptsize, anchor=north west] at (1.05,-1.20) {$z^{-}(\rho)=(\rho,\,-g(\rho)-d/2)$};
  \node[font=\scriptsize, anchor=south east] at (-1.05,1.45) {$\omega^{+}$ on upper branch};
  \node[font=\scriptsize, anchor=north east] at (-1.05,-1.22) {$\omega^{-}$ on lower branch};
  \draw[->, thick] (-1.55,{0.65+0.26*2.4}) -- (-0.85,{0.65+0.26*0.7225});
  \node[font=\scriptsize, anchor=north east] at (-0.92,0.80) {$\rho$};
\end{tikzpicture}
\caption{Local description of the interface near a candidate splash point (the origin). Near the
pair of almost-touching points the two sides of the interface are graphs
$z^{+}(\rho)=(\rho,f(\rho)+d/2)$ and $z^{-}(\rho)=(\rho,-g(\rho)-d/2)$ over the horizontal
parameter $\rho$, with $f(0)=g(0)=0$; the vertical gap at $\rho=0$ is $d$, and it is the gap that
closes as $d\to0$, producing the splash. The interface vorticities $\omega^{+},\omega^{-}$ live on
the two branches. This is the setup of the singular-integral estimate of
Proposition~\ref{prop:key}.}
\label{fig:branches}
\end{figure}
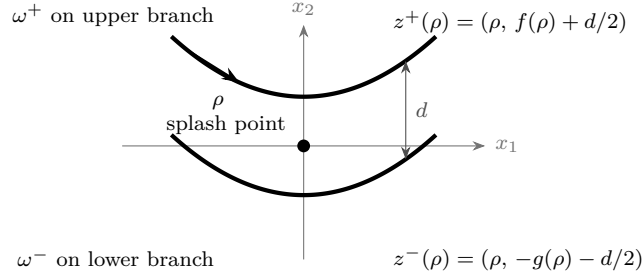
With these choices $f(0)=f_{2}(0)-\epsilon=\frac d2-\epsilon=0$, $g(0)=-f_{1}(0)-\epsilon=\frac d2-\epsilon=0$
and $f'(0)=f'_{2}(0)=0$, $g'(0)=-f'_{1}(0)=0$, which is the normalization required in Appendix~\ref{sec:singular}; moreover $f_{2}(\rho)-f_{1}(\rho)\geqslant d=2\epsilon$ on the
common domain by \eqref{e30}. Since $f'_{j}(0)=0$ and $\|f_{j}\|_{C^{4}}\lesssim1$, we have
$|f'_{j}(\rho)|\leqslant C_{A}|\rho|$ and $|f_{j}(\rho)-f_{j}(0)|\leqslant C_{A}\rho^{2}$; hence,
for $\alpha,\beta\in(-4\sqrt{\epsilon},4\sqrt{\epsilon})$ and $C_{A}$ denoting a constant
depending only on $A$,
\[
\begin{aligned}
f(\alpha)+g(\beta)&=f_{2}(\alpha)-f_{1}(\beta)-2\epsilon\\
&=\big[f_{2}(\alpha)-f_{2}(\beta)\big]+\big[f_{2}(\beta)-f_{1}(\beta)\big]-2\epsilon\\
&\geqslant-C_{A}\big(|\alpha|+|\beta|\big)|\alpha-\beta|-2\epsilon+d
\geqslant-64C_{A}\epsilon,
\end{aligned}
\]
while
\[
f(\alpha)+g(\beta)\leqslant\big[f_{2}(\alpha)-f_{2}(0)\big]+\big[f_{1}(0)-f_{1}(\beta)\big]+d-2\epsilon
\leqslant C_{A}\big(\alpha^{2}+\beta^{2}\big)\leqslant32C_{A}\epsilon.
\]
Thus
\begin{equation}\label{f.g.two}
\big|f(\alpha)+g(\beta)\big|\leqslant 64C_{A}\epsilon
\qquad\text{whenever }|\alpha|,|\beta|\leqslant4\sqrt{\epsilon},
\end{equation}
while on the diagonal $f(\alpha)+g(\alpha)=f_{2}(\alpha)-f_{1}(\alpha)-2\epsilon\geqslant0$. Together
with $f(0)=f'(0)=g(0)=g'(0)=0$ and $\|f\|_{C^{3}}+\|g\|_{C^{3}}\lesssim_{A}1$, these are exactly the
assumptions imposed in Appendix~\ref{sec:singular}, with the constant $A$ of that section denoting
$64C_{A}$: both the diagonal lower bound $f+g\geqslant-\epsilon$ and the two-sided estimate
\eqref{f.g.two} hold. The three properties \eqref{e:N.range.0}--\eqref{e:a.low} hold with the
constant $64C_{A}$, and the estimates of the appendix are uniform in it.
The smallness conditions there are conditions on $\epsilon$ alone, and they hold here as
well: $\epsilon=\frac d2\lesssim\varepsilon_{4}$ by \eqref{e22} and the choice of $t_{2}$, and
$\varepsilon_{4}$ is subject to no lower bound, so it may be chosen so small that in particular
$\epsilon\leqslant(64\cdot 64C_{A})^{-2}$, the condition used in \eqref{e:a.low}, and
$\frac s2=\frac{\sqrt{\epsilon}}{2^{7}}\geqslant\epsilon$.
The $C^{4}$-bounds of \eqref{e30} on $(-\varepsilon_{5}^{4},\varepsilon_{5}^{4})$ are used below
only for arguments with $|x|\lesssim\sqrt{\epsilon}$.
In Appendix~\ref{sec:singular} we therefore take $\delta:=\varepsilon_{5}^{4}$ and use the original
normalized profiles $f=f_{2}-\epsilon$ and $g=-f_{1}-\epsilon$ on $(-\delta,\delta)$; a cutoff on an
interval of width $\varepsilon_{5}^{4}$ would destroy the uniform $C^{3}$ and $C^{4}$ bounds of
\eqref{e30}. Since $\varepsilon_{4}\leqslant\varepsilon_{5}^{100}$ and
$\epsilon\lesssim\varepsilon_{4}$, we may choose $\varepsilon_{4}$ so small that
$4\sqrt{\epsilon}<\delta$. All cutoff-supported integrals used below then sample these profiles
only at arguments in $(-4\sqrt{\epsilon},4\sqrt{\epsilon})$. The local bounds established above
verify the hypotheses of Appendix~\ref{sec:singular}, with its parameter $A$ denoting a constant
depending only on the present $A$.
Then using the definitions of Appendix~\ref{sec:singular}:
\begin{equation}\label{T.defs1}
\begin{aligned}
T_{1,f}\omega^{+}(\alpha)&=\frac{1}{2\pi}p.v.\int_{\mathbb{R}}\frac{f_{2}(\alpha)-f_{2}(\beta)}{(\alpha-\beta)^{2}+(f_{2}(\alpha)-f_{2}(\beta))^{2}}\omega_{2}(\beta)\psi_{\sqrt{\epsilon}}(\alpha-\beta)d\beta, \\
-T_{3,(f,g)}\omega^{-}(\alpha)&=\frac{1}{2\pi}p.v.\int_{\mathbb{R}}\frac{f_{2}(\alpha)-f_{1}(\beta)}{(\alpha-\beta)^{2}+(f_{2}(\alpha)-f_{1}(\beta))^{2}}\omega_{1}(\beta)\psi_{\sqrt{\epsilon}}(\alpha-\beta)d\beta, \\
T_{1,g}\omega^{-}(\alpha)&=\frac{1}{2\pi}p.v.\int_{\mathbb{R}}\frac{f_{1}(\alpha)-f_{1}(\beta)}{(\alpha-\beta)^{2}+(f_{1}(\alpha)-f_{1}(\beta))^{2}}\omega_{1}(\beta)\psi_{\sqrt{\epsilon}}(\alpha-\beta)d\beta, \\
-T_{3,(g,f)}\omega^{+}(\alpha)&=\frac{1}{2\pi}p.v.\int_{\mathbb{R}}\frac{f_{1}(\alpha)-f_{2}(\beta)}{(\alpha-\beta)^{2}+(f_{1}(\alpha)-f_{2}(\beta))^{2}}\omega_{2}(\beta)\psi_{\sqrt{\epsilon}}(\alpha-\beta)d\beta,
\end{aligned}
\end{equation}
and
\begin{equation}\label{T.defs2}
\begin{aligned}
T_{4,(g,f)}\omega^{+}(\alpha)&=\frac{1}{2\pi}p.v.\int_{\mathbb{R}}\frac{\alpha-\beta}{(\alpha-\beta)^{2}+(f_{1}(\alpha)-f_{2}(\beta))^{2}}\omega_{2}(\beta)\psi_{\sqrt{\epsilon}}(\alpha-\beta)d\beta, \\
T_{2,f}\omega^{+}(\alpha)&=\frac{1}{2\pi}p.v.\int_{\mathbb{R}}\frac{\alpha-\beta}{(\alpha-\beta)^{2}+(f_{2}(\alpha)-f_{2}(\beta))^{2}}\omega_{2}(\beta)\psi_{\sqrt{\epsilon}}(\alpha-\beta)d\beta, \\
T_{4,(f,g)}\omega^{-}(\alpha)&=-\frac{1}{2\pi}p.v.\int_{\mathbb{R}}\frac{\alpha-\beta}{(\alpha-\beta)^{2}+(f_{2}(\alpha)-f_{1}(\beta))^{2}}\omega_{1}(\beta)\psi_{\sqrt{\epsilon}}(\alpha-\beta)d\beta, \\
T_{2,g}\omega^{-}(\alpha)&=-\frac{1}{2\pi}p.v.\int_{\mathbb{R}}\frac{\alpha-\beta}{(\alpha-\beta)^{2}+(f_{1}(\alpha)-f_{1}(\beta))^{2}}\omega_{1}(\beta)\psi_{\sqrt{\epsilon}}(\alpha-\beta)d\beta.
\end{aligned}
\end{equation}
The signs on the left-hand sides of \eqref{T.defs1} and \eqref{T.defs2} record how the operator of
Appendix~\ref{sec:singular}, evaluated on $\omega^{\mp}$, compares with the integral written on the
right; three sign changes are involved. First, the
density convention $\omega_{1}=-\omega^{-}$, $\omega_{2}=\omega^{+}$ fixed above. Second, the
numerator of a $T_{1}$-kernel is the difference of a single profile at two points, and for the lower
profile $g(\alpha)-g(\beta)=-\big(f_{1}(\alpha)-f_{1}(\beta)\big)$. Third, the operators with
interchanged profiles use $2\epsilon+g(\alpha)+f(\beta)=f_{2}(\beta)-f_{1}(\alpha)$, the negative of
the corresponding $f_{1}(\alpha)-f_{2}(\beta)$. The entries containing $\omega_{2}$ and no profile
interchange, namely $T_{1,f}\omega^{+}$ and $T_{2,f}\omega^{+}$, carry no leading sign. The entries
$T_{2,g}\omega^{-}$, $T_{3,(f,g)}\omega^{-}$ and $T_{4,(f,g)}\omega^{-}$ carry exactly one minus,
coming from the first change alone (their numerators are $\alpha-\beta$ or
$2\epsilon+f(\alpha)+g(\beta)$, neither of which changes sign under $g=-f_{1}-\epsilon$). The entry
$T_{3,(g,f)}\omega^{+}$ carries one minus, coming from the third change alone. The entry
$T_{1,g}\omega^{-}$ is the only one in which two changes interact: the first would supply a minus,
but the second supplies another one and the two cancel, so no sign appears. The resulting
asymmetry (a minus sign in front of $T_{2,g}\omega^{-}$ but not in front of
$T_{1,g}\omega^{-}$, and none in front of $T_{1,f}\omega^{+}$ and $T_{2,f}\omega^{+}$) reappears
in \eqref{Ue1.pm}, in \eqref{Ue2.disp} and in the identities of
Proposition \ref{prop:key}.
Then we have the formula for $U_{n}(\mu)$:
\[
\begin{aligned}
U_{n}(\mu)=&\frac{\omega_{n}(\mu)(1,f'_{n}(\mu))}{2(1+|f'_{n}(\mu)|^{2})}+Q^{-1}(S_{t}^{\infty}(z(\alpha_{n}+\beta_{n}(\mu)))) \\
&+\sum_{m=1}^{2}\frac{1}{2\pi}p.v.\int_{\mathbb{R}}\frac{(-f_{n}(\mu)+f_{m}(\rho),\mu-\rho)}{|\mu-\rho|^{2}+|f_{n}(\mu)-f_{m}(\rho)|^{2}}\omega_{m}(\rho)\,\psi_{\varepsilon_{5}^{10}}(\beta_{m}(\rho))\,d\rho.
\end{aligned}
\]
Each of the three terms of \eqref{V.decomp.U} contributes explicitly to the
$e_{1}$-component: reading off the first coordinate in
\eqref{V.decomp.U} and recalling $\omega_{1}=-\omega^{-}$, $\omega_{2}=\omega^{+}$, the local
term gives $\frac{\omega_{n}(\mu)}{2(1+|f'_{n}(\mu)|^{2})}$, the far-field term gives
$[Q^{-1}(S_{t}^{\infty}(z(\alpha_{n}+\beta_{n}(\mu))))]\cdot e_{1}$, and the principal-value sum
gives
\[
-\frac{1}{2\pi}\sum_{m=1}^{2}\int_{\mathbb{R}}\frac{f_{n}(\mu)-f_{m}(\rho)}{|\mu-\rho|^{2}+|f_{n}(\mu)-f_{m}(\rho)|^{2}}\,\omega_{m}(\rho)\,\psi_{\varepsilon_{5}^{10}}(\beta_{m}(\rho))\,d\rho.
\]
In the last expression we replace the cutoff $\psi_{\varepsilon_{5}^{10}}(\beta_{m}(\rho))$ by
$\psi_{\sqrt{\epsilon}}(\mu-\rho)$, the error being exactly the term $E_{3}^{\mp}$ defined in
\eqref{E.def1} and \eqref{E.def2} below, and then identify the resulting integrals by means of
\eqref{T.defs1} and \eqref{T.defs2}. Absorbing at the same time the difference between the local
term and the Hilbert term $\pm\frac{\omega^{\mp}(\mu)}{2}$ into $E_{1}^{\mp}$ (that is,
$E_{1}^{-}(\mu)=\frac{\omega^{-}(\mu)}{2(1+|f'_{1}(\mu)|^{2})}-\frac{\omega^{-}(\mu)}{2}$ and
$E_{1}^{+}(\mu)=\frac{\omega^{+}(\mu)}{2}-\frac{\omega^{+}(\mu)}{2(1+|f'_{2}(\mu)|^{2})}$)
and the far-field term into $E_{2}^{\mp}$, we arrive at%
\begin{equation}\label{Ue1.pm}
\begin{aligned}
-(U_{1}(\mu)\cdot e_{1})&= \frac{\omega^{-}(\mu)}{2}+T_{1,g}\omega^{-}(\mu)-T_{3,(g,f)}\omega^{+}(\mu)+E_{1}^{-}(\mu)+E_{2}^{-}(\mu)+E_{3}^{-}(\mu), \\
-(U_{2}(\mu)\cdot e_{1})&=-\frac{\omega^{+}(\mu)}{2}-T_{3,(f,g)}\omega^{-}(\mu)+T_{1,f}\omega^{+}(\mu)+E_{1}^{+}(\mu)+E_{2}^{+}(\mu)+E_{3}^{+}(\mu),
\end{aligned}
\end{equation}
where 
\begin{equation}\label{E.def1}
\begin{aligned}
E_{1}^{-}(\mu)&:=\frac{\omega_{1}(\mu)|f'_{1}(\mu)|^{2}}{2(1+|f'_{1}(\mu)|^{2})}, \\
E_{2}^{-}(\mu)&:=-[Q^{-1}(S_{t}^{\infty}(z(\alpha_{1}+\beta_{1}(\mu))))]\cdot e_{1}, \\
E_{3}^{-}(\mu)&=\sum_{m=1}^{2}\frac{1}{2\pi}p.v.\int_{\mathbb{R}}\frac{f_{1}(\mu)-f_{m}(\rho)}{|\mu-\rho|^{2}+|f_{1}(\mu)-f_{m}(\rho)|^{2}}\omega_{m}(\rho)\,[\psi_{\varepsilon_{5}^{10}}(\beta_{m}(\rho))-\psi_{\sqrt{\epsilon}}(\mu-\rho)]\,d\rho,
\end{aligned}
\end{equation}
and
\begin{equation}\label{E.def2}
\begin{aligned}
	E_{1}^{+}(\mu)&:=\frac{\omega_{2}(\mu)|f'_{2}(\mu)|^{2}}{2(1+|f'_{2}(\mu)|^{2})}, \\
	E_{2}^{+}(\mu)&:=-[Q^{-1}(S_{t}^{\infty}(z(\alpha_{2}+\beta_{2}(\mu))))]\cdot e_{1}, \\
	E_{3}^{+}(\mu)&=\sum_{m=1}^{2}\frac{1}{2\pi}p.v.\int_{\mathbb{R}}\frac{f_{2}(\mu)-f_{m}(\rho)}{|\mu-\rho|^{2}+|f_{2}(\mu)-f_{m}(\rho)|^{2}}\omega_{m}(\rho)\,[\psi_{\varepsilon_{5}^{10}}(\beta_{m}(\rho))-\psi_{\sqrt{\epsilon}}(\mu-\rho)]\,d\rho.
\end{aligned}
\end{equation}
We verify the assumptions of Proposition \ref{prop:key}. First,
\begin{equation}\label{Z.Esum}
\sum_{m=1}^{3}[\|E_{m}^{+}\|_{Z}+\|E_{m}^{-}\|_{Z}]\lesssim\tilde{d}\log(\frac{1}{\tilde{d}}).
\end{equation}
For $\|E_{1}^{\pm}\|_{Z}$, we estimate, using $\|\omega_{i}\|_{L^{\infty}}\lesssim1$ (which
follows from $\|\omega_{\Gamma}\|_{L^{\infty}}\lesssim1$ (see the beginning of the proof of
\eqref{S.C3}) together with $\omega_{i}(\rho)=\omega_{\Gamma}(t,\alpha_{i}+\beta_{i}(\rho))\beta'_{i}(\rho)$
of \eqref{e29} and $\|\beta_{i}\|_{C^{4}}\lesssim1$) and $f'_{i}(0)=0$ together with
$\|f_{i}\|_{C^{4}}\lesssim1$ (whence $|f'_{i}(\mu)|\lesssim|\mu|$),
\[
|E_{1}^{-}(\mu)|+|E_{1}^{+}(\mu)|\leqslant|\omega_{1}(\mu)||f'_{1}(\mu)|^{2}+|\omega_{2}(\mu)||f'_{2}(\mu)|^{2}\lesssim|\mu|^{2},
\]
then we have
\[
\|E_{1}^{-}\|_{Z}+\|E_{1}^{+}\|_{Z}\lesssim\int_{(-\frac{\sqrt{\epsilon}}{8},\frac{\sqrt{\epsilon}}{8})}\frac{\epsilon^{2}|\beta|^{2}}{\epsilon^{3}+|\beta|^{3}}\,d\beta=\frac{2}{3}\epsilon^{2}\log\Big(1+\frac{1}{8^{3}\epsilon^{\frac{3}{2}}}\Big)\lesssim\epsilon\log(\frac{1}{\epsilon})\lesssim\tilde{d}\log(\frac{1}{\tilde{d}}),
\]

For $\|E_{2}^{\pm}\|_{Z}$ we use the first estimate of Lemma \ref{lem:elem}. By \eqref{S.C3} and $\|\beta_{m}\|_{C^{4}}\lesssim1$ (from \eqref{e30}) the map $\mu\mapsto S_{t}^{\infty}(z(t,\alpha_{m}+\beta_{m}(\mu)))$ is bounded and Lipschitz with constant $\lesssim1$, so $|E_{2}^{\pm}(\mu)|\lesssim1$ and $|\partial_{\mu}E_{2}^{\pm}(\mu)|\lesssim1$; hence
\[
\|E_{2}^{\pm}\|_{Z}\lesssim\epsilon\|E_{2}^{\pm}\|_{C^{1}((-\frac{\sqrt{\epsilon}}{8},\frac{\sqrt{\epsilon}}{8}))}\lesssim\epsilon\lesssim\tilde{d}\log(\frac{1}{\tilde{d}}).
\]

To bound $\|E_{3}^{\pm}\|_{Z}$, write the integral (dropping the factor $1/(2\pi)$) as
\[
\begin{aligned}
\mathcal{I}_{n,m}(\mu)&:=\int_{\mathbb{R}}K_{n,m}(\mu,\rho)\,\omega_{m}(\rho)\,[\psi_{\varepsilon_{5}^{10}}(\beta_{m}(\rho))-\psi_{\sqrt{\epsilon}}(\mu-\rho)]\,d\rho,\\
K_{n,m}(\mu,\rho)&:=\frac{f_{n}(\mu)-f_{m}(\rho)}{|\mu-\rho|^{2}+|f_{n}(\mu)-f_{m}(\rho)|^{2}}.
\end{aligned}
\]
and recall from \eqref{omV.def} that $\omega_{2}=\omega^{+}$ and $\omega_{1}=-\omega^{-}$; by \eqref{om.C1}
together with $|\beta_{m}'(0)|\approx 1$ and $\|\beta_{m}\|_{C^{4}}\lesssim1$ from
\eqref{e30}, we have, for $m\in\{1,2\}$,
\begin{equation}\label{om.m}
\|\omega_{m}\|_{C^{1}((-2\sqrt{\epsilon},2\sqrt{\epsilon}))}\lesssim 1.
\end{equation} We claim that
\begin{equation}\label{I.C1}
\|\mathcal{I}_{n,m}\|_{C_{\mu}^{1}((-\frac{\sqrt{\epsilon}}{8},\frac{\sqrt{\epsilon}}{8}))}\lesssim\log(\frac{1}{\epsilon})\qquad\text{for }m,n\in\{1,2\}.
\end{equation}
Throughout the proof of \eqref{I.C1} we write
\[
b:=\mu-\rho,\qquad a:=f_{n}(\mu)-f_{m}(\rho),\qquad
\Psi(\mu,\rho):=\psi_{\varepsilon_{5}^{10}}(\beta_{m}(\rho))-\psi_{\sqrt{\epsilon}}(\mu-\rho),
\]
so that $K_{n,m}=\frac{a}{b^{2}+a^{2}}$ and $\mathcal{I}_{n,m}(\mu)=\int_{\mathbb{R}}K_{n,m}\,\omega_{m}\,\Psi\,d\rho$.

\textit{Step (i): support of $\Psi$.} First, $\Psi(\mu,\rho)=0$ whenever $|b|\leqslant\frac{\sqrt{\epsilon}}{2}$.
Since $|\mu|\leqslant\frac{\sqrt{\epsilon}}{8}$ we get $|\rho|\leqslant|\mu|+|b|\leqslant\frac58\sqrt{\epsilon}$;
and because $d\lesssim\varepsilon_{4}\leqslant\varepsilon_{5}^{100}$ by \eqref{e22} and the choice of
$t_{2}$, while $\epsilon=\frac d2$, we have $\sqrt{\epsilon}\lesssim\varepsilon_{5}^{50}\ll\varepsilon_{5}^{10}$;
moreover
$|\beta_{m}(\rho)|\lesssim|\rho|$ by $\beta_{m}(0)=0$ and $\|\beta_{m}\|_{C^{4}}\lesssim1$ (see \eqref{e30});
hence $|\beta_{m}(\rho)|\leqslant\frac{\varepsilon_{5}^{10}}{2}$ and both cutoffs equal $1$. Second, $\Psi\ne0$
forces either $|\beta_{m}(\rho)|\leqslant\varepsilon_{5}^{10}$ or $|b|\leqslant\sqrt{\epsilon}$; in the first case
$|\rho|\lesssim\varepsilon_{5}^{10}$, because $\beta_{m}(0)=0$, $|\beta_{m}'(0)|\approx1$ and
$\|\beta_{m}\|_{C^{4}}\lesssim1$ make $\beta_{m}$ bilipschitz near $0$, and in the second case
$|\rho|\leqslant|\mu|+\sqrt{\epsilon}\lesssim\sqrt{\epsilon}\ll\varepsilon_{5}^{10}$. Thus, on the support of $\Psi$,
\begin{equation}\label{e:E3supp}
\frac{\sqrt{\epsilon}}{2}\leqslant|b|\lesssim\varepsilon_{5}^{10},\qquad|\rho|\lesssim\varepsilon_{5}^{10},
\end{equation}
and since $|b|\leqslant|\rho|+\frac{\sqrt{\epsilon}}{8}\leqslant|\rho|+\frac{|b|}{4}$ we have
$|b|\leqslant\frac43|\rho|$, while $|\rho|\leqslant|\mu|+|b|\leqslant\frac54|b|$; that is, $|\rho|\approx|b|$.
In particular the support of $\Psi$ is contained in the interval
$[-\varepsilon_{5}^{9},\varepsilon_{5}^{9}]$ of the convention stated after \eqref{omV.def}, and it
is contained in $(-2\sqrt{\epsilon},2\sqrt{\epsilon})$, where $\omega_{m}$ is $C^{1}$ by
\eqref{om.m}; consequently $\mathcal{I}_{n,m}$ is an integral of a compactly supported integrand
over a subset of the chart.

\textit{Step (ii): size of the kernel there.} By \eqref{e30} we have $f_{1}(0)=-\frac d2$, $f_{2}(0)=\frac d2$,
$f'_{j}(0)=0$ and $\|f_{j}\|_{C^{4}}\lesssim1$; Taylor's formula at the origin therefore gives
$|f_{j}(\sigma)\mp\frac d2|\lesssim\sigma^{2}$. Since $d=2\epsilon\lesssim\rho^{2}$ on the support
\eqref{e:E3supp}, this yields
\[
|a|\lesssim d+\mu^{2}+\rho^{2}\lesssim\rho^{2}\lesssim b^{2},\qquad
|\partial_{\mu}a|=|f'_{n}(\mu)|\lesssim|\mu|\leqslant\frac{\sqrt{\epsilon}}{8}.
\]
Hence, on the support of $\Psi$,
\[
|K_{n,m}|=\frac{|a|}{b^{2}+a^{2}}\leqslant\frac{|a|}{b^{2}}\lesssim1,
\]
while, differentiating at fixed $\rho$ (so that $\partial_{\mu}b=1$, $\partial_{\mu}a=f'_{n}(\mu)$),
\[
\partial_{\mu}K_{n,m}=\frac{\partial_{\mu}a}{b^{2}+a^{2}}-\frac{2ab}{(b^{2}+a^{2})^{2}}-\frac{2a^{2}\,\partial_{\mu}a}{(b^{2}+a^{2})^{2}},
\]
so that, by $(b^{2}+a^{2})^{-1}\leqslant b^{-2}$ and $|a|\lesssim b^{2}$,
\[
|\partial_{\mu}K_{n,m}|\lesssim\frac{\sqrt{\epsilon}}{b^{2}}+\frac{|a|}{|b|^{3}}+\frac{a^{2}\sqrt{\epsilon}}{b^{4}}
\lesssim\frac{\sqrt{\epsilon}}{b^{2}}+\frac{1}{|b|}+\sqrt{\epsilon}.
\]

\textit{Step (iii): integration.} By \eqref{om.m}, $\|\omega_{m}\|_{L^{\infty}}\lesssim1$; moreover
$|\Psi|\leqslant1$ and
$|\partial_{\mu}\Psi|=\frac{1}{\sqrt{\epsilon}}\big|\psi'\big(\frac{b}{\sqrt{\epsilon}}\big)\big|$ is supported in
$\{\frac{\sqrt{\epsilon}}{2}\leqslant|b|\leqslant\sqrt{\epsilon}\}$, a set of measure $\lesssim\sqrt{\epsilon}$,
so that $\int_{\mathbb{R}}|\partial_{\mu}\Psi|\,d\rho\lesssim1$. Since for each fixed $\mu$ the integrand is
smooth and compactly supported in $\rho$, we may differentiate under the integral sign, and \eqref{e:E3supp} gives
\[
|\mathcal{I}_{n,m}(\mu)|\lesssim\int_{\{|b|\lesssim\varepsilon_{5}^{10}\}}d\rho\lesssim\varepsilon_{5}^{10},
\]
\[
|\partial_{\mu}\mathcal{I}_{n,m}(\mu)|
\lesssim\int_{\mathbb{R}}\big(|\partial_{\mu}K_{n,m}|\,|\Psi|+|K_{n,m}|\,|\partial_{\mu}\Psi|\big)d\rho
\lesssim\sqrt{\epsilon}\int_{\{|b|\geqslant\frac{\sqrt{\epsilon}}{2}\}}\frac{d\rho}{b^{2}}
+\int_{\{\frac{\sqrt{\epsilon}}{2}\leqslant|b|\lesssim\varepsilon_{5}^{10}\}}\frac{d\rho}{|b|}+1.
\]
The first integral equals $\frac{4}{\sqrt{\epsilon}}$, so the first term is $\lesssim1$; the second is
\[
\lesssim\log\frac{C\varepsilon_{5}^{10}}{\sqrt{\epsilon}}\lesssim\log\frac1\epsilon,
\]
because $\log(C\varepsilon_{5}^{10})$ is a fixed constant while $\frac12\log\frac1\epsilon\to\infty$; and the
last term was estimated above. This proves \eqref{I.C1}. Finally
$E_{3}^{-}=\frac{1}{2\pi}\sum_{m=1}^{2}\mathcal{I}_{1,m}$ and
$E_{3}^{+}=\frac{1}{2\pi}\sum_{m=1}^{2}\mathcal{I}_{2,m}$ by \eqref{E.def1}--\eqref{E.def2}, so \eqref{I.C1}
together with the first estimate of Lemma \ref{lem:elem}, namely
$\|h\|_{Z}\lesssim\epsilon\|h\|_{C^{1}((-\sqrt{\epsilon}/8,\sqrt{\epsilon}/8))}$, yields
\[
\|E_{3}^{\pm}\|_{Z}\lesssim\epsilon\log(\frac{1}{\epsilon})\lesssim\tilde{d}\log(\frac{1}{\tilde{d}}).
\]

Then by calculating we have
\begin{equation}\label{Ue2.disp}
(U_{1}(\mu)\cdot e_{2})-(U_{2}(\mu)\cdot e_{2})=[T_{4,(g,f)}\omega^{+}(\mu)-T_{2,f}\omega^{+}(\mu)]+[T_{4,(f,g)}\omega^{-}(\mu)-T_{2,g}\omega^{-}(\mu)]+E(\mu),
\end{equation}
where 
\[
E=E_{1}+E_{2}+E_{3},
\]
and 
\[
\begin{aligned}
E_{1}(\mu)&:=\frac{\omega_{1}(\mu)f'_{1}(\mu)}{2(1+|f'_{1}(\mu)|^{2})}-\frac{\omega_{2}(\mu)f'_{2}(\mu)}{2(1+|f'_{2}(\mu)|^{2})}, \\
E_{2}(\mu)&:=[Q^{-1}(S_{t}^{\infty}(z(\alpha_{1}+\beta_{1}(\mu))))]\cdot e_{2}-[Q^{-1}(S_{t}^{\infty}(z(\alpha_{2}+\beta_{2}(\mu))))]\cdot e_{2}, \\
E_{3}(\mu)&:=\sum_{m=1}^{2}\frac{1}{2\pi}\int_{\mathbb{R}}[\frac{\mu-\rho}{|\mu-\rho|^{2}+|f_{1}(\mu)-f_{m}(\rho)|^{2}}-\frac{\mu-\rho}{|\mu-\rho|^{2}+|f_{2}(\mu)-f_{m}(\rho)|^{2}}] \\
&\quad\times\omega_{m}(\rho)\,[\psi_{\varepsilon_{5}^{10}}(\beta_{m}(\rho))-\psi_{\sqrt{\epsilon}}(\mu-\rho)]\,d\rho.
\end{aligned}
\]
For $E_{1}$, by its definition in the preceding formula,
\[
|E_{1}(\mu)|\leqslant\frac{|\omega_{1}(\mu)|\,|f'_{1}(\mu)|}{2(1+|f'_{1}(\mu)|^{2})}
+\frac{|\omega_{2}(\mu)|\,|f'_{2}(\mu)|}{2(1+|f'_{2}(\mu)|^{2})}
\lesssim|\mu|
\qquad\text{for }|\mu|\leqslant2\sqrt{\epsilon},
\]
by \eqref{om.m}, $f'_{i}(0)=0$ and $\|f_{i}\|_{C^{4}}\lesssim1$ (whence $|f'_{i}(\mu)|\lesssim|\mu|$);
in particular $|E_{1}(\mu)|\lesssim\epsilon+|\mu|$ there. For $E_{2}$, by \eqref{S.C3} and
$\|\beta_{m}\|_{C^{4}}\lesssim1$ the maps $\mu\mapsto S_{t}^{\infty}(z(t,\alpha_{m}+\beta_{m}(\mu)))$
are bounded and Lipschitz with constant $\lesssim1$: for $|\mu|\leqslant2\sqrt{\epsilon}$ we have
$|z(t,\alpha_{m}+\beta_{m}(\mu))-z(t,\alpha_{m})|
=|Q(\mu,f_{m}(\mu)-f_{m}(0))|\lesssim|\mu|\lesssim\sqrt{\epsilon}\lesssim\varepsilon_{5}^{50}
\ll\varepsilon_{5}^{12}$ by \eqref{e29} and $2\epsilon=\tilde d\lesssim\varepsilon_{4}
\leqslant \varepsilon_{5}^{100}$, so the points at which $S_{t}^{\infty}$ is evaluated lie in the balls of
\eqref{S.C3}, where its $C^{1}$-norm is controlled. Therefore
\[
\begin{aligned}
|E_{2}(\mu)|&\lesssim|z(\alpha_{1}+\beta_{1}(\mu))-z(\alpha_{2}+\beta_{2}(\mu))|\\
&=|f_{2}(\mu)-f_{1}(\mu)|\\
&\leqslant|f_{2}(\mu)-f_{2}(0)|+|f_{2}(0)-f_{1}(0)|+|f_{1}(\mu)-f_{1}(0)|\\
&\lesssim\mu^{2}+d+\mu^{2}\lesssim\epsilon+|\mu|,
\end{aligned}
\]
where the middle equality uses \eqref{e29} again and the last step uses $d=2\epsilon$ and
$|\mu|\leqslant1$. For $E_{3}$ we claim that
\begin{equation}\label{E3.pt}
|E_{3}(\mu)|\lesssim\epsilon\log(\frac{1}{\epsilon})\qquad\text{whenever }|\mu|\leqslant2\sqrt{\epsilon}.
\end{equation}
To locate the region of integration, note that the cutoff difference
$\Psi:=\psi_{\varepsilon_{5}^{10}}(\beta_{m}(\rho))-\psi_{\sqrt{\epsilon}}(\mu-\rho)$ vanishes
unless $|\mu-\rho|\geqslant\frac{\sqrt{\epsilon}}{2}$: indeed, if $|\mu-\rho|<\frac{\sqrt\epsilon}{2}$
then $|\rho|\leqslant|\mu|+\frac{\sqrt\epsilon}{2}\leqslant\frac{5\sqrt\epsilon}{2}$, so
$|\beta_{m}(\rho)|\lesssim|\rho|\lesssim\sqrt\epsilon\leqslant\varepsilon_{5}^{50}\ll\varepsilon_{5}^{10}$
by $\beta_{m}(0)=0$ and $\|\beta_{m}\|_{C^{4}}\lesssim1$, and both cutoffs are equal to $1$.
It is also zero unless $|\beta_{m}(\rho)|\lesssim\varepsilon_{5}^{10}$ or
$|\mu-\rho|\lesssim\sqrt{\epsilon}$, and in either case $|\rho|\lesssim\varepsilon_{5}^{10}$:
this is the bilipschitz character of $\beta_{m}$ near $0$ in the first case, while in the second
$|\rho|\leqslant|\mu|+\sqrt{\epsilon}\lesssim\sqrt{\epsilon}\ll\varepsilon_{5}^{10}$. Hence the
integral runs over the region where
\[
|\mu-\rho|\gtrsim\sqrt{\epsilon}\quad\text{and}\quad|\rho|\lesssim\varepsilon_{5}^{10}.
\]
On this region we write the difference of the two fractions as
\[
\begin{aligned}
&\frac{\mu-\rho}{|\mu-\rho|^{2}+A_{1}^{2}}-\frac{\mu-\rho}{|\mu-\rho|^{2}+A_{2}^{2}}
=\frac{(\mu-\rho)\big(A_{2}^{2}-A_{1}^{2}\big)}{\big((\mu-\rho)^{2}+A_{1}^{2}\big)\big((\mu-\rho)^{2}+A_{2}^{2}\big)},\\
& A_{j}:=f_{j}(\mu)-f_{m}(\rho).
\end{aligned}
\]
By \eqref{e30} we have $f_{j}(0)=\mp\frac d2$, $f'_{j}(0)=0$ and $\|f_{j}\|_{C^{4}}\lesssim1$, so
Taylor's formula gives $|f_{j}(\sigma)\mp\frac d2|\lesssim\sigma^{2}$; since $|\mu|\leqslant2\sqrt\epsilon$,
this yields
\[
|A_{j}|=|f_{j}(\mu)-f_{m}(\rho)|\lesssim d+\mu^{2}+\rho^{2}\lesssim\epsilon+\rho^{2},
\qquad j,m\in\{1,2\}.
\]
Likewise $|f_{2}(\mu)-f_{1}(\mu)|\leqslant d+C\mu^{2}\lesssim\epsilon$, because
$\mu^{2}\leqslant4\epsilon\lesssim d$. Therefore, using $|X^{2}-Y^{2}|\leqslant|X-Y|(|X|+|Y|)$,
\[
\big||A_{2}|^{2}-|A_{1}|^{2}\big|
\leqslant|f_{2}(\mu)-f_{1}(\mu)|\big(|A_{1}|+|A_{2}|\big)
\lesssim\epsilon\,(\epsilon+\rho^{2}).
\]
It remains to integrate this bound over the region above, and for this we split it into two parts.
If $|\rho|\geqslant8\sqrt{\epsilon}$, then $\epsilon\lesssim\rho^{2}$ and, since
$|\mu|\leqslant2\sqrt{\epsilon}\leqslant\frac{|\rho|}{4}$, we have $|\mu-\rho|\approx|\rho|$;
hence each of the two denominators is at least $(\mu-\rho)^{2}\gtrsim\rho^{2}$, so their product
is $\gtrsim\rho^{4}$, whereas the numerator is $\lesssim|\rho|\,\epsilon\,\rho^{2}$; the
difference written above is therefore bounded by $\frac{\epsilon}{|\rho|}$, which is integrable over
$\{8\sqrt{\epsilon}\leqslant|\rho|\lesssim\varepsilon_{5}^{10}\}$ with integral
$\lesssim\epsilon\log\frac{\varepsilon_{5}^{10}}{\sqrt{\epsilon}}\lesssim\epsilon\log\frac1\epsilon$
(the last step because $\log\varepsilon_{5}^{10}$ is a fixed constant while
$\frac12\log\frac1\epsilon\to\infty$). If instead $|\rho|\leqslant8\sqrt{\epsilon}$, then
$|A_{j}|\lesssim\epsilon$ and the numerator of this difference is
$\lesssim|\mu-\rho|\,\epsilon^{2}\lesssim\sqrt{\epsilon}\,\epsilon^{2}$, while each denominator
is at least $(\mu-\rho)^{2}\gtrsim\epsilon$, so that this difference is
$\lesssim\sqrt{\epsilon}$; multiplying by $|\omega_{m}(\rho)|\lesssim1$ and integrating over
$\{|\rho|\leqslant8\sqrt{\epsilon}\}$ gives
$\lesssim\sqrt{\epsilon}\cdot\sqrt{\epsilon}=\epsilon\lesssim\epsilon\log\frac1\epsilon$.
This proves \eqref{E3.pt}.
As a result, for any $\phi\in\mathcal{C}$ supported in $(-2\sqrt{\epsilon},2\sqrt{\epsilon})$, we have
\[
\Big|\int_{\mathbb{R}}E\cdot\phi\,d\alpha\Big|
\lesssim\int_{|\alpha|\leqslant2\sqrt{\epsilon}}\big[\epsilon\log(\frac{1}{\epsilon})+\epsilon+|\alpha|\big]\frac{\epsilon}{\epsilon^{2}+\alpha^{2}}d\alpha
\lesssim\epsilon\log(\frac{1}{\epsilon}),
\]
where the first step uses the pointwise bounds just proved together with
$|\phi(\alpha)|\leqslant\frac{\epsilon}{\epsilon^{2}+\alpha^{2}}$ and the support of $\phi$.

\emph{Application of Proposition \ref{prop:key}.} We put
\begin{equation}\label{FG.def}
F^{+}:=-(U_{2}\cdot e_{1}),\qquad F^{-}:=-(U_{1}\cdot e_{1}),\qquad
G:=(U_{1}\cdot e_{2})-(U_{2}\cdot e_{2}),
\end{equation}
and $E^{\pm}:=E_{1}^{\pm}+E_{2}^{\pm}+E_{3}^{\pm}$, $E:=E_{1}+E_{2}+E_{3}$. With this notation
the relation \eqref{Ue1.pm} reads
\[
F^{+}=T_{1,f}\omega^{+}-T_{3,(f,g)}\omega^{-}-\tfrac12\omega^{+}+E^{+},
\qquad
F^{-}=-T_{3,(g,f)}\omega^{+}+T_{1,g}\omega^{-}+\tfrac12\omega^{-}+E^{-},
\]
while \eqref{Ue2.disp} reads
$G=(T_{4,(g,f)}-T_{2,f})\omega^{+}+(T_{4,(f,g)}-T_{2,g})\omega^{-}+E$; these are exactly the three
structural identities required in Proposition \ref{prop:key}, valid on
$(-\sqrt{\epsilon},\sqrt{\epsilon})$. It remains to verify the hypotheses of that proposition.
\begin{itemize}
\item[(a)] \emph{The identities hold pointwise.} All the operators
\[
T_{1,f},\quad T_{2,f},\quad T_{3,(f,g)},\quad T_{3,(g,f)},\quad T_{4,(f,g)},\quad T_{4,(g,f)}
\]
are defined by integrals with smooth kernels: where the kernel is bounded the integral converges
absolutely, and where it is singular it is understood as an improper integral or as a Cauchy
principal value, as the case may be. Since the constant
$\epsilon$ is fixed and positive and the cutoffs localize the integration to a bounded set, each of
these integrals is well defined; hence
all three relations are pointwise identities on $(-\sqrt{\epsilon},\sqrt{\epsilon})$, in agreement
with the requirement of Proposition \ref{prop:key}.
\item[(b)] \emph{Regularity of $\omega^{\pm}$.} One has
$\|\omega^{\pm}\|_{C^{1}((-2\sqrt{\epsilon},2\sqrt{\epsilon}))}\lesssim1$ by \eqref{om.m}; note
also that $\omega^{\pm}$ are defined on all of $(-2\sqrt{\epsilon},2\sqrt{\epsilon})$, because
$2\sqrt{\epsilon}=\sqrt{2d}\lesssim\varepsilon_{5}^{50}$ by
$d\lesssim\varepsilon_{4}\leqslant\varepsilon_{5}^{100}$, so that this
interval is contained in the common domain $(-\varepsilon_{5}^{4},\varepsilon_{5}^{4})$ of the
maps $f_{j}$ and $\beta_{j}$ of \eqref{e30}.
\item[(c)] \emph{Size of the remainders.} The bound
$\|E^{+}\|_{Z}+\|E^{-}\|_{Z}\lesssim\epsilon\log\frac1\epsilon$ follows from \eqref{Z.Esum}
(recall that $E^{\pm}=E_{1}^{\pm}+E_{2}^{\pm}+E_{3}^{\pm}$, so that
$\|E^{\pm}\|_{Z}\leqslant\sum_{m=1}^{3}\|E_{m}^{\pm}\|_{Z}$), and the second bound required in
Proposition \ref{prop:key},
\[
\sup_{\phi\in\mathcal{C},\ \mathrm{supp}\,\phi\subseteq(-2\sqrt{\epsilon},2\sqrt{\epsilon})}
\Big|\int_{\mathbb{R}}E\phi\,d\alpha\Big|\lesssim\epsilon\log\frac1\epsilon,
\]
was proved immediately above. In addition, $E^{+},E^{-},E$ belong to $L^{2}((-\delta,\delta))$: extended
by zero outside $(-2\sqrt{\epsilon},2\sqrt{\epsilon})$, each
$E_{m}^{\pm}$ is bounded on the bounded interval $(-\delta,\delta)$ ($E_{1}^{\pm}$ by
\eqref{om.m}, $E_{2}^{\pm}$ by \eqref{S.C3}), and $E=\sum_{m}E_{m}$. For $E_{3}^{\pm}$ we argue
separately, since \eqref{I.C1} only covers $|\mu|\leqslant\frac{\sqrt{\epsilon}}{8}$. On the
support of the bracket $\Psi$ one has, as in Step (i) of the proof of \eqref{I.C1}, either
$|\rho|\lesssim\varepsilon_{5}^{10}$ or $|\mu-\rho|\leqslant\sqrt{\epsilon}$. In the first case,
if $|\mu|\geqslant\frac{\sqrt{\epsilon}}{8}$ then $|\mu-\rho|\gtrsim|\mu|$ and, by
$|f_{j}(\sigma)\mp\frac d2|\lesssim\sigma^{2}$,
$|a|\lesssim d+\mu^{2}+\rho^{2}$, so that
$|K_{n,m}|\leqslant|a|/|\mu-\rho|^{2}\lesssim(d+\mu^{2})/\mu^{2}\lesssim1$, while the length in
$\rho$ of that set is $\lesssim\varepsilon_{5}^{10}$; the first case contributes nothing when
$|\mu|<\frac{\sqrt{\epsilon}}{8}$, since $|\mu-\rho|\leqslant|\mu|+|\rho|\leqslant\frac{\sqrt{\epsilon}}{8}+\varepsilon_{5}^{10}$
would then violate $|\mu-\rho|\geqslant\frac{\sqrt{\epsilon}}{2}$. In the second case the two
indices $n,m$ have to be treated separately. Note first that here $|\rho|\leqslant\varepsilon_{5}^{9}$
by the convention on the integrals in $\rho$ stated after \eqref{omV.def}, so that
$|\mu|\leqslant|\mu-\rho|+|\rho|\leqslant\sqrt{\epsilon}+\varepsilon_{5}^{9}\leqslant2\sqrt{\epsilon}$:
both indices lie in $(-4\sqrt{\epsilon},4\sqrt{\epsilon})$, and \eqref{e:a.range}, \eqref{e:a.low}
and \eqref{om.m} apply without further qualification. If $n=m$, the amplitude
$a=f_{n}(\mu)-f_{m}(\rho)$ vanishes to first order on the diagonal, so
$a/((\mu-\rho)^{2}+a^{2})$ is a Hilbert kernel with bounded amplitude
$\bigl(|f_{n}'(\xi)|\leqslant C_{A}(|\mu|+|\rho|)\leqslant 3C_{A}\sqrt{\epsilon}$, with $\xi$
between $\mu$ and $\rho$, so that the amplitude of the kernel is bounded by $1/2\bigr)$ by
\eqref{e30}, and since $\omega_{m}\in C^{1}$ by
\eqref{om.m} the principal value is bounded by $\|\omega_{m}\|_{C^{1}}\lesssim1$. If $n\neq m$,
the amplitude does not vanish: $|a|\leqslant(2+A)\epsilon$ by \eqref{e:a.range}, while
$|a|\geqslant\frac{\epsilon}{2}$ whenever $|\mu-\rho|\leqslant\epsilon$, by \eqref{e:a.low}
applied to $a(\mu,\rho)$ or to $-a(\rho,\mu)$ according as $(n,m)=(2,1)$ or $(1,2)$. Thus on
$\{|\mu-\rho|\leqslant\epsilon\}$ the kernel is bounded by $1/|a|\leqslant2/\epsilon$, so that the
contribution of that set is at most $4$, and on $\{\epsilon\leqslant|\mu-\rho|\leqslant\sqrt{\epsilon}\}$
it satisfies $|K_{n,m}|\leqslant|a|/|\mu-\rho|^{2}\leqslant C_{A}\epsilon/|\mu-\rho|^{2}$, which
integrates to at most $C_{A}$; altogether the second case contributes $\lesssim_{A}1$ also for
$n\neq m$, the principal value being in fact an ordinary integral there. Hence
$|E_{3}^{\pm}(\mu)|\lesssim_{A}1$ for every
$|\mu|\leqslant\delta$, which is the asserted $L^{2}$-bound.
\item[(d)] \emph{$C^{1}$-bounds for $F^{\pm}$ and $G$.} By \eqref{V.decomp} and
\eqref{V.decomp.W} we have
\[
U_{n}=V_{n}-W_{n}-Q^{-1}(c(t_{2}))-Q^{-1}w(t_{2},z(t_{2},\alpha_{n}+\beta_{n}(\mu))),
\]
where
$V_{n}(\mu)=Q^{-1}\big(v_{1}(t_{2},\alpha_{n}+\beta_{n}(\mu))\big)$ and $W_{n}$ is the bulk part.
Since $Q$ is orthogonal it preserves Euclidean norms, $\|V_{n}\|_{L^{\infty}}\leqslant\|v_{1}\|_{L^{\infty}}\leqslant A$
by \eqref{e9} while $\|W_{n}\|_{L^{\infty}}\leqslant A$ by \eqref{e11.bis} and
$\|Q^{-1}w(t_{2},z(t_{2},\alpha_{n}+\beta_{n}(\mu)))\|_{L^{\infty}}\leqslant A$ by \eqref{e9.norm}, so that
$\|U_{n}\|_{L^{\infty}}\leqslant4A$; and since
$\partial_{\mu}\big[v_{1}(t_{2},\alpha_{n}+\beta_{n}(\mu))\big]
=\beta'_{n}(\mu)\,\partial_{\alpha}v_{1}(t_{2},\alpha_{n}+\beta_{n}(\mu))$, the a priori bound
$\|\partial_{\alpha}v_{1}\|_{L^{\infty}}\leqslant  A$ of \eqref{e9} together with
$\|\beta'_{n}\|_{L^{\infty}}\lesssim1$ from \eqref{e30} gives
$\|\partial_{\mu}V_{n}\|_{L^{\infty}}\lesssim_{A}1$, while the chain rule, the bound
$\|\nabla w\|_{L^{\infty}(\mathbb{R}^{2})}\leqslant A$ of \eqref{e9.norm} and
$\|\beta'_{n}\|_{L^{\infty}}\lesssim1$ from \eqref{e30} give
$\|\partial_{\mu}\big(w(t_{2},z(t_{2},\alpha_{n}+\beta_{n}(\mu)))\big)\|_{L^{\infty}}
\lesssim_{A}1$; by the Cauchy--Schwarz inequality
$|\partial_{\mu}(V_{n}\cdot e_{j})|\leqslant|\partial_{\mu}V_{n}|$, the same bound holds for each
component. On the other hand, by \eqref{V.decomp.W} we have
$W_{n}(\mu)=Q^{-1}\big(\Phi_{2}(t_{2},z(t_{2},\alpha_{n}+\beta_{n}(\mu)))\big)$ with
$\Phi_{2}$ as in \eqref{e11.bis}, so that
\[
\partial_{\mu}W_{n}(\mu)=\beta'_{n}(\mu)\,Q^{-1}\Big(\partial_{\alpha}\big[\Phi_{2}(t_{2},z(t_{2},\cdot))\big]_{\,\alpha=\alpha_{n}+\beta_{n}(\mu)}\Big),
\]
whence $\|\partial_{\mu}W_{n}\|_{L^{\infty}}\lesssim_{A}1$ by \eqref{e11.bis}. These bounds hold on
the whole common domain of the $\beta_{n}$, in particular on $(-2\sqrt{\epsilon},2\sqrt{\epsilon})$.

Combining these three bounds, $\|U_{n}\|_{C^{1}((-2\sqrt{\epsilon},2\sqrt{\epsilon}))}\lesssim_{A}1$, and therefore
\[
\|F^{+}\|_{C^{1}((-2\sqrt{\epsilon},2\sqrt{\epsilon}))}+\|F^{-}\|_{C^{1}((-2\sqrt{\epsilon},2\sqrt{\epsilon}))}+\|G\|_{C^{1}((-2\sqrt{\epsilon},2\sqrt{\epsilon}))}\lesssim_{A}1
\] 
as required.
\end{itemize}
Hence Proposition \ref{prop:key} applies and yields
$|G(0)|\lesssim\epsilon\log\frac1\epsilon$; since $\epsilon=\tilde d/2$, this is precisely
\eqref{U.step2.claim}.

\textit{Step 3: we estimate $|(W_{1}(0)-W_{2}(0))\cdot e_{2}|$.} 

We argue at the time $t=t_{2}$ and use the same
rigid motion $R$ and the same functions $f_{1},f_{2}$ as in Step 2, so that $d=\tilde{d}=2\epsilon$.

We have

\[
\begin{aligned}
&2\pi\,Q(W_{1}(0)-W_{2}(0)) \\
=&\int_{\Omega_{2}(t_{2})}(\frac{(z(t_{2},\alpha_{1})-y)^{\bot}}{|z(t_{2},\alpha_{1})-y|^{2}}-\frac{(z(t_{2},\alpha_{2})-y)^{\bot}}{|z(t_{2},\alpha_{2})-y|^{2}})\omega(t_{2},y)dy \\
=&\int_{\Omega_{2}(t_{2})}(\frac{(R(0,f_{1}(0))-y)^{\bot}}{|R(0,f_{1}(0))-y|^{2}}-\frac{(R(0,f_{2}(0))-y)^{\bot}}{|R(0,f_{2}(0))-y|^{2}})\omega(t_{2},y)dy \\
=&\int_{R^{-1}(\Omega_{2}(t_{2}))}(\frac{(R(0,f_{1}(0))-R(y))^{\bot}}{|R(0,f_{1}(0))-R(y)|^{2}}-\frac{(R(0,f_{2}(0))-R(y))^{\bot}}{|R(0,f_{2}(0))-R(y)|^{2}})\omega(t_{2},R(y))\,dy \\
=&\int_{R^{-1}(\Omega_{2}(t_{2}))}(\frac{(Q(0,f_{1}(0))-Q(y))^{\bot}}{|y_{1}|^{2}+|f_{1}(0)-y_{2}|^{2}}-\frac{(Q(0,f_{2}(0))-Q(y))^{\bot}}{|y_{1}|^{2}+|f_{2}(0)-y_{2}|^{2}})\omega(t_{2},R(y))\,dy.
\end{aligned}
\]
The factor $(2\pi)^{-1}$ carried by $W_{n}$ in \eqref{V.decomp.W} is kept on the left-hand
side throughout this step. Each of the two integrals in the first line is understood in the
principal-value sense fixed after \eqref{e9.bs}. Their difference is absolutely integrable, and this
is exactly what the estimates of this step prove, since they bound $\int_{U}|h_{d}|$ by
$d(1+|\log d|)$; consequently the symmetric truncations defining the two principal values may be
combined, and the change of variables below is justified by dominated convergence. The third
equality is the change of variables
$y\mapsto R(y)$, whose Jacobian determinant equals $\det Q=1$, and the fourth uses $R(0,f_{j}(0))-R(y)=Q\big((0,f_{j}(0))-y\big)$ together with
$|Qv|=|v|$. For a $2\times2$
orthogonal matrix $Q$ one has
\[
Q^{-1}[(Qy)^{\bot}]=\det(Q)\,y^{\bot}
\]
for all $y\in\mathbb{R}^{2}$; since $\det Q=1$ this is precisely \eqref{Q.perp}. Applying $Q^{-1}$
to the last relation we obtain
\[
\begin{aligned}
2\pi\big(W_{1}(0)-W_{2}(0)\big)=\int_{R^{-1}(\Omega_{2}(t_{2}))}(\frac{(y_{2}-f_{1}(0),-y_{1})}{|y_{1}|^{2}+|f_{1}(0)-y_{2}|^{2}}-\frac{(y_{2}-f_{2}(0),-y_{1})}{|y_{1}|^{2}+|f_{2}(0)-y_{2}|^{2}})& \\
\omega(t_{2},R(y))dy&.
\end{aligned}
\]
As a result,
\[
\begin{aligned}
&2\pi\,\big|(W_{1}(0)-W_{2}(0))\cdot e_{2}\big| \\
=&|\int_{R^{-1}(\Omega_{2}(t_{2}))}\frac{y_{1}((f_{1}(0)-y_{2})^{2}-(f_{2}(0)-y_{2})^{2})}{(y_{1}^{2}+(f_{1}(0)-y_{2})^{2})(y_{1}^{2}+(f_{2}(0)-y_{2})^{2})}\omega(t_{2},R(y))\,dy| \\
=&|\int_{R^{-1}(\Omega_{2}(t_{2}))}\frac{2dy_{1}y_{2}\omega(t_{2},R(y_{1},y_{2}))}{(y_{1}^{2}+(\frac{d}{2}+y_{2})^{2})(y_{1}^{2}+(\frac{d}{2}-y_{2})^{2})}dy|.
\end{aligned}
\]
Set $U:=R^{-1}(\Omega_{2}(t_{2}))$ and 
\[
h_{d}(z_{1},z_{2}):=\frac{2dz_{1}z_{2}\omega(t_{2},R(z_{1},z_{2}))}{(z_{1}^{2}+(\frac{d}{2}-z_{2})^{2})(z_{1}^{2}+(\frac{d}{2}+z_{2})^{2})}, \quad z=(z_{1},z_{2})\in\mathbb{R}^{2}.
\]
In the following, we break the integral into two parts and consider
\[
I_{1}:=\int_{U_{1}}h_{d}(z)dz, \quad I_{2}:=\int_{U_{2}}h_{d}(z)dz,
\]
where 
\[
\begin{aligned}
&U_{1}:=U\cap\{z:|z-(0,\tfrac{d}{2})|>\tfrac{d}{4},|z-(0,-\tfrac{d}{2})|>\tfrac{d}{4}\}, \\
&U_{2}:=U\cap(\{z:|z-(0,\tfrac{d}{2})|<\tfrac{d}{4}\}\cup\{z:|z-(0,-\tfrac{d}{2})|<\tfrac{d}{4}\}).
\end{aligned}
\]
We will use repeatedly that, by \eqref{e11}, $\|\omega\|_{L^{\infty}}\leqslant A$, and that
$\omega$ decays exponentially in the vertical direction, as in \eqref{e11.loc}.

\emph{Normalization of the representative.} The estimate of $I_{1}$ below involves the distance
from the origin of the region where $\omega(t_{2},R(\cdot))$ is concentrated, and hence the position of the midpoint
$R(0)=\frac12\big(z(t_{2},\alpha_{1})+z(t_{2},\alpha_{2})\big)$ of the chord. That position is not
controlled by the a priori bounds, because the pair $(\alpha_{1},\alpha_{2})$ may be shifted by any
common multiple of $2\pi$ without changing the geometric configuration; we therefore fix the
representatives once and for all. By the periodicity
requirement of Definition \ref{adm} we have $z(t,\alpha+2\pi k)=z(t,\alpha)+(2\pi k,0)$ for every
$k\in\mathbb{Z}$, and all the data of the problem are $2\pi$-periodic in $x_{1}$; replacing
$\alpha_{m}$ by $\alpha_{m}+2\pi k$ therefore translates the entire configuration by
$(2\pi k,0)$, and leaves unchanged both the set $U=R^{-1}(\Omega_{2}(t_{2}))$ and the quantity
$2\pi(W_{1}(0)-W_{2}(0))\cdot e_{2}$ under consideration, since $\omega$ and $\Omega_{2}(t_{2})$
are themselves invariant under that translation. We choose $k\in\mathbb{Z}$ such that
$\big|\frac12(\alpha_{1}+\alpha_{2})+2\pi k\big|\leqslant\pi$. Then, since
$|z(t_{2},\alpha)-(\alpha,0)|\leqslant A$ by \eqref{e8}, we get
\[
|R(0)|=\Big|\frac{z(t_{2},\alpha_{1})+z(t_{2},\alpha_{2})}{2}\Big|
\leqslant \pi+2A,
\]
and consequently
\begin{equation}\label{R.bound}
R_{\ast}:=|R(0)|+R_{0}\leqslant \pi+2A+R_{0}.
\end{equation}
Here $R$ denotes the rigid motion of \eqref{e28} and $R_{\ast}$ the scalar bound of
\eqref{R.bound}. Thus $R_{\ast}$ is bounded in terms of
$A$ and $R_{0}$ alone, and every constant below that is stated to depend only on $A$ and $R_{0}$
may depend on $R_{\ast}$. The periodicity is used here, and again in the far-field estimate in the
proof of \eqref{S.C3}, to make the constants independent of the position of the chord.

First we estimate $I_{1}$. Put $P_{+}:=(0,\frac d2)$, $P_{-}:=(0,-\frac d2)$ and
\[
r_{1}:=|z-P_{-}|,\qquad r_{2}:=|z-P_{+}|,\qquad\text{so that}\qquad r_{1}+r_{2}\geqslant |P_{+}-P_{-}|=d
\quad\text{for every }z.
\]
With $\widetilde{z}:=z-P_{+}$ (that is, $r_{2}=|\widetilde z|$ and $r_{1}=|\widetilde z+(0,d)|$), the
definition of $h_{d}$ gives
\[
|h_{d}(z)|\leqslant \frac{2d\,|\omega(t_{2},R(z))|\,|z_{1}z_{2}|}{r_{1}^{2}r_{2}^{2}}.
\]
On $U_{1}$ both radii exceed $\frac d4$; $\omega(t_{2},\cdot)$ is localized: by
\eqref{e11.loc} it decays exponentially,
$|\omega(t_{2},y)|\leqslant A e^{-|y_{2}|/R_{0}}$, so that
$|\omega(t_{2},R(z))|\leqslant A e^{-|(Qz+R(0))_{2}|/R_{0}}$. The strip
\begin{equation}\label{S.def}
S:=\{z:|(Qz+R(0))_{2}|\leqslant R_{0}\},
\end{equation}
which has width $2R_{0}$ and whose center line has distance at most $|R(0)|$ from the origin.
Thus its transverse offset is bounded by $R_{\ast}:=|R(0)|+R_{0}$, with $R_{\ast}$ bounded by
\eqref{R.bound}; it is the region where the weight is at least $e^{-1}$,
so that only $|\omega|\leqslant A$ is used there. Outside $S$ the weight decays, but only away
from the points $P_{+}$ and $P_{-}$; the two pieces of $S^{c}$ are estimated after
\eqref{club}. We bound the part of $I_{1}$ inside the strip and claim that
\begin{equation}\label{club}
\int_{U_{1}\cap S}\frac{|z_{1}z_{2}|}{r_{1}^{2}r_{2}^{2}}\,dz\lesssim1+\log\frac1d,
\end{equation}
with an implied constant depending only on $A$ and $R_{0}$; since $\|\omega\|_{L^{\infty}}\leqslant A$,
this gives the bound $d(1+\log\frac1d)$ for the part of $I_{1}$ over $U_{1}\cap S$. To prove the
claim, split $U_{1}$ into the three regions
\[
\mathcal{A}_{1}:=\{r_{2}<\tfrac d2\},\qquad \mathcal{A}_{2}:=\{r_{1}<\tfrac d2\},\qquad
\mathcal{A}_{3}:=\{r_{1}\geqslant \tfrac d2,\ r_{2}\geqslant \tfrac d2\}.
\]
On $\mathcal{A}_{1}$ we have $r_{1}\geqslant d-r_{2}>\frac d2$, $|z_{1}|\leqslant r_{2}$ and
$|z_{2}|\leqslant \frac d2+r_{2}\leqslant d$; hence
$\frac{|z_{1}z_{2}|}{r_{1}^{2}r_{2}^{2}}\leqslant \frac{4}{d\,r_{2}}$ and, using the polar coordinates
$z-P_{+}=r_{2}(\cos\theta,\sin\theta)$ (the constraint $r_{1}>\frac d4$ only shrinks the domain),
\[
\int_{\mathcal{A}_{1}}\frac{|z_{1}z_{2}|}{r_{1}^{2}r_{2}^{2}}\,dz
\leqslant \frac4d\int_{d/4<r_{2}<d/2}\int_{0}^{2\pi}\frac{r_{2}\,dr_{2}\,d\theta}{r_{2}}
\leqslant4\pi.
\]
The region $\mathcal{A}_{2}$ is handled by the same computation with the roles of $r_{1}$ and $r_{2}$
interchanged (equivalently, by the reflection $z_{2}\mapsto-z_{2}$, which exchanges $P_{+}$ and $P_{-}$ and
leaves $|z_{1}z_{2}|$ and $dz$ unchanged), and gives again $4\pi$; since only an upper bound for
the integral over $\mathcal{A}_{2}\cap U_{1}\subseteq \mathcal{A}_{2}$ is needed, the presence of $U_{1}$ is
irrelevant here. On $\mathcal{A}_{3}$
write $z=(\xi,\eta)$, $s:=\frac d2+\eta$, $t:=\frac d2-\eta$, so that $s+t=d$,
$r_{1}=\sqrt{\xi^{2}+s^{2}}$, $r_{2}=\sqrt{\xi^{2}+t^{2}}$, and split
$\mathcal{A}_{3}=\mathcal{A}_{3}^{\mathrm{in}}\cup\mathcal{A}_{3}^{\mathrm{out}}$ according to whether
$|\eta|\leqslant\frac d2$ or $|\eta|>\frac d2$. In both parts the estimates
$|z_{1}|=|\xi|\leqslant \min(r_{1},r_{2})=:r_{<}$ and
$|z_{2}|=|\eta|\leqslant \frac{r_{1}+r_{2}}{2}$, the latter because
$r_{1}+r_{2}\geqslant|\eta+\tfrac d2|+|\eta-\tfrac d2|\geqslant2|\eta|$, give
\begin{align}\label{dagger}
\frac{|z_{1}z_{2}|}{r_{1}^{2}r_{2}^{2}}
\leqslant \frac{r_{<}\cdot\frac{r_{1}+r_{2}}{2}}{r_{<}^{2}r_{>}^{2}}
\leqslant \frac{1}{r_{<}r_{>}},\qquad r_{>}:=\max(r_{1},r_{2}).
\end{align}
On $\mathcal{A}_{3}^{\mathrm{in}}$ one has $s,t\geqslant0$, and the identity
$\big(\xi^{2}+s^{2}\big)\big(\xi^{2}+t^{2}\big)=\big(\xi^{2}+st\big)^{2}+\xi^{2}(s-t)^{2}
\geqslant \big(\xi^{2}+st\big)^{2}$ gives $\frac1{r_{<}r_{>}}\leqslant \frac1{\xi^{2}+st}$, so that
\[
\int_{\mathcal{A}_{3}^{\mathrm{in}}\cap S}\frac{|z_{1}z_{2}|}{r_{1}^{2}r_{2}^{2}}\,dz
\leqslant \int_{-d/2}^{d/2}\int_{\mathbb{R}}\frac{d\xi\,d\eta}{\xi^{2}+s t}
=\pi\int_{-d/2}^{d/2}\frac{d\eta}{\sqrt{s t}}
=\pi^{2},
\]
since $s\,t=\frac{d^{2}}{4}-\eta^{2}$,
$\int_{-d/2}^{d/2}\frac{d\eta}{\sqrt{d^{2}/4-\eta^{2}}}=\pi$ and
$\int_{\mathbb{R}}\frac{d\xi}{\xi^{2}+st}=\frac{\pi}{\sqrt{st}}$ for $st>0$; here the middle
integral is over the whole slab $|\eta|\leqslant \frac d2$, $\xi\in\mathbb{R}$, which contains
$\mathcal{A}_{3}^{\mathrm{in}}\cap S$, so that dropping the constraint $z\in S$ is legitimate.
On $\mathcal{A}_{3}^{\mathrm{out}}$ the same bound \eqref{dagger} applies, and there the smaller of the two
radii is the vertical one: for $\eta>\frac d2$ put $u:=\eta-\frac d2>0$, so that
$0<u<u+d$, $r_{2}=\sqrt{\xi^{2}+u^{2}}\leqslant r_{1}$, $r_{<}=r_{2}$, and hence
$\frac{|z_{1}z_{2}|}{r_{1}^{2}r_{2}^{2}}\leqslant \frac1{r_{2}^{2}}=\frac1{\xi^{2}+u^{2}}$; the case
$\eta<-\frac d2$ is identical with $r_{1}$ in place of $r_{2}$. Since $r_{2}\geqslant\frac d2$
throughout $\mathcal{A}_{3}$, the region $\mathcal{A}_{3}^{\mathrm{out}}$ is disjoint from the disc of radius
$\frac d2$ about $P_{+}$. Write $z-P_{+}=\rho(\cos\theta,\sin\theta)$. The strip $S$ of \eqref{S.def} has
width $2R_{0}$, and its center line is at distance at most $|R(0)|+d/2$ from $P_{+}$. Since $d\leqslant1$,
this transverse offset, including the half-width $R_{0}$, is at most $R_{\ast}+1$; hence, for
$\rho\geqslant2(R_{\ast}+1)$, the circle of radius $\rho$ about $P_{+}$ meets $S$ in at most two arcs
whose total angular length is $\lesssim R_{0}/\rho$, while for $\rho\leqslant2(R_{\ast}+1)$ we simply
allow all angles. On $\mathcal{A}_{3}$ the two radii are comparable to $\rho:=|z-P_{+}|$, since both are
$\geqslant \frac d2$ and $|r_{1}-r_{2}|\leqslant |P_{+}-P_{-}|=d\leqslant 2\min(r_{1},r_{2})$; hence \eqref{dagger} gives
$\frac{|z_{1}z_{2}|}{r_{1}^{2}r_{2}^{2}}\leqslant \frac1{r_{<}r_{>}}\leqslant \frac{C}{\rho^{2}}$, and therefore
\[
\int_{\mathcal{A}_{3}^{\mathrm{out}}\cap S}\frac{|z_{1}z_{2}|}{r_{1}^{2}r_{2}^{2}}\,dz
\leqslant \int_{d/2}^{2(R_{\ast}+1)}\frac{2\pi\,d\rho}{\rho}
+C R_{0}\int_{2(R_{\ast}+1)}^{\infty}\frac{d\rho}{\rho^{2}}
\lesssim\log\frac1d+1,
\]
the implied constant depending only on $A$ and $R_{0}$, since $R_{\ast}\leqslant \pi+2A+R_{0}$ by \eqref{R.bound}
and $d\leqslant 1$. The logarithmic term comes only from the lower limit $d/2$ of the first integral; the
upper limit contributes $\log\frac{2(R_{\ast}+1)}{d/2}=\log\frac1d+O(1)$, the $O(1)$ depending on $A$
and $R_{0}$ through \eqref{R.bound} alone. The normalization of the representative enters here:
without it the constant in \eqref{club} would depend on the position of the chord.
Combining the two parts,
\[
\int_{\mathcal{A}_{3}\cap S}\frac{|z_{1}z_{2}|}{r_{1}^{2}r_{2}^{2}}\,dz\lesssim1+\log\frac1d,
\]
and, together with the bounds for $\mathcal{A}_{1}$ and $\mathcal{A}_{2}$, this proves \eqref{club}.

\emph{The part outside the strip.} Let $\rho:=|z-P_{+}|$ and put
\begin{equation}\label{rho0.def}
\rho_{0}:=4\big(R_{\ast}+R_{0}+1\big),
\end{equation}
which is bounded in terms of $A$ and $R_{0}$ alone by \eqref{R.bound}. On
$S^{c}\cap\{\rho\leqslant\rho_{0}\}$ we drop the weight and use only $|\omega|\leqslant A$ from
\eqref{e11}. The bounds for $\mathcal{A}_{1}$ and $\mathcal{A}_{2}$ above do not involve $S$ and give $4\pi$ each;
on $\mathcal{A}_{3}$ the estimate \eqref{dagger} and the comparability of $r_{1},r_{2}$ with $\rho$ give
$\frac{|z_{1}z_{2}|}{r_{1}^{2}r_{2}^{2}}\leqslant C\rho^{-2}$, hence, in polar coordinates
about $P_{+}$,
\begin{equation}\label{club.near}
\int_{U_{1}\cap S^{c}\cap\{\rho\leqslant\rho_{0}\}}\frac{|z_{1}z_{2}|}{r_{1}^{2}r_{2}^{2}}\,dz
\leqslant 8\pi+C\int_{d/2}^{\rho_{0}}\frac{d\rho}{\rho}
\lesssim1+\log\frac1d,
\end{equation}
where the lower limit is $\frac d2$ because $\mathcal{A}_{3}$ forces $r_{2}\geqslant\frac d2$, and the
last step uses only that $\rho_{0}$ is bounded in terms of $A$ and $R_{0}$, so that
$\log\frac{\rho_{0}}{d}=\log\rho_{0}+\log\frac1d\lesssim_{A,R_{0}}1+\log\frac1d$.

On the far field $\{\rho>\rho_{0}\}$ we use the weight. Since $\rho>\rho_{0}\geqslant2d$ for
$d\leqslant1$, we have $\rho-\frac d2\leqslant|z|\leqslant\rho+\frac d2\leqslant\frac54\rho$
and $r_{1},r_{2}\geqslant\rho-d\geqslant\frac{\rho}{2}$, hence $\rho\geqslant\frac{|z|}{2}$.
Since $\rho>\rho_{0}\geqslant4R_{\ast}$,  $|z|\geqslant\frac{\rho}{2}\geqslant2R_{\ast}$, and therefore
\[
|y|=|Qz+R(0)|\leqslant|z|+R_{\ast}\leqslant\tfrac32|z|,
\qquad
|y|\geqslant|z|-R_{\ast}\geqslant\tfrac{|z|}{2}.
\]
The upper bound is the one used in the kernel comparison, and it gives
$|z|\geqslant\frac23|y|$; the lower bound locates the domain of integration in $y$. Thus
\[
\frac{|z_{1}z_{2}|}{r_{1}^{2}r_{2}^{2}}
\leqslant\frac{|z|^{2}/2}{\rho^{4}/16}\leqslant\frac{128}{|z|^{2}}
\leqslant\frac{288}{|y|^{2}}\leqslant\frac{512}{|y|^{2}},
\]
and, since $dz=dy$ and $|y|\geqslant\frac{\rho}{4}>\frac{\rho_{0}}{4}$ on $\{\rho>\rho_{0}\}$,
\begin{align*}
\int_{U_{1}\cap\{\rho>\rho_{0}\}}e^{-|y_{2}|/R_{0}}
\frac{|z_{1}z_{2}|}{r_{1}^{2}r_{2}^{2}}\,dz
&\leqslant 512\int_{\{|y|>\rho_{0}/4\}}\frac{e^{-|y_{2}|/R_{0}}}{|y|^{2}}\,dy \\
&\leqslant 512\Big[\frac{16}{\rho_{0}}\int_{\mathbb{R}}e^{-|y_{2}|/R_{0}}\,dy_{2}
+4\,e^{-\rho_{0}/(8R_{0})}\Big]\leqslant C(A,R_{0}).
\end{align*}
The first summand corresponds to $|y_{1}|\geqslant\frac{\rho_{0}}{8}$, where
$\int_{|y_{1}|\geqslant\rho_{0}/8}\frac{dy_{1}}{y_{1}^{2}+y_{2}^{2}}\leqslant\frac{16}{\rho_{0}}$
and $\int_{\mathbb{R}}e^{-|y_{2}|/R_{0}}dy_{2}=2R_{0}$; the second to $|y_{1}|\leqslant\frac{\rho_{0}}{8}$,
where $|y|>\frac{\rho_{0}}{4}$ forces $|y_{2}|\geqslant\frac{\rho_{0}}{8}$ and hence
$e^{-|y_{2}|/R_{0}}\leqslant e^{-\rho_{0}/(8R_{0})}$, the remaining integral being at most $4$.

Combining the two estimates and recalling that
$|h_{d}(z)|\leqslant 2dA\,|z_{1}z_{2}|/(r_{1}^{2}r_{2}^{2})$ by \eqref{e11},
\[
\int_{U_{1}\cap S^{c}}|h_{d}(z)|\,dz\lesssim d\Big(1+\log\frac1d\Big),
\]
and together with \eqref{club} this gives $|I_{1}|\lesssim d(1+\log\frac1d)$ for $0<d\leqslant 1$.
The logarithm comes from the bound $C\rho^{-2}$. In the proof of Theorem \ref{thm} one has
$d=\tilde d\leqslant\varepsilon_{4}/2\leqslant e^{-1}$ by the choice of $\varepsilon_{4}$ made
above; then $1+\log\frac1d\leqslant2\log\frac1d$, and only
$|I_{1}|\lesssim d|\log d|$ is used below.

It remains to estimate $I_{2}$:
\begin{equation}
\begin{aligned}
|I_{2}|\leqslant &(\int_{R^{-1}(\Omega_{2}(t_{2}))\cap\{z:|z-(0,\frac{d}{2})|<\frac{d}{4}\}}|h_{d}|dz+\int_{R^{-1}(\Omega_{2}(t_{2}))\cap\{z:|z-(0,-\frac{d}{2})|<\frac{d}{4}\}}|h_{d}|dz) \\
=&L_{1}+L_{2}.
\end{aligned}
\end{equation}
For $L_{1}$, note that 
\[
\begin{aligned}
L_{1}\lesssim&d\int_{\{z:|z|<\frac{d}{4}\}}\frac{|z_{1}(z_{2}+\frac{d}{2})|}{|z|^{2}|z+(0,d)|^{2}}dz \\
=&d\int_{|z|<\frac{d}{4}}\frac{|z_{1}|(|z_{2}|+\frac{d}{2})}{|z|^{2}(z_{1}^{2}+(z_{2}+d)^{2})}dz \\
\leqslant&d\int_{|z|<\frac{d}{4}}\Big[\frac{1}{z_{1}^{2}+(z_{2}+d)^{2}}+\frac{d/2}{|z|\big(z_{1}^{2}+(z_{2}+d)^{2}\big)}\Big]dz \\
\lesssim&d\int_{|z-(0,d)|<\frac{d}{4}}\frac{1}{|z|^{2}}dz+d^{2}\int_{|z|<\frac{d}{4}}\frac{dz}{|z|(z_{1}^{2}+(z_{2}+d)^{2})}.
\end{aligned}
\]
Notice that 
\[
d\int_{|z-(0,d)|<\frac{d}{4}}\frac{1}{|z|^{2}}dz\lesssim d\int_{\frac{3d}{4}}^{\frac{5d}{4}}\int_{0}^{2\pi}r^{-1}\,d\theta\,dr\lesssim d
\]
and
\[
d^{2}\int_{|z|<\frac{d}{4}}\frac{dz}{|z|(z_{1}^{2}+(z_{2}+d)^{2})}\lesssim\int_{|z|<\frac{d}{4}}\frac{1}{|z|}dz\lesssim\int_{0}^{\frac{d}{4}}dr\lesssim d.
\]
The term $L_{2}$ is bounded by the same computation with $z_{2}$ replaced by $-z_{2}$. Hence 
\[
|I|:=\Big|\int_{U}h_{d}(z)\,dz\Big|\leqslant|I_{1}|+|I_{2}|\lesssim d\big(1+|\log d|\big)
\quad\text{for } \ 0<d\leqslant  1,
\]
and hence, for $0<d\leqslant e^{-1}$,
$|(W_{1}(0)-W_{2}(0))\cdot e_{2}|\lesssim d|\log d|$, which is the assertion of Step 3.

\textit{Step 4: the integral inequality.}
Combining Steps 2 and 3 we have proved that
\begin{equation}\label{D.gronwall2}
|\widetilde{D}'(t_{2})|\lesssim\widetilde{D}(t_{2})\big(\log\widetilde{D}(t_{2})+A\big);
\end{equation}
at the time $t_{2}$ the additive constant is dominated, because
$\widetilde{D}(t_{2})=\tilde d^{-2}$ with $\tilde d\lesssim\varepsilon_{4}$, so that
$\widetilde{D}(t_{2})\log\widetilde{D}(t_{2})\gtrsim\varepsilon_{4}^{-2}\log\frac1{\varepsilon_{4}}\geqslant A$
after shrinking $\varepsilon_{4}$; it is retained because \eqref{D.gronwall2} is the form that passes
to \eqref{D.gronwall.a.e}.
The estimate \eqref{D.gronwall2} is derived at the single time $t_{2}$. Since $D$ is defined by a
supremum, the passage to an integral inequality for $D$ is made by the envelope theorem. Recall from
\eqref{D.def} that $D$ is a supremum over the open
window $\varepsilon_{1}<|\alpha-\beta|<\varepsilon_{1}^{-1}$. Put
\[
m(t):=\min\Big\{|z(t,\alpha)-z(t,\beta)|:\alpha,\beta\in\mathbb{R},\
\varepsilon_{1}\leqslant|\alpha-\beta|\leqslant\tfrac1{\varepsilon_{1}}\Big\},\qquad t\in[0,T],
\]
a minimum over a compact set of pairs: indeed $\alpha\mapsto z(t,\alpha)-(\alpha,0)$ is
$2\pi$-periodic, so $\alpha$ may be restricted to one period, after which $\beta$ ranges over a
compact interval. Hence the minimum is attained. Since the pairs with
$|\alpha-\beta|\geqslant\varepsilon_{1}$ satisfy
$|z(t,\alpha)-z(t,\beta)|\geqslant\varepsilon_{1}|F(z)(t,\alpha,\beta)|$, the lower bound
\eqref{g.range} gives
$m(t)\geqslant\varepsilon_{1}\inf_{\alpha,\beta}F(z)(t,\alpha,\beta)\geqslant
\varepsilon_{1}\varepsilon_{4}/K=:m_{0}>0$ for all $t\in[t_{0},\bar t]$.

By \eqref{t0.value} we have $\inf_{\alpha,\beta}F(z)(t,\alpha,\beta)\leqslant\varepsilon_{4}/2
\leqslant\varepsilon_{2}$ for every $t\in[t_{0},\bar t]$, so Lemma \ref{chord}(2) applies at
every such $t$ and furnishes a pair $(\alpha_{1},\alpha_{2})$ with
\begin{equation}\label{m.minpair}
|\alpha_{1}-\alpha_{2}|\in[2\varepsilon_{1},\tfrac1{2\varepsilon_{1}}],\qquad
|z(t,\alpha_{1})-z(t,\alpha_{2})|=m(t),
\end{equation}
the second identity being the first line of \eqref{e22}, whose right-hand side is precisely the
minimum defining $m(t)$; moreover $(\alpha_{1},\alpha_{2})$ is a conjugate pair, that is, it
satisfies \eqref{e23} and \eqref{e24}. In particular the infimum of $|z(t,\alpha)-z(t,\beta)|$
over the open window $\varepsilon_{1}<|\alpha-\beta|<\varepsilon_{1}^{-1}$ is at most $m(t)$,
because the pair of \eqref{m.minpair} lies in that window, and it is at least $m(t)$, because
the open window is contained in the closed one. It therefore equals $m(t)$, and consequently
\begin{equation}\label{D.equals}
D(t)=m(t)^{-2}\qquad\text{for }t\in[t_{0},\bar t].
\end{equation}

\emph{$m$ is Lipschitz.} For every admissible pair the function
$t\mapsto|z(t,\alpha)-z(t,\beta)|$ is Lipschitz with constant $2\|\partial_{t}z\|_{L^{\infty}}
=2\|v_{2}\|_{L^{\infty}}\leqslant 2(A+\|u\|_{L^{\infty}})$, because $\partial_{t}z=v_{2}$ by
the no-slip condition and $v_{2}=v_{1}-u\tau$ by \eqref{curl.jump}, so that
$\|v_{2}\|_{L^{\infty}}\leqslant\|v_{1}\|_{L^{\infty}}+\|u\|_{L^{\infty}}\leqslant A+A C_{1}$
by \eqref{e9} and \eqref{e19}, the last step using $u=\omega_{\Gamma}/a$ and
$a\geqslant A^{-1}$. The bound on $\|\partial_{t}z\|_{L^{\infty}}$ needed here is
therefore not part of the a priori information of Definition \ref{adm}, but of the derived
bounds of the interface-vorticity section. A pointwise infimum of a family of functions sharing a Lipschitz
constant is Lipschitz with the same constant; hence $m$ is Lipschitz on $[0,T]$, therefore
absolutely continuous and, by Rademacher's theorem, differentiable at almost every point of
$[0,T]$.

\emph{The derivative at a point of differentiability.} Let $t\in[t_{0},\bar t]$ be a point at
which $m$ is differentiable, and let $(\alpha_{1},\alpha_{2})$ be the pair of
\eqref{m.minpair}: it minimizes the distance at the time $t$ and it is a conjugate pair.
We claim that
\[
m'(t)=\partial_{\tau}\Big[|z(\tau,\alpha_{1})-z(\tau,\alpha_{2})|\Big]_{\tau=t}
=\frac{z(t,\alpha_{1})-z(t,\alpha_{2})}{|z(t,\alpha_{1})-z(t,\alpha_{2})|}\cdot
\big(\partial_{t}z(t,\alpha_{1})-\partial_{t}z(t,\alpha_{2})\big).
\]

To prove this, set $f(\tau,\alpha,\beta):=|z(\tau,\alpha)-z(\tau,\beta)|$ and let $A(t')$
denote the set of pairs at which the infimum defining $m(t')$ is attained. Since
$f(\tau,\alpha,\beta)\geqslant |\alpha-\beta|-2A$, every minimizing pair satisfies
$|\alpha-\beta|\leqslant m(t')+2A$; and since the periodicity condition of Definition \ref{adm} gives
$z(\tau,\alpha+2\pi k)=z(\tau,\alpha)+(2\pi k,0)$ for every $k\in\mathbb{Z}$, both $f$ and
$\partial_{\tau}f$ are invariant under $(\alpha,\beta)\mapsto(\alpha+2\pi k,\beta+2\pi k)$,
so that we may normalize minimizing pairs by $0\leqslant \alpha\leqslant2\pi$, after which $A(t')$ is
contained in a compact box (and hence, being closed, is compact). Let $p$ be a minimizing pair at time
$t$ and let $p'$ be a minimizing pair at time $t'$. From $m(t')\leqslant f(t',p)$,
$m(t)=f(t,p)$, $m(t')=f(t',p')$ and $m(t)\leqslant f(t,p')$ we obtain
\[
\begin{split}
(t'-t)\,\partial_{\tau}f(t,p')+o(t'-t)
&=f(t',p')-f(t,p')\\
&\leqslant m(t')-m(t)\\
&\leqslant f(t',p)-f(t,p)
=(t'-t)\,\partial_{\tau}f(t,p)+o(t'-t),
\end{split}
\]
where the $o(t'-t)$ remainder is uniform for $p,p'$ in the compact range described above. We
let $t'\to t$ twice, once from above and once from below. Both times we may pass to a
subsequence for which $p'$ converges to some $\bar p$; since $m$ is continuous and
$m(t')=f(t',p')$, taking limits gives $m(t)=f(t,\bar p)$, i.e.,\ $\bar p\in A(t)$. 

\emph{Right limit $t'\downarrow t$.} Dividing the relation by $t'-t>0$ and letting
$t'\downarrow t$ we obtain
$\partial_{\tau}f(t,\bar p)\leqslant m'(t)\leqslant\partial_{\tau}f(t,p)$; since this holds for
{\em every} $p\in A(t)$ (the pair $p$ was arbitrary), we get
\begin{equation}\label{e.env.right}
m'(t)\leqslant\inf_{p\in A(t)}\partial_{\tau}f(t,p).
\end{equation}

\emph{Left limit $t'\uparrow t$.} Now $t'-t<0$, so dividing the relation by $t'-t$ reverses
both inequalities:
$\partial_{\tau}f(t,\bar p)\geqslant\frac{m(t')-m(t)}{t'-t}\geqslant\partial_{\tau}f(t,p)$.
Letting $t'\uparrow t$ and using again that $p\in A(t)$ was arbitrary gives
\begin{equation}\label{e.env.left}
m'(t)\geqslant\sup_{p\in A(t)}\partial_{\tau}f(t,p).
\end{equation}
Combining \eqref{e.env.right} and \eqref{e.env.left} we obtain
\[
\sup_{p\in A(t)}\partial_{\tau}f(t,p)\leqslant m'(t)\leqslant\inf_{p\in A(t)}\partial_{\tau}f(t,p),
\]
so that all the numbers $\partial_{\tau}f(t,p)$, $p\in A(t)$, are equal to one another and to
$m'(t)$ (this is the envelope theorem for a minimum over a compact set). Applying this to the pair of \eqref{m.minpair}
proves the claim. Since that
pair is a conjugate pair and since
$\tilde d:=|z(t,\alpha_{1})-z(t,\alpha_{2})|=m(t)\approx\inf_{\alpha,\beta}F(z)(t,\alpha,\beta)
\leqslant\varepsilon_{4}/2$ by \eqref{e22} and \eqref{t0.value}, Proposition
\ref{chord-arc prop} applies at the time $t$ with this pair and with $d=\tilde d$. Steps 1--3, run at the time $t$
instead of $t_{2}$, therefore give
\[
|m'(t)|=\big|(V_{1}(0)-V_{2}(0))\cdot e_{2}\big|\lesssim\tilde d\big(|\log\tilde d|+A\big),\qquad \tilde d=m(t),
\]
where the additional summand $A\tilde d$ is the contribution of the harmonic drift $w$ of
\eqref{e9.norm} (see Step 1); it is of smaller order than $\tilde d|\log\tilde d|$ as soon as
$\tilde d\leqslant e^{-A}$, and in general it is absorbed in the Gronwall argument of Step 5 once
$\log D(t)\geqslant\log(\varepsilon_{1}^{2}\varepsilon_{4}^{-2})\geqslant A$; all constants
occurring here depend only on $A$, on $T$ and on the a
priori bounds \eqref{e8}--\eqref{e11.bis}, and hence neither on $t$ nor on the pair.

Before turning this into a differential inequality we record that the logarithm occurring in it is
taken of a quantity which is at least $1$, so that $|\log m|=\frac12\log D$ below. By
\eqref{t0.value} we have $\inf_{\alpha,\beta}F(z)(t,\alpha,\beta)\leqslant\varepsilon_{4}/2$ for
every $t\in[t_{0},\bar t]$, so there is a pair $(\alpha,\beta)$ with
$F(z)(t,\alpha,\beta)\leqslant\varepsilon_{4}$. By \eqref{e20} such a pair cannot satisfy
$|\alpha-\beta|\leqslant\varepsilon_{1}$, where $F(z)\geqslant\varepsilon_{1}$; and the periodicity
estimate proved above, $F(z)(t,\alpha,\beta)\geqslant1-2A|\alpha-\beta|^{-1}$, forces
$|\alpha-\beta|<4A$, since otherwise $F(z)\geqslant\frac12>\varepsilon_{4}$. As
$4A\leqslant\varepsilon_{1}^{-1}$ by \eqref{eps1.list}, the pair lies in the window of
\eqref{D.def}, so that
\[
m(t)\leqslant|z(t,\alpha)-z(t,\beta)|=F(z)(t,\alpha,\beta)\,|\alpha-\beta|
\leqslant\varepsilon_{4}\cdot4A\leqslant\frac{\varepsilon_{4}}{\varepsilon_{1}}\leqslant1
\]
once $\varepsilon_{4}\leqslant\varepsilon_{1}$, an upper bound which we impose from now on. Hence
$m(t)\leqslant1$ and $D(t)=m(t)^{-2}\geqslant1$ for every $t\in[t_{0},\bar t]$. This lower bound
uses only \eqref{t0.value}, \eqref{e20} and the periodicity estimate, and no consequence of
\eqref{D.gronwall}; Step 5 below sharpens it to
$D(t)\geqslant\varepsilon_{1}^{2}\varepsilon_{4}^{-2}$.

Hence, for almost every $t\in[t_{0},\bar t]$, using $D=m^{-2}$ and $|\log m|=\frac12\log D$,
\begin{equation}\label{D.gronwall.a.e}
|D'(t)|=\frac{2|m'(t)|}{m(t)^{3}}\lesssim\frac{m(t)\big(|\log m(t)|+A\big)}{m(t)^{3}}
=\frac{|\log m(t)|+A}{m(t)^{2}}\lesssim D(t)\big(\log D(t)+A\big).
\end{equation}
Since $D=m^{-2}$ is the composition of the Lipschitz function $m$ with a smooth function on
$[m_{0},\infty)$, it is absolutely continuous on $[t_{0},\bar t]$, so \eqref{D.gronwall.a.e} may be
integrated: for all $t_{0}\leqslant r\leqslant s\leqslant\bar t$,
\[
D(s)\leqslant D(r)+C\int_{r}^{s}D(t')\big(\log D(t')+A\big)\,dt',
\]
which is \eqref{D.gronwall} with the additive $A$ shown. That summand is absorbed in
Step 5 below, where the lower bound $D(t)\geqslant\varepsilon_{1}^{2}\varepsilon_{4}^{-2}
\geqslant e^{A}$ is established; that bound is obtained from \eqref{t0.value}, \eqref{e20} and the
periodicity estimate alone, so the absorption does not circularize the argument. Only the
differential form \eqref{D.gronwall.a.e} is used in what follows.

\textit{Step 5: conclusion.}

By \eqref{t0.value}, $\inf_{\alpha,\beta}F(z)(t_{0},\alpha,\beta)=\varepsilon_{4}/2$. By
Lemma \ref{chord}(2), applied with $\vartheta=\varepsilon_{4}/2\leqslant \varepsilon_{2}$, the
shortest admissible chord at time $t_{0}$ has length
$d_{0}\approx\varepsilon_{4}/2$; consequently
\begin{equation}\label{D.t0}
D(t_{0})=\frac{1}{d_{0}^{2}}\approx4\,\varepsilon_{4}^{-2}\approx\varepsilon_{4}^{-2}.
\end{equation}
The differential form of \eqref{D.gronwall} is \eqref{D.gronwall.a.e}, valid for almost every
$t\in[t_{0},\bar{t}]$ with $D$ absolutely continuous. The argument given before
\eqref{D.gronwall.a.e} yields in fact
$D(t)=m(t)^{-2}\geqslant \varepsilon_{1}^{2}\varepsilon_{4}^{-2}\gg1$ after shrinking
$\varepsilon_{4}$ (which is allowed because every constraint imposed on $\varepsilon_{4}$ so far is
an upper bound, namely $\varepsilon_{4}\leqslant \varepsilon_{2}$,
$\varepsilon_{4}\leqslant \varepsilon_{5}^{100}$, $\varepsilon_{4}\leqslant 2/A$,
$\varepsilon_{4}<\varepsilon_{1}$, $\varepsilon_{4}\leqslant e^{-1}$ (used in Step 3) and the
smallness conditions of Appendix~\ref{sec:singular}, while no lower bound is needed anywhere); in particular,
shrinking $\varepsilon_{4}$ further, we may and will assume that
$\log D(t)\geqslant\log(\varepsilon_{1}^{2}\varepsilon_{4}^{-2})\geqslant A$ for every
$t\in[t_{0},\bar t]$, which absorbs the additive summand $A$ of \eqref{D.gronwall.a.e}, and hence
of the integral form \eqref{D.gronwall}, into $C\log D(t)$. Thus
$y(t):=\log D(t)$ is absolutely continuous and satisfies
$y'(t)=\frac{D'(t)}{D(t)}\leqslant C\log D(t)=C\,y(t)$ for almost every $t$, so Gronwall's
inequality gives $y(\bar{t})\leqslant y(t_{0})e^{C(\bar{t}-t_{0})}\leqslant y(t_{0})e^{CT}$, and therefore
\[
D(\bar{t})\leqslant \big(D(t_{0})\big)^{e^{C_{0}T}}\leqslant C_{3}\,\varepsilon_{4}^{-2e^{C_{0}T}},
\]
for constants $C_{0},C_{3}$ depending only on $A,T$ and $R_{0}$, where we used
\eqref{D.t0}.
It follows that
\[
\inf_{\alpha,\beta\in\mathbb{R}}F(z)(\bar{t},\alpha,\beta)\gtrsim\varepsilon_{1}\,D(\bar{t})^{-1/2}
\gtrsim C_{3}^{-1/2}\,\varepsilon_{1}\,\varepsilon_{4}^{\,e^{C_{0}T}}.
\]
For pairs in the window \eqref{D.def} we have $|z(\bar t,\alpha)-z(\bar t,\beta)|\geqslant D(\bar t)^{-1/2}$
and $|\alpha-\beta|\leqslant \varepsilon_{1}^{-1}$, hence $F\geqslant \varepsilon_{1}D(\bar t)^{-1/2}$; pairs with
$|\alpha-\beta|\leqslant \varepsilon_{1}$ satisfy $F\geqslant \varepsilon_{1}$ by \eqref{e20}, while pairs with
$|\alpha-\beta|\geqslant \varepsilon_{1}^{-1}$ satisfy $F\geqslant 1-2A\varepsilon_{1}\geqslant \frac12$ by the
periodicity estimate, and both bounds dominate $\varepsilon_{1}D(\bar t)^{-1/2}$ because
$D(\bar t)\geqslant 1$ and $\varepsilon_{1}\leqslant \frac{1}{2A^{2}}\leqslant \frac12$.
Choosing $K$ so large that $K>C_{3}^{1/2}\varepsilon_{1}^{-1}\varepsilon_{4}^{\,1-e^{C_{0}T}}$, a constant
depending only on $A$, $T$ and $R_{0}$, we obtain
$\inf_{\alpha,\beta}F(z)(\bar{t},\alpha,\beta)>\varepsilon_{4}/K$, which contradicts
\eqref{tbar.range} and hence the assumption \eqref{hyp}. Therefore \eqref{hyp} is false, that is,
\[
\inf_{(t,\alpha,\beta)\in[0,T]\times\mathbb{R}^{2}}F(z)(t,\alpha,\beta)\geqslant \frac{\varepsilon_{4}}{K}=:c>0,
\]
which is \eqref{e13}. The constant $c$ depends only on $T$, $A$ and $R_{0}$, with $A$ the
constant of Definition \ref{adm} after the enlargement of Remark \ref{rm:normalize}. With $K$
chosen minimally in the preceding relation, the value of $c$ is
\[
c\ \approx\ \varepsilon_{1}\,C_{3}^{-1/2}\,\varepsilon_{4}^{\,e^{C_{0}T}},
\]
Thus $c$ is the power $\varepsilon_{4}^{\,e^{C_{0}T}}$ of the threshold, with an exponent exponential
in $T$. The upper bounds on $\varepsilon_{4}$ involve $\varepsilon_{3}$,
$\varepsilon_{5}$ and the smallness conditions of Appendix~\ref{sec:singular}, all of which depend
on $T$; the threshold therefore decreases with $T$, and Theorem \ref{thm} asserts the positivity of
$c$ for the resulting threshold. The table below records the order in which the constants are fixed
and the constraints they satisfy.

\begin{center}
\begin{tabular}{lll>{\raggedright\arraybackslash}p{6.1cm}}
\hline
constant & depends on & fixed in & constraint \\
\hline
$\varepsilon_{1},\varepsilon_{2}$ & $A$ & Lemma \ref{chord} & --- \\
$\varepsilon_{3}$ & $A,T,\varepsilon_{1}$ & Lemma \ref{chord-arc} & --- \\
$\varepsilon_{5}$ & $A,T,\varepsilon_{1},\varepsilon_{3}$ & Proposition \ref{chord-arc prop} & \eqref{eps45.list} \\
$\varepsilon_{4}$ & $A,T$ & Proposition \ref{chord-arc prop} & upper bounds only \\
$C_{0},C_{3}$ & $A,T,R_{0}$ & Step 4, \eqref{D.gronwall.a.e} & --- \\
$K$ & $A,T,R_{0}$ & after \eqref{D.gronwall.a.e} & $K>C_{3}^{1/2}\varepsilon_{1}^{-1}\varepsilon_{4}^{\,1-e^{C_{0}T}}$ \\
$c$ & $A,T,R_{0}$ & \eqref{e13} & $c=\varepsilon_{4}/K$ \\
\hline
\end{tabular}
\end{center}

No constant depends on the particular solution beyond $A$, $T$ and $R_{0}$; only
$\varepsilon_{4}$ may still be decreased.
\end{proof}

\appendix
\section{Singular-integral estimates}\label{sec:singular}
In this section we collect the singular-integral lemmas used in Step 2 of the proof of
Theorem \ref{thm}; they are the analogs of the estimates of
\cite{CfAiVl} in the two-fluid Euler setting, adapted to the presence of the bulk vorticity
of the viscous MHD fluid as in \cite{AB}.

The objects of the section are two $C^{4}$ functions $f,g$ subject to the
hypotheses below, together with the operators $T_{1},T_{2},T_{3},T_{4},H,M$ introduced in the
course of it; no quantity belonging to the fluid problem occurs. Proposition \ref{prop:key} and
Lemma \ref{lem:elem} are used in the proof of Theorem \ref{thm}, Lemmas \ref{lem:key} and
\ref{lem:comm} only in the proof of Proposition \ref{prop:key}.

Let $f,g:(-\delta,\delta)\to\mathbb{R}$ be $C^{4}$ functions satisfying 
\[
\begin{aligned}
&\|f\|_{C^{3}}+\|g\|_{C^{3}}\leqslant  A, \qquad f(0)=f'(0)=g(0)=g'(0)=0, \\
&f(\alpha)+g(\alpha)\geqslant -\epsilon\quad\text{for }|\alpha|<4\sqrt{\epsilon},\qquad
\big|f(\alpha)+g(\beta)\big|\leqslant A\epsilon
\quad\text{if} \ \alpha,\beta\in(-4\sqrt{\epsilon},4\sqrt{\epsilon}).
\end{aligned}
\]
Here $\delta>0$, and throughout the section $\delta\geqslant4\sqrt{\epsilon}$. The implied constants below
depend only on $A$, and on $\delta$ where this is indicated explicitly; they are uniform in
$\epsilon\in(0,\varepsilon_{5})$. We write $C_{A}$ for a
constant depending only on $A$ whose value may differ at each occurrence, and $C_{\psi}$ for a
constant depending only on the cutoff $\psi$.

The second line is used below only through the following three
consequences. First, since $f(0)=f'(0)=g(0)=g'(0)=0$ and
$\|f''\|_{L^{\infty}}+\|g''\|_{L^{\infty}}\leqslant A$, Taylor's theorem gives
$|f(\alpha)|+|g(\alpha)|\leqslant A\alpha^{2}$; together with the diagonal hypothesis
$f+g\geqslant -\epsilon$ this yields, for $|\alpha|\leqslant\sqrt{\epsilon}$,
\begin{equation}\label{e:N.range.0}
\epsilon\leqslant N(\alpha):=2\epsilon+f(\alpha)+g(\alpha)\leqslant(2+A)\epsilon,
\end{equation}
and, on the whole window $|\alpha|\leqslant4\sqrt{\epsilon}$, where $A\alpha^{2}\leqslant16A\epsilon$,
\begin{equation}\label{e:N.range.1}
\epsilon\leqslant N(\alpha)\leqslant(2+16A)\epsilon.
\end{equation}
The sharper bound \eqref{e:N.range.0} is used whenever $|\alpha|\leqslant\sqrt{\epsilon}$, in
particular for $\alpha$ in the interval $I$ of Lemma \ref{lem:key}; \eqref{e:N.range.1} is used
only when the wider range $|\alpha|\leqslant2\sqrt{\epsilon}$ is needed. Second, for
all $\alpha,\beta\in(-4\sqrt{\epsilon},4\sqrt{\epsilon})$,
\begin{equation}\label{e:a.range}
\big|a(\alpha,\beta)\big|\leqslant(2+A)\epsilon,
\qquad a(\alpha,\beta):=2\epsilon+f(\alpha)+g(\beta).
\end{equation}
Third, if $|\alpha-\beta|\leqslant\epsilon$ then, by
$|f(\alpha)-f(\beta)|\leqslant A\max\{|\alpha|,|\beta|\}\,|\alpha-\beta|\leqslant
4A\epsilon^{3/2}$ and by the diagonal lower bound $N(\beta)\geqslant\epsilon$,
\begin{equation}\label{e:a.low}
a(\alpha,\beta)=N(\beta)+f(\alpha)-f(\beta)\geqslant\epsilon-4A\epsilon^{3/2}
\geqslant\tfrac{\epsilon}{2}
\qquad\text{whenever }|\alpha-\beta|\leqslant\epsilon,
\end{equation}
the last inequality using the smallness condition $\epsilon\leqslant(64A)^{-2}$, which
we impose once and for all; the other smallness conditions appearing below are written
out where they occur.
In particular the kernel $a/((\alpha-\beta)^{2}+a^{2})$ is comparable to the Poisson kernel of
width $\epsilon$ on the diagonal $|\alpha-\beta|\leqslant\epsilon$, while for
$|\alpha-\beta|\geqslant\epsilon$ it is dominated by $1/|\alpha-\beta|$; these are the only two
properties of it that occur in the proofs below.

Assume that $\psi:\mathbb{R}\to[0,1]$ is a smooth, even function supported in the interval $[-1,1]$ and equal to 1 in the interval $[-\frac{1}{2},\frac{1}{2}]$, and let $\psi_{r}(\xi):=\psi(\xi/r)$, consistently with \eqref{cutoff}. The operator formulas below are interpreted only on the local cutoff windows used in the lemmas and proposition. In all applications, $|\alpha|\leqslant2\sqrt{\epsilon}$; the factors $\psi_{\sqrt{\epsilon}}(\alpha-\beta)$ then imply $|\beta|\leqslant3\sqrt{\epsilon}<\delta$, while the kernels of $H$ and $M$ additionally have $|\alpha|\leqslant s$ and $|\alpha-\beta|\leqslant s$, where $s=\sqrt{\epsilon}/2^{6}$. Thus the profiles and input functions are sampled only where they are defined. Given $w\in L^{2}((-\delta,\delta))$, we define
\[
\begin{aligned}
T_{1,f}w(\alpha)&:=\frac{1}{2\pi}p.v.\int_{\mathbb{R}}\frac{f(\alpha)-f(\beta)}{(\alpha-\beta)^{2}+(f(\alpha)-f(\beta))^{2}}w(\beta)\psi_{\sqrt{\epsilon}}(\alpha-\beta)d\beta, \\
T_{2,f}w(\alpha)&=\frac{1}{2\pi}p.v.\int_{\mathbb{R}}\frac{\alpha-\beta}{(\alpha-\beta)^{2}+(f(\alpha)-f(\beta))^{2}}w(\beta)\psi_{\sqrt{\epsilon}}(\alpha-\beta)d\beta, \\
T_{3,(f,g)}w(\alpha)&=\frac{1}{2\pi}p.v.\int_{\mathbb{R}}\frac{2\epsilon+f(\alpha)+g(\beta)}{(\alpha-\beta)^{2}+(2\epsilon+f(\alpha)+g(\beta))^{2}}w(\beta)\psi_{\sqrt{\epsilon}}(\alpha-\beta)d\beta, \\
T_{4,(f,g)}w(\alpha)&=\frac{1}{2\pi}p.v.\int_{\mathbb{R}}\frac{\alpha-\beta}{(\alpha-\beta)^{2}+(2\epsilon+f(\alpha)+g(\beta))^{2}}w(\beta)\psi_{\sqrt{\epsilon}}(\alpha-\beta)d\beta, 
\end{aligned}
\]
and
\[
\begin{aligned}
Hw(\alpha)&:=\frac{1}{2\pi}p.v.\int_{\mathbb{R}}\frac{w(\beta)}{\alpha-\beta}\frac{(2\epsilon+f(\alpha)+g(\alpha))^{2}}{(\alpha-\beta)^{2}+(2\epsilon+f(\alpha)+g(\alpha))^{2}}\psi_{\frac{\sqrt{\epsilon}}{2^{6}}}(\alpha)\psi_{\frac{\sqrt{\epsilon}}{2^{6}}}(\alpha-\beta)d\beta, \\
Mw(\alpha)&:=\frac{1}{2\pi}\int_{\mathbb{R}}w(\beta)\frac{2\epsilon+f(\alpha)+g(\alpha)}{(\alpha-\beta)^{2}+(2\epsilon+f(\alpha)+g(\alpha))^{2}}\psi_{\frac{\sqrt{\epsilon}}{2^{6}}}(\alpha)\psi_{\frac{\sqrt{\epsilon}}{2^{6}}}(\alpha-\beta)d\beta.
\end{aligned}
\]
The operators $T_{3,(g,f)}$ and $T_{4,(g,f)}$ are defined by the very same formulas with the roles of
$f$ and $g$ interchanged, i.e.,\ with the quantity $2\epsilon+f(\alpha)+g(\beta)$ replaced by
$2\epsilon+g(\alpha)+f(\beta)$ in the numerator of $T_{3}$ and in the denominator of $T_{4}$; no
separate notation is needed for $T_{1}$ and $T_{2}$, which involve only one profile. In the
application in the proof of the main theorem the profile $f$ describes the upper interface and $g$ the
lower one, so that $2\epsilon+f(\alpha)+g(\beta)=f_{2}(\alpha)-f_{1}(\beta)$ and
$2\epsilon+g(\alpha)+f(\beta)=f_{2}(\beta)-f_{1}(\alpha)$; thus $T_{3,(f,g)}$ pairs the upper
profile at $\alpha$ with the lower one at $\beta$, while $T_{3,(g,f)}$ pairs the lower one at
$\alpha$ with the upper one at $\beta$. It is the latter that occurs in \eqref{Ue1.pm} and in
Proposition \ref{prop:key}, and the two operators are transposes of one another, the kernel of
one being obtained from that of the other by the substitution $(\alpha,\beta)\mapsto(\beta,\alpha)$.
We define the function space:
\begin{definition}Let
\[
\mathcal{C}:=\{\phi\in C_{0}^{1}((-\delta,\delta)):|\phi^{(k)}(\alpha)|\leqslant \frac{\epsilon}{\epsilon^{2+k}+|\alpha|^{2+k}} \text{for any}\ \alpha\in(-\delta,\delta)\ \text{and}\ k\in\{0,1\}\},
\]
and for any $h\in L^{2}((-\delta,\delta))$, let
\[
\|h\|_{Z}:=\sup_{\phi\in\mathcal{C}}|\int_{\mathbb{R}}h(\alpha)H^{*}\phi(\alpha)d\alpha|,
\]
where $H^{*}$ is the dual of $H$.
\end{definition}
\begin{lemma}\label{lem:elem}
$(1)$ For any $h\in C^{1}((-\frac{\sqrt{\epsilon}}{8},\frac{\sqrt{\epsilon}}{8}))$ we have
\[
\|h\|_{Z}\lesssim\epsilon\|h\|_{C^{1}((-\frac{\sqrt{\epsilon}}{8},\frac{\sqrt{\epsilon}}{8}))},
\]
while for any $h\in L^{2}((-\delta,\delta))$ we have
\[
\|h\|_{Z}\lesssim\int_{(-\frac{\sqrt{\epsilon}}{8},\frac{\sqrt{\epsilon}}{8})}|h(\beta)|\frac{\epsilon^{2}}{\epsilon^{3}+|\beta|^{3}}d\beta.
\]
$(2)$ There is a constant $C_{M}=C_{M}(A,\delta)$ such that 
\[
C_{M}^{-1}M^{*}\phi\in\mathcal{C} \quad \text{for any}\ \phi\in\mathcal{C}.
\]
Also, for any $h\in L^{\infty}((-\sqrt{\epsilon},\sqrt{\epsilon}))$,
\[
\|Mh\|_{L^{\infty}((-\delta,\delta))}\leqslant \|h\|_{L^{\infty}((-\sqrt{\epsilon},\sqrt{\epsilon}))}.
\]
\end{lemma}
\begin{proof}
For the first inequality in $(1)$, it suffices to show 
\[
\Big|\int_{\mathbb{R}}Hh(\alpha)\cdot\phi(\alpha)d\alpha\Big|\lesssim\epsilon\|h\|_{C^{1}((-\frac{\sqrt{\epsilon}}{8},\frac{\sqrt{\epsilon}}{8}))},
\]
given $\phi\in\mathcal{C}$. By calculation
\[
\begin{aligned}
|Hh(\alpha)|&=\frac{1}{2\pi}|\int_{\mathbb{R}}\frac{h(\beta)-h(\alpha)}{\alpha-\beta}\frac{(2\epsilon+f(\alpha)+g(\alpha))^{2}}{(\alpha-\beta)^{2}+(2\epsilon+f(\alpha)+g(\alpha))^{2}}\psi_{\frac{\sqrt{\epsilon}}{2^{6}}}(\alpha)\psi_{\frac{\sqrt{\epsilon}}{2^{6}}}(\alpha-\beta)d\beta| \\
&\lesssim\|h\|_{C^{1}((-\frac{\sqrt{\epsilon}}{8},\frac{\sqrt{\epsilon}}{8}))}(2\epsilon+f(\alpha)+g(\alpha))\psi_{\frac{\sqrt{\epsilon}}{2^{6}}}(\alpha) \\
&\lesssim\epsilon\psi_{\frac{\sqrt{\epsilon}}{2^{6}}}(\alpha)\|h\|_{C^{1}((-\frac{\sqrt{\epsilon}}{8},\frac{\sqrt{\epsilon}}{8}))},
\end{aligned}
\]
then we get
\[
|\int_{\mathbb{R}}Hh(\alpha)\cdot\phi(\alpha)d\alpha|\lesssim\epsilon\|h\|_{C^{1}((-\frac{\sqrt{\epsilon}}{8},\frac{\sqrt{\epsilon}}{8}))}\int_{\mathbb{R}}\frac{\epsilon}{\epsilon^{2}+|\alpha|^{2}}d\alpha=\pi\epsilon\|h\|_{C^{1}((-\frac{\sqrt{\epsilon}}{8},\frac{\sqrt{\epsilon}}{8}))}.
\]
Similarly for the second one it suffices to prove
\[
|H^{*}\phi(\beta)|\lesssim\frac{\epsilon^{2}}{\epsilon^{3}+|\beta|^{3}}, \quad \beta\in(-\delta,\delta),
\]
for any $\phi\in\mathcal{C}$. In doing so we may assume that $|\beta|\lesssim\sqrt\epsilon$:
the integrand defining $H^{*}$ carries the two factors
$\psi_{\frac{\sqrt{\epsilon}}{2^{6}}}(\beta+\rho)$ and $\psi_{\frac{\sqrt{\epsilon}}{2^{6}}}(\rho)$, so
that $H^{*}\phi(\beta)=0$ unless
\begin{align}\label{heartsuit}
|\beta|\leqslant |\rho|+|\beta+\rho|\leqslant 2s=\frac{\sqrt\epsilon}{32}.
\end{align}
This is the point at which the restriction $|\beta|\lesssim\sqrt\epsilon$ enters the
argument below, where it is used to compare $\frac{\epsilon}{\epsilon+|\beta|}$ with
$\frac{\epsilon^{2}}{\epsilon^{3}+|\beta|^{3}}$. We write
\[
\begin{aligned}
H^{*}\phi(\beta)&=\frac{1}{2\pi}\int_{\mathbb{R}}\frac{\phi(\beta+\rho)}{\rho}\frac{(2\epsilon+f(\beta+\rho)+g(\beta+\rho))^{2}}{\rho^{2}+(2\epsilon+f(\beta+\rho)+g(\beta+\rho))^{2}}\\
&\qquad\times\psi_{\frac{\sqrt{\epsilon}}{2^{6}}}(\beta+\rho)\psi_{\frac{\sqrt{\epsilon}}{2^{6}}}(\rho)\,d\rho \\
&=\frac{1}{2\pi}\int_{\mathbb{R}}\frac{\phi(\beta+\rho)}{\rho}\frac{(2\epsilon+f(\beta+\rho)+g(\beta+\rho))^{2}}{\rho^{2}+(2\epsilon+f(\beta+\rho)+g(\beta+\rho))^{2}}\\
&\qquad\times\psi_{\frac{\sqrt{\epsilon}}{2^{6}}}(\beta+\rho)\psi_{\frac{\sqrt{\epsilon}}{2^{6}}}(\rho)\psi_{\kappa}(\rho)\,d\rho \\
&+\frac{1}{2\pi}\int_{\mathbb{R}}\frac{\phi(\beta+\rho)}{\rho}\frac{(2\epsilon+f(\beta+\rho)+g(\beta+\rho))^{2}}{\rho^{2}+(2\epsilon+f(\beta+\rho)+g(\beta+\rho))^{2}}\\
&\qquad\times\psi_{\frac{\sqrt{\epsilon}}{2^{6}}}(\beta+\rho)\psi_{\frac{\sqrt{\epsilon}}{2^{6}}}(\rho)(1-\psi_{\kappa}(\rho))\,d\rho \\
&=:H_{\leqslant \kappa}^{*}\phi(\beta)+H_{\geqslant \kappa}^{*}\phi(\beta).
\end{aligned}
\]
For the second part, we choose $\kappa=\frac{|\beta|+\epsilon}{10}$ and estimate as follows:
\[
\begin{aligned}
|H_{\geqslant \kappa}^{*}\phi(\beta)|&\lesssim\int_{\mathbb{R}}\frac{|\phi(\beta+\rho)|}{|\rho|}\frac{\epsilon^{2}}{\rho^{2}}\psi_{\frac{\sqrt{\epsilon}}{2^{6}}}(\rho)(1-\psi_{\kappa}(\rho))d\rho \\
&\lesssim\int_{\mathbb{R}}\frac{\epsilon}{\epsilon^{2}+(\beta+\rho)^{2}}\frac{\epsilon^{2}}{\kappa^{3}}d\rho \\
&\lesssim\frac{\epsilon^{2}}{\epsilon^{3}+|\beta|^{3}}.
\end{aligned}
\]
To estimate $|H_{\leqslant \kappa}^{*}\phi(\beta)|$, in the first summand below the
factor $\psi_{\frac{\sqrt{\epsilon}}{2^{6}}}(\rho)\psi_{\kappa}(\rho)$ is even and supported in
$[-\kappa,\kappa]$, so that the contribution of the constant $\phi(\beta)$, which is the only
singular part, cancels by symmetry and $\phi(\beta+\rho)$ may be replaced by
$\phi(\beta+\rho)-\phi(\beta)$:
\[
\begin{aligned}
|H_{\leqslant \kappa}^{*}\phi(\beta)|\lesssim&\int_{\mathbb{R}}\Big|\frac{\phi(\beta+\rho)-\phi(\beta)}{\rho}\Big|\frac{(2\epsilon+f(\beta)+g(\beta))^{2}}{\rho^{2}+(2\epsilon+f(\beta)+g(\beta))^{2}}\psi_{\frac{\sqrt{\epsilon}}{2^{6}}}(\beta)\psi_{\frac{\sqrt{\epsilon}}{2^{6}}}(\rho)\psi_{\kappa}(\rho)d\rho \\
&+\int_{\mathbb{R}}\frac{|\phi(\beta+\rho)|}{|\rho|}|\frac{(2\epsilon+f(\beta+\rho)+g(\beta+\rho))^{2}}{\rho^{2}+(2\epsilon+f(\beta+\rho)+g(\beta+\rho))^{2}}\psi_{\frac{\sqrt{\epsilon}}{2^{6}}}(\beta+\rho) \\
&-\frac{(2\epsilon+f(\beta)+g(\beta))^{2}}{\rho^{2}+(2\epsilon+f(\beta)+g(\beta))^{2}}\psi_{\frac{\sqrt{\epsilon}}{2^{6}}}(\beta)|d\rho \\
=:&J_{1}+J_{2}.
\end{aligned}
\]
For $J_{1}$ we use $\epsilon\leqslant X:=2\epsilon+f(\beta)+g(\beta)\leqslant C_{A}\epsilon$ (which gives
$\frac{X^{2}}{\rho^{2}+X^{2}}\leqslant C_{A}^{2}\frac{\epsilon^{2}}{\rho^{2}+\epsilon^{2}}$) together with
\begin{align}\label{spadesuit}
\frac{|\phi(\beta+\rho)-\phi(\beta)|}{|\rho|}
\leqslant \int_{0}^{1}|\phi'(\beta+t\rho)|\,dt
\leqslant \int_{0}^{1}\frac{\epsilon\,dt}{\epsilon^{3}+|\beta+t\rho|^{3}},
\end{align}
where the first inequality follows from
$\phi(\beta+\rho)-\phi(\beta)=\int_{0}^{\rho}\phi'(\beta+t)\,dt$. Since
$\psi_{\kappa}(\rho)$ vanishes for $|\rho|>\kappa=\frac{|\beta|+\epsilon}{10}$, all the numbers
$\beta+t\rho$ lie in $[|\beta|-\kappa,|\beta|+\kappa]$. If $|\beta|>2\kappa$ then
$|\beta+t\rho|\geqslant \frac{1}{2}|\beta|$ and the last integral in \eqref{spadesuit} is
$\leqslant 8\frac{\epsilon}{\epsilon^{3}+|\beta|^{3}}$; if instead $|\beta|\leqslant 2\kappa$ then
$|\beta|\leqslant \epsilon/4$ and $|\beta+t\rho|\leqslant 3\epsilon/8$, so that the integral is
$\leqslant \epsilon^{-2}\leqslant 64\frac{\epsilon}{\epsilon^{3}+|\beta|^{3}}$. Hence in all cases
\[
\frac{|\phi(\beta+\rho)-\phi(\beta)|}{|\rho|}
\lesssim\frac{\epsilon}{\epsilon^{3}+|\beta|^{3}},
\]
and therefore, using $\psi_{\kappa}(\rho)\leqslant 1$ and
$\int_{\mathbb{R}}\frac{\epsilon^{2}\,d\rho}{\rho^{2}+\epsilon^{2}}=\pi\epsilon$,
\[
J_{1}\lesssim\frac{\epsilon}{\epsilon^{3}+|\beta|^{3}}
\int_{\mathbb{R}}\frac{\epsilon^{2}}{\rho^{2}+\epsilon^{2}}\,d\rho
\lesssim\frac{\epsilon^{2}}{\epsilon^{3}+|\beta|^{3}}.
\]
For $J_{2}$ write $\Phi(x):=\frac{N(x)^{2}}{\rho^{2}+N(x)^{2}}\psi_{\frac{\sqrt{\epsilon}}{2^{6}}}(x)$,
so that the bracket in the integrand of $J_{2}$ is $\Phi(\beta+\rho)-\Phi(\beta)$, with $\rho$
held fixed. Since $|f(x)|+|g(x)|\leqslant Ax^{2}\leqslant A\epsilon$ and $f(x)+g(x)\geqslant -\epsilon$ for
$|x|\leqslant s$, we have $\epsilon\leqslant N(x)\leqslant (2+A)\epsilon$ there; hence, by the mean value theorem,
\[
|\Phi(\beta+\rho)-\Phi(\beta)|\leqslant |\rho|\sup_{x}\big|\partial_{x}\Phi(x)\big|
\lesssim|\rho|\Big[\frac{\epsilon(|\beta|+|\rho|+\epsilon)}{\rho^{2}+\epsilon^{2}}
+\frac{\epsilon^{2}}{s(\rho^{2}+\epsilon^{2})}\Big],
\qquad s:=\frac{\sqrt{\epsilon}}{2^{6}},
\]
the first summand coming from $\partial_{x}\big(N^{2}/(\rho^{2}+N^{2})\big)$ and the second from
$\psi'_{s}$. The second summand is in fact dominated by the first. It is nonzero only when
$\psi'_{s}(\beta+\rho)\ne0$, that is, when $\frac{s}{2}\leqslant |\beta+\rho|\leqslant s$; since
$|\rho|\leqslant \kappa=\frac{|\beta|+\epsilon}{10}$, this forces $|\beta|\geqslant \frac{s}{3}$: if instead
$|\beta|\leqslant 2\kappa$ then $|\beta|\leqslant \frac{\epsilon}{4}$ and
$|\beta+\rho|\leqslant \frac{3\epsilon}{8}<\frac{s}{2}$ for $\epsilon$ small, so that this case is void.
In the remaining case $\epsilon(\epsilon+|\beta|)\geqslant \frac{\epsilon s}{3}$, and since
$s^{2}=2^{-12}\epsilon$ we get
$\frac{\epsilon^{2}}{s}\leqslant 3\cdot2^{12}\,\epsilon(\epsilon+|\beta|)$. Hence
\[
|\Phi(\beta+\rho)-\Phi(\beta)|\lesssim|\rho|\,\frac{\epsilon(\epsilon+|\beta|)}{\rho^{2}+\epsilon^{2}},
\]
whence
\[
\begin{aligned}
J_{2}\lesssim&\int_{\mathbb{R}}\frac{|\phi(\beta+\rho)|}{|\rho|}\frac{\epsilon|\rho|(\epsilon+|\beta|)}{\rho^{2}+\epsilon^{2}}\psi_{\frac{\sqrt{\epsilon}}{2^{6}}}(\rho)\psi_{\kappa}(\rho)d\rho \\
\lesssim&\int_{\mathbb{R}}\frac{\epsilon}{\epsilon^{2}+\beta^{2}}\frac{\epsilon(\epsilon+|\beta|)}{\rho^{2}+\epsilon^{2}}\psi_{\kappa}(\rho)d\rho \\
\lesssim&\frac{\epsilon(\epsilon+|\beta|)}{\epsilon^{2}+\beta^{2}}
\leqslant \frac{2\epsilon}{\epsilon+|\beta|}
\lesssim\frac{\epsilon^{2}}{\epsilon^{3}+|\beta|^{3}}.
\end{aligned}
\]
Here the second line uses $|\phi(\beta+\rho)|\lesssim\frac{\epsilon}{\epsilon^{2}+\beta^{2}}$, which
follows from the definition of $\mathcal C$ together with the localization
$|\rho|\leqslant \kappa$: when $|\beta|>2\kappa$ one has $|\beta+\rho|\geqslant \frac12|\beta|$, while in
the complementary case $|\beta|\leqslant 2\kappa$ one has $|\beta|\leqslant \epsilon/4$ and
$|\phi|\leqslant \epsilon^{-1}\lesssim\frac{\epsilon}{\epsilon^{2}+\beta^{2}}$. The third line uses
$\int_{\mathbb{R}}\frac{d\rho}{\rho^{2}+\epsilon^{2}}=\frac{\pi}{\epsilon}$ and
$\epsilon^{2}+|\beta|^{2}\geqslant \frac{(\epsilon+|\beta|)^{2}}{2}$. Finally, the last comparison is a
consequence of \eqref{heartsuit}: with $b:=|\beta|\leqslant \frac{\sqrt\epsilon}{32}$ one has
$b^{2}\leqslant \frac{\epsilon}{2^{10}}\leqslant \epsilon$, hence $b^{3}\leqslant \epsilon b$, and together with
$\epsilon^{3}\leqslant \epsilon^{2}$ this gives $\epsilon^{3}+b^{3}\leqslant \epsilon(\epsilon+b)$; therefore
\[
\frac{\epsilon}{\epsilon+b}\leqslant \frac{\epsilon^{2}}{\epsilon^{3}+b^{3}}.
\]
Thus, we obtain the desired bound.

For $(2)$ set
\[
s:=\frac{\sqrt{\epsilon}}{2^{6}},\qquad N(\alpha):=2\epsilon+f(\alpha)+g(\alpha),
\]
so that the kernel of $M$ reads $\frac{N(\alpha)}{(\alpha-\beta)^{2}+N(\alpha)^{2}}\psi_{s}(\alpha)\psi_{s}(\alpha-\beta)$. Three remarks will be used repeatedly.\par\noindent
\textit{(i) Range of $N$.} Since $f(0)=f'(0)=g(0)=g'(0)=0$ and $\|f''\|_{L^{\infty}}+\|g''\|_{L^{\infty}}\leqslant A$, Taylor's theorem gives $|f(\alpha)|+|g(\alpha)|\leqslant A\alpha^{2}$; hence, for $|\alpha|\leqslant s$,
\begin{equation}\label{e:N.range}
\epsilon\leqslant N(\alpha)\leqslant C_{A}\epsilon,
\end{equation}
the lower bound by the hypothesis $f(\alpha)+g(\alpha)\geqslant-\epsilon$ (applicable since $s<4\sqrt{\epsilon}$) and the upper bound by $A\alpha^{2}\leqslant A\epsilon$.\par\noindent
\textit{(ii) Support.} The integrand defining $M^{*}\phi(\beta)$ carries the two factors $\psi_{s}(\alpha)$ and $\psi_{s}(\alpha-\beta)$, both supported in $[-s,s]$; hence $M^{*}\phi(\beta)\ne0$ forces $|\beta|\leqslant|\alpha|+|\alpha-\beta|\leqslant2s$, that is,
\begin{equation}\label{e:M.support}
M^{*}\phi(\beta)=0\qquad\text{whenever }|\beta|>2s.
\end{equation}
Hence $\operatorname{supp}M^{*}\phi\subseteq(-\delta,\delta)$ as soon as $2s<\delta$, and it suffices to bound $(M^{*}\phi)'$ for $|\beta|\leqslant2s$; throughout that range
\begin{equation}\label{e:M.target}
\frac{\epsilon}{\epsilon^{3}+|\beta|^{3}}\geqslant\frac{\epsilon}{\epsilon^{3}+8s^{3}}\geqslant\frac{1}{3}\,\epsilon^{-\frac12},
\end{equation}
because $8s^{3}=2^{-15}\epsilon^{\frac32}$ and $\epsilon^{3}\leqslant\epsilon^{\frac32}$ for $\epsilon\leqslant1$.\par\noindent
\textit{(iii) Differentiation under the integral sign.} Since $\phi\in C_{0}^{1}((-\delta,\delta))$ and, for fixed $\epsilon>0$, every kernel occurring below is smooth in $(\alpha,\beta)$ with $N(\alpha)\geqslant\epsilon$, each differentiated integrand is absolutely integrable in $\alpha$, locally uniformly in $\beta$. This legitimizes the computations below and shows $M^{*}\phi\in C^{1}$; with \eqref{e:M.support} it gives $M^{*}\phi\in C_{0}^{1}((-\delta,\delta))$.

The $L^{\infty}$-bound for $M$ is immediate: for $\alpha\in(-\delta,\delta)$, by \eqref{e:N.range} and $\int_{\mathbb{R}}\frac{\epsilon\,d\rho}{\rho^{2}+\epsilon^{2}}=\pi$,
\[
|Mh(\alpha)|\lesssim\int_{\mathbb{R}}|h(\alpha-\rho)|\frac{\epsilon}{\rho^{2}+\epsilon^{2}}\psi_{s}(\alpha)\psi_{s}(\rho)\,d\rho\lesssim\|h\|_{L^{\infty}((-\sqrt{\epsilon},\sqrt{\epsilon}))},
\]
because $\psi_{s}(\alpha)\ne0$ forces $|\alpha|\leqslant s$, and then $\psi_{s}(\rho)\ne0$ forces
$|\alpha-\rho|\leqslant s$, so that $|\rho|\leqslant2s=\sqrt{\epsilon}/32<\sqrt{\epsilon}$; for
$|\alpha|>s$ the left-hand side vanishes identically.

It remains to exhibit $C_{M}$ with $C_{M}^{-1}M^{*}\phi\in\mathcal{C}$ for every $\phi\in\mathcal{C}$. By (ii), (iii) and the definition of $\mathcal{C}$ this amounts to proving
\begin{equation}\label{e:M.goal}
|M^{*}\phi(\beta)|\lesssim\frac{\epsilon}{\epsilon^{2}+\beta^{2}},\qquad |(M^{*}\phi)'(\beta)|\lesssim\frac{\epsilon}{\epsilon^{3}+|\beta|^{3}},\qquad \beta\in(-\delta,\delta),
\end{equation}
with constants depending only on $A$ and $\psi$; $C_{M}$ is then the largest of them. Throughout we use, for $\phi\in\mathcal{C}$,
\begin{equation}\label{e:phi.ab}
|\phi(\alpha)|\leqslant\frac{1}{\epsilon}\quad\text{pointwise},\qquad \int_{\mathbb{R}}|\phi(\alpha)|\,d\alpha\leqslant\int_{\mathbb{R}}\frac{\epsilon\,d\alpha}{\epsilon^{2}+\alpha^{2}}=\pi.
\end{equation}

\emph{The first bound in \eqref{e:M.goal}.} By $|\phi(\alpha)|\leqslant\frac{\epsilon}{\epsilon^{2}+\alpha^{2}}$ and \eqref{e:N.range},
\[
|M^{*}\phi(\beta)|\lesssim\int_{\mathbb{R}}\frac{\epsilon}{\epsilon^{2}+\alpha^{2}}\frac{\epsilon}{(\alpha-\beta)^{2}+\epsilon^{2}}\psi_{s}(\alpha)\psi_{s}(\alpha-\beta)\,d\alpha.
\]
We detail this estimate, splitting the integral according to $|\alpha-\beta|\geqslant\frac{|\beta|}{10}$ or $|\alpha-\beta|\leqslant\frac{|\beta|}{10}$. In the first region $\frac{\epsilon}{(\alpha-\beta)^{2}+\epsilon^{2}}\lesssim\frac{\epsilon}{\beta^{2}+\epsilon^{2}}$ and the remaining integral satisfies $\int_{\mathbb{R}}\frac{\epsilon\,d\alpha}{\epsilon^{2}+\alpha^{2}}=\pi$. In the second region $|\alpha|\geqslant\frac{9}{10}|\beta|$; if $|\beta|\geqslant\epsilon$ then $\frac{\epsilon}{\epsilon^{2}+\alpha^{2}}\lesssim\frac{\epsilon}{\beta^{2}+\epsilon^{2}}$, while if $|\beta|\leqslant\epsilon$ then simply $\frac{\epsilon}{\epsilon^{2}+\alpha^{2}}\leqslant\frac{1}{\epsilon}\lesssim\frac{\epsilon}{\beta^{2}+\epsilon^{2}}$; in both cases the remaining factor integrates to $\int_{\mathbb{R}}\frac{\epsilon\,d\alpha}{(\alpha-\beta)^{2}+\epsilon^{2}}=\pi$. Hence $|M^{*}\phi(\beta)|\lesssim\frac{\epsilon}{\epsilon^{2}+\beta^{2}}$.

\emph{The second bound in \eqref{e:M.goal}.} By \eqref{e:M.support} we may assume $|\beta|\leqslant2s$. Differentiating in $\beta$ at fixed $\alpha$ and writing $\rho:=\alpha-\beta$ (so that $\partial_{\beta}\rho=-1$ and $\partial_{\beta}(\rho^{2}+N(\alpha)^{2})^{-1}=2\rho\,(\rho^{2}+N(\alpha)^{2})^{-2}$) gives
\begin{equation}\label{e:M.prime}
\begin{aligned}
(M^{*}\phi)'(\beta)=&\frac{1}{2\pi}\int_{\mathbb{R}}\phi(\alpha)\frac{2\rho\,N(\alpha)}{[\rho^{2}+N(\alpha)^{2}]^{2}}\psi_{s}(\alpha)\psi_{s}(\rho)\,d\alpha\\
&-\frac{1}{2\pi s}\int_{\mathbb{R}}\phi(\alpha)\frac{N(\alpha)}{\rho^{2}+N(\alpha)^{2}}\psi_{s}(\alpha)\psi'\Big(\frac{\rho}{s}\Big)\,d\alpha.
\end{aligned}
\end{equation}
\textit{Case 1: $|\beta|\leqslant10\epsilon$.} Here $|\beta|^{3}\leqslant10^{3}\epsilon^{3}$, so $\frac{\epsilon}{\epsilon^{3}+|\beta|^{3}}\geqslant\frac{1}{1001}\epsilon^{-2}$. For the first integral of \eqref{e:M.prime}, \eqref{e:N.range} together with $[\rho^{2}+N(\alpha)^{2}]^{2}\geqslant\rho^{4}$ gives
\[
\frac{2|\rho|N(\alpha)}{[\rho^{2}+N(\alpha)^{2}]^{2}}\leqslant\frac{2C_{A}\epsilon|\rho|}{(\rho^{2}+\epsilon^{2})^{2}}=:k(\rho),\qquad \int_{\mathbb{R}}k(\rho)\,d\rho=\frac{2C_{A}}{\epsilon},
\]
since $\int_{\mathbb{R}}\frac{|\rho|\,d\rho}{(\rho^{2}+\epsilon^{2})^{2}}=\frac{1}{\epsilon^{2}}$; hence by \eqref{e:phi.ab} this term is at most $\|\phi\|_{L^{\infty}}\|k\|_{L^{1}}\leqslant2C_{A}\epsilon^{-2}$. For the second integral of \eqref{e:M.prime}, $\frac{N(\alpha)}{\rho^{2}+N(\alpha)^{2}}\leqslant\frac{1}{N(\alpha)}\leqslant\frac{1}{\epsilon}$, so it is at most $\frac{C_{\psi}}{\epsilon s}\int_{\mathbb{R}}|\phi|\,d\alpha\lesssim\epsilon^{-\frac32}\leqslant\epsilon^{-2}$. Both terms are therefore $\lesssim_{A}\frac{\epsilon}{\epsilon^{3}+|\beta|^{3}}$ in this case.

\textit{Case 2: $10\epsilon\leqslant|\beta|\leqslant2s$.} Here $|\beta|^{3}\geqslant10^{3}\epsilon^{3}\geqslant\epsilon^{3}$, hence
\begin{equation}\label{e:M.case2}
\frac{\epsilon}{\epsilon^{3}+|\beta|^{3}}\geqslant\frac{1}{2}\,\frac{\epsilon}{|\beta|^{3}}.
\end{equation}
Write $1=\big[1-\psi\big(\frac{10\rho}{\beta}\big)\big]+\psi\big(\frac{10\rho}{\beta}\big)$ inside the integrand defining $M^{*}\phi$ and change variables $\alpha=\beta+\sigma$ in the second contribution: this yields $M^{*}\phi=M_{1}+M_{2}$ with
\[
\begin{aligned}
M_{1}(\beta)&:=\frac{1}{2\pi}\int_{\mathbb{R}}\phi(\alpha)\frac{N(\alpha)}{\rho^{2}+N(\alpha)^{2}}\psi_{s}(\alpha)\psi_{s}(\rho)\Big[1-\psi\Big(\frac{10\rho}{\beta}\Big)\Big]d\alpha,\qquad \rho=\alpha-\beta,\\
M_{2}(\beta)&:=\frac{1}{2\pi}\int_{\mathbb{R}}\phi(\beta+\sigma)\frac{N(\beta+\sigma)}{\sigma^{2}+N(\beta+\sigma)^{2}}\psi_{s}(\beta+\sigma)\psi_{s}(\sigma)\psi\Big(\frac{10\sigma}{\beta}\Big)d\sigma.
\end{aligned}
\]
Since $\psi\equiv1$ on $[-\frac12,\frac12]$ and $\operatorname{supp}\psi\subseteq[-1,1]$, the factor $1-\psi(\frac{10\rho}{\beta})$ vanishes for $|\rho|\leqslant\frac{|\beta|}{20}$, while $\psi(\frac{10\sigma}{\beta})$ vanishes for $|\sigma|\geqslant\frac{|\beta|}{10}$. Both factors are smooth in $\beta$ here, because $|\beta|\geqslant10\epsilon>0$.

\emph{Estimate of $M_{1}'$.} Differentiating at fixed $\alpha$, using $\partial_{\beta}\rho=-1$ and $\frac{\rho}{\beta}=\frac{\alpha}{\beta}-1$,
\[
\begin{aligned}
\partial_{\beta}\Big[\frac{N(\alpha)}{\rho^{2}+N(\alpha)^{2}}\Big]&=\frac{2\rho N(\alpha)}{[\rho^{2}+N(\alpha)^{2}]^{2}},\qquad
\partial_{\beta}\psi_{s}(\rho)=-\frac{1}{s}\psi'\Big(\frac{\rho}{s}\Big),\\
\partial_{\beta}\Big[1-\psi\Big(\frac{10\rho}{\beta}\Big)\Big]&=\frac{10\alpha}{\beta^{2}}\psi'\Big(\frac{10\rho}{\beta}\Big).
\end{aligned}
\]
The three resulting integrals are estimated in turn.
\begin{itemize}
\item[(a)] Its integrand is supported in $\{|\rho|\geqslant\frac{|\beta|}{20}\}$, where by \eqref{e:N.range} and $[\rho^{2}+N(\alpha)^{2}]^{2}\geqslant\rho^{4}$,
\[
\frac{2|\rho|N(\alpha)}{[\rho^{2}+N(\alpha)^{2}]^{2}}\leqslant\frac{2C_{A}\epsilon}{|\rho|^{3}}\leqslant\frac{2\cdot20^{3}C_{A}\epsilon}{|\beta|^{3}};
\]
by \eqref{e:phi.ab} this term is therefore $\lesssim_{A}\frac{\epsilon}{|\beta|^{3}}$, which is admissible by \eqref{e:M.case2}.
\item[(b)] Its integrand is supported in $\{\frac{s}{2}\leqslant|\rho|\leqslant s\}$, where $\rho^{2}+N(\alpha)^{2}\geqslant\frac{s^{2}}{4}$, whence $\frac{N(\alpha)}{\rho^{2}+N(\alpha)^{2}}\leqslant\frac{4C_{A}\epsilon}{s^{2}}=2^{14}C_{A}$ (recall $s^{2}=2^{-12}\epsilon$). By \eqref{e:phi.ab} this term is at most $\frac{2^{14}C_{A}C_{\psi}\pi}{s}\lesssim_{A}\epsilon^{-\frac12}$, admissible by \eqref{e:M.target}.
\item[(c)] Its integrand is supported in $\{\frac{|\beta|}{20}\leqslant|\rho|\leqslant\frac{|\beta|}{10}\}$, a set of measure at most $\frac{|\beta|}{5}$, on which $\frac{9}{10}|\beta|\leqslant|\alpha|=|\beta+\rho|\leqslant\frac{11}{10}|\beta|$; since $|\beta|\geqslant10\epsilon$, there holds
\[
|\phi(\alpha)|\leqslant\frac{\epsilon}{\epsilon^{2}+\alpha^{2}}\lesssim\frac{\epsilon}{\beta^{2}},\qquad
\frac{N(\alpha)}{\rho^{2}+N(\alpha)^{2}}\leqslant\frac{C_{A}\epsilon}{\rho^{2}}\lesssim_{A}\frac{\epsilon}{\beta^{2}},\qquad
\frac{10|\alpha|}{\beta^{2}}\Big|\psi'\Big(\frac{10\rho}{\beta}\Big)\Big|\lesssim\frac{C_{\psi}}{|\beta|},
\]
so that this term is $\lesssim_{A}|\beta|\cdot\frac{\epsilon}{\beta^{2}}\cdot\frac{\epsilon}{\beta^{2}}\cdot\frac{1}{|\beta|}=\frac{\epsilon^{2}}{|\beta|^{4}}\leqslant\frac{\epsilon}{|\beta|^{3}}$, admissible by \eqref{e:M.case2}.
\end{itemize}
Hence $|M_{1}'(\beta)|\lesssim_{A}\frac{\epsilon}{\epsilon^{3}+|\beta|^{3}}$.

\emph{Estimate of $M_{2}'$.} Put
\[
\mathcal{K}(\beta,\sigma):=\frac{N(\beta+\sigma)}{\sigma^{2}+N(\beta+\sigma)^{2}}\psi_{s}(\beta+\sigma)\psi_{s}(\sigma)\psi\Big(\frac{10\sigma}{\beta}\Big),
\]
so that $M_{2}'(\beta)=\frac{1}{2\pi}\int_{\mathbb{R}}\big[\phi'(\beta+\sigma)\mathcal{K}(\beta,\sigma)+\phi(\beta+\sigma)\partial_{\beta}\mathcal{K}(\beta,\sigma)\big]d\sigma$. Differentiating at fixed $\sigma$,
\[
\begin{aligned}
\partial_{\beta}\mathcal{K}(\beta,\sigma)=&N'(\beta+\sigma)\frac{\sigma^{2}-N(\beta+\sigma)^{2}}{[\sigma^{2}+N(\beta+\sigma)^{2}]^{2}}\psi_{s}(\beta+\sigma)\psi_{s}(\sigma)\psi\Big(\frac{10\sigma}{\beta}\Big)\\
&+\frac{N(\beta+\sigma)}{\sigma^{2}+N(\beta+\sigma)^{2}}\frac{1}{s}\psi'\Big(\frac{\beta+\sigma}{s}\Big)\psi_{s}(\sigma)\psi\Big(\frac{10\sigma}{\beta}\Big)\\
&-\frac{N(\beta+\sigma)}{\sigma^{2}+N(\beta+\sigma)^{2}}\psi_{s}(\beta+\sigma)\psi_{s}(\sigma)\psi'\Big(\frac{10\sigma}{\beta}\Big)\frac{10\sigma}{\beta^{2}}.
\end{aligned}
\]
Every piece is supported in $\{|\sigma|\leqslant\frac{|\beta|}{10}\}$, hence $|\beta+\sigma|\geqslant\frac{9}{10}|\beta|$; moreover $|N'(\beta+\sigma)|\leqslant|f'(\beta+\sigma)|+|g'(\beta+\sigma)|\leqslant2A|\beta+\sigma|\leqslant2As$ wherever $\psi_{s}(\beta+\sigma)\ne0$, and, by \eqref{e:N.range} applied on the range $|\beta+\sigma|\leqslant s$ to which each term below is confined by its factor $\psi_{s}(\beta+\sigma)$ (or $\psi'\big(\frac{\beta+\sigma}{s}\big)$, which confines even more),
\begin{equation}\label{e:M.poisson}
\int_{\mathbb{R}}\frac{N(\beta+\sigma)}{\sigma^{2}+N(\beta+\sigma)^{2}}\,d\sigma\leqslant C_{A}\epsilon\int_{\mathbb{R}}\frac{d\sigma}{\sigma^{2}+\epsilon^{2}}\leqslant C_{A}\pi.
\end{equation}
The four contributions are bounded as follows.
\begin{itemize}
\item[(d)] Since $|\beta+\sigma|\geqslant\frac{9}{10}|\beta|$ and $\phi\in\mathcal{C}$,
\[
|\phi'(\beta+\sigma)|\leqslant\frac{\epsilon}{\epsilon^{3}+|\beta+\sigma|^{3}}\leqslant\Big(\frac{10}{9}\Big)^{3}\frac{\epsilon}{\epsilon^{3}+|\beta|^{3}},
\]
so by \eqref{e:M.poisson} this term is $\lesssim_{A}\frac{\epsilon}{\epsilon^{3}+|\beta|^{3}}$.
\item[(e)] Using $\frac{|\sigma^{2}-N^{2}|}{[\sigma^{2}+N^{2}]^{2}}\leqslant\frac{1}{\sigma^{2}+N^{2}}\leqslant\frac{1}{\sigma^{2}+\epsilon^{2}}$ together with $|\phi(\beta+\sigma)|\leqslant\frac{\epsilon}{\epsilon^{2}+(\beta+\sigma)^{2}}\lesssim\frac{\epsilon}{\beta^{2}}$ and $|N'(\beta+\sigma)|\leqslant2As$, this term is at most
\[
2A\,C_{A}\,\frac{\epsilon}{\beta^{2}}\,s\int_{\mathbb{R}}\frac{d\sigma}{\sigma^{2}+\epsilon^{2}}\lesssim_{A}\frac{s}{\beta^{2}}.
\]
Since $|\beta|\leqslant2s$ gives $s|\beta|\leqslant2s^{2}\leqslant\epsilon$, we have $\frac{s}{\beta^{2}}\lesssim\frac{\epsilon}{|\beta|^{3}}$, so this term is admissible by \eqref{e:M.case2}.
\item[(f)] This contribution is supported in $\{\frac{s}{2}\leqslant|\beta+\sigma|\leqslant s\}$, where $|\phi(\beta+\sigma)|\leqslant\frac{\epsilon}{\epsilon^{2}+(\beta+\sigma)^{2}}\leqslant\frac{4\epsilon}{s^{2}}\lesssim1$; hence by \eqref{e:M.poisson} it is $\lesssim_{A}s^{-1}\lesssim\epsilon^{-\frac12}$, admissible by \eqref{e:M.target}.
\item[(g)] This contribution is supported in $\{\frac{|\beta|}{20}\leqslant|\sigma|\leqslant\frac{|\beta|}{10}\}$, where $\frac{10|\sigma|}{\beta^{2}}\lesssim\frac{1}{|\beta|}$ and, by \eqref{e:N.range}, $\frac{N(\beta+\sigma)}{\sigma^{2}+N(\beta+\sigma)^{2}}\leqslant\frac{C_{A}\epsilon}{\sigma^{2}}\lesssim_{A}\frac{\epsilon}{\beta^{2}}$, while $|\phi(\beta+\sigma)|\lesssim\frac{\epsilon}{\beta^{2}}$; since the support has measure $\lesssim|\beta|$, it is $\lesssim_{A}\frac{\epsilon^{2}}{|\beta|^{4}}\leqslant\frac{\epsilon}{|\beta|^{3}}$, admissible by \eqref{e:M.case2}.
\end{itemize}
Adding (a)--(g) settles Case 2, and together with Case 1 this proves the second bound in \eqref{e:M.goal}. Taking $C_{M}$ to be the largest of the constants produced above (they depend only on $A$ and on $\psi$, hence only on $A$ and $\delta$ in the notation of the statement), we obtain $C_{M}^{-1}M^{*}\phi\in\mathcal{C}$.
\end{proof}
\begin{lemma}\label{lem:key}
$(1)$ For any $h\in L^\infty((-\delta,\delta))$ we have
\[
\|T_{1,f}h\|_Z\lesssim \epsilon\log(1/\epsilon)\|h\|_{L^\infty((-\delta,\delta))},
\]
and
\[
\|(M-T_{3,(f,g)})h\|_Z+\|(M-T_{3,(g,f)})h\|_Z\lesssim \epsilon\log(1/\epsilon)\|h\|_{L^\infty(-2\sqrt\epsilon,2\sqrt{\epsilon})}.
\]
$(2)$ For any $h\in C^{1}((-2\sqrt{\epsilon},2\sqrt{\epsilon}))$ and any $\phi\in\mathcal{C}$ supported in $(-\sqrt{\epsilon},\sqrt{\epsilon})$,
\[
\begin{split}
&\Big|\int_{(-\sqrt{\epsilon},\sqrt{\epsilon})}(H+T_{4,(g,f)}-T_{2,f})h(\alpha)\phi(\alpha)\,d\alpha\Big|\\
&+\Big|\int_{(-\sqrt{\epsilon},\sqrt{\epsilon})}(H+T_{4,(f,g)}-T_{2,g})h(\alpha)\phi(\alpha)\,d\alpha\Big|\lesssim \epsilon\log(1/\epsilon)\|h\|_{C^{1}((-2\sqrt{\epsilon},2\sqrt{\epsilon}))}.
\end{split}
\]
\end{lemma}
\begin{proof}
We write
\[
\frac{f(\alpha)-f(\beta)}{(\alpha-\beta)^{2}+(f(\alpha)-f(\beta))^{2}}=\frac{f'(\beta)}{(\alpha-\beta)(1+f'(\beta)^{2})}+K(\alpha,\beta),
\] 
where
\[
\begin{aligned}
K(\alpha,\beta):=&\frac{f(\alpha)-f(\beta)}{(\alpha-\beta)^{2}+(f(\alpha)-f(\beta))^{2}}-\frac{f'(\beta)}{(\alpha-\beta)(1+f'(\beta)^{2})} \\
=&K_{1}(\alpha,\beta)+K_{2}(\alpha,\beta), \\
K_{1}(\alpha,\beta):=&\frac{f(\alpha)-f(\beta)-(\alpha-\beta)f'(\beta)}{(\alpha-\beta)^{2}(1+\widetilde{f}(\alpha,\beta)^{2})}, \\
K_2(\alpha,\beta):=&\frac{f'(\beta)[f'(\beta)+\widetilde{f}(\alpha,\beta)][f'(\beta)-\widetilde{f}(\alpha,\beta)]}{(\alpha-\beta)(1+\widetilde{f}(\alpha,\beta)^2)(1+f'(\beta)^2)},\\
\widetilde{f}(\alpha,\beta):=&\frac{f(\alpha)-f(\beta)}{\alpha-\beta}.
\end{aligned}
\]
By the fundamental theorem of calculus and Taylor's formula, we get 
\begin{equation}\label{K.bound}
\sup_{\alpha,\beta\in (-2\sqrt{\epsilon},2\sqrt{\epsilon})}\big[|K(\alpha,\beta)|+|\partial_\alpha K(\alpha,\beta)|+|\partial_\beta K(\alpha,\beta)|\big]\lesssim 1.
\end{equation}
Here no singularity is hidden in the diagonal $\alpha=\beta$: setting
$\tau:=\alpha-\beta$ and $\widehat\Lambda:=\Lambda/\tau^{2}$, where
$\Lambda(\alpha,\beta):=f(\alpha)-f(\beta)-(\alpha-\beta)f'(\beta)$, Taylor's formula gives
$\widehat\Lambda=\frac12f''(\beta)+O(\tau)$ and
$\widetilde f=f'(\beta)+O(\tau)$, both smooth, and accordingly
$K_{1}=\widehat\Lambda/(1+\widetilde f^{2})$ and
$K_{2}=f'(\beta)(f'(\beta)+\widetilde f)(f'(\beta)-\widetilde f)\big/\big(\tau(1+\widetilde f^{2})(1+f'(\beta)^{2})\big)$
extend smoothly to $\tau=0$, because $f'(\beta)-\widetilde f=O(\tau)$; hence $K$, its
$\alpha$- and $\beta$-derivatives are bounded on the compact range by a constant depending only
on $\|f\|_{C^{3}}\leqslant A$.
Thus, setting
\[
T_{1,f}^{1}h(\alpha):=\frac{1}{2\pi}\int_{\mathbb{R}}K(\alpha,\beta)h(\beta)\psi_{\sqrt{\epsilon}}(\alpha-\beta)\,d\beta,
\]
we obtain from Lemma \ref{lem:elem} that
\[
\|T_{1,f}^1h\|_Z\lesssim \epsilon\|T_{1,f}^1h\|_{C^1((-\sqrt{\epsilon},\sqrt{\epsilon}))}\lesssim\epsilon.
\]

It remains to check that $\|T_{1,f}^{1}h\|_{C^{1}((-\sqrt{\epsilon},\sqrt{\epsilon}))}\lesssim1$
under the normalization $\|h\|_{L^{\infty}}\leqslant1$. The bound \eqref{K.bound} says that $K$
and $\partial_{\alpha}K$ are bounded on the relevant range by a constant depending only on $A$,
while $\psi_{\sqrt{\epsilon}}$ is supported in an interval of length $2\sqrt{\epsilon}$ and
satisfies $|\partial_{\alpha}\psi_{\sqrt{\epsilon}}|\lesssim\epsilon^{-1/2}$. Since the integrand
is smooth and compactly supported in $\beta$ for fixed $\alpha$, differentiation under the
integral sign is legitimate, and therefore
\[
|T_{1,f}^{1}h(\alpha)|\lesssim\int_{\{|\alpha-\beta|\leqslant\sqrt{\epsilon}\}}d\beta
\lesssim\sqrt{\epsilon},
\]
\[
|\partial_{\alpha}T_{1,f}^{1}h(\alpha)|
\lesssim\int_{\{|\alpha-\beta|\leqslant\sqrt{\epsilon}\}}\Big(1+\epsilon^{-\frac12}\Big)d\beta
\lesssim\sqrt{\epsilon}+1\lesssim1.
\]
To continue, we introduce a dyadic partition of the cutoff: let $j_0$ denote the largest integer with the property that $2^{j_0}\leqslant \epsilon$ (we will always assume $\epsilon$ small enough that $\frac{j_0}{2}+4\leqslant0$, which holds as soon as $\epsilon<2^{-7}$) and let
\[
\eta_j(\alpha):=
\begin{cases}
\psi_{2^{j_0}}(\alpha)\qquad&\text{ if }j=j_0;\\
\psi_{2^{j}}(\alpha)-\psi_{2^{j-1}}(\alpha)\qquad&\text{ if }j\in[j_0+1,0]\cap\mathbb{Z}.
\end{cases}
\]
The functions satisfy
\[
\sum_{j=j_0}^0\eta_j=\psi.
\]
For $k_1,k_2\in[j_0,j_0/2+4]\cap\mathbb{Z}$ we define
\[
\begin{aligned}
h_{k_1}(\beta)&:=h(\beta)\eta_{k_1}(\beta), \\ H_{k_1,k_2}(\alpha)&:=\eta_{k_2}(\alpha)\frac{1}{2\pi}\int_{\mathbb{R}}\frac{f'(\beta)}{(\alpha-\beta)(1+f'(\beta)^2)}h_{k_1}(\beta)\psi_{\sqrt{\epsilon}}(\alpha-\beta)\,d\beta.
\end{aligned}
\]
Without loss of generality we assume $\|h\|_{L^{\infty}(-2\sqrt{\epsilon},2\sqrt{\epsilon})}\leqslant 1$; the general case follows by linearity, since all the operators above are linear in $h$. 

\textit{Step 1 (reduction to the indices $k_{1},k_{2}\in[j_{0},\frac{j_{0}}{2}+4]$).} Let
\[
\mathcal{T}h(\alpha):=\frac{1}{2\pi}\mathrm{p.v.}\int_{\mathbb{R}}\frac{f'(\beta)}{(\alpha-\beta)(1+f'(\beta)^{2})}h(\beta)\psi_{\sqrt{\epsilon}}(\alpha-\beta)\,d\beta
\]
denote the principal part of $T_{1,f}$, so that $T_{1,f}h=T_{1,f}^{1}h+\mathcal{T}h$ and
\[
H_{k_{1},k_{2}}(\alpha)=\eta_{k_{2}}(\alpha)\mathcal{T}(h\eta_{k_{1}})(\alpha).
\]
We check that the $Z$-norm only probes the region $|\alpha|\leqslant\frac{\sqrt{\epsilon}}{32}$ and, through the factor $\psi_{\sqrt{\epsilon}}(\alpha-\beta)$, only the values $|\beta|\leqslant\frac{33}{32}\sqrt{\epsilon}$. By the definition of $H^{*}$, the integrand of $H^{*}\phi(\beta)$ contains the two factors $\psi_{\frac{\sqrt{\epsilon}}{2^{6}}}(\beta+\rho)$ and $\psi_{\frac{\sqrt{\epsilon}}{2^{6}}}(\rho)$, whence $H^{*}\phi(\beta)\ne0$ forces $|\beta|\leqslant|\beta+\rho|+|\rho|\leqslant \frac{\sqrt{\epsilon}}{2^{5}}$; and if $|\alpha|\leqslant\frac{\sqrt{\epsilon}}{32}$ then $\psi_{\sqrt{\epsilon}}(\alpha-\beta)\ne0$ forces $|\beta|\leqslant|\alpha|+\sqrt{\epsilon}\leqslant\frac{33}{32}\sqrt{\epsilon}$.\par\noindent
Since $2^{j_{0}}$ is the largest power of two with $2^{j_{0}}\leqslant\epsilon$, we have $2^{j_{0}}>\frac{\epsilon}{2}$, hence
\[
2^{\frac{j_{0}}{2}+4}>\frac{16}{\sqrt{2}}\sqrt{\epsilon}>11\sqrt{\epsilon}\geqslant\frac{33}{32}\sqrt{\epsilon}.
\]
On the other hand $\eta_{k}=\psi_{2^{k}}-\psi_{2^{k-1}}$ is supported in $\{2^{k-2}\leqslant|\alpha|\leqslant2^{k}\}$ for $k>j_{0}$, so for every $k\geqslant\frac{j_{0}}{2}+5$ the function $\eta_{k}$ vanishes on $\{|\alpha|\leqslant\frac{33}{32}\sqrt{\epsilon}\}$, because there $|\alpha|\leqslant\frac{33}{32}\sqrt{\epsilon}<2^{\frac{j_{0}}{2}+3}\leqslant2^{k-2}$. Since $\sum_{j=j_{0}}^{0}\eta_{j}=\psi$ and $\psi\equiv1$ on $[-1,1]\supseteq[-\frac{33}{32}\sqrt{\epsilon},\frac{33}{32}\sqrt{\epsilon}]$, we conclude that
\[
\sum_{k=j_{0}}^{\frac{j_{0}}{2}+4}\eta_{k}(\alpha)=1\qquad\text{for all }|\alpha|\leqslant\frac{33}{32}\sqrt{\epsilon}.
\]
Therefore, for every $\phi\in\mathcal{C}$ and every $|\alpha|\leqslant\frac{\sqrt{\epsilon}}{32}$,
\[
\sum_{k_{1},k_{2}\in[j_{0},\frac{j_{0}}{2}+4]\cap\mathbb{Z}}H_{k_{1},k_{2}}(\alpha)=\Big(\sum_{k_{2}}\eta_{k_{2}}(\alpha)\Big)\mathcal{T}\Big(h\sum_{k_{1}}\eta_{k_{1}}\Big)(\alpha)=\mathcal{T}h(\alpha),
\]
and consequently
\[
\|\mathcal{T}h\|_{Z}=\Big\|\sum_{k_1,k_2\in[j_0,j_0/2+4]\cap\mathbb{Z}}H_{k_1,k_2}\Big\|_Z.
\]
Thus it suffices to show that
\[
\Big\|\sum_{k_1,k_2\in[j_0,j_0/2+4]\cap\mathbb{Z}}H_{k_1,k_2}\Big\|_Z\lesssim \epsilon\log(\frac{1}{\epsilon}).
\]

\textit{Step 2 (the $Z$-weight on $\operatorname{supp}\eta_{k_{2}}$).} Fix $k_{1},k_{2}\in[j_{0},\frac{j_{0}}{2}+4]\cap\mathbb{Z}$. Since $\operatorname{supp}H_{k_{1},k_{2}}\subseteq\operatorname{supp}\eta_{k_{2}}\subseteq[-2^{k_{2}},2^{k_{2}}]$ and $2^{k_{2}}\geqslant2^{j_{0}}>\frac{\epsilon}{2}$, the weight occurring in the second estimate of Lemma \ref{lem:elem} satisfies
\begin{equation}\label{e:Zweight}
\frac{\epsilon^{2}}{\epsilon^{3}+|\alpha|^{3}}\lesssim\epsilon^{2}2^{-3k_{2}}\qquad\text{uniformly for }\alpha\in\operatorname{supp}\eta_{k_{2}},
\end{equation}
and hence
\[
\|H_{k_{1},k_{2}}\|_{Z}\lesssim\int_{\mathbb{R}}|H_{k_{1},k_{2}}(\alpha)|\frac{\epsilon^{2}}{\epsilon^{3}+|\alpha|^{3}}\,d\alpha\lesssim\epsilon^{2}2^{-3k_{2}}\|H_{k_{1},k_{2}}\|_{L^{1}}.
\]
If $k_{2}=j_{0}$ then $\epsilon^{3}+|\alpha|^{3}\geqslant\epsilon^{3}$, so the weight is at most $\epsilon^{-1}$, while $2^{-3j_{0}}\geqslant\epsilon^{-3}$ gives $\epsilon^{2}2^{-3j_{0}}\geqslant\epsilon^{-1}$; and if $k_{2}>j_{0}$ then $\eta_{k_{2}}=\psi_{2^{k_{2}}}-\psi_{2^{k_{2}-1}}$ is supported in $\{2^{k_{2}-2}\leqslant|\alpha|\leqslant2^{k_{2}}\}$ (the smaller cutoff $\psi_{2^{k_{2}-1}}$ equals $1$ on $[-\frac{1}{2}2^{k_{2}-1},\frac{1}{2}2^{k_{2}-1}]$), so that $\epsilon^{3}+|\alpha|^{3}\geqslant|\alpha|^{3}\geqslant2^{3k_{2}-6}$ and the weight is at most $\epsilon^{2}2^{-3k_{2}+6}$. 

\textit{Step 3 (Cauchy--Schwarz).} Since $|\operatorname{supp}\eta_{k_{2}}|\leqslant2^{k_{2}+1}$,
\[
\|H_{k_{1},k_{2}}\|_{L^{1}}\leqslant\|\mathbf{1}\|_{L^{2}(\operatorname{supp}\eta_{k_{2}})}\|H_{k_{1},k_{2}}\|_{L^{2}}\leqslant2^{\frac{k_{2}+1}{2}}\|H_{k_{1},k_{2}}\|_{L^{2}}.
\]

\textit{Step 4 (boundedness of the truncated Hilbert transform on $L^{2}$, uniform in $\epsilon$).} Put
$G(\beta):=\frac{f'(\beta)}{1+f'(\beta)^{2}}h_{k_{1}}(\beta)\in L^{2}(\mathbb{R})$ and
\[
\mathcal{H}_{\sqrt{\epsilon}}G(\alpha):=\frac{1}{2\pi}\mathrm{p.v.}\int_{\mathbb{R}}\frac{G(\beta)}{\alpha-\beta}\,\psi_{\sqrt{\epsilon}}(\alpha-\beta)\,d\beta,
\]
so that $H_{k_{1},k_{2}}(\alpha)=\eta_{k_{2}}(\alpha)\,\mathcal{H}_{\sqrt{\epsilon}}G(\alpha)$. We claim that
\[
\|\mathcal{H}_{\sqrt{\epsilon}}G\|_{L^{2}(\mathbb{R})}\leqslant C_{\psi}\|G\|_{L^{2}(\mathbb{R})},\qquad C_{\psi}:=\tfrac{1}{2}\|\widehat{\psi}\|_{L^{1}(\mathbb{R})},
\]
with $C_{\psi}$ independent of $\epsilon$. The kernel $K(\rho):=\frac{1}{2\pi}\frac{\psi_{\sqrt{\epsilon}}(\rho)}{\rho}$ is the compactly supported distribution $\psi_{\sqrt{\epsilon}}\cdot\mathrm{p.v.}\frac{1}{2\pi\rho}$, i.e., the product of the even cutoff $\psi_{\sqrt{\epsilon}}$ with the odd distribution $\mathrm{p.v.}\frac{1}{2\pi\rho}$. Passing to the Fourier side (with the convention $\widehat{u}(\xi)=\int_{\mathbb{R}}u(\rho)e^{-i\rho\xi}d\rho$), a product becomes, up to the factor $(2\pi)^{-1}$, the convolution of the Fourier transforms, and
$\big|\mathcal{F}\big(\mathrm{p.v.}\frac{1}{\pi\rho}\big)(\xi)\big|=|\!-\!i\,\mathrm{sgn}\,\xi|\leqslant1$. Since $\psi$ is smooth and compactly supported, $\widehat{\psi}$ decays rapidly, so $\|\widehat{\psi}\|_{L^{1}}<\infty$, and since
$\widehat{\psi_{\sqrt{\epsilon}}}(\xi)=\sqrt{\epsilon}\,\widehat{\psi}(\sqrt{\epsilon}\xi)$ we have $\|\widehat{\psi_{\sqrt{\epsilon}}}\|_{L^{1}}=\|\widehat{\psi}\|_{L^{1}}$, independently of $\epsilon$. Young's inequality for convolutions then yields $\|\widehat{K}\|_{L^{\infty}}\leqslant C_{\psi}$, and the claim follows from Plancherel's theorem. 

\textit{Step 5 (the region $k_{2}\geqslant k_{1}-10$).} By Steps 3 and 4, by $|\eta_{k_{2}}|\leqslant1$, and by the mean value theorem together with $f'(0)=0$, $\|f''\|_{L^{\infty}}\leqslant A$, which give $\|f'\|_{L^{\infty}(\operatorname{supp}\eta_{k_{1}})}\leqslant A\,2^{k_{1}}$ (because $\operatorname{supp}\eta_{k_{1}}\subseteq[-2^{k_{1}},2^{k_{1}}$]), we get
\[
\|H_{k_{1},k_{2}}\|_{L^{2}}\leqslant\|\mathcal{H}_{\sqrt{\epsilon}}G\|_{L^{2}}\leqslant C_{\psi}\|G\|_{L^{2}}\leqslant C_{\psi}A\,2^{k_{1}}\|h_{k_{1}}\|_{L^{2}}.
\]
Also, $\|h_{k_{1}}\|_{L^{2}}\leqslant\|h\|_{L^{\infty}}\|\eta_{k_{1}}\|_{L^{2}}\leqslant2^{\frac{k_{1}+1}{2}}$. Combining this with Steps 2 and 3,
\[
\begin{aligned}
\|H_{k_{1},k_{2}}\|_{Z}&\lesssim\epsilon^{2}2^{-3k_{2}}\|H_{k_{1},k_{2}}\|_{L^{1}}
\lesssim\epsilon^{2}2^{-\frac{5k_{2}}{2}}\|H_{k_{1},k_{2}}\|_{L^{2}} \\
&\lesssim_{A}\epsilon^{2}2^{-\frac{5k_{2}}{2}}2^{k_{1}}\|h_{k_{1}}\|_{L^{2}}
\lesssim\epsilon^{2}2^{-\frac{5k_{2}}{2}}2^{\frac{3k_{1}}{2}}.
\end{aligned}
\]
In the region $k_{2}\geqslant k_{1}-10$, i.e., $k_{1}\leqslant k_{2}+10$, we may bound $2^{\frac{3k_{1}}{2}}\leqslant2^{\frac{3k_{2}}{2}+15}$, hence
\[
\|H_{k_{1},k_{2}}\|_{Z}\lesssim_{A}\epsilon^{2}2^{-k_{2}}\qquad\text{whenever }k_{2}\geqslant k_{1}-10.
\]
Summing over this region, and writing $k_{2}=j_{0}+m$ (for a given $k_{2}$ the number of admissible indices $k_{1}\in[j_{0},\min(k_{2}+10,\frac{j_{0}}{2}+4)]$ is at most $m+11$),
\[
\sum_{\substack{k_{1},k_{2}\in[j_{0},\frac{j_{0}}{2}+4]\cap\mathbb{Z}\\k_{2}\geqslant k_{1}-10}}\|H_{k_{1},k_{2}}\|_{Z}
\lesssim\epsilon^{2}\sum_{m\geqslant0}(m+11)2^{-j_{0}-m}
=\epsilon^{2}2^{-j_{0}}\sum_{m\geqslant0}(m+11)2^{-m}\lesssim\epsilon,
\]
because $\sum_{m\geqslant0}(m+11)2^{-m}=\sum_{m\geqslant0}m2^{-m}+11\sum_{m\geqslant0}2^{-m}=2+22=24$ and $2^{-j_{0}}<2\epsilon^{-1}$.

\textit{Step 6 (the separated region $k_{2}\leqslant k_{1}-11$).} Here the two scales are well separated. First, $k_{2}\leqslant k_{1}-11$ together with $k_{2}\geqslant j_{0}$ forces $k_{1}\geqslant j_{0}+11>j_{0}$, so that $\eta_{k_{1}}$ is supported in the annulus $\{2^{k_{1}-2}\leqslant|\beta|\leqslant2^{k_{1}}\}$, whereas $\eta_{k_{2}}$ is supported in $\{|\alpha|\leqslant2^{k_{2}}\}\subseteq\{|\alpha|\leqslant2^{k_{1}-11}\}$; moreover
$2^{k_{1}}\leqslant2^{\frac{j_{0}}{2}+4}\leqslant16\sqrt{\epsilon}$, because $2^{j_{0}}\leqslant\epsilon$. Hence for $\alpha\in\operatorname{supp}\eta_{k_{2}}$ and $\beta\in\operatorname{supp}\eta_{k_{1}}$,
\[
|\alpha-\beta|\geqslant|\beta|-|\alpha|\geqslant2^{k_{1}-2}-2^{k_{1}-11}\geqslant2^{k_{1}-3},\qquad |f'(\beta)|\leqslant A|\beta|\leqslant A\,2^{k_{1}},
\]
where we again used $f'(0)=0$ and $\|f''\|_{L^{\infty}}\leqslant A$.

The functions $H_{k_{1},k_{2}}$ are summed over $k_{2}$ before being estimated: each of them carries the cutoff $\eta_{k_{2}}$, which oscillates on the scale $2^{k_{2}}\ll2^{k_{1}}$, so its $C^{1}$-norm is not bounded independently of $\epsilon$. Since $\eta_{j_{0}}=\psi_{2^{j_{0}}}$ and $\eta_{k}=\psi_{2^{k}}-\psi_{2^{k-1}}$ for $k>j_{0}$, the sum telescopes:
\[
\sum_{k_{2}=j_{0}}^{K}\eta_{k_{2}}=\psi_{2^{K}}\qquad\text{for every integer }K\geqslant j_{0}.
\]
Therefore, for each fixed $k_{1}\geqslant j_{0}+11$,
\[
\sum_{k_{2}=j_{0}}^{k_{1}-11}H_{k_{1},k_{2}}(\alpha)
=\psi_{2^{k_{1}-11}}(\alpha)\,\mathcal{T}(h\eta_{k_{1}})(\alpha)=:\mathcal{S}_{k_{1}}(\alpha).
\]
We claim that
\begin{equation}\label{e:Sk1}
\|\mathcal{S}_{k_{1}}\|_{Z}\lesssim_{A}\epsilon\qquad\text{for }j_{0}+11\leqslant k_{1}\leqslant\tfrac{j_{0}}{2}+4.
\end{equation}
Write
\[
\begin{aligned}
\mathcal{T}(h\eta_{k_{1}})(\alpha)&=\frac{1}{2\pi}\int_{\mathbb{R}}G(\beta)\,P(\alpha-\beta)\,d\beta,\\
G(\beta)&:=\frac{f'(\beta)}{1+f'(\beta)^{2}}\,h(\beta)\eta_{k_{1}}(\beta),\qquad P(\rho):=\frac{\psi_{\sqrt{\epsilon}}(\rho)}{\rho}.
\end{aligned}
\]
By $\|h\|_{L^{\infty}}\leqslant1$, by $|f'(\beta)|\leqslant A\,2^{k_{1}}$ on $\operatorname{supp}\eta_{k_{1}}$ and by
$|\operatorname{supp}\eta_{k_{1}}|\leqslant2^{k_{1}+1}$,
\[
\|G\|_{L^{1}}\leqslant A\,2^{2k_{1}+1}.
\]
On the set where $\psi_{2^{k_{1}-11}}(\alpha)\ne0$ we have $|\alpha|\leqslant2^{k_{1}-11}$ while
$|\beta|\geqslant2^{k_{1}-2}$ on $\operatorname{supp}G$, so $|\alpha-\beta|\geqslant2^{k_{1}-3}$; there the
kernel $P$ never meets its singularity at the origin and satisfies
\[
|P(\alpha-\beta)|\leqslant2^{-k_{1}+3},\qquad
|\partial_{\alpha}P(\alpha-\beta)|\leqslant2^{-2k_{1}+6}+\|\psi'\|_{L^{\infty}}\,\epsilon^{-\frac12}\,2^{-k_{1}+3}.
\]
Differentiation under the integral sign is therefore legitimate, $\mathcal{T}(h\eta_{k_{1}})$ is smooth there, and
\[
|\mathcal{T}(h\eta_{k_{1}})(\alpha)|\lesssim_{A}2^{2k_{1}}\cdot2^{-k_{1}}\lesssim_{A}2^{k_{1}},\qquad
|\partial_{\alpha}\mathcal{T}(h\eta_{k_{1}})(\alpha)|
\lesssim_{A}2^{2k_{1}}\Big(2^{-2k_{1}}+\epsilon^{-\frac12}2^{-k_{1}}\Big)\lesssim_{A}1,
\]
the last step by $2^{k_{1}}\leqslant16\sqrt{\epsilon}$. Multiplying by $\psi_{2^{k_{1}-11}}$, whose
gradient is bounded by $\|\psi'\|_{L^{\infty}}2^{-k_{1}+11}$, we get
\[
|\mathcal{S}_{k_{1}}(\alpha)|+|\partial_{\alpha}\mathcal{S}_{k_{1}}|
\lesssim_{A}2^{k_{1}}+2^{k_{1}}2^{-k_{1}+11}+1\lesssim_{A}1,
\]
the last step because $2^{k_{1}}\leqslant2^{\frac{j_{0}}{2}+4}\leqslant16\sqrt{\epsilon}\lesssim1$,
and \eqref{e:Sk1} follows from the \emph{first} estimate of Lemma \ref{lem:elem} applied on
$I_{\epsilon}:=(-\tfrac{\sqrt{\epsilon}}{8},\tfrac{\sqrt{\epsilon}}{8})$. Finally, the
number of admissible indices $k_{1}$ is at most $\frac{|j_{0}|}{2}+5\lesssim\log\frac1\epsilon$
(because $2^{j_{0}}>\frac{\epsilon}{2}$ gives $|j_{0}|\leqslant\log_{2}\frac{2}{\epsilon}$), whence
\[
\Big\|\sum_{k_{1}}\ \sum_{k_{2}=j_{0}}^{k_{1}-11}H_{k_{1},k_{2}}\Big\|_{Z}
\leqslant\sum_{k_{1}}\|\mathcal{S}_{k_{1}}\|_{Z}\lesssim_{A}\epsilon\log(\tfrac1\epsilon).
\]

\textit{Step 7 (conclusion).} Every pair $(k_{1},k_{2})$ satisfies exactly one of
$k_{2}\geqslant k_{1}-10$ and $k_{2}\leqslant k_{1}-11$, so Steps 5 and 6 (together with the
contribution of $T_{1,f}^{1}$, for which $\|T_{1,f}^{1}h\|_{Z}\lesssim\epsilon$ was established
above) give
\[
\|T_{1,f}h\|_{Z}\leqslant\|T_{1,f}^{1}h\|_{Z}+\Big\|\sum_{k_{1},k_{2}\in[j_{0},\frac{j_{0}}{2}+4]\cap\mathbb{Z}}H_{k_{1},k_{2}}\Big\|_{Z}
\lesssim_{A}\epsilon+\epsilon\log(\tfrac1\epsilon)\lesssim\epsilon\log(\tfrac1\epsilon).
\]
This proves the first bound.

For the second inequality in $(1)$, by linearity we may again assume
$\|h\|_{L^{\infty}(-2\sqrt{\epsilon},2\sqrt{\epsilon})}\leqslant1$; and since $M$ is symmetric in
$(f,g)$ (it only involves the sum $f(\alpha)+g(\alpha)$), the very same argument with $(f,g)$
replaced by $(g,f)$ yields the bound for $\|(M-T_{3,(g,f)})h\|_{Z}$. So it suffices to prove
\[
\|(M-T_{3,(f,g)})h\|_{Z}\lesssim_{A}\epsilon\log(\tfrac{1}{\epsilon}).
\]
We will use the two estimates of Lemma \ref{lem:elem} alternately: for every $w$,
\begin{equation}\label{e:Ztwo}
\|w\|_{Z}\lesssim\epsilon\|w\|_{C^{1}(I_{\epsilon})}
\qquad\text{or}\qquad
\|w\|_{Z}\lesssim\int_{I}|w(\beta)|\,\frac{\epsilon^{2}}{\epsilon^{3}+|\beta|^{3}}\,d\beta,
\qquad I:=(-\tfrac{\sqrt{\epsilon}}{8},\tfrac{\sqrt{\epsilon}}{8}).
\end{equation}
Throughout this proof we abbreviate
\[
\begin{aligned}
&s:=\frac{\sqrt{\epsilon}}{2^{6}},\qquad
\widetilde{a}(\alpha):=2\epsilon+f(\alpha)+g(\alpha),\\
&a(\alpha,\beta):=2\epsilon+f(\alpha)+g(\beta),\qquad
\delta(\alpha,\beta):=g(\beta)-g(\alpha),
\end{aligned}
\]
so that $a=\widetilde{a}+\delta$, and we write $\rho:=\alpha-\beta$. 

\textit{Step 1 (the parameters remain in the admissible range).} Fix $\alpha\in I$ and let $\beta$
be such that the kernels below do not vanish. The operator $T_{3,(f,g)}$ carries the factor
$\psi_{\sqrt{\epsilon}}(\rho)$ and $M$ carries $\psi_{s}(\rho)$, so $|\rho|\leqslant\sqrt{\epsilon}$ and hence
\[
|\beta|\leqslant|\alpha|+\sqrt{\epsilon}\leqslant\tfrac{9}{8}\sqrt{\epsilon}<2\sqrt{\epsilon};
\]
only the values of $h$ on $(-2\sqrt{\epsilon},2\sqrt{\epsilon})$ are therefore used, and
$|\alpha|,|\beta|\leqslant4\sqrt{\epsilon}$; also $\alpha\in I\subseteq(-\sqrt{\epsilon},\sqrt{\epsilon})$.
Hence \eqref{e:a.range}, \eqref{e:a.low} and \eqref{e:N.range.0} apply and give
\begin{equation}\label{e:arange}
\big|a(\alpha,\beta)\big|\leqslant(2+A)\epsilon,\quad
a(\alpha,\beta)\geqslant\frac{\epsilon}{2}\ \text{ if }|\rho|\leqslant\epsilon,\quad
\epsilon\leqslant\widetilde{a}(\alpha)\leqslant(2+A)\epsilon.
\end{equation}
When $|\rho|\geqslant\epsilon$ the kernel is controlled instead by
$|a|/(\rho^{2}+a^{2})\leqslant1/(2|\rho|)$, as recorded in
\eqref{e:a.low}. By the mean value theorem and
$g'(0)=0$, $\|g''\|_{L^{\infty}}\leqslant A$,
\begin{equation}\label{e:deltabd}
|g'(\alpha)|\leqslant A|\alpha|,\qquad |g'(\alpha)-g'(\beta)|\leqslant A|\rho|,\qquad
|\delta(\alpha,\beta)|\leqslant A\big(|\alpha|+|\rho|\big)|\rho|.
\end{equation}

\textit{Step 2 (exact splitting of the difference of the two kernels).} Put
\[
K_{M}(\alpha,\beta):=\frac{\widetilde{a}(\alpha)}{\rho^{2}+\widetilde{a}(\alpha)^{2}}\,\psi_{s}(\alpha)\psi_{s}(\rho),
\qquad
K_{T}(\alpha,\beta):=\frac{a(\alpha,\beta)}{\rho^{2}+a(\alpha,\beta)^{2}}\,\psi_{\sqrt{\epsilon}}(\rho),
\]
so that $(M-T_{3,(f,g)})h(\alpha)=\frac{1}{2\pi}\int_{\mathbb{R}}h(\beta)\,[K_{M}-K_{T}](\alpha,\beta)\,d\beta$.
A direct computation gives
\[
\frac{a}{\rho^{2}+a^{2}}-\frac{\widetilde{a}}{\rho^{2}+\widetilde{a}^{2}}
=\frac{a\big(\rho^{2}+\widetilde{a}^{2}\big)-\widetilde{a}\big(\rho^{2}+a^{2}\big)}
{(\rho^{2}+a^{2})(\rho^{2}+\widetilde{a}^{2})}
=\frac{(a-\widetilde{a})\big(\rho^{2}-a\widetilde{a}\big)}
{(\rho^{2}+a^{2})(\rho^{2}+\widetilde{a}^{2})}
=\delta\,Q,
\]
where
\[
Q=Q(\alpha,\beta):=\frac{\rho^{2}-a\widetilde{a}}{(\rho^{2}+a^{2})(\rho^{2}+\widetilde{a}^{2})}.
\]
Adding and subtracting $\frac{a}{\rho^{2}+a^{2}}\psi_{s}(\rho)$ we thus obtain the identity
\begin{equation}\label{e:split}
K_{M}-K_{T}=-\delta\,Q\,\psi_{s}(\rho)+\frac{a}{\rho^{2}+a^{2}}\,\Theta(\rho)
+\frac{\widetilde{a}}{\rho^{2}+\widetilde{a}^{2}}\,\psi_{s}(\rho)\,\big[\psi_{s}(\alpha)-1\big],
\qquad \Theta:=\psi_{s}-\psi_{\sqrt{\epsilon}},
\end{equation}
and accordingly we decompose $(M-T_{3,(f,g)})h=G_{1}+G_{2}+G_{3}$ with
\[
\begin{aligned}
G_{1}(\alpha)&:=-\frac{1}{2\pi}\int_{\mathbb{R}}h(\beta)\,\delta(\alpha,\beta)\,Q(\alpha,\beta)\,\psi_{s}(\alpha-\beta)\,d\beta,\\
G_{2}(\alpha)&:=\frac{1}{2\pi}\int_{\mathbb{R}}h(\beta)\,\frac{a(\alpha,\beta)}{\rho^{2}+a(\alpha,\beta)^{2}}\,\Theta(\alpha-\beta)\,d\beta,\\
G_{3}(\alpha)&:=\big[\psi_{s}(\alpha)-1\big]\,\frac{1}{2\pi}\int_{\mathbb{R}}h(\beta)\,
\frac{\widetilde{a}(\alpha)}{\rho^{2}+\widetilde{a}(\alpha)^{2}}\,\psi_{s}(\alpha-\beta)\,d\beta.
\end{aligned}
\]
Since both cutoffs equal $1$ there, $\Theta\equiv0$ on $\{|\rho|\leqslant\frac{s}{2}\}$; also $\|\Theta\|_{L^{\infty}}\leqslant1$ and $\|\Theta'\|_{L^{\infty}}\leqslant\|\psi'\|_{L^{\infty}}(\frac1s+\frac{1}{\sqrt{\epsilon}})\lesssim\frac1s$, and, by \eqref{e:arange},
\begin{equation}\label{e:Qbd}
|Q|\leqslant\frac{\rho^{2}+(2+A)^{2}\epsilon^{2}}{(\rho^{2}+\epsilon^{2})^{2}}\leqslant\frac{C_{A}}{\rho^{2}+\epsilon^{2}}.
\end{equation}
For the lower bound of the denominator we distinguish two cases: if
$|\rho|\leqslant\epsilon$, then $a\geqslant\frac{\epsilon}{2}$ and $\widetilde{a}\geqslant\epsilon$ by
\eqref{e:arange}, so that
$(\rho^{2}+a^{2})(\rho^{2}+\widetilde{a}^{2})\geqslant\frac14(\rho^{2}+\epsilon^{2})^{2}$;
if $|\rho|\geqslant\epsilon$, then
$\rho^{2}+a^{2}\geqslant\rho^{2}\geqslant\frac12(\rho^{2}+\epsilon^{2})$ and likewise for
$\widetilde{a}$, so that the product is at least $\frac14(\rho^{2}+\epsilon^{2})^{2}$ in this case
as well; in both cases $|Q|\lesssim_{A}(\rho^{2}+\epsilon^{2})^{-1}$ follows from
$|a|,|\widetilde{a}|\leqslant(2+A)\epsilon$ and
$|\rho^{2}-a\widetilde{a}|\leqslant\rho^{2}+(2+A)^{2}\epsilon^{2}$.

\textit{Step 3 (the term $G_{3}$: it vanishes near the origin).} Since
$\int_{\mathbb{R}}\frac{\widetilde{a}}{\rho^{2}+\widetilde{a}^{2}}d\rho=\pi$, we have
\[
0\leqslant\frac{1}{2\pi}\int_{\mathbb{R}}\frac{\widetilde{a}(\alpha)}{\rho^{2}+\widetilde{a}(\alpha)^{2}}\psi_{s}(\rho)\,d\rho\leqslant\frac12,
\]
so $|G_{3}(\alpha)|\leqslant\frac12$, while $\psi_{s}(\alpha)-1=0$ for $|\alpha|\leqslant\frac s2$. Hence,
by the second estimate in \eqref{e:Ztwo}, and using $\frac{\epsilon^{2}}{\epsilon^{3}+|\alpha|^{3}}\leqslant\frac{\epsilon^{2}}{|\alpha|^{3}}$
for $|\alpha|\geqslant\frac s2$ (recall $\frac s2=\frac{\sqrt{\epsilon}}{128}\geqslant\epsilon$ for $\epsilon$ small),
\[
\|G_{3}\|_{Z}\lesssim\int_{\{|\alpha|\geqslant\frac s2\}}\frac{\epsilon^{2}}{|\alpha|^{3}}\,d\alpha
=\frac{4\epsilon^{2}}{s^{2}}=2^{14}\,\epsilon.
\]

\textit{Step 4 (the term $G_{2}$: the cutoff mismatch is slowly varying).} We claim that
$\|G_{2}\|_{C^{1}(I_{\epsilon})}\lesssim_{A}\log\frac1\epsilon$; the first estimate in
\eqref{e:Ztwo} then gives $\|G_{2}\|_{Z}\lesssim_{A}\epsilon\log\frac1\epsilon$. Write $F(\alpha,\beta):=\frac{a\,\Theta(\rho)}{\rho^{2}+a^{2}}$ with
$\rho=\alpha-\beta$, so that $G_{2}(\alpha)=\frac{1}{2\pi}\int_{\mathbb{R}}h(\beta)F(\alpha,\beta)d\beta$. For fixed
$\alpha\in I$ the function $\beta\mapsto F(\alpha,\beta)$ is smooth and compactly supported (it vanishes unless
$\frac s2\leqslant|\rho|\leqslant\sqrt{\epsilon}$), uniformly in $\alpha$; since $h$ is merely bounded, we keep the
derivative on the kernel, i.e., we differentiate under the integral sign. Holding $\beta$ fixed we have
$\partial_{\alpha}a=f'(\alpha)$ and $\partial_{\alpha}\rho=1$, whence
\[
\partial_{\alpha}G_{2}(\alpha)=\frac{1}{2\pi}\int_{\mathbb{R}}h(\beta)\Big[F_{a}\,f'(\alpha)+F_{\rho}\Big]d\beta,
\]
with
\[
F_{a}=\Theta(\rho)\frac{\rho^{2}-a^{2}}{(\rho^{2}+a^{2})^{2}},\qquad
F_{\rho}=a\Big[\frac{\Theta'(\rho)}{\rho^{2}+a^{2}}-\frac{2\rho\,\Theta(\rho)}{(\rho^{2}+a^{2})^{2}}\Big].
\]
Using $|f'(\alpha)|\leqslant A|\alpha|\leqslant A\sqrt{\epsilon}$, the elementary bound
$\frac{|\rho^{2}-a^{2}|}{(\rho^{2}+a^{2})^{2}}\leqslant\frac{1}{\rho^{2}+a^{2}}$, the inclusion
$\operatorname{supp}\Theta\subseteq\{|\rho|\geqslant\frac s2\}$, the bound $\|\Theta'\|_{L^{\infty}}\lesssim\frac1s$,
the estimate $|a|\leqslant(2+A)\epsilon$ from \eqref{e:arange}, and
\[
\int_{\{|\rho|\geqslant\frac s2\}}\frac{d\rho}{\rho^{2}}=\frac{4}{s},\qquad
\int_{\{|\rho|\geqslant\frac s2\}}\frac{2|\rho|}{(\rho^{2}+a^{2})^{2}}\,d\rho\leqslant\int_{\{|\rho|\geqslant\frac s2\}}\frac{2\,d\rho}{|\rho|^{3}}=\frac{8}{s^{2}},
\]
we obtain
\[
\begin{aligned}
\int_{\mathbb{R}}|F|\,d\beta&\leqslant\int_{\{|\rho|\geqslant\frac s2\}}\frac{|a|}{\rho^{2}+a^{2}}\,d\rho
\leqslant\int_{\{|\rho|\geqslant\frac s2\}}\frac{d\rho}{2|\rho|}
\leqslant\frac12\log\frac{2\sqrt{\epsilon}}{s}\lesssim\log\frac1\epsilon,\\
\int_{\mathbb{R}}|F_{a}\,f'(\alpha)|\,d\beta&\leqslant A\sqrt{\epsilon}\int_{\{|\rho|\geqslant\frac s2\}}\frac{d\rho}{\rho^{2}}\lesssim_{A}1,\\
\int_{\mathbb{R}}|F_{\rho}|\,d\beta&\leqslant\frac{4|a|\,\|\Theta'\|_{L^{\infty}}}{s}+\frac{8|a|}{s^{2}}
\lesssim_{A}\frac{\epsilon}{s^{2}}\lesssim1.
\end{aligned}
\]
(The first line uses $|a|/(\rho^{2}+a^{2})\leqslant1/(2|\rho|)$, which holds for every real $a$ and
every $\rho\ne0$; no sign condition on $a$ is needed. The last line uses
$s^{2}=2^{-12}\epsilon$, so that $\frac{\epsilon}{s^{2}}=2^{12}$.) Therefore
$\|G_{2}\|_{C^{1}(I_{\epsilon})}\lesssim_{A}\log\frac1\epsilon$, and the first estimate in
\eqref{e:Ztwo} gives $\|G_{2}\|_{Z}\lesssim_{A}\epsilon\log\frac1\epsilon$. 

\textit{Step 5 (the term $G_{1}$).} We split $\delta=\delta_{1}+\delta_{2}$ with
\[
\delta_{1}(\alpha,\beta):=-g'(\alpha)\rho,\qquad
\delta_{2}(\alpha,\beta):=g(\beta)-g(\alpha)+g'(\alpha)\rho=g(\alpha-\rho)-g(\alpha)+g'(\alpha)\rho,
\]
so that indeed $\delta_{1}+\delta_{2}=g(\beta)-g(\alpha)=\delta$, and we write $G_{1}=G_{1,1}+G_{1,2}$
accordingly. By Taylor's theorem and \eqref{e:deltabd},
\begin{equation}\label{e:delta12}
|\delta_{1}|\leqslant A|\alpha|\,|\rho|,\qquad
|\delta_{2}|\leqslant\tfrac{A}{2}\rho^{2},\qquad
|\partial_{\alpha}\delta_{2}|\leqslant2A|\rho|,\qquad
|\partial_{\rho}\delta_{2}|\leqslant A|\rho|;
\end{equation}
indeed $\partial_{\alpha}\delta_{2}=g'(\alpha-\rho)-g'(\alpha)+g''(\alpha)\rho$ has modulus at most
$A|\rho|+A|\rho|$ and $\partial_{\rho}\delta_{2}=g'(\alpha)-g'(\alpha-\rho)$ has modulus at most $A|\rho|$,
both by \eqref{e:deltabd} and $\|g''\|_{L^{\infty}}\leqslant A$. 

\textit{Step 5a ($G_{1,1}$: a pointwise bound suffices).} By \eqref{e:Qbd} and \eqref{e:delta12},
\[
|G_{1,1}(\alpha)|\lesssim_{A}|\alpha|\int_{\{|\rho|\leqslant s\}}\frac{|\rho|}{\rho^{2}+\epsilon^{2}}\,d\rho
=|\alpha|\log\Big(\frac{s^{2}+\epsilon^{2}}{\epsilon^{2}}\Big)\lesssim|\alpha|\log(\tfrac1\epsilon),
\]
since $\frac{s^{2}+\epsilon^{2}}{\epsilon^{2}}=\frac{1}{2^{12}\epsilon}+1\leqslant\frac{1}{\epsilon}$ for $\epsilon$ small.
Hence, by the second estimate in \eqref{e:Ztwo} together with
\[
\int_{\mathbb{R}}|\alpha|\,\frac{\epsilon^{2}}{\epsilon^{3}+|\alpha|^{3}}\,d\alpha
=2\epsilon\int_{0}^{\infty}\frac{t}{1+t^{3}}\,dt\lesssim\epsilon
\]
(substituting $|\alpha|=\epsilon t$), which yields $\|G_{1,1}\|_{Z}\lesssim\epsilon\log(\frac1\epsilon)$. 

\textit{Step 5b ($G_{1,2}$: bounded in $C^{1}$ up to a logarithm).} We claim
$\|G_{1,2}\|_{C^{1}(I_{\epsilon})}\lesssim_{A}\log(\frac1\epsilon)$, which by the first estimate in \eqref{e:Ztwo} gives
$\|G_{1,2}\|_{Z}\lesssim_{A}\epsilon\log(\frac1\epsilon)$. First, by \eqref{e:Qbd} and $|\delta_{2}|\leqslant\frac{A}{2}\rho^{2}$,
\[
|G_{1,2}(\alpha)|\lesssim_{A}\int_{\{|\rho|\leqslant s\}}\frac{\rho^{2}}{\rho^{2}+\epsilon^{2}}\,d\rho\leqslant2s\lesssim\sqrt{\epsilon}.
\]
For the derivative we again keep it on the kernel; since $\rho=\alpha-\beta$, differentiating with respect to
$\alpha$ at fixed $\beta$ amounts to applying $(\partial_{\alpha}+\partial_{\rho})$ to the integrand written in the
variables $(\alpha,\rho)$. Thus, with $W:=\delta_{2}\,Q\,\psi_{s}(\rho)$,
\begin{equation}\label{e:dW}
\int_{\mathbb{R}}|\partial_{\alpha}W|\,d\beta\leqslant\int_{\mathbb{R}}
\Big[\,\big|(\partial_{\alpha}+\partial_{\rho})\delta_{2}\big|\,|Q|
+|\delta_{2}|\big(|\partial_{\alpha}Q|+|\partial_{\rho}Q|\big)
+|\delta_{2}|\,|Q|\,|\psi_{s}'(\rho)|\,\Big]d\rho.
\end{equation}
To bound $\partial_{\tau}Q$ for $\tau\in\{\alpha,\rho\}$ we differentiate $Q=N/D$ with
$N=\rho^{2}-a\widetilde{a}$ and $D=(\rho^{2}+a^{2})(\rho^{2}+\widetilde{a}^{2})$. By \eqref{e:arange},
by $|\alpha|\leqslant\sqrt{\epsilon}$ and by the two-case argument used for \eqref{e:Qbd} we have
$D\geqslant\frac14(\rho^{2}+\epsilon^{2})^{2}$, $|N|\leqslant C_{A}(\rho^{2}+\epsilon^{2})$,
$|\partial_{\tau}a|\leqslant A(\sqrt{\epsilon}+|\rho|)$, $|\partial_{\alpha}\widetilde{a}|\leqslant2A\sqrt{\epsilon}$,
$\partial_{\rho}\widetilde{a}=0$, and therefore
\[
|\partial_{\tau}N|\leqslant2|\rho|+C_{A}\epsilon\big(\sqrt{\epsilon}+|\rho|\big),\qquad
|\partial_{\tau}D|\leqslant C_{A}\big(|\rho|+\epsilon(\sqrt{\epsilon}+|\rho|)\big)(\rho^{2}+\epsilon^{2}),
\]
whence, using $\epsilon|\rho|\leqslant|\rho|$,
\begin{equation}\label{e:dQ}
|\partial_{\tau}Q|\leqslant\frac{|\partial_{\tau}N|}{D}+\frac{|N|\,|\partial_{\tau}D|}{D^{2}}
\leqslant C_{A}\frac{|\rho|}{(\rho^{2}+\epsilon^{2})^{2}}+C_{A}\frac{\epsilon\sqrt{\epsilon}}{(\rho^{2}+\epsilon^{2})^{2}}.
\end{equation}
Substituting \eqref{e:Qbd}, \eqref{e:delta12} and \eqref{e:dQ} into \eqref{e:dW}, and using
\[
\begin{aligned}
&\int_{\{|\rho|\leqslant s\}}\frac{|\rho|}{\rho^{2}+\epsilon^{2}}\,d\rho\lesssim\log(\tfrac1\epsilon)
\quad(\text{as in Step 5a}),\qquad
\int_{\{|\rho|\leqslant\epsilon\}}\frac{|\rho|^{3}}{(\rho^{2}+\epsilon^{2})^{2}}\,d\rho\lesssim1,\\
&\int_{\epsilon}^{s}\frac{d\rho}{\rho}\leqslant\log(\tfrac1\epsilon),\qquad
\epsilon\sqrt{\epsilon}\int_{\mathbb{R}}\frac{\rho^{2}\,d\rho}{(\rho^{2}+\epsilon^{2})^{2}}
=\epsilon\sqrt{\epsilon}\cdot\frac{\pi}{2\epsilon}\lesssim\sqrt{\epsilon},
\end{aligned}
\]
together with
\[
\int_{\mathbb{R}}|\delta_{2}|\,|Q|\,|\psi_{s}'(\rho)|\,d\rho
\lesssim_{A}\frac{1}{s}\Big|\big\{\tfrac s2\leqslant|\rho|\leqslant s\big\}\Big|\lesssim_{A}1
\]
(for $|\delta_{2}||Q|\leqslant\frac{A}{2}\rho^{2}\cdot\frac{C_{A}}{\rho^{2}+\epsilon^{2}}\leqslant\frac{AC_{A}}{2}$ and
$|\psi_{s}'|\lesssim\frac1s$), we arrive at $\|\partial_{\alpha}G_{1,2}\|_{L^{\infty}(I_{\epsilon})}\lesssim_{A}\log(\frac1\epsilon)$,
which proves the claim. 

\textit{Step 6 (conclusion).} Adding the four bounds and using $\epsilon\lesssim\epsilon\log\frac1\epsilon$,
\[
\|(M-T_{3,(f,g)})h\|_{Z}\leqslant\|G_{1,1}\|_{Z}+\|G_{1,2}\|_{Z}+\|G_{2}\|_{Z}+\|G_{3}\|_{Z}
\lesssim_{A}\epsilon\log(\tfrac1\epsilon),
\]
under the normalization $\|h\|_{L^{\infty}(-2\sqrt{\epsilon},2\sqrt{\epsilon})}\leqslant1$; the general case follows
by linearity. By the symmetry $(f,g)\leftrightarrow(g,f)$ the same bound holds for
$\|(M-T_{3,(g,f)})h\|_{Z}$, and the second inequality in $(1)$ is proved.

Finally we prove $(2)$. By linearity we may assume $\|h\|_{C^{1}((-2\sqrt{\epsilon},2\sqrt{\epsilon}))}\leqslant1$;
by symmetry it suffices to prove that
\[
\Big|\int_{(-\sqrt{\epsilon},\sqrt{\epsilon})}(H+T_{4,(g,f)}-T_{2,f})h(\alpha)\phi(\alpha)\,d\alpha\Big|\lesssim \epsilon\log(\frac{1}{\epsilon}),
\]
we write 
\[
\begin{split}
&[-T_{4,(g,f)}+T_{2,f}]h(\alpha)\\
=&\frac{1}{2\pi}\int_{\mathbb{R}}\frac{(\alpha-\beta)\,h(\beta)\psi_{\sqrt{\epsilon}}(\alpha-\beta)}
{[(\alpha-\beta)^2+(f(\alpha)-f(\beta))^2]\cdot[(\alpha-\beta)^2+(2\epsilon+f(\beta)+g(\alpha))^2]}\\
&\qquad\times\big[(2\epsilon+f(\alpha)+g(\alpha))^2+2(2\epsilon+f(\alpha)+g(\alpha))(f(\beta)-f(\alpha))\big]\,d\beta\\
=&J^1(\alpha)+J^2(\alpha),
\end{split}
\]
where
\[
\begin{split}
J^1(\alpha):=&\frac{1}{2\pi}\int_{\mathbb{R}}\frac{(\alpha-\beta)\big[2(2\epsilon+f(\alpha)+g(\alpha))(f(\beta)-f(\alpha))\big]}{[(\alpha-\beta)^2+(f(\alpha)-f(\beta))^2]\cdot[(\alpha-\beta)^2+(2\epsilon+f(\beta)+g(\alpha))^2]} \\
&h(\beta)\psi_{\sqrt{\epsilon}}(\alpha-\beta)\,d\beta,\\
J^2(\alpha):=&\frac{1}{2\pi}\int_{\mathbb{R}}\frac{(\alpha-\beta)(2\epsilon+f(\alpha)+g(\alpha))^2}{[(\alpha-\beta)^2+(f(\alpha)-f(\beta))^2]\cdot[(\alpha-\beta)^2+(2\epsilon+f(\beta)+g(\alpha))^2]} \\
&h(\beta)\psi_{\sqrt{\epsilon}}(\alpha-\beta)\,d\beta.
\end{split}
\]
For $J^{1}(\alpha)$ we put $\rho:=\alpha-\beta$, $\Delta:=f(\beta)-f(\alpha)$ and
$N:=2\epsilon+f(\alpha)+g(\alpha)$, so that $N=\big(2\epsilon+f(\beta)+g(\alpha)\big)-\Delta$; since
$|\alpha|,|\beta|\leqslant2\sqrt{\epsilon}$ here, \eqref{e:N.range.1} gives
$\epsilon\leqslant N\lesssim_{A}\epsilon$. As
$\psi_{\sqrt{\epsilon}}(\rho)=0$ for $|\rho|\geqslant\sqrt{\epsilon}$,
\[
|J^{1}(\alpha)|\lesssim\int_{\{|\rho|\leqslant\sqrt{\epsilon}\}}
\frac{|\rho|\,\epsilon\,|\Delta|}{(\rho^{2}+\Delta^{2})\big(\rho^{2}+(N+\Delta)^{2}\big)}\,d\rho.
\]
We use the difference quotient, valid because $f'(0)=0$ and
$\|f''\|_{L^{\infty}}\leqslant A$,
\[ |\Delta|\leqslant A\big(|\alpha|+|\rho|\big)|\rho| . \]
The weaker bound $|\Delta|\lesssim\sqrt{\epsilon}|\rho|$, which follows from
$|f'(\alpha)|\leqslant A|\alpha|$ alone, would give only $|J^{1}|\lesssim|\alpha|+\sqrt{\epsilon}$.
Split the integral into
$\mathcal{R}_{\leqslant}:=\{|\rho|\leqslant\epsilon\}$ and
$\mathcal{R}_{\geqslant}:=\{\epsilon\leqslant|\rho|\leqslant\sqrt{\epsilon}\}$.
On $\mathcal{R}_{\leqslant}$ we have $|\Delta|\leqslant A(|\alpha|+\epsilon)\epsilon
\leqslant3A\epsilon^{3/2}\leqslant\frac{N}{2}$ for $\epsilon$ small, so that
$\rho^{2}+(N+\Delta)^{2}\gtrsim\rho^{2}+\epsilon^{2}$ and $\rho^{2}+\Delta^{2}\geqslant\rho^{2}$; hence
\[
\int_{\mathcal{R}_{\leqslant}}\lesssim_{A}\epsilon\int_{\{|\rho|\leqslant\epsilon\}}
\frac{|\alpha|+|\rho|}{\rho^{2}+\epsilon^{2}}\,d\rho\lesssim|\alpha|+\epsilon.
\]
On $\mathcal{R}_{\geqslant}$ we keep only $\rho^{2}+\Delta^{2}\geqslant\rho^{2}$ and
$\rho^{2}+(N+\Delta)^{2}\geqslant\rho^{2}$, which gives
\[
\int_{\mathcal{R}_{\geqslant}}\lesssim_{A}\epsilon\int_{\{\epsilon\leqslant|\rho|\leqslant\sqrt{\epsilon}\}}
\frac{|\alpha|+|\rho|}{\rho^{2}}\,d\rho\lesssim|\alpha|+\epsilon\log(\tfrac1\epsilon).
\]
Therefore $|J^{1}(\alpha)|\lesssim_{A}|\alpha|+\epsilon\log(\frac1\epsilon)$.
Then we set
\[
J^3(\alpha):=\frac{1}{2\pi}\int_{\mathbb{R}}\frac{(2\epsilon+f(\alpha)+g(\alpha))^2}{[1+f'(\alpha)^2]\cdot[(\alpha-\beta)^2+(f(\alpha)+g(\alpha)+2\epsilon)^2]}\frac{h(\beta)}{\alpha-\beta}\psi_{\sqrt{\epsilon}}(\alpha-\beta)\,d\beta.
\]

At this point we compare $J^{2}$ with the Hilbert-type kernel $J^{3}$. Write
\[
\rho:=\alpha-\beta,\qquad N:=2\epsilon+f(\alpha)+g(\alpha),\qquad \Delta:=f(\beta)-f(\alpha),
\]
so that $2\epsilon+f(\beta)+g(\alpha)=N+\Delta$ and $(f(\alpha)-f(\beta))^{2}=\Delta^{2}=\widetilde{f}(\alpha,\beta)^{2}\rho^{2}$, whence
\[
J^{2}(\alpha)=\frac{1}{2\pi}\int_{\mathbb{R}}\frac{N^{2}}{\rho}\,\frac{h(\beta)\psi_{\sqrt{\epsilon}}(\rho)}{D_{2}}\,d\beta,
\qquad D_{2}:=\big[1+\widetilde{f}(\alpha,\beta)^{2}\big]\big[\rho^{2}+(N+\Delta)^{2}\big],
\]
while, by its definition,
\[
J^{3}(\alpha)=\frac{1}{2\pi}\int_{\mathbb{R}}\frac{N^{2}}{\rho}\,\frac{h(\beta)\psi_{\sqrt{\epsilon}}(\rho)}{D_{3}}\,d\beta,
\qquad D_{3}:=\big[1+f'(\alpha)^{2}\big]\big[\rho^{2}+N^{2}\big].
\]
For any $\alpha,\beta\in(-2\sqrt{\epsilon},2\sqrt{\epsilon})$ we claim that
\[
\begin{aligned}
|D_{2}-D_{3}|&=\Big|\big(\widetilde{f}(\alpha,\beta)^{2}-f'(\alpha)^{2}\big)\big[\rho^{2}+(N+\Delta)^{2}\big]
+\big(1+f'(\alpha)^{2}\big)\big[\Delta^{2}+2N\Delta\big]\Big|\\
&\lesssim_{A}\epsilon|\rho|\,(|\alpha|+|\beta|).
\end{aligned}
\]
The identity itself follows from $D_{2}-D_{3}=(1+\widetilde{f}^{2})[\rho^{2}+(N+\Delta)^{2}]-(1+f'^{2})[\rho^{2}+N^{2}]$ by adding and subtracting $(1+\widetilde f^{2})[\rho^{2}+N^{2}]$. For the estimate, note first that
\[
\big|\widetilde{f}(\alpha,\beta)-f'(\alpha)\big|=\Big|\frac{1}{\rho}\int_{\beta}^{\alpha}\big(f'(s)-f'(\alpha)\big)ds\Big|\leqslant\tfrac{A}{2}|\rho|,
\qquad
\big|\widetilde{f}(\alpha,\beta)\big|+\big|f'(\alpha)\big|\lesssim_{A}|\alpha|+|\beta|,
\]
using $f'(0)=0$ and $\|f''\|_{L^{\infty}}\leqslant A$; hence $|\widetilde{f}^{2}-f'^{2}|\lesssim_{A}|\rho|(|\alpha|+|\beta|)$. Since $|\rho|\leqslant\sqrt{\epsilon}$ and since \eqref{e:N.range.0} and \eqref{e:a.range}, applied to the two mixed sums $2\epsilon+f(\alpha)+g(\alpha)$ and $2\epsilon+g(\alpha)+f(\beta)$ whose arguments all lie in $(-4\sqrt{\epsilon},4\sqrt{\epsilon})$, give $\epsilon\leqslant N\lesssim_{A}\epsilon$ and $|N+\Delta|\lesssim_{A}\epsilon$, we have $\rho^{2}+(N+\Delta)^{2}\lesssim\epsilon$ and $1+f'(\alpha)^{2}\lesssim1$. Also, $|\Delta|\leqslant\sup|f'|\,|\rho|\lesssim_{A}(|\alpha|+|\rho|)|\rho|$ and $|\rho|\leqslant|\alpha|+|\beta|$, so
\[
|2N\Delta|\lesssim_{A}\epsilon\,(|\alpha|+|\rho|)|\rho|\lesssim\epsilon|\rho|(|\alpha|+|\beta|),
\qquad
|\Delta|^{2}\lesssim_{A}\epsilon\,\rho^{2}\lesssim\epsilon|\rho|(|\alpha|+|\beta|),
\]
which proves the claim. Next, each of $D_{2}$ and $D_{3}$ satisfies $\gtrsim\rho^{2}+\epsilon^{2}$: the first factors are $\geqslant1$, while $|N|\geqslant\epsilon$ by the diagonal hypothesis of Appendix~\ref{sec:singular}, which holds for every $|\alpha|\leqslant2\sqrt{\epsilon}$ and is noted in \eqref{e:N.range.1}, and as for the second factor we argue by cases: if $|\rho|\leqslant\epsilon$, then $N+\Delta=a(\alpha,\beta)\geqslant\frac{\epsilon}{2}$ by \eqref{e:arange}, so that $\rho^{2}+(N+\Delta)^{2}\gtrsim\epsilon^{2}\gtrsim\rho^{2}+\epsilon^{2}$; if instead $|\rho|\geqslant\epsilon$, then $\rho^{2}+(N+\Delta)^{2}\geqslant\rho^{2}\geqslant\frac12(\rho^{2}+\epsilon^{2})$, no matter what the sign of $N+\Delta$ is. Combining this lower bound with $|N^{2}|\lesssim_{A}\epsilon^{2}$, with the estimate for $|D_{2}-D_{3}|$, and with $\|h\|_{L^{\infty}}\leqslant1$, we obtain
\[
\begin{aligned}
|J^{2}(\alpha)-J^{3}(\alpha)|
&\lesssim\int_{\mathbb{R}}\frac{\epsilon^{2}\cdot\epsilon|\rho|\,(|\alpha|+|\beta|)}{(\rho^{2}+\epsilon^{2})^{2}}\frac{|h(\beta)|}{|\rho|}\psi_{\sqrt{\epsilon}}(\rho)\,d\beta\\
&\lesssim_{A}\epsilon^{3}\int_{\{|\rho|\leqslant\sqrt{\epsilon}\}}\frac{|\alpha|+|\rho|}{(\rho^{2}+\epsilon^{2})^{2}}\,d\rho
\lesssim|\alpha|+\epsilon,
\end{aligned}
\]
where we used $|\beta|=|\alpha-\rho|\leqslant|\alpha|+|\rho|$ together with the elementary bounds
\[
\int_{\{|\rho|\leqslant\sqrt{\epsilon}\}}\frac{d\rho}{(\rho^{2}+\epsilon^{2})^{2}}\lesssim\epsilon^{-3},
\qquad
\int_{\{|\rho|\leqslant\sqrt{\epsilon}\}}\frac{|\rho|\,d\rho}{(\rho^{2}+\epsilon^{2})^{2}}\leqslant\int_{\mathbb{R}}\frac{|\rho|\,d\rho}{(\rho^{2}+\epsilon^{2})^{2}}=\frac{1}{\epsilon^{2}}.
\]
(The first of these follows by splitting at $|\rho|=\epsilon$: the inner part contributes at most $\epsilon^{-4}\cdot\epsilon=\epsilon^{-3}$ and the outer part at most $\int_{\epsilon}^{\sqrt{\epsilon}}\rho^{-4}d\rho\leqslant\frac13\epsilon^{-3}$.)
Thus, for any $\phi\in\mathcal{C}$,
\[
\begin{aligned}
&\int_{(-\sqrt{\epsilon},\sqrt{\epsilon})}\Big(|J^{1}(\alpha)|+|J^{2}(\alpha)-J^{3}(\alpha)|\Big)\,|\phi(\alpha)|\,d\alpha \\
\lesssim&\int_{(-\sqrt{\epsilon},\sqrt{\epsilon})}\Big(|\alpha|+\epsilon\log(\tfrac1\epsilon)\Big)\frac{\epsilon}{\epsilon^{2}+\alpha^{2}}\,d\alpha \\
&\lesssim\epsilon\log(\tfrac1\epsilon)+\epsilon\log(\tfrac1\epsilon) \\
&\lesssim\epsilon\log(\frac{1}{\epsilon}).
\end{aligned}
\]
Thus, it suffices to prove that, for any $\phi\in\mathcal{C}$ supported in $(-\sqrt{\epsilon},\sqrt{\epsilon})$, 
\[
\Big|\int_{\mathbb{R}}[Hh(\alpha)-J^3(\alpha)]\phi(\alpha)\,d\alpha\Big|\lesssim \epsilon\log(\frac{1}{\epsilon}).
\]
Write $\widetilde a(\alpha):=2\epsilon+f(\alpha)+g(\alpha)$ and put $\rho:=\alpha-\beta$ (so that
$\beta=\alpha-\rho$), and recall $s=\sqrt{\epsilon}/2^{6}$. Both kernels
\[
\frac{1}{\rho}\frac{\widetilde a(\alpha)^{2}}{\rho^{2}+\widetilde a(\alpha)^{2}}\psi_{s}(\alpha)\psi_{s}(\rho),
\qquad
\frac{1}{\rho}\frac{\widetilde a(\alpha)^{2}}{(1+f'(\alpha)^{2})\big(\rho^{2}+\widetilde a(\alpha)^{2}\big)}\psi_{\sqrt\epsilon}(\rho)
\]
occurring in $Hh$ and in $J^{3}$ are \emph{odd} functions of $\rho$: each is $1/\rho$ times a
factor which is even in $\rho$. The principal value therefore annihilates the contributions
containing $h(\alpha)$, because the truncation is symmetric about $\rho=0$ and the integral of an
odd function over a symmetric interval vanishes; hence
\[
\begin{aligned}
\int_{\mathbb{R}}[Hh-J^{3}](\alpha)\phi(\alpha)\,d\alpha
&=\frac{1}{2\pi}\iint_{\mathbb{R}^{2}}
\frac{h(\alpha-\rho)-h(\alpha)}{\rho}\,
\frac{\widetilde a(\alpha)^{2}}{\rho^{2}+\widetilde a(\alpha)^{2}}\\
&\qquad\times\Big[\psi_{s}(\alpha)\psi_{s}(\rho)-\frac{\psi_{\sqrt\epsilon}(\rho)}{1+f'(\alpha)^{2}}\Big]
\phi(\alpha)\,d\rho\,d\alpha.
\end{aligned}
\]
Here
$\big|\frac{h(\alpha-\rho)-h(\alpha)}{\rho}\big|\leqslant\|h'\|_{L^{\infty}}$ and
$\big|\psi_{s}(\alpha)\psi_{s}(\rho)-\frac{\psi_{\sqrt\epsilon}(\rho)}{1+f'(\alpha)^{2}}\big|\leqslant2$; moreover
$\operatorname{supp}\psi_{\sqrt\epsilon}\subseteq[-2\sqrt{\epsilon},2\sqrt{\epsilon}]$ while, by \eqref{e:arange},
\[
\int_{\mathbb{R}}\frac{\widetilde a(\alpha)^{2}}{\rho^{2}+\widetilde a(\alpha)^{2}}\,d\rho=\pi\widetilde a(\alpha)\leqslant\pi(2+A)\epsilon,
\]
and $\int_{\mathbb{R}}|\phi(\alpha)|\,d\alpha\leqslant\int_{\mathbb{R}}\frac{\epsilon}{\epsilon^{2}+\alpha^{2}}d\alpha=\pi$. Hence
\[
\Big|\int_{\mathbb{R}}[Hh-J^{3}](\alpha)\phi(\alpha)\,d\alpha\Big|
\lesssim \|h\|_{C^{1}}\,\epsilon,
\]
which is even stronger than the required bound.
\end{proof}
Using Lemma \ref{lem:elem} and Lemma \ref{lem:key}, we now prove the key proposition of this
section.
\begin{proposition}\label{prop:key}
Assume that $F^{+}$, $F^{-}$, $G$, $\omega^{+}$, $\omega^{-}$ are of class $C^{1}$ on $(-2\sqrt{\epsilon},2\sqrt{\epsilon})$, that $E^{+},E^{-},E\in L^{2}((-\delta,\delta))$, and that they satisfy the following estimates
\[
\begin{split}
&\|F^+\|_{C^1((-2\sqrt{\epsilon},2\sqrt{\epsilon}))}+\|F^-\|_{C^1((-2\sqrt{\epsilon},2\sqrt{\epsilon}))}+\|G\|_{C^1((-2\sqrt{\epsilon},2\sqrt{\epsilon}))} \\
&+\|\omega^+\|_{C^1((-2\sqrt{\epsilon},2\sqrt{\epsilon}))}+\|\omega^-\|_{C^1((-2\sqrt{\epsilon},2\sqrt{\epsilon}))}\lesssim 1,\\
&\|E^+\|_{Z}+\|E^-\|_Z+\sup_{\phi\in\mathcal{C},\,\mathrm{supp}\,\phi\subseteq (-2\sqrt{\epsilon},2\sqrt{\epsilon})}\Big|\int_{\mathbb{R}}E\cdot\phi\,d\alpha\Big|\lesssim\epsilon\log(\frac{1}{\epsilon}).
\end{split}
\]
Also, assume that the following identities hold in $(-\sqrt{\epsilon},\sqrt{\epsilon})$:
\[
\begin{split}
F^+&=T_{1,f}\omega^+-T_{3,(f,g)}\omega^--(1/2)\omega^++E^+,\\
F^-&=-T_{3,(g,f)}\omega^++T_{1,g}\omega^-+(1/2)\omega^-+E^-,\\
G&=(T_{4,(g,f)}-T_{2,f})\omega^++(T_{4,(f,g)}-T_{2,g})\omega^-+E.
\end{split}
\]
Then
\[
|G(0)|\lesssim\epsilon\log(\frac{1}{\epsilon}).
\]
\end{proposition}
\begin{proof}
Throughout the proof we regard $\omega^{\pm}$ as elements of $L^{2}((-\delta,\delta))$, extending
them $C^{1}$ and boundedly from $(-2\sqrt{\epsilon},2\sqrt{\epsilon})$ (recall
$2\sqrt{\epsilon}<\delta$ for $\epsilon$ small). Every kernel below carries
the factor $\psi_{\sqrt{\epsilon}}(\alpha-\beta)$ or $\psi_{\sqrt{\epsilon}/2^{6}}(\alpha-\beta)$,
so for $|\alpha|\leqslant\sqrt{\epsilon}$ the values of $T_{j}w(\alpha)$, $Hw(\alpha)$ and
$Mw(\alpha)$ depend only on the restriction of $w$ to $(-2\sqrt{\epsilon},2\sqrt{\epsilon})$,
and the extension can be chosen without increasing $\|w\|_{L^{\infty}}$.
Notice 
\[
\begin{aligned}
&\omega^{+}+2M\omega^{-} \\
=&2(M\omega^{-}-T_{3,(f,g)}\omega^{-})+\omega^{+}+2T_{3,(f,g)}\omega^{-} \\
=&2(M\omega^{-}-T_{3,(f,g)}\omega^{-})+2T_{1,f}\omega^{+}+2E^{+}-2F^{+},
\end{aligned}
\]
where the second line uses the identity for $F^{+}$ in the statement of Proposition
\ref{prop:key} in the form
$\omega^{+}+2T_{3,(f,g)}\omega^{-}=2T_{1,f}\omega^{+}+2E^{+}-2F^{+}$. Taking the $Z$-norm, we get
\[
\begin{aligned}
&\|\omega^{+}+2M\omega^{-}\|_{Z} \\
\lesssim&\|(M-T_{3,(f,g)})\omega^{-}\|_{Z}+\|T_{1,f}\omega^{+}\|_{Z}+\|E^{+}\|_{Z}+\epsilon\|F^{+}\|_{C^{1}} \\
\lesssim&\epsilon\log(\frac{1}{\epsilon}).
\end{aligned}
\]
Interchanging the roles of $\omega^{+}$ and $\omega^{-}$, we get
\[
\|\omega^{-}-2M\omega^{+}\|_{Z}\lesssim\epsilon\log(\frac{1}{\epsilon}).
\]
We will also use that $M$ is bounded on $Z$ up to an $\epsilon\log\frac1\epsilon$ error:
for every $g$ with $\|g\|_{L^{\infty}}\lesssim1$,
\begin{equation}\label{e:Mbdd}
\begin{aligned}
\|Mg\|_{Z}&=\sup_{\phi\in\mathcal{C}}\Big|\int_{\mathbb{R}}g\,M^{*}H^{*}\phi\,d\alpha\Big|\\
&\leqslant\sup_{\phi\in\mathcal{C}}\Big|\int_{\mathbb{R}}g\,H^{*}M^{*}\phi\,d\alpha\Big|
+\|g\|_{L^{\infty}}\sup_{\phi\in\mathcal{C}}\|(M^{*}H^{*}-H^{*}M^{*})\phi\|_{L^{1}}\\
&\lesssim\|g\|_{Z}+\epsilon\log(\tfrac1\epsilon),
\end{aligned}
\end{equation}
where we used $C_{M}^{-1}M^{*}\phi\in\mathcal{C}$ (Lemma \ref{lem:elem}) to bound the first
term by $C_{M}\|g\|_{Z}$, and the commutator lemma proved below to bound the second.
Then
\[
\begin{aligned}
&\|(I+4M^{2})\omega^{+}\|_{Z} \\
\leqslant&\|\omega^{+}+2M\omega^{-}\|_{Z}+\|-2M(-2M\omega^{+}+\omega^{-})\|_{Z} \\
\lesssim&\epsilon\log(\frac{1}{\epsilon}),
\end{aligned}
\]
because $(I+4M^{2})\omega^{+}=(\omega^{+}+2M\omega^{-})-2M(\omega^{-}-2M\omega^{+})$ and
$\|\omega^{-}-2M\omega^{+}\|_{Z}\lesssim\epsilon\log\frac1\epsilon$ by the previous relation.
The same argument with $\omega^{+}$ and $\omega^{-}$ interchanged gives
\[
\|(I+4M^{2})\omega^{-}\|_{Z}\lesssim\epsilon\log(\frac{1}{\epsilon}).
\]
As a result, for any $\phi\in\mathcal{C}$ we have 
\[
\Big|\int_{\mathbb{R}}(H\omega^++4HM^2\omega^+)\cdot\phi\,d\alpha\Big|+\Big|\int_{\mathbb{R}}(H\omega^-+4HM^2\omega^-)\cdot\phi\,d\alpha\Big|\lesssim\epsilon\log(\frac{1}{\epsilon}).
\]
Also,
\[
\begin{split}
&\Big|\int_{\mathbb{R}}(HM^2\omega^+-M^2H\omega^+)\cdot\phi\,d\alpha\Big| \\
\lesssim&\Big|\int_{\mathbb{R}}(HM-MH)M\omega^+\cdot\phi\,d\alpha\Big|+\Big|\int_{\mathbb{R}}M(HM-MH)\omega^+\cdot\phi\,d\alpha\Big|\\
\lesssim&\Big|\int_{\mathbb{R}}M\omega^+\cdot(M^\ast H^\ast-H^\ast M^\ast)\phi\,d\alpha\Big|+\Big|\int_{\mathbb{R}}\omega^+\cdot(M^\ast H^\ast-H^\ast M^\ast)M^\ast\phi\,d\alpha\Big|\\
\lesssim&\epsilon\log(\frac{1}{\epsilon}),
\end{split}
\]
where the last inequality follows from the commutator Lemma \ref{lem:comm} proved below. \\
The same argument with $\omega^{+}$ and $\omega^{-}$ interchanged gives
\[
\Big|\int_{\mathbb{R}}(HM^2\omega^--M^2H\omega^-)\cdot\phi\,d\alpha\Big|\lesssim\epsilon\log(\frac{1}{\epsilon}).
\]
Combining the last two relations, we obtain
\[
\begin{aligned}
&\Big|\int_{\mathbb{R}}H\omega^{+}[I+4(M^{*})^{2}]\phi\,d\alpha\Big|+\Big|\int_{\mathbb{R}}H\omega^{-}[I+4(M^{*})^{2}]\phi\,d\alpha\Big| \\
&\quad\leqslant\Big|\int_{\mathbb{R}}(H\omega^{+}+4HM^{2}\omega^{+})\phi+4(M^{2}H-HM^{2})\omega^{+}\,\phi\,d\alpha\Big| \\
&\qquad+\Big|\int_{\mathbb{R}}(H\omega^{-}+4HM^{2}\omega^{-})\phi+4(M^{2}H-HM^{2})\omega^{-}\,\phi\,d\alpha\Big| \\
&\quad\lesssim\epsilon\log(\frac{1}{\epsilon}),
\end{aligned}
\]
the equality in the first step being the identity
$H\omega^{\pm}[I+4(M^{\ast})^{2}]\phi=(H\omega^{\pm}+4HM^{2}\omega^{\pm})\phi
+4(M^{2}H-HM^{2})\omega^{\pm}\phi$, which is obtained by moving the adjoint
$(M^{\ast})^{2}$ onto $\omega^{\pm}$, and the last step combining the two
preceding bounds.
As a result, for any $\phi\in\mathcal{C}$ supported in $(-\frac{\sqrt{\epsilon}}{8},\frac{\sqrt{\epsilon}}{8})$,
\[
\begin{aligned}
&|\int_{\mathbb{R}}G[I+4(M^{*})^{2}]\phi| \\
\lesssim&\Big|\int_{\mathbb{R}}[(T_{4,(g,f)}-T_{2,f})\omega^++(T_{4,(f,g)}-T_{2,g})\omega^-]\cdot[I+4(M^\ast)^2]\phi\,d\alpha\Big| \\
&+\Big|\int_{\mathbb{R}}E\cdot[I+4(M^\ast)^2]\phi\,d\alpha\Big| \\
\lesssim&\epsilon\log(\frac{1}{\epsilon}).
\end{aligned}
\]
In the last step we applied Lemma \ref{lem:key}(2) with test function
$[I+4(M^\ast)^2]\phi$ and the assumption on $E$ with test function
$C^{-1}[I+4(M^\ast)^2]\phi$, where the factor $C$ depends only on $A$. This is
legitimate: by Lemma \ref{lem:elem}(2) applied twice, $C_{M}^{-1}M^{\ast}\phi\in\mathcal{C}$
and $C_{M}^{-2}(M^{\ast})^{2}\phi\in\mathcal{C}$, so $[I+4(M^\ast)^2]\phi$ belongs to
$\mathcal{C}$ up to multiplication by a constant depending only on $A$ (both estimates of
Lemma \ref{lem:key}(2), and the supremum in the assumption on $E$, are linear in the test
function), and its support is contained in
$(-\frac{\sqrt{\epsilon}}{8},\frac{\sqrt{\epsilon}}{8})\cup(-\frac{\sqrt{\epsilon}}{16},\frac{\sqrt{\epsilon}}{16})
=(-\frac{\sqrt{\epsilon}}{8},\frac{\sqrt{\epsilon}}{8})\subset(-\sqrt{\epsilon},\sqrt{\epsilon})$
by \eqref{e:M.support}. The $H$-part of Lemma \ref{lem:key}(2) is then absorbed by the
preceding relation: by the decomposition given above,
$\int_{\mathbb{R}}H\omega^{\pm}[I+4(M^\ast)^2]\phi\,d\alpha$ differs from
$\int_{\mathbb{R}}(H\omega^{\pm}+4HM^{2}\omega^{\pm})\phi\,d\alpha$ by
$4\int_{\mathbb{R}}(M^{2}H-HM^{2})\omega^{\pm}\phi\,d\alpha$, and this commutator term is
bounded by $\epsilon\log\frac1\epsilon$.
Since $G$ is $C^{1}$, we have $|G(\alpha)-G(0)|\lesssim|\alpha|$. For any $\phi\in\mathcal{C}$,
\[
\begin{aligned}
&|G(0)|\,\Big|\int_{\mathbb{R}}[1+4(M^{*})^{2}]\phi\,d\alpha\Big| \\
\lesssim&\Big|\int_{\mathbb{R}}G[I+4(M^{*})^{2}]\phi\,d\alpha\Big|+\int_{\mathbb{R}}|\alpha|\,\big|[I+4(M^{*})^{2}]\phi\big|\,d\alpha \\
&\lesssim\epsilon\log(\frac{1}{\epsilon})+\int_{-\frac{\sqrt{\epsilon}}{8}}^{\frac{\sqrt{\epsilon}}{8}}\frac{|\alpha|\,\epsilon}{\epsilon^{2}+\alpha^{2}}\,d\alpha \\
&\lesssim\epsilon\log(\frac{1}{\epsilon}).
\end{aligned}
\]
Now choose
\[
\phi_{0}:=(C\epsilon)^{-1}\psi_{\epsilon},\qquad
C:=\max\Big\{\sup_{\xi\in\mathbb{R}}|\psi(\xi)|(1+\xi^{2}),\,
\sup_{\xi\in\mathbb{R}}|\psi'(\xi)|(1+|\xi|^{3}),\,1\Big\},
\]
a finite quantity because $\psi$ is compactly supported. Then $\phi_{0}\in C_{0}^{1}((-\delta,\delta))$ (its support is $[-\epsilon,\epsilon]\subseteq(-\delta,\delta)$ for
$\epsilon$ small), and for every $\alpha$, writing $\xi:=\alpha/\epsilon$,
\[
\begin{aligned}
|\phi_{0}(\alpha)|&=\frac{|\psi(\xi)|}{C\epsilon}
\leqslant \frac{1}{\epsilon(1+\xi^{2})}=\frac{\epsilon}{\epsilon^{2}+\alpha^{2}},\\
|\phi_{0}'(\alpha)|&=\frac{|\psi'(\xi)|}{C\epsilon^{2}}
\leqslant \frac{1}{\epsilon^{2}(1+|\xi|^{3})}=\frac{\epsilon}{\epsilon^{3}+|\alpha|^{3}},
\end{aligned}
\]
so $\phi_{0}\in\mathcal{C}$. The kernel of $M^{\ast}$ is nonnegative (it is
$\frac{1}{2\pi}\frac{N(\alpha)}{(\alpha-\beta)^{2}+N(\alpha)^{2}}\psi_{s}(\alpha)\psi_{s}(\alpha-\beta)$ with
$N>0$), hence $(M^{\ast})^{2}\phi_{0}\geqslant0$ pointwise and
\[
\Big|\int_{\mathbb{R}}[I+4(M^\ast)^2]\phi_0\,d\alpha\Big|\geqslant\int_{\mathbb{R}}\phi_{0}\,d\alpha=\frac{1}{C}\int_{\mathbb{R}}\psi(\xi)\,d\xi\gtrsim 1.
\]
Dividing the previous relation by this quantity gives $|G(0)|\lesssim\epsilon\log\frac1\epsilon$, as claimed.
\end{proof}
Finally, we prove the commutator lemma used above.
\begin{lemma}\label{lem:comm}
If $\phi\in \mathcal{C}$ then
\[
\|(H^\ast M^\ast-M^\ast H^\ast)\phi\|_{L^1}\lesssim \epsilon\log(\frac{1}{\epsilon}).
\]
\end{lemma}
\begin{proof}
By definition, we can rewrite,
\[
(H^\ast M^\ast-M^\ast H^\ast)\phi(\alpha)=\frac{1}{4\pi^2}\int_{\mathbb{R}}K(\alpha,\beta)\phi(\beta)\,d\beta,
\]
where, writing $N(\theta):=2\epsilon+f(\theta)+g(\theta)$,
\[
\begin{split}
K(\alpha,\beta):&=\int_{\mathbb{R}}\frac{\psi_{\sqrt\epsilon/2^6}(\mu-\alpha)}{\mu-\alpha}\frac{N(\mu)^2\psi_{\sqrt\epsilon/2^6}(\mu)}{(\mu-\alpha)^2+N(\mu)^2}
\cdot\frac{N(\beta)\psi_{\sqrt\epsilon/2^6}(\beta)}{(\beta-\mu)^2+N(\beta)^2}\psi_{\sqrt\epsilon/2^6}(\beta-\mu)\,d\mu\\
&\quad-\int_{\mathbb{R}}\frac{N(\mu)\psi_{\sqrt\epsilon/2^6}(\mu)}{(\mu-\alpha)^2+N(\mu)^2}\psi_{\sqrt\epsilon/2^6}(\mu-\alpha)\cdot\frac{\psi_{\sqrt\epsilon/2^6}(\beta-\mu)}{\beta-\mu}\frac{N(\beta)^2\psi_{\sqrt\epsilon/2^6}(\beta)}{(\beta-\mu)^2+N(\beta)^2}\,d\mu.
\end{split}
\]
The first integral is the kernel of $H^\ast M^\ast$. Indeed the kernel of $H$ is
\[
\frac{1}{2\pi}\frac{1}{\alpha-\beta}\frac{N(\alpha)^2}{(\alpha-\beta)^2+N(\alpha)^2}\psi_s(\alpha)\psi_s(\alpha-\beta),
\]
that of $M$ is
\[
\frac{1}{2\pi}\frac{N(\alpha)}{(\alpha-\beta)^2+N(\alpha)^2}\psi_s(\alpha)\psi_s(\alpha-\beta),
\]
the kernel of $H^\ast$ is its transpose, and composing the two gives
\[
\begin{aligned}
&\frac{1}{4\pi^2}\int\frac{1}{\mu-\alpha}\frac{N(\mu)^2}{(\mu-\alpha)^2+N(\mu)^2}\psi_s(\mu)\psi_s(\mu-\alpha)\\
&\qquad\times\frac{N(\beta)}{(\beta-\mu)^2+N(\beta)^2}\psi_s(\beta)\psi_s(\beta-\mu)\,d\mu;
\end{aligned}
\]
the second integral is obtained in the same way from $M^\ast H^\ast$.

In the first integral we change variables $\mu=\alpha+\rho$, and in the second one
$\mu=\beta-\rho$; both produce the common factor
$\frac{\psi_{\sqrt\epsilon/2^6}(\rho)}{\rho}\psi_{\sqrt\epsilon/2^6}(\beta-\alpha-\rho)$, so that
\[
K(\alpha,\beta)=\int_{\mathbb{R}}\frac{\psi_{\sqrt\epsilon/2^6}(\rho)}{\rho}\psi_{\sqrt\epsilon/2^6}(\beta-\alpha-\rho)\widetilde{K}(\alpha,\beta;\rho)\,d\rho,
\]
where
\[
\begin{split}
\widetilde{K}(\alpha,\beta;\rho):&=\frac{(2\epsilon+f(\alpha+\rho)+g(\alpha+\rho))^2\psi_{\sqrt\epsilon/2^6}(\alpha+\rho)}{\rho^2+(2\epsilon+f(\alpha+\rho)+g(\alpha+\rho))^2}\\
&\quad\times\frac{(2\epsilon+f(\beta)+g(\beta))\psi_{\sqrt\epsilon/2^6}(\beta)}{(\beta-\alpha-\rho)^2+(2\epsilon+f(\beta)+g(\beta))^2}\\
&\quad-\frac{(2\epsilon+f(\beta)+g(\beta))^2\psi_{\sqrt\epsilon/2^6}(\beta)}{\rho^2+(2\epsilon+f(\beta)+g(\beta))^2}\\
&\qquad\times\frac{(2\epsilon+f(\beta-\rho)+g(\beta-\rho))\psi_{\sqrt\epsilon/2^6}(\beta-\rho)}{(\beta-\alpha-\rho)^2+(2\epsilon+f(\beta-\rho)+g(\beta-\rho))^2}.
\end{split}
\]
We claim that for any $\alpha,\beta\in\mathbb{R}$,
\[
|K(\alpha,\beta)|\lesssim \psi_{\sqrt\epsilon}(\beta)\psi_{\sqrt\epsilon}(\alpha)\frac{\epsilon(|\alpha|+|\beta|+\epsilon)}{(\beta-\alpha)^2+\epsilon^2}.
\]
Then for any $\phi\in\mathcal{C}$ we estimate
\[
\Big\|\int_{\mathbb{R}}\phi(\beta)K(\alpha,\beta)\,d\beta\Big\|_{L^1_\alpha}\lesssim \int_{\mathbb{R}}\int_{\mathbb{R}}\psi_{\sqrt\epsilon}(\beta)\psi_{\sqrt\epsilon}(\alpha)\frac{\epsilon(|\alpha|+|\beta|+\epsilon)}{(\beta-\alpha)^2+\epsilon^2}\cdot\frac{\epsilon}{\epsilon^2+\beta^2}\,d\beta d\alpha\lesssim \epsilon\log(\frac{1}{\epsilon}),
\]
as desired.\par\noindent
It remains to prove the claim. We write 
\[
\begin{split}
\widetilde{K}(\alpha,\beta;\rho)&=\widetilde{K}_1(\alpha,\beta;\rho)+\widetilde{K}_2(\alpha,\beta;\rho)+\widetilde{K}_3(\alpha,\beta;\rho),\\
\widetilde{K}_1(\alpha,\beta;\rho):&=\frac{(2\epsilon+f(\beta)+g(\beta))\psi_{\sqrt\epsilon/2^6}(\beta)}{(\beta-\alpha-\rho)^2+(2\epsilon+f(\beta)+g(\beta))^2} \\
&\Big[\frac{(2\epsilon+f(\alpha+\rho)+g(\alpha+\rho))^2\psi_{\sqrt\epsilon/2^6}(\alpha+\rho)}{\rho^2+(2\epsilon+f(\alpha+\rho)+g(\alpha+\rho))^2}-\frac{(2\epsilon+f(\alpha)+g(\alpha))^2\psi_{\sqrt\epsilon/2^6}(\alpha)}{\rho^2+(2\epsilon+f(\alpha)+g(\alpha))^2}\Big],\\
\widetilde{K}_2(\alpha,\beta;\rho):&=\frac{(2\epsilon+f(\beta)+g(\beta))\psi_{\sqrt\epsilon/2^6}(\beta)}{(\beta-\alpha-\rho)^2+(2\epsilon+f(\beta)+g(\beta))^2} \\
&\Big[\frac{(2\epsilon+f(\alpha)+g(\alpha))^2\psi_{\sqrt\epsilon/2^6}(\alpha)}{\rho^2+(2\epsilon+f(\alpha)+g(\alpha))^2}-\frac{(2\epsilon+f(\beta)+g(\beta))^2\psi_{\sqrt\epsilon/2^6}(\beta)}{\rho^2+(2\epsilon+f(\beta)+g(\beta))^2}\Big],\\
\widetilde{K}_3(\alpha,\beta;\rho):&=\frac{(2\epsilon+f(\beta)+g(\beta))^2\psi_{\sqrt\epsilon/2^6}(\beta)}{\rho^2+(2\epsilon+f(\beta)+g(\beta))^2} \\
&\quad\times\Big[\frac{(2\epsilon+f(\beta)+g(\beta))\psi_{\sqrt\epsilon/2^6}(\beta)}{(\beta-\alpha-\rho)^2+(2\epsilon+f(\beta)+g(\beta))^2}\\
&\qquad-\frac{(2\epsilon+f(\beta-\rho)+g(\beta-\rho))\psi_{\sqrt\epsilon/2^6}(\beta-\rho)}{(\beta-\alpha-\rho)^2+(2\epsilon+f(\beta-\rho)+g(\beta-\rho))^2}\Big].
\end{split}
\]
For $\widetilde{K}_1$ and $\widetilde{K}_3$, we have the following estimates, valid for $\rho,\alpha-\beta-\rho\in(-\sqrt{\epsilon},\sqrt{\epsilon})$:
\[
\begin{split}
&|\widetilde{K}_1(\alpha,\beta;\rho)|\lesssim \frac{\epsilon\psi_{\sqrt\epsilon/2^6}(\beta)}{(\beta-\alpha-\rho)^2+\epsilon^2}\cdot \frac{|\rho|\epsilon(|\alpha|+|\rho|+\epsilon)}{\rho^2+\epsilon^2},\\
&|\widetilde{K}_3(\alpha,\beta;\rho)|\lesssim \frac{\epsilon^2\psi_{\sqrt\epsilon/2^6}(\beta)}{\rho^2+\epsilon^2}\cdot \frac{|\rho|(|\beta|+|\rho|+\epsilon)}{(\beta-\alpha-\rho)^2+\epsilon^2}.
\end{split}
\]
Write $s:=\sqrt\epsilon/2^{6}$, $F(x,\rho):=\frac{N(x)^{2}\psi_{s}(x)}{\rho^{2}+N(x)^{2}}$ and
$O_{x}:=\frac{N(x)\psi_{s}(x)}{(\beta-\alpha-\rho)^{2}+N(x)^{2}}$, so that
$\widetilde{K}_{1}=O_{\beta}\,[F(\alpha+\rho,\rho)-F(\alpha,\rho)]$ with
$|O_{\beta}|\lesssim\frac{\epsilon\,\psi_{s}(\beta)}{(\beta-\alpha-\rho)^{2}+\epsilon^{2}}$ by \eqref{e:N.range}.
Since $F(\alpha+\rho,\rho)-F(\alpha,\rho)=\int_{0}^{\rho}\partial_{t}F(\alpha+t,\rho)\,dt$, it
suffices to bound $\partial_{x}F$; differentiating and using $|N|\leqslant C_{A}\epsilon$,
$|N'|\leqslant A|x|$ (valid on all of $(-\delta,\delta)$ since $f'(0)=g'(0)=0$) and
$|\psi'_{s}|\leqslant C_{\psi}/s$ gives
\[
|\partial_{x}F(x,\rho)|\lesssim_{A}\frac{\epsilon(|\alpha|+|\rho|)}{\rho^{2}+\epsilon^{2}}
+\frac{\epsilon^{2}}{s(\rho^{2}+\epsilon^{2})}
+\frac{\epsilon^{3}(|\alpha|+|\rho|)}{(\rho^{2}+\epsilon^{2})^{2}}
\lesssim\frac{\epsilon(|\alpha|+|\rho|+\epsilon)}{\rho^{2}+\epsilon^{2}}.
\]
Here the middle term is handled by noting that $\psi'_{s}(x)\ne0$ forces $|x|\geqslant\frac{s}{2}$,
whence $|\alpha|+|\rho|\gtrsim s$ and therefore
$\frac{\epsilon^{2}}{s}\lesssim\epsilon(|\alpha|+|\rho|+\epsilon)$ (a pure constant, since
$s^{2}=\frac{\epsilon}{2^{12}}$), while the last term is bounded by
$\frac{\epsilon(|\alpha|+|\rho|)}{\rho^{2}+\epsilon^{2}}\cdot\frac{\epsilon^{2}}{\rho^{2}+\epsilon^{2}}
\leqslant\frac{\epsilon(|\alpha|+|\rho|+\epsilon)}{\rho^{2}+\epsilon^{2}}$.
This proves the first estimate. The second one is identical:
$\widetilde{K}_{3}=\frac{N(\beta)^{2}\psi_{s}(\beta)}{\rho^{2}+N(\beta)^{2}}\,[O_{\beta}-O_{\beta-\rho}]$
with $\frac{N(\beta)^{2}\psi_{s}(\beta)}{\rho^{2}+N(\beta)^{2}}
\lesssim\frac{\epsilon^{2}\psi_{s}(\beta)}{\rho^{2}+\epsilon^{2}}$ and, by the very same
differentiation in $x$ with the denominator $(\beta-\alpha-\rho)^{2}+N(x)^{2}$ in place of
$\rho^{2}+N(x)^{2}$,
\[
|O_{\beta}-O_{\beta-\rho}|\leqslant|\beta-(\beta-\rho)|\cdot\sup_{x}|\partial_{x}O_{x}|
\lesssim\frac{|\rho|\,\epsilon(|\beta|+|\rho|+\epsilon)}{(\beta-\alpha-\rho)^{2}+\epsilon^{2}}.
\]
Therefore, letting
\[
K_j(\alpha,\beta):=\int_{\mathbb{R}}\frac{\psi_{\sqrt\epsilon/2^6}(\rho)}{\rho}\psi_{\sqrt\epsilon/2^6}(\beta-\alpha-\rho)\widetilde{K}_j(\alpha,\beta;\rho)\,d\rho,\qquad j\in\{1,2,3\},
\]
we have
\[
\begin{split}
|K_1(\alpha,\beta)|+|K_3(\alpha,\beta)|&\lesssim \psi_{\sqrt\epsilon}(\beta)\psi_{\sqrt\epsilon}(\alpha)\int_{\mathbb{R}}\frac{\epsilon\psi_{\sqrt\epsilon/2^6}(\rho)}{\rho^2+\epsilon^2}\frac{\epsilon\psi_{\sqrt\epsilon/2^6}(\beta-\alpha-\rho)}{(\beta-\alpha-\rho)^2+\epsilon^2}\\
&\qquad\qquad\cdot(|\alpha|+|\beta|+|\rho|+\epsilon)\,d\rho\\
&\lesssim \psi_{\sqrt\epsilon}(\beta)\psi_{\sqrt\epsilon}(\alpha)\frac{\epsilon(|\alpha|+|\beta|+\epsilon)}{(\beta-\alpha)^2+\epsilon^2}.
\end{split}
\]
The cutoffs $\psi_s(\rho)$ and $\psi_s(b-\rho)$ confine the integration to the overlap
$[-s,s]\cap[b-s,b+s]$, where $s:=\sqrt\epsilon/64$ and $b:=\beta-\alpha$, so that $|\rho|\leqslant s$ and
$|b-\rho|\leqslant s$. In addition, $\widetilde{K}_1$ vanishes unless $|\alpha|\leqslant2s$, because it contains one
of the factors $\psi_s(\alpha+\rho)$ or $\psi_s(\alpha)$ and $|\rho|\leqslant s$; likewise
$\widetilde{K}_3$ vanishes unless $|\beta|\leqslant2s$. As $2s=\sqrt\epsilon/32\leqslant \sqrt\epsilon/2$ and
$\psi\equiv1$ on $[-\frac12,\frac12]$, the function $\psi_{\sqrt\epsilon}$ equals $1$ on $[-2s,2s]$;
inserting the factors $\psi_{\sqrt\epsilon}(\alpha)$ and $\psi_{\sqrt\epsilon}(\beta)$ into the
bounds above therefore changes nothing on the support of the integrand.

For the second line write $\Lambda:=|\alpha|+|\beta|+\epsilon$, $v:=|b|/\epsilon$ and
$U:=s/\epsilon=\frac1{64}\epsilon^{-1/2}$ (so that $|v|\leqslant4U$, since $|\alpha|,|\beta|\leqslant2s$ and
$|b|\leqslant4s$). Since $0\leqslant \psi_s\leqslant1$, the change
of variables $u:=\rho/\epsilon$, which maps $[-s,s]\cap[b-s,b+s]$ onto
$\{|u|\leqslant U,\ |v-u|\leqslant U\}$, bounds the integral above by $\frac{\Lambda}{\epsilon}\,B_{0}(v)+B_{1}(v)$,
where
\[
\begin{gathered}
B_{0}(v):=\int_{|u|\leqslant U,\ |v-u|\leqslant U}\frac{du}{(1+u^{2})(1+(v-u)^{2})},\\[1mm]
B_{1}(v):=\int_{|u|\leqslant U,\ |v-u|\leqslant U}\frac{|u|\,du}{(1+u^{2})(1+(v-u)^{2})}.
\end{gathered}
\]
Write $D(u,v):=(1+u^{2})(1+(v-u)^{2})$. For all $u,v\in\mathbb{R}$,
\[
\frac{1}{D(u,v)}\leqslant \frac{4}{1+v^{2}}\Big(\frac{1}{1+u^{2}}+\frac{1}{1+(v-u)^{2}}\Big),
\tag{$\diamondsuit$}
\]
because $v^{2}=(u+(v-u))^{2}\leqslant2u^{2}+2(v-u)^{2}$ and hence
$1+v^{2}\leqslant4(1+u^{2})+4(1+(v-u)^{2})$. Integrating $(\diamondsuit)$ over the region
$\{|u|\leqslant U,\ |v-u|\leqslant U\}$ and using $\int_{\mathbb{R}}\frac{dt}{1+t^{2}}=\pi$, we obtain
\[
B_{0}(v)\leqslant \frac{8\pi}{1+v^{2}}.
\]
For $B_{1}$ we first drop the cutoffs (this only enlarges the domain of integration, so
that the resulting integral dominates the one above) and then use the substitution
$u\mapsto v-u$, which shows that $B_{1}$ is even in $v$ and rewrites it as
$\frac12\int_{\mathbb{R}}\frac{|u|+|v-u|}{D(u,v)}\,du=\int_{\mathbb{R}}\frac{|u|}{D(u,v)}\,du$,
so as to reduce to the case $v\geqslant0$. Splitting
$\mathbb{R}=(-\infty,0)\cup(0,v)\cup(v,\infty)$ and using $(v+t)^{2}\geqslant v^{2}$ on $[0,v]$ together
with $(v+t)^{2}\geqslant t^{2}$ on $[v,\infty)$, we get
\[
\begin{aligned}
\int_{-\infty}^{0}\frac{|u|\,du}{D(u,v)}
&\leqslant \frac{\log(1+v^{2})+1}{2(1+v^{2})},\\
\int_{0}^{v}\frac{|u|\,du}{D(u,v)}&\leqslant \frac{2\pi v}{4+v^{2}},\\
\int_{v}^{\infty}\frac{|u|\,du}{D(u,v)}
&\leqslant \frac{\pi v}{2(1+v^{2})}+\frac{\log(1+v^{2})+1}{2(1+v^{2})}.
\end{aligned}
\]
After the substitution $t=-u$ the first integral becomes
$\int_{0}^{\infty}\frac{t\,dt}{(1+t^{2})(1+(v+t)^{2})}$, which is at most
$\int_{0}^{v}\frac{t\,dt}{(1+t^{2})(1+v^{2})}+\int_{v}^{\infty}\frac{t\,dt}{(1+t^{2})^{2}}$.
The second integral equals $\frac v2\int_{0}^{v}\frac{du}{D(u,v)}$ (average the two forms obtained
from each other by $u\mapsto v-u$), and the latter is at most $\frac{4\pi}{4+v^{2}}$: write it as
twice its restriction to $[0,v/2]$, where $1+(v-u)^{2}\geqslant1+v^{2}/4$, and use
$\int_{0}^{v/2}\frac{du}{1+u^{2}}\leqslant \frac{\pi}{2}$. The third integral follows by
putting $t=u-v$, upon which the $t$-part is bounded exactly as the first integral, while the
$v$-part is bounded by $\frac{v}{1+v^{2}}\int_{0}^{\infty}\frac{dt}{1+t^{2}}$.
Each of the three quantities above is bounded by a constant if $0\leqslant v\leqslant1$, and by
$\frac{C}{v}\leqslant \frac{C(1+|v|)}{1+v^{2}}$ if $v\geqslant1$; consequently
\[
B_{1}(v)\lesssim\frac{1+|v|}{1+v^{2}}
=\frac{\epsilon(1+|v|)}{\epsilon(1+v^{2})}
\leqslant \frac{\Lambda}{\epsilon(1+v^{2})},
\]
where the last inequality uses $\epsilon(1+|v|)=\epsilon+|b|\leqslant|\alpha|+|\beta|+\epsilon=\Lambda$.
Combining the two estimates,
\[
\frac{\Lambda}{\epsilon}\,B_{0}(v)+B_{1}(v)
\lesssim\frac{\Lambda}{\epsilon(1+v^{2})}
=\frac{\epsilon\,\Lambda}{b^{2}+\epsilon^{2}},
\]
which proves the second line of the claimed estimate.
Thus, it remains to prove
\[
|K_2(\alpha,\beta)|\lesssim \psi_{\sqrt\epsilon}(\beta)\psi_{\sqrt\epsilon}(\alpha)\frac{\epsilon(|\alpha|+|\beta|+\epsilon)}{(\beta-\alpha)^2+\epsilon^2}.
\]
For this we decompose 
\[
K_2(\alpha,\beta)=K_2^1(\alpha,\beta)+K_2^2(\alpha,\beta),
\]
where
\[
\begin{split}
&K_2^1(\alpha,\beta):=\int_{\mathbb{R}}\frac{\psi_{\kappa}(\rho)}{\rho}\psi_{\sqrt\epsilon/2^6}(\beta-\alpha-\rho)\widetilde{K}_2(\alpha,\beta;\rho)\,d\rho,\\
&K_2^2(\alpha,\beta):=\int_{\mathbb{R}}\frac{\psi_{\sqrt\epsilon/2^6}(\rho)-\psi_{\kappa}(\rho)}{\rho}\psi_{\sqrt\epsilon/2^6}(\beta-\alpha-\rho)\widetilde{K}_2(\alpha,\beta;\rho)\,d\rho,
\end{split}
\]
and carefully select $\kappa:=\frac{|\alpha-\beta|+\epsilon}{10}$. 

For $\widetilde{K}_{2}(\alpha,\beta;\rho)$ we can estimate, for $\rho,\alpha-\beta-\rho\in(-\sqrt{\epsilon},\sqrt{\epsilon})$:
\[
|\widetilde{K}_2(\alpha,\beta;\rho)|\lesssim \frac{\epsilon\psi_{\sqrt\epsilon/2^6}(\beta)}{(\beta-\alpha-\rho)^2+\epsilon^2}\cdot \frac{|\alpha-\beta|\epsilon(|\alpha|+|\beta|+\epsilon)}{\rho^2+\epsilon^2}
\]
Writing $\widetilde{K}_{2}=O_{\beta}\,[F(\alpha,\rho)-F(\beta,\rho)]$ with $F,O_{\beta}$ as above,
The denominator $\rho^{2}+N(x)^{2}$ of the bracket depends on the variable
$x$, since $N(x)=2\epsilon+f(x)+g(x)$, so comparing the two numerators alone does not suffice.
Differentiating,
\[
\partial_{x}F(x,\rho)=\frac{2N(x)N'(x)\rho^{2}}{(\rho^{2}+N(x)^{2})^{2}}\psi_{s}(x)
+\frac{N(x)^{2}}{\rho^{2}+N(x)^{2}}\psi'_{s}(x),
\]
and the mean value theorem in the first variable gives
$F(\alpha,\rho)-F(\beta,\rho)=\int_{\beta}^{\alpha}\partial_{x}F(x,\rho)\,dx$. For $x$ between
$\alpha$ and $\beta$ we have $|x|\leqslant|\alpha|+|\beta|$, so that
$\epsilon\leqslant N\leqslant C_{A}\epsilon$ by \eqref{e:N.range},
$|N'|\leqslant A|x|\leqslant A(|\alpha|+|\beta|)$ and
$|\psi'_{s}|\leqslant C_{\psi}/s$ yield
\[
|\partial_{x}F(x,\rho)|\lesssim_{A}\frac{\epsilon(|\alpha|+|\beta|)}{\rho^{2}+\epsilon^{2}}
+\frac{\epsilon^{2}}{s(\rho^{2}+\epsilon^{2})}
\lesssim\frac{\epsilon(|\alpha|+|\beta|+\epsilon)}{\rho^{2}+\epsilon^{2}},
\]
the middle term being handled as before: $\psi'_{s}(x)\ne0$ forces $|x|\geqslant\frac{s}{2}$,
hence $|\alpha|+|\beta|\geqslant\frac{s}{2}$, and therefore
$\frac{\epsilon^{2}}{s}\lesssim\epsilon(|\alpha|+|\beta|)$ because $s^{2}=2^{-12}\epsilon$.
Hence
\[
|F(\alpha,\rho)-F(\beta,\rho)|\leqslant|\alpha-\beta|\sup_{x}|\partial_{x}F(x,\rho)|
\lesssim\frac{|\alpha-\beta|\,\epsilon(|\alpha|+|\beta|+\epsilon)}{\rho^{2}+\epsilon^{2}}.
\]
No factor $|\rho|$ appears here, in contrast with the corresponding step for $\widetilde{K}_{1}$,
because the mean value is taken over the interval between $\alpha$ and $\beta$ rather than over
$[\alpha,\alpha+\rho]$. Combined with the bound on $O_{\beta}$ this gives the announced estimate.

Hence,
\[
\begin{split}
|K_2^2(\alpha,\beta)|\lesssim& \int_{\mathbb{R}}\frac{|\psi_{\sqrt\epsilon/2^6}(\rho)-\psi_{\kappa}(\rho)|}{|\rho|}\psi_{\sqrt\epsilon/2^6}(\beta-\alpha-\rho)\frac{\epsilon^2\psi_{\sqrt\epsilon/2^6}(\beta)}{(\beta-\alpha-\rho)^2+\epsilon^2} \\
&\cdot \frac{|\alpha-\beta|(|\alpha|+|\beta|+\epsilon)}{\rho^2+\epsilon^2}\,d\rho\\
\lesssim&\epsilon\psi_{\sqrt\epsilon}(\beta)\psi_{\sqrt\epsilon}(\alpha)(|\alpha|+|\beta|+\epsilon) \\
&\int_{\mathbb{R}}\frac{|\psi_{\sqrt\epsilon/2^6}(\rho)-\psi_{\kappa}(\rho)|}{|\rho|}\frac{\epsilon\psi_{\sqrt\epsilon/2^6}(\beta-\alpha-\rho)}{(\beta-\alpha-\rho)^2+\epsilon^2}\cdot \frac{|\alpha-\beta|}{\rho^2+\epsilon^2}\,d\rho\\
\lesssim&\frac{\epsilon\psi_{\sqrt\epsilon}(\beta)\psi_{\sqrt\epsilon}(\alpha)(|\alpha|+|\beta|+\epsilon)}{(\beta-\alpha)^2+\epsilon^2}.
\end{split}
\]
In the last step we used the following. Both cutoffs equal $1$ on
$[-\frac12\min\{s,\kappa\},\frac12\min\{s,\kappa\}]$, so the difference
$\psi_{s}(\rho)-\psi_{\kappa}(\rho)$ vanishes there, and since $0\leqslant\psi\leqslant1$ its
absolute value is at most $1$ everywhere; on the support of the difference, therefore,
$|\rho|>m:=\frac12\min\{s,\kappa\}$, and $\psi_{s}(\beta-\alpha-\rho)$ forces
$|\beta-\alpha-\rho|\leqslant s$. Writing $b=\beta-\alpha$, we bound
$\frac1{|\rho|}\leqslant \frac1m$, $\frac{|b|}{\rho^{2}+\epsilon^{2}}\leqslant \frac{|b|}{m^{2}}$ and
$\int_{\mathbb{R}}\frac{\epsilon\,d\rho}{(b-\rho)^{2}+\epsilon^{2}}=\pi$, so that the integral is at
most $\frac{\pi|b|}{m^{3}}\lesssim\frac{|b|}{\min\{s,\kappa\}^{3}}$, and this is in turn
$\lesssim\frac{1}{b^{2}+\epsilon^{2}}$: indeed, if $\kappa\leqslant s$ then
$\frac{|b|}{\min^{3}}=\frac{1000|b|}{(|b|+\epsilon)^{3}}$ and
$\frac{t}{(t+1)^{3}}\lesssim\frac{1}{1+t^{2}}$ for $t\geqslant0$, while if $\kappa>s$ the support
condition also gives $|\rho|\leqslant\kappa$, hence
$|b|\leqslant s+\kappa=\frac{|b|+\epsilon}{10}+s$, which forces $|b|\lesssim s$ and the same bound
follows.
Finally we bound $|K_{2}^{1}|$; first we write
\[
K_{2}^{1}(\alpha,\beta)=\int_{\mathbb{R}}\frac{\psi_{\kappa}(\rho)}{\rho}\big[\psi_{\sqrt\epsilon/2^6}(\beta-\alpha-\rho)\widetilde{K}_2(\alpha,\beta;\rho)-\psi_{\sqrt\epsilon/2^6}(\beta-\alpha)\widetilde{K}'_2(\alpha,\beta;\rho)\big]\,d\rho,
\]
where 
\[
\begin{aligned}
\widetilde{K}'_2(\alpha,\beta;\rho):=&\frac{(2\epsilon+f(\beta)+g(\beta))\psi_{\sqrt\epsilon/2^6}(\beta)}{(\beta-\alpha)^2+(2\epsilon+f(\beta)+g(\beta))^2} \\
&\Big[\frac{(2\epsilon+f(\alpha)+g(\alpha))^2\psi_{\sqrt\epsilon/2^6}(\alpha)}{\rho^2+(2\epsilon+f(\alpha)+g(\alpha))^2}-\frac{(2\epsilon+f(\beta)+g(\beta))^2\psi_{\sqrt\epsilon/2^6}(\beta)}{\rho^2+(2\epsilon+f(\beta)+g(\beta))^2}\Big].
\end{aligned}
\]
Then 
\[
\begin{aligned}
&|\psi_{\frac{\sqrt{\epsilon}}{2^{6}}}(\beta-\alpha-\rho)\widetilde{K}_{2}(\alpha,\beta;\rho)-\psi_{\frac{\sqrt{\epsilon}}{2^{6}}}(\beta-\alpha)\widetilde{K'}_{2}(\alpha,\beta;\rho)| \\
\lesssim&|\psi_{\frac{\sqrt{\epsilon}}{2^{6}}}(\beta-\alpha-\rho)\frac{(2\epsilon+f(\beta)+g(\beta))\psi_{\frac{\sqrt{\epsilon}}{2^{6}}}(\beta)}{(\beta-\alpha-\rho)^{2}+(2\epsilon+f(\beta)+g(\beta))^{2}}- \\
&\psi_{\frac{\sqrt{\epsilon}}{2^{6}}}(\beta-\alpha)\frac{(2\epsilon+f(\beta)+g(\beta))\psi_{\frac{\sqrt{\epsilon}}{2^{6}}}(\beta)}{(\beta-\alpha)^{2}+(2\epsilon+f(\beta)+g(\beta))^{2}}| \\
&\times|\frac{(2\epsilon+f(\alpha)+g(\alpha))^{2}\psi_{\frac{\sqrt{\epsilon}}{2^{6}}}(\alpha)}{(\rho)^{2}+(2\epsilon+f(\alpha)+g(\alpha))^{2}}-\frac{(2\epsilon+f(\beta)+g(\beta))^{2}\psi_{\frac{\sqrt{\epsilon}}{2^{6}}}(\beta)}{(\rho)^{2}+(2\epsilon+f(\beta)+g(\beta))^{2}}| \\
\lesssim&\frac{\epsilon|\rho|(|\beta-\alpha|+\epsilon)}{[(\beta-\alpha)^2+\epsilon^2]^2}\cdot \frac{|\alpha-\beta|\epsilon(|\alpha|+|\beta|+\epsilon)}{\rho^2+\epsilon^2}.
\end{aligned}
\]
Here the first factor bounds the difference of the two
$\psi_s N/((\cdot)^{2}+N^{2})$ factors: separating the increment of $\psi_s$ from that of
$((\cdot)^{2}+N^{2})^{-1}$, and using $|b-\rho|\geqslant \frac45|b|$ (which follows from
$|\rho|\leqslant \kappa=\frac1{10}(|b|+\epsilon)$) when $|b|\geqslant \epsilon$, gives the stated bound;
when $|b|\leqslant \epsilon$ one has $|b|,|b-\rho|\leqslant|b|+\kappa\leqslant\frac65\epsilon\leqslant\frac{s}{2}$
for $\epsilon$ small, so that $\psi_{s}$ equals $1$ at both points $b$ and $b-\rho$ and the
increment of $\psi_{s}$ vanishes; every other denominator is then comparable to $\epsilon^{2}$
and the same bound follows in the same way.
Combining this with the definition of $K_2^1$ we get
\[
|K_2^1(\alpha,\beta)|\lesssim \frac{\epsilon^{2}\,|b|\,\Lambda\,(|b|+\epsilon)}{[b^{2}+\epsilon^{2}]^{2}}\int_{\mathbb{R}}\frac{\psi_{\kappa}(\rho)}{\rho^{2}+\epsilon^{2}}\,d\rho
\leqslant \frac{\epsilon^{2}\,|b|\,\Lambda\,(|b|+\epsilon)}{[b^{2}+\epsilon^{2}]^{2}}\cdot\frac{2}{\epsilon}\arctan\frac{\kappa}{\epsilon}
\lesssim \frac{\epsilon\Lambda}{b^{2}+\epsilon^{2}},
\]
where the middle inequality uses $0\leqslant \psi_\kappa\leqslant1$ together with
$\operatorname{supp}\psi_\kappa\subseteq[-\kappa,\kappa]$, and the last one the identity
$\kappa=\frac1{10}(|b|+\epsilon)$ and the uniform estimate
\[
\frac{2\,|b|\,(|b|+\epsilon)}{b^{2}+\epsilon^{2}}\arctan\frac{\kappa}{\epsilon}
=\frac{2t(t+1)}{1+t^{2}}\arctan\frac{t+1}{10}
\leqslant3\arctan\frac{t+1}{10}\leqslant \frac{3\pi}{2},
\qquad t:=\frac{|b|}{\epsilon}\geqslant0,
\]
which uses $\frac{2t(t+1)}{1+t^{2}}\leqslant3$ for $t\geqslant0$.
\end{proof}

\smallskip
\noindent\textit{Acknowledgment.}
C. Hao was partially supported by the National Natural Science Foundation of China under Grant No. 12171460, the CAS Project for Young Scientists in Basic Research under Grant No. YSBR-031 and the National Key R\&D Program of China under Grant No. 2021YFA1000800.

\end{document}